\documentclass[10pt]{article}

\def\VersionPath{versions/revision-1}
\makeatletter
\def\input@path{{\VersionPath}}
\makeatother

\usepackage{definitions}
\usepackage{float}
\usepackage{graphicx}
\usepackage{wrapfig}
\usepackage{caption}
\usepackage{subcaption}
\usepackage{upgreek}

\graphicspath{{\VersionPath}}

\usepackage{amsmath,amssymb,amsthm,amsfonts,mathtools,braket,bm,bbm}
\usepackage[inline]{enumitem}
\usepackage{mymacros}
\usepackage{graphicx}
\usepackage[backref=false,natbib=true,maxbibnames=99,doi=false,url=false,giveninits=true,style=numeric,maxcitenames=2,mincitenames=1,backend=bibtex,
uniquelist=false]{biblatex}
\usepackage{hyperref}
\usepackage{comment}
\usepackage[title]{appendix}

\AtBeginEnvironment{thebibliography}{\small}
\setlist[enumerate]{wide=0pt}

\theoremstyle{definition} \newtheorem{theorem}{Theorem}

\theoremstyle{remark} \newtheorem{remark}{Remark}

\makeatletter
\newcommand{\neutralize}[1]{\expandafter\let\csname c@#1\endcsname\count@}
\makeatother

\theoremstyle{theorem}
\newenvironment{assprime}[1]
  {\renewcommand{\theass}{#1}\addtocounter{ass}{-1}\begin{ass}}
  {\end{ass}}

\excludecomment{notfinished}

\newcommand{\inrevise}[1]{{\color{black}#1}}
\newenvironment{bulkrevise}{\begingroup\color{black}}{\endgroup}

\newcommand{\kmax}{k_{\text{max}}}
\newcommand{\rhomax}{\ensuremath{\rho_{\text{max}}}}

\newcommand{\taumin}{\ensuremath{\tau_{\text{min}}}}
\newcommand{\ximin}{\ensuremath{\xi_{\text{min}}}}
\newcommand{\cAmin}{\ensuremath{\mathcal{A}^{\text{min}}}}
\newcommand{\cAmax}{\ensuremath{\mathcal{A}^{\text{max}}}}
\newcommand{\tauzero}{\ensuremath{\bar{\mathcal{\tau}}_0}}
\newcommand{\nonE}{\ensuremath{\mathcal{E}_{\tau\xi}^{\prime}}}
\newcommand{\natF}{\ensuremath{\mathcal{G}}}

\newcommand{\supp}[1]{\mathop{\textbf{supp}}(#1)}
\newcommand{\bdsemicolon}{\;\bm{;}\;}

 \renewcommand{\top}{T}
\allowdisplaybreaks

\begin{document}

\title{Sequential Quadratic Optimization for Solving Expectation Equality Constrained Stochastic Optimization Problems}

\author{Haoming Shen\footnote{Department of Industrial Engineering, University of Arkansas (haomings@uark.edu).} \and Yang Zeng\footnote{School of Computing and Augmented Intelligence, Arizona State University (yzeng87@asu.edu).} \and Baoyu Zhou\footnote{School of Computing and Augmented Intelligence, Arizona State University (baoyu.zhou@asu.edu).}}
\date{}
\maketitle

\begin{abstract}
  A sequential quadratic programming method is designed for solving general smooth nonlinear stochastic optimization problems subject to expectation equality constraints. We consider the setting where the objective and constraint function values, as well as their derivatives, are not directly available. The algorithm applies an adaptive step size policy and only relies on objective gradient estimates, constraint function estimates, and constraint derivative estimates to update iterates. Both asymptotic and non-asymptotic convergence properties of the algorithm are analyzed. Under reasonable assumptions, the algorithm generates a sequence of iterates whose first-order stationary measure diminishes in expectation. In addition, we identify the iteration and sample complexity for obtaining a first-order $\varepsilon$-stationary iterate in expectation. The results of numerical experiments demonstrate the efficiency and efficacy of our proposed algorithm compared to a penalty method and an augmented Lagrangian method.
 \end{abstract}

\section{Introduction}
In this paper, we design, analyze, and implement a sequential quadratic programming (SQP) method for minimizing a general smooth nonlinear function, which is written in an expectation format, subject to expectation equality constraints. Such optimization problems arise in various science and engineering applications, including but not limited to machine learning fairness~\citep{oneto2020fairness}, dynamic systems~\citep{simon2002kalman}, PDE-constrained optimization~\citep{Gahu21}, structural model estimation~\citep{su2012constrained}, and training physics-informed neural networks~\citep{basir2022physics}.

Numerous algorithms have been designed and analyzed for solving general nonlinear constrained \textit{deterministic} optimization problems, including penalty methods, interior-point methods and SQP methods. The fundamental idea of penalty methods~\citep{courant1994variational,fletcher2000practical,zangwill1967non} is to incorporate constraint violations into the objective function using a weighted penalty term and then apply unconstrained optimization algorithms, such as subgradient methods~\citep{kiwiel2006methods,shor2012minimization}, to solve the resulting penalty subproblem. Penalty methods are easy to implement, however, they usually have inferior numerical performance compared to interior-point methods and SQP methods due to the difficulty in selecting suitable penalty parameters and the ill-conditioning of penalty subproblems. Meanwhile, interior-point methods~\citep{dikin1967iterative,renegar1988polynomial,wright1997primal} apply barrier functions to guide iterates moving along a central path within the feasible set to a solution. Interior-point methods are well-known for their outstanding numerical behaviors as they serve as base algorithms for many advanced computational optimization softwares, including \texttt{IPOPT} \citep{WachBieg06} and \texttt{KNITRO} \citep{ByrdNoceWalt06}. Finally, SQP methods~\citep{boggs1995sequential,wilson1963simplicial} are another class of optimization algorithms that have superior performance in both theory and practice. Line-search SQP methods~\citep{han1977globally,han1979exact,powell2006fast} are widely recognized as a class of state-of-the-art algorithms for solving equality-constrained optimization problems. At every iteration, SQP methods construct and solve a subproblem model of quadratic objective function subject to linearized constraints. This helps SQP methods allow infeasible iterates while (under reasonable conditions) enjoy global and fast local convergence behaviors at the same time.

To handle optimization problems arising in areas of machine learning and data science, it becomes crucial to design advanced algorithms for solving \textit{stochastic} optimization problems subject to nonlinear constraints. When solving \textit{stochastic} optimization problems with \textit{deterministic} constraints, most of existing algorithms are either stochastic penalty methods, including stochastic augmented Lagrangian algorithms, which offer worst-case complexity guarantees but often exhibit relatively inferior numerical performance, or imposing strong conditions (e.g., convexity and boundedness) on the feasible set so that projection-type operations or linear-type operators (such as Frank-Wolfe methods) are tractable~\citep{boob2025level,hazan2016variance,lu2021generalized,lu2024variance,ma2020quadratically,reddi2016stochastic,shi2025momentum}. Recently, a few papers proposed novel algorithms such as stochastic interior-point methods and stochastic SQP methods. In particular,~\cite{CurtKungRobiWang23,DezfWach24} develop stochastic interior-point methods for solving \textit{stochastic} (or \textit{noisy}) optimization problems with \textit{deterministic} box constraints, while~\cite{curtis2025interior,CurtJianWang24} extend the problem setting to general \textit{deterministic} (potentially \textit{noisy}) nonlinear constraints. Compared to stochastic interior-point methods, stochastic SQP methods have been studied more extensively in the literature. \citet{BeraCurtRobiZhou21} proposed the very first stochastic SQP algorithm for solving \textit{deterministic} nonlinear equality constrained \textit{stochastic} optimization problems, while~\cite{NaAnitKola23b} is the first work that considers general \textit{deterministic} functional constrained setting. There are multiple follow-up works based on~\cite{BeraCurtRobiZhou21,NaAnitKola23b}, including showing almost-sure convergence and non-asymptotic convergence properties of stochastic SQP algorithms~\cite{curtis2023almost,CurtRobiOnei24}; designing and analyzing stochastic line-search SQP methods~\citep{BeraXieZhou23,NaAnitKola23a,OztoByrdNoce23,QiuKung23}; relaxing requirements of constraint qualifications~\citep{BeraCurtOneiRobi23,berahas2025optimistic}; allowing inexactness when solving SQP subproblems~\citep{BeraBollShi24,BeraBollZhou22,CurtRobiZhou24,NaMaho22}; improving convergence behaviors via variance-reduced techniques~\citep{BeraShiYiZhou23}; designing stochastic trust-region SQP methods~\citep{FangNaMahoKola24,fang2024trust,sun2024trust}; and solving general \textit{deterministic} nonlinear constrained \textit{stochastic} optimization problems with an objective-function-free stochastic SQP method~\citep{CurtRobiZhou23}.

On the other hand, there are much less literature discussing general expectation-constrained stochastic optimization algorithms~\cite{alacaoglu2024complexity, BoobDengLan23, CuiWangXiao24, DzahKokkLeDi23, FaccKung23, LanZhou20, lew2024sample, li2024stochastic, MenhAuguBungYous17}. We note that some proximal-type algorithms are proposed in \cite{BoobDengLan23,LanZhou20}, however, relatively restrictive assumptions (including convexity of functions and ``strong feasibility'' conditions; see \cite[Assumption~3]{BoobDengLan23}) are required for proving theoretical convergence results. Meanwhile,~\cite{DzahKokkLeDi23, MenhAuguBungYous17} consider the ``black-box'' optimization setting, where trust-region and direct-search methods are developed respectively. Recently, several papers~\cite{alacaoglu2024complexity,CuiWangXiao24,li2024stochastic} have proposed variants of stochastic penalty and augmented Lagrangian algorithms, along with corresponding complexity analyses. To the best of our knowledge,~\cite{FaccKung23} proposes the only existing stochastic SQP method for solving expectation-constrained stochastic optimization problems, however, such algorithm relies on a special Monte Carlo process that generates stochastic search directions as unbiased estimators of the corresponding true search direction, while our proposed algorithm does not require such a restrictive unbiasedness condition.

\subsection{Contributions}
Our paper makes three-fold major contributions, which are summarized as follows.

\textbf{Algorithmistic Perspective.}\ \ We design \inrevise{a stochastic $\ell_1$-SQP method} to solve general nonlinear equality-constrained optimization problems, where both objective and constraint functions are represented as expectations. Our algorithm framework shares the spirit of \cite[Algorithm~4.1]{bottou2018optimization} and \cite[Algorithm~3.1]{BeraCurtRobiZhou21} that only a few stochastic estimates need to be evaluated at each iteration, while these stochastic estimates are not necessarily satisfying any probablistic accuracy conditions as \cite{BeraXieZhou23,NaAnitKola23a}. Moreover, our proposed algorithm takes adaptive step sizes and never requires objective function estimates.

\textbf{Theoretical Perspective.}\ \ Although our algorithmic framework is inspired by~\cite{BeraCurtRobiZhou21}, our theoretical analysis differs as we focus on the \textit{stochastic} constrained setting. In~\cite{BeraCurtRobiZhou21}, when solving \textit{deterministic} equality-constrained \textit{stochastic} optimization problems, under reasonable assumptions, the proposed algorithm always computes stochastic search directions behaving as unbiased estimators of ``real'' search directions, which would be computed via corresponding \textit{deterministic} quantities. However, such nice properties no longer hold in our setting due to uncertainty within constraints, which results in a different analysis from \cite{BeraCurtRobiZhou21,CurtRobiZhou23}. Because of loose conditions required for stochastic estimates, we only present convergence in expectation properties of our proposed algorithm. We demonstrate asymptotic and non-asymptotic convergence properties of our method, including both iteration and sample complexity results.
\inrevise{In particular, the theoretical performance of our algorithm remains competitive with that of variance-reduced stochastic gradient type algorithms proposed in~\cite{alacaoglu2024complexity} for the same problem setting.}

\textbf{Numerical Perspective.}\ \ We present the empirical performance of our proposed algorithm by comparing it to the stochastic momentum-based method proposed in~\cite{CuiWangXiao24} on a constrained binary classification problem with five different datasets from the LIBSVM collection~\citep{chang2011libsvm}. We also test and show the results of numerical experiments of our algorithm in contrast to a stochastic subgradient method on problems from the CUTEst collection~\citep{gould2015cutest}. These numerical comparisons illustrate the benefits of applying \inrevise{our proposed $\ell_1$-SQP framework} for solving equality-expectation-constrained stochastic optimization problems.

\subsection{Notation}
Let $\RR$ denote the set of real numbers, $\RR_{>r}$ (resp., $\RR_{\geq r}$) denote the set of real numbers that are strictly greater than (resp., greater than or equal to) $r\in\RR$, $\RR^n$ denote the set of $n$-dimensional real vectors, $\RR^{m\times n}$ denote the set of $m$-by-$n$-dimensional real matrices, $\mathbb{S}^n$ denote the set of $n$-by-$n$-dimensional real symmetric matrices, and $\NN$ denote the set of positive natural numbers $\{1,2,3,\ldots\}$. For any $m\in\NN$, we define $[m] := \{1,2,\ldots, m\}$. Let $\Null(A)$ and $\Range(A^T)$ denote the null space of matrix $A\in\RR^{s\times t}$ and the range space of matrix $A^T\in\RR^{t\times s}$, respectively, i.e., $\Null(A) := \{z^{\Null}\in\RR^t: Az^{\Null} = 0\}$ and $\Range(A^T) := \{z^{\Range}\in\RR^t: z^{\Range} = A^Tz \text{ for some }z\in\RR^s\}$. Let $\|\cdot\|$ denote the $\ell_2$-norm, $\|\cdot\|_F$ denote the Frobenius norm, and $\lceil\cdot\rceil$ denote the ceiling function. We define $\ones(\cdot)$ as the boolean function showing the input statement is true or false, and we use $\Vectorize{(\cdot)}$ to denote the vectorization function. For a random vector \(X \in \reals^n\), we use \(\supp{X}\) to denote its support.

\subsection{Organization}
The rest of this paper is structured as follows. In Section~\ref{sec.prob}, we introduce our problem setting of interest. We present the algorithm framework in Section~\ref{sec.alg}. The theoretical results of our proposed algorithm, including both asymptotic and non-asymptotic convergence behaviors, are provided in Section~\ref{sec.analysis}. We demonstrate the numerical performance of the algorithm in Section~\ref{sec.numerical}. Finally, Section~\ref{sec.conclusion} includes some concluding remarks.

\section{Problem Setting}\label{sec.prob}
This paper designs, analyzes, and tests a stochastic SQP algorithm for solving problems of the form
\bequation\label{eq.prob}
\min_{x\in\RR^{n}}\ f(x)\ \text{s.t.}\ c(x)=0\ \text{with}\ f(x) = \EE[F(x,\omega)]\text{ and }c(x) = \EE[C(x,\omega)],
\eequation
where both functions $f:\RR^{n}\to\RR$ and $c:\RR^{n}\to\RR^{m}$ are continuously differentiable, \(\omega\) is a random vector defined on probability space $(\Omega,\cF_{\Omega},\Prob)$ and $\EE[\cdot]$ is the expectation with respect to \(\Prob\). Since both objective and constraint functions in \eqref{eq.prob} could be nonlinear and nonconvex, instead of looking for global solutions of~\eqref{eq.prob}, we target on finding primal-dual iterates $(x,y)\in\RR^n \times\RR^m$ satisfying the stationary condition of the Lagrangian function $\cL(x,y) = f(x) + c(x)^Ty$, i.e.,
\bequation\label{eq.KKT}
0 = \bbmatrix \nabla \cL_x(x,y) \\ \nabla \cL_y(x,y) \ebmatrix = \bbmatrix \nabla f(x) + \nabla c(x)y \\ c(x) \ebmatrix.
\eequation
From~\eqref{eq.KKT}, one may see that the algorithm aims to find a primal iterate $x\in\RR^n$ that satisfies the feasibility condition $c(x) = 0$ while ensuring that its objective gradient lies in the range space of the constraint jacobian, i.e., $\nabla f(x) \in \Range (\nabla c(x))$.
\inrevise{Next, we make the following assumption regarding problem~\eqref{eq.prob}.}

\inrevise{
\begin{ass}\label{ass.prob}
The objective function $f:\RR^{n}\to\RR$ is continuously differentiable and bounded below. For every $i\in[m]$, the constraint function $c^i:\RR^{n}\to\RR$ is continuously differentiable and bounded.
The objective gradient function $\nabla f:\RR^{n}\to\RR^{n}$ and constraint gradient function $\nabla c^i:\RR^{n}\to\RR^{n}$, where $i\in [m]$, are Lipschitz continuous and bounded. Moreover, the minimum singular value of the constraint jacobian function $\nabla c^T:\RR^{n}\to\RR^{m \times n}$ is uniformly bounded away from zero.
\end{ass}
}

By Assumption~\ref{ass.prob}, there exist constants $f_{\inf}\in \RR$ and $\{\kappa_c,\kappa_{\nabla f},\kappa_{\nabla c},\lambda_{\min}\}\subset\RR_{>0}$ such that for any \inrevise{$x\in\RR^n$},
\bequation\label{eq.basic_condition}
f(x) \geq f_{\inf}, \ \|c(x)\|_1 \leq \kappa_c,\ \|\nabla f(x)\| \leq \kappa_{\nabla f},\ \sum_{i=1}^m\|\nabla c^i(x)\|_1 \leq \kappa_{\nabla c},\ \text{and} \ \nabla c(x)^T\nabla c(x) \succeq \lambda_{\min}^2 I.
\eequation
Meanwhile, when Assumption~\ref{ass.prob} holds, by the Lipschitz continuity of functions $\nabla f$ and $\nabla c^i$ for $i\in[m]$, there also exist Lipschitz constants $L > 0$, $\{\gamma_i\}_{i\in[m]} \subset \RR_{\geq 0}$, and $\Gamma = \sum_{i=1}^m \gamma_i \geq 0$
such that for any $\alpha \geq 0$ and \inrevise{$\{x,x+\alpha d\}\subset\RR^n$} that
\bequation\label{eq.Lipschitz_continuity}
\baligned
f(x+\alpha d) &\leq f(x) + \alpha\nabla f(x)^Td + \frac{L}{2}\alpha^2\|d\|^2 \quad \text{and} \\
|c^i(x+\alpha d)| &\leq |c^i(x) + \alpha\nabla c^i(x)^Td| + \frac{\gamma_i}{2}\alpha^2\|d\|^2 \  \text{and} \  \|c(x+\alpha d)\|_1 \leq \|c(x) + \alpha\nabla c(x)^Td \|_1 + \frac{\Gamma}{2}\alpha^2\|d\|^2.
\ealigned
\eequation
\inrevise{In this paper, we will propose an objective-function-free algorithm (see Section~\ref{sec.alg}), which never requires objective function values or estimates.} Moreover, we consider the setting that at any \inrevise{$x\in\RR^n$}, the objective gradient vector, the constraint function, and the constraint jacobian matrix are all unavailable, while we only have access to estimates of corresponding values, i.e., $(\bar{g}(x),\bar{c}(x),\bar{j}(x)) \approx (\nabla f(x),c(x),\nabla c(x)^T)$.

Similar to other SQP methods, we aim to solve the nonlinear system in~\eqref{eq.KKT} using Newton's method. At each iteration, we compute the gradient of the right-hand side of~\eqref{eq.KKT} and determine the search directions by solving the corresponding linear system. However, due to uncertainty in both the objective and constraint functions, we may only obtain noisy estimates for the objective gradient $\nabla f(x)$, the constraints $c(x)$, and the constraint Jacobian $\nabla c(x)^\top$, potentially through sampling oracles. 
Additionally, an accurate estimate of $\nabla_{xx}^2 \cL(x,y)$ is difficult to obtain especially when the gradient estimates are already noisy. To address this, instead of directly estimating second-order information, we rely on a simple oracle to generate \inrevise{crude ``estimates''} $\{H_k\} \subset \symm^n$ of $\{\nabla_{xx}^2 \cL(x_k, y_k)\}$ as \textit{inputs} to our algorithm, only requiring them to be positive definite without imposing further assumptions on their accuracy.

\begin{ass}\label{ass.H}
There exist constants $\kappa_H \geq \zeta > 0$ such that $\kappa_H I \succeq H_k \succeq \zeta I$ for all $k\in\NN$. 
\end{ass}

Therefore, to construct such oracles for \(\{H_k\}\), one may have it consistently return the identity matrix at all iterations (see Section~\ref{sec.numerical}). Alternatively, our analysis also allows one to generate $\{H_k\}$ as stochastic estimates, so long as each $H_k$ only relies on information evaluated \textit{prior to} iteration $k$. For example, one could use any information collected before the current iteration to predict the primal part of the Hessian of the Lagrangian and then add a multiple of the identity matrix to ensure that \(\{H_k\}\) remain bounded and sufficiently positive definite. It is worth mentioning that one could further relax the positive definiteness requirement by only asking $H_k$ to be sufficiently positive definite on the space $\Null(\nabla c(x_k)^T)\cup\Null(\bar{j}(x_k))$, to match the conditions in classic SQP methods~\cite{nocedal2006numerical}. However, such relaxation of the positive definiteness requirement is not the main focus of our paper, and we consider Assumption~\ref{ass.H} for the ease of presentation.
In essence, our algorithm is Hessian-free---that is, it does not require Hessian information---and our theoretical analyses are quite general, relying only on the uniform boundedness of \(\{H_k\}\) assumed in Assumption~\ref{ass.H}.

\section{Algorithm Framework}\label{sec.alg}
We propose \inrevise{a stochastic $\ell_1$-SQP method} for solving \eqref{eq.prob}. At every iterate $x_k\in\RR^{n}$, our algorithm computes a search direction $\bar{d}_k\in\RR^{n}$ by solving a quadratic-objective-linear-constrained subproblem
\bequation\label{eq.SQP}
\min_{\bar{d}\in\RR^n}\ \bar{g}(x_k)^T\bar{d} + \frac{1}{2}\bar{d}^TH_k\bar{d}\quad\text{s.t.}\ \bar{c}(x_k) + \bar{j}(x_k)\bar{d} = 0,
\eequation
where $(\bar{g}(x_k),\bar{c}(x_k),\bar{j}(x_k))$ are stochastic estimates of $(\nabla f(x_k),c(x_k),\nabla c(x_k)^T)$. \inrevise{Under Assumption~\ref{ass.H}, the solution of \eqref{eq.SQP}, denoted by $\bar{d}_k$, is unique, since~\eqref{eq.SQP} amounts to minimizing a strongly convex quadratic function over an affine set. Moreover, under Assumption~\ref{ass.H}, $\bar{d}_k\in\RR^n$ can be achieved by solving the following linear system}
\bequation\label{eq.linear_system}
\bbmatrix H_k & \bar{j}(x_k)^T \\ \bar{j}(x_k) & 0 \ebmatrix \bbmatrix \bar{d}_k \\ \bar{y}_k \ebmatrix = -\bbmatrix \bar{g}(x_k) \\ \bar{c}(x_k) \ebmatrix,
\eequation
where $\bar{y}_k\in\RR^m$, the dual multiplier corresponding to equality constraints in \eqref{eq.SQP}, is also unique if the minimum singular value of $\bar{j}(x_k)$ is \inrevise{positive}. After computing search direction $\bar{d}_k$, we construct a model based on an $\ell_1$-norm merit function to monitor the algorithm's progress. In particular, we consider the merit function $\phi:\RR^{n}\times\RR_{>0} \to \RR$ as
\bequation\label{eq.merit_func}
\phi(x,\tau) := \tau f(x) + \|c(x)\|_1,
\eequation
where $\tau > 0$ is the merit parameter that weighs the balance between objective function value and the $\ell_1$-norm of constraint violation. Moreover, we define $l:\RR^n\times\RR_{>0}\times\RR^n\times\RR^m\times\RR^{m\times n}\times\RR^n \to \RR$ as a linear model of merit function $\phi(x,\tau)$ along direction $d\in\RR^n$ with estimates $(\bar{g}(x),\bar{c}(x),\bar{j}(x))\approx (\nabla f(x),c(x),\nabla c(x)^T)$, i.e.,
\bequationn
l(x,\tau,\bar{g}(x),\bar{c}(x),\bar{j}(x),d) = \tau(f(x) + \bar{g}(x)^Td) + \|\bar{c}(x) + \bar{j}(x)d\|_1.
\eequationn
Furthermore, we consider a model reduction function $\Delta l:\RR^n\times\RR_{>0}\times\RR^n\times\RR^m\times\RR^{m\times n}\times \RR^n \to \RR$, where
\bequationn
\Delta l(x,\tau,\bar{g}(x),\bar{c}(x),\bar{j}(x),d) := l(x,\tau,\bar{g}(x),\bar{c}(x),\bar{j}(x),0) - l(x,\tau,\bar{g}(x),\bar{c}(x),\bar{j}(x),d).
\eequationn
In particular, when the given search direction $\bar{d}\in\RR^n$ satisfies $\bar{c}(x) + \bar{j}(x)\bar{d} = 0$, we have
\bequation\label{eq.linear_model_reduction}
\baligned
\Delta l(x,\tau,\bar{g}(x),\bar{c}(x),\bar{j}(x),\bar{d}) = \ &l(x,\tau,\bar{g}(x),\bar{c}(x),\bar{j}(x),0) - l(x,\tau,\bar{g}(x),\bar{c}(x),\bar{j}(x),\bar{d}) \\
=\ &-\tau \bar{g}(x)^T\bar{d} + \|\bar{c}(x)\|_1 - \|\bar{c}(x)+\bar{j}(x)\bar{d}\|_1 = -\tau \bar{g}(x)^T\bar{d} + \|\bar{c}(x)\|_1.
\ealigned
\eequation
\inrevise{This model reduction function serves as a key measure that is closely related to progress toward stationarity (in expectation). In the special case of stochastic gradient methods in unconstrained optimization, where $(\tau,\bar{d},\bar{c}(x)) = (1,-\bar{g}(x),0)$, substituting into~\eqref{eq.linear_model_reduction} shows that the model reduction function reduces to $-\tau \bar{g}(x)^T\bar{d} + \|\bar{c}(x)\|_1 = \|\bar{g}(x)\|_2^2$, which is a standard progress measure in unconstrained stochastic optimization. Model reduction functions of a similar spirit have also been proposed in~\cite{BeraCurtOneiRobi23,BeraCurtRobiZhou21,CurtRobiOnei24,CurtRobiZhou24}.}

Next, we introduce the update rule of stochastic merit parameters used in our algorithm. At each iterate $x_k$, after evaluating stochastic objective gradient and constraint function estimates $\bar{g}(x_k)$ and $\bar{c}(x_k)$, stochastic search direction $\bar{d}_k$ (by solving~\eqref{eq.linear_system}), and the previous stochastic merit parameter $\bar\tau_{k-1}$, we update $\bar\tau_k$ by setting
\bequation\label{eq.merit_parameter_update}
\bar\tau_k \leftarrow \bcases \bar\tau_{k-1} \quad\quad\quad\quad\ \ \ \text{if }\bar\tau_{k-1}\leq (1-\epsilon_{\tau})\cdot\bar\tau^{\trial}_k \\ (1-\epsilon_{\tau})\cdot\min\left\{\bar\tau_{k-1}, \bar\tau_k^{\trial}\right\}
\quad\text{otherwise,} \ecases\ \text{with} \ \bar\tau_k^{\trial} = \bcases +\infty \quad\quad\quad\quad\text{if }\bar{g}(x_k)^T\bar{d}_k + \frac{1}{2}\bar{d}_k^TH_k\bar{d}_k \leq 0 \\ \frac{(1-\sigma)\|\bar{c}(x_k)\|_1}{\bar{g}(x_k)^T\bar{d}_k + \frac{1}{2}\bar{d}_k^TH_k\bar{d}_k} \quad\quad\quad\quad\quad\quad \text{otherwise,} \ecases
\eequation
where $\{\sigma,\epsilon_{\tau}\}\subset (0,1)$ are user-defined parameters. \inrevise{The main goal of updating stochastic merit parameters via \eqref{eq.merit_parameter_update} is to ensure that $\bar\tau_k$ is always sufficiently small relative to $\bar\tau_k^{\trial}$ (see Lemma~\ref{lem.merit_parameter}) that the model reduction function remains sufficiently large at every iteration $k\in\mathbb{N}$. Specifically, for the stochastic search direction $\bar{d}_k$ obtained from \eqref{eq.linear_system}, the model reduction condition
\bequation\label{eq.model_reduction_cond}
\Delta l(x_k,\bar\tau_k,\bar{g}(x_k),\bar{c}(x_k),\bar{j}(x_k),\bar{d}_k) = -\bar\tau_k\bar{g}(x_k)^T\bar{d}_k + \|\bar{c}(x_k)\|_1 \geq \frac{\zeta}{2}\bar\tau_k\|\bar{d}_k\|^2 + \sigma\|\bar{c}(x_k)\|_1 \geq 0
\eequation
is guaranteed to hold (see Lemma~\ref{lem.model_reduction_suff_large}). This condition also fundamentally ensures that the stochastic search direction $\bar{d}_k$ constitutes a descent direction for the model function $l(x_k,\bar\tau_k,\bar{g}(x_k),\bar{c}(x_k),\bar{j}(x_k),d)$.}
Meanwhile, our algorithm generates a sequence of ratio parameters $\{\bar\xi_k\}$ that helps to determine adaptive step sizes. In fact, for any $k\in\NN$, $\bar\xi_k$ always behaves as a lower bound of the ratio between $\Delta l(x_k,\bar\tau_k,\bar{g}(x_k),\bar{c}(x_k),\bar{j}(x_k),\bar{d}_k)$ and $\bar\tau_k\|\bar{d}_k\|^2$. In particular, whenever $\bar{d}_k\neq 0$, we update $\{\bar\xi_k\}$ by setting
\bequation\label{eq.ratio_parameter_update}
\bar\xi_k \leftarrow \bcases \bar\xi_{k-1} &\text{if }\bar\xi_{k-1} \leq \bar\xi_k^{\trial} \\ \min\{(1-\epsilon_{\xi})\bar\xi_{k-1},\bar\xi_k^{\trial}\} &\text{otherwise,} \ecases \ \text{ with } \ \bar\xi_k^{\trial} = \frac{\Delta l(x_k,\bar\tau_k,\bar{g}(x_k),\bar{c}(x_k),\bar{j}(x_k),\bar{d}_k)}{\bar\tau_k\|\bar{d}_k\|^2},
\eequation
where $\epsilon_{\xi}\in (0,1)$ is a prescribed parameter. Even though ratio parameters rely on stochastic estimates of objective gradient and constraint information $\{(\bar{g}(x_k),\bar{c}(x_k),\bar{j}(x_k))\}$, we show in Lemma~\ref{lem.ratio_parameter} that the sequence $\{\bar\xi_k\}\subset\mathbb{R}_{>0}$ is uniformly bounded away from zero with a deterministic positive lower bound.
\inrevise{This ratio parameter update shares a similar spirit with those in~\cite{BeraCurtRobiZhou21,BeraShiYiZhou23,CurtRobiZhou24}.
To further illustrate this behavior, consider the special case of stochastic gradient methods for unconstrained stochastic optimization. In this setting, $(\bar{d}_k,\bar{c}(x_k)) = (-\bar{g}(x_k),0)$ for any $k\in\mathbb{N}$, then it follows from~\eqref{eq.linear_model_reduction} and~\eqref{eq.ratio_parameter_update} that $\bar\xi_k^{\trial} = \frac{-\bar\tau_k\bar{g}(x_k)^T\bar{d}_k + \|\bar{c}(x_k)\|_1}{\bar\tau_k\|\bar{d}_k\|^2} = 1$ and $\bar\xi_k$ remains constant for all $k\in\mathbb{N}$.}

The last centerpiece of our algorithm framework is the adaptive step size selection strategy.
Given prescribed parameters $(\bar\tau_0,\bar\xi_0,\eta)\in \mathbb{R}_{>0}\times\mathbb{R}_{>0}\times(0,1)$ and Lipschitz constants $(L,\Gamma)$ from \eqref{eq.Lipschitz_continuity}, we choose a sequence of~$\{\beta_k\}\subset (0,1]$ satisfying
\bequation\label{eq.beta_requirement}
\frac{2(1-\eta)\beta_k\bar\xi_0\bar\tau_0}{(\bar\tau_0L + \Gamma)} \in (0,1] \text{ for all }k\in\NN{}.
\eequation
When $\bar{d}_k \neq 0$, with user-defined parameters $\theta \geq 0$ and $\eta \in (0,1)$, we set
\bequation\label{eq.stepsizes}
\bar\alpha_k^{\min} := \frac{2(1-\eta)\beta_k\bar\xi_k\bar\tau_k}{\bar\tau_kL+\Gamma}, \quad  \bar\alpha_k^{\max} := \min\{\bar\alpha_k^{\min} + \theta\beta_k, \bar\alpha_k^{\varphi}\} \quad \text{and} \quad \bar\alpha_k^{\varphi} = \max\left\{\alpha > 0\ | \ \varphi_k(\alpha) \leq 0\right\},
\eequation
where for any $k\in\mathbb{N}$ function $\varphi_k:\RR \to \RR$ is defined as
\bequation\label{eq.alphaphi}
\varphi_k(\alpha) :=  (\eta - 1)\alpha\beta_k\Delta l(x_k,\bar\tau_k,\bar{g}(x_k),\bar{c}(x_k),\bar{j}(x_k),\bar{d}_k) + (|1-\alpha| - (1-\alpha))\|\bar{c}(x_k)\|_1 + \frac{1}{2}(\bar\tau_k L+\Gamma)\alpha^2\|\bar{d}_k\|^2.
\eequation
The $\varphi_k(\alpha)$ function is constructed in a way that for any iteration $k\in\NN$, when $(\bar{g}(x_k),\bar{c}(x_k),\bar{j}(x_k)) = (\nabla f(x_k),c(x_k),\nabla c(x_k)^T)$ and $\beta_k = 1$, any $\alpha > 0$ satisfying $\varphi_k(\alpha) \leq 0$ would also guarantee that $\alpha\bar{d}_k$ provides a sufficient decay in the merit function, i.e., $\phi(x_k + \alpha\bar{d}_k,\bar\tau_k) - \phi(x_k,\bar\tau_k) \leq -\eta\alpha\Delta l(x_k,\bar\tau_k,\bar{g}(x_k),\bar{c}(x_k),\bar{j}(x_k),\bar{d}_k)$; see~\cite[Equation (25)]{CurtRobiZhou24}.
We notice in~\eqref{eq.alphaphi} that when $\bar{d}_k\neq 0$, $\varphi_k(\alpha)$ is the maximum of two strongly convex quadratic functions, whose nonsmoothness only appears at $\alpha = 1$. Meanwhile, $\varphi_k(\alpha)$ is also a convex function, and there exist two solutions satisfying $\varphi_k(\alpha) = 0$ when $\bar{d}_k\neq 0$, which are $\alpha = 0$ and $\alpha = \bar\alpha_k^{\varphi}$ (see Lemma~\ref{lem.stepsize}). Finally, we choose an adaptive step size $\bar\alpha_k \in [\bar\alpha_k^{\min}, \bar\alpha_k^{\max}]$ and update $x_{k+1} \leftarrow x_k + \bar\alpha_k \bar{d}_k$. On the other hand, if $\bar{d}_k = 0$, we set $\bar\alpha_k^{\min} = \frac{2(1-\eta)\beta_k\bar\xi_k\bar\tau_k}{\bar\tau_kL+\Gamma}$ and $(\bar\alpha_k^{\max},\bar\alpha_k^{\varphi},\bar\alpha_k) = \left(\bar\alpha_k^{\min} + \theta\beta_k, \bar\alpha_k^{\min} + \theta\beta_k, \bar\alpha_k^{\min}\right)$ then update the next iterate $x_{k+1} \leftarrow x_k + \bar\alpha_k \bar{d}_k$.

In the rest of this paper, for brevity, we abbreviate function values evaluated at iterate $x_k\in\RR^n$ by adding iteration counter $k$ as a subscript to the corresponding function, e.g., $f_k := f(x_k)$ and $\bar{g}_k := \bar{g}(x_k)$.
The detailed algorithm is presented in Algorithm~\ref{alg.main}.

\begin{algorithm}[tb]
	\SetAlgoLined
	\textbf{Require:} initial iterate $x_1\in\RR^n$; initial merit parameter $\bar\tau_0 > 0$; initial ratio parameter $\bar\xi_0 > 0$; $\{H_k\}\subset\mathbb{S}^n$ satisfying Assumption~\ref{ass.H}; Lipschitz constants $L>0$ and $\Gamma>0$; constant parameters $\eta\in(0,1)$, $\{\sigma,\epsilon_{\tau},\epsilon_{\xi}\}\subset (0,1)$ and $\theta > 0$; step-size parameter sequence $\{\beta_k\}\subset (0,1]$ satisfying \eqref{eq.beta_requirement}; and variance control sequences $\{\rho_k^g\}\subset\RR_{>0}$, $\{\rho_k^c\}\subset\RR_{>0}$, and $\{\rho_k^j\}\subset\RR_{>0}$.\\
	\For{$k\in \NN$}{
	    Evaluate stochastic estimates $(\bar{g}_k,\bar{c}_k,\bar{j}_k) \approx (\nabla f(x_k),c(x_k),\nabla c^T(x_k))$ satisfying Assumption~\ref{ass.estimate}. \\
            Compute $(\bar{d}_k,\bar{y}_k)$ by solving \eqref{eq.linear_system} \\
            \uIf{$\bar{d}_k \neq 0$}
            {Update $\bar\tau_k$ and $\bar\tau_k^{\trial}$ via \eqref{eq.merit_parameter_update} \\
            Update $\bar\xi_k$ and $\bar\xi_k^{\trial}$ via \eqref{eq.ratio_parameter_update} \\
            Compute $\bar\alpha_k^{\min}$ and $\bar\alpha_k^{\max}$ via \eqref{eq.stepsizes} \\
            Choose $\bar\alpha_k \in [\bar\alpha_k^{\min}, \bar\alpha_k^{\max}]$}
            \Else{Set $(\bar\tau_k,\bar\tau_k^{\trial},\bar\xi_k,\bar\xi_k^{\trial})\leftarrow(\bar\tau_{k-1},+\infty,\bar\xi_{k-1},+\infty)$ \\
            Set $\bar\alpha_k^{\min} \leftarrow \frac{2(1-\eta)\beta_k\bar\xi_k\bar\tau_k}{\bar\tau_kL+\Gamma}$ and $(\bar\alpha_k^{\max},\bar\alpha_k^{\varphi},\bar\alpha_k)\leftarrow \left(\bar\alpha_k^{\min} + \theta\beta_k, \bar\alpha_k^{\min} + \theta\beta_k, \bar\alpha_k^{\min}\right)$}
	    Update $x_{k+1}\leftarrow x_k + \bar\alpha_k \bar{d}_k$
	}
	\caption{\inrevise{Stochastic $\ell_1$-SQP Algorithm with Adaptive Step-size Selection}}
	\label{alg.main}
\end{algorithm}

Our proposed Algorithm~\ref{alg.main} shares the similar spirit as \cite[Algorithm~3.1]{BeraCurtRobiZhou21} but with major differences in step size selection and merit parameter update. In particular, \cite[Algorithm~3.1]{BeraCurtRobiZhou21} chooses adaptive step sizes within an interval whose length is proportional to $\beta_k^2$, while our Algorithm~\ref{alg.main} considers $(\bar\alpha_k^{\max} - \bar\alpha_k^{\min}) \leq \theta\beta_k$ (see~\eqref{eq.stepsizes}).
\inrevise{This distinction arises from the intrinsic difference between the stochastic search directions $\bar{d}_k$ defined in~\eqref{eq.linear_system} and those considered in~\cite[Equation~(3.2)]{BeraCurtRobiZhou21}.
Specifically, the stochastic search directions in~\cite{BeraCurtRobiZhou21} are \textit{unbiased} estimators of their deterministic counterparts, which allows exact asymptotic convergence in expectation even when stochastic gradient estimates have non-diminishing but bounded variance, provided that suitably diminishing step sizes (corresponding to an interval length of order $\mathcal{O}(\beta_k^2)$) are used; see \cite[Corollary~3.14]{BeraCurtRobiZhou21}. In contrast, our work addresses a more challenging setting in which the stochastic search directions are \textit{biased} estimators of the associated deterministic quantities. This necessitates an adaptive step-size selection strategy that is closer in spirit to that of~\cite{CurtRobiZhou23} than to~\cite{BeraCurtRobiZhou21}. When only \textit{biased} stochastic search directions are available, diminishing variance of the stochastic gradient estimates is required to ensure exact asymptotic convergence in expectation, which in turn permits the use of potentially larger step sizes corresponding to an adaptive step-size interval of length $\mathcal{O}(\beta_k)$.
Meanwhile, our merit parameter update rule ensures that $\bar\tau_k \leq (1-\epsilon_{\tau})\cdot\bar\tau_k^{\trial}$ for all $k\in\mathbb{N}$ (see Lemma~\ref{lem.merit_parameter}), which is a more conservative update compared to \cite{BeraCurtRobiZhou21}, but guarantees that stochastic merit parameters generated by Algorithm~\ref{alg.main} are not excessively large with high probability at each iteration (see Lemma~\ref{lem.prob-tau-good-k}).}
In addition, we note that Algorithm~\ref{alg.main} generates a stochastic sequence
\bequationn
\left\{(x_k,\bar{y}_k,\bar{g}_k,\bar{c}_k,\bar{j}_k,\bar{d}_k,\bar\tau_k,\bar\tau_k^{\trial},\bar\xi_k,\bar\xi_k^{\trial},\bar\alpha_k^{\min},\bar\alpha_k^{\max}, \bar\alpha_k^{\varphi}, \bar\alpha_k)\right\},
\eequationn
which can be considered as a realization of the stochastic process
\bequationn
\left\{(X_k,\bar{Y}_k,\bar{G}_k,\bar{C}_k,\bar{J}_k,\bar{D}_k,\bar\cT_k,\bar\cT_k^{\trial},\bar\Xi_k,\bar\Xi_k^{\trial},\bar\cA_k^{\min},\bar\cA_k^{\max},\bar\cA_k^{\varphi},\bar\cA_k)\right\}.
\eequationn
For the sake of analysis, we further define $(d_k,y_k)$ as the solution to \eqref{eq.linear_system} with estimates $(\bar{g}(x_k),\bar{c}(x_k),\bar{j}(x_k))$ being replaced by corresponding true values $(\nabla f(x_k),c(x_k),\nabla c(x_k)^T)$, where the uniqueness of $(d_k,y_k)$ follows Assumption~\ref{ass.prob}, i.e.,
\bequation\label{eq.linear_system_true}
\bbmatrix H_k & \nabla c(x_k) \\ \nabla c(x_k)^T & 0 \ebmatrix \bbmatrix d_k \\ y_k \ebmatrix = -\bbmatrix \nabla f(x_k) \\ c(x_k) \ebmatrix.
\eequation
We also consider $\{(d_k,y_k)\}$ as a realization of $\{(D_k,Y_k)\}$. In addition, we define $\{\cT_k\}$ and $\{\cT_k^{\trial}\}$ as the deterministic counterpart of $\{\bar\cT_k\}$ and $\{\bar\cT_k^{\trial}\}$ (with $(\bar{g}(x_k),\bar{c}(x_k),\bar{d}_k)$ replaced by $(\nabla f(x_k),c(x_k),d_k)$ in \eqref{eq.merit_parameter_update}). Note that $\{(D_k,Y_k,\cT_k,\cT_k^{\trial})\}$ is merely for theoretical analysis and we never compute such values during the implementation of Algorithm~\ref{alg.main}.
Moreover, \inrevise{we let \(\cG_1 := \{\varnothing, \Omega\}\) be the initial \(\sigma\)-algebra}, and for all iterations \(k \geq 2\), we use \(\cG_k\) to denote the \(\sigma\)-algebra generated by random variables $\{\bar{G}_1,\ldots,\bar{G}_{k-1}\}\cup\{\bar{C}_1,\ldots,\bar{C}_{k-1}\}\cup\{\bar{J}_1,\ldots,\bar{J}_{k-1}\}$.

\inrevise{In the remainder of this paper, and unless stated otherwise, all equalities and inequalities involving random variables are understood to hold almost surely. We also impose the following assumption.}
\begin{ass}\label{ass.estimate}
\inrevise{
The following two statements hold for stochastic constraint jacobian estimates generated by Algorithm~\ref{alg.main}: (i) constraint jacobian estimates $\{\bar{J}_k\}$ are uniformly upper bounded in norm, and (ii) the minimum singular values of the constraint jacobian estimates $\{\bar{J}_k\}$ are uniformly bounded away from zero. In particular, we assume that there exist constants $\{\kappa_{\bar{J}},\bar\lambda_{\min}\}\subset\RR_{>0}$ such that for any $k\in\NN$, $\|\bar{J}_k\|_1 \leq \kappa_{\bar{J}}$ and $\bar{J}_k\bar{J}_k^T \succeq \bar\lambda_{\min}^2 I$.
}
\end{ass}
\inrevise{Now we conclude this section by showing that under aforementioned assumptions, singular values of matrices in~\eqref{eq.linear_system} and~\eqref{eq.linear_system_true} are always uniformly bounded above and away from zero.}

\blemma\label{lem.singular_values}
Suppose Assumptions~\ref{ass.prob}--\ref{ass.estimate} hold. There exist constants $q_{\max} \geq q_{\min} > 0$ such that for any iteration~$k\in\NN$, the singular values of $\bbmatrix H_k & \nabla c(X_k) \\ \nabla c(X_k)^T & 0 \ebmatrix$ and $\bbmatrix H_k & \bar{J}_k^T \\ \bar{J}_k & 0 \ebmatrix$ always stay in $[q_{\min}, q_{\max}]$.
\elemma
\bproof
For any iteration $k\in\mathbb{N}$, let's denote $M_k=\begin{bmatrix} H_k & S_k \\ S_k^T & 0 \end{bmatrix}$, then by Assumptions~\ref{ass.prob}--\ref{ass.estimate} and \cite[Theorem~3.1]{gill1991inertia}, for any $S_k\in\{\nabla c(X_k), \bar{J}_k^T\}$, $M_k$ is always invertible and we always have
\bequationn
M_k^{-1} = \begin{bmatrix} H_k^{-1} - H_k^{-1}S_k(S_k^TH_k^{-1}S_k)^{-1}S_k^TH_k^{-1} & H_k^{-1}S_k(S_k^TH_k^{-1}S_k)^{-1} \\ (S_k^TH_k^{-1}S_k)^{-1}S_k^TH_k^{-1} & -(S_k^TH_k^{-1}S_k)^{-1} \end{bmatrix}.
\eequationn
From Assumptions~\ref{ass.prob}--\ref{ass.estimate}, we know that the singular values of $H_k$ and $S_k$ are uniformly bounded above and away from zero. Thus, every block of $M_k$ and $M_k^{-1}$ is uniformly bounded. Therefore, there exists $q_{\max}>1$ such that the singular values of both $M_k$ and $M_k^{-1}$ are upper bounded by $q_{\max}$. Meanwhile, since the singular values of $M_k$ are the reciprocals of the singular values of $M_k^{-1}$, we may set $q_{\min}:=\frac{1}{q_{\max}}\in (0,1)$, where $q_{\min}$ performs as a lower bound of singular values of $M_k$.
\eproof

\section{Convergence Properties and Complexity Analysis}\label{sec.analysis}
In this section, we analyze the theoretical performance of Algorithm~\ref{alg.main}. In Section~\ref{sec.fund_lemmas}, we present some fundamental lemmas that would be useful for the analysis of both asymptotic and non-asymptotic convergence properties of our algorithm. In Section~\ref{sec.converg_results}, we demonstrate the asymptotic convergence (in expectation) behavior of our stochastic SQP algorithm (Algorithm~\ref{alg.main}). Moreover, we provide the non-asymptotic convergence performance of Algorithm~\ref{alg.main} in Section~\ref{sec.complex_results}, which includes both iteration and sample complexity results for finding a near-stationary iterate in expectation. Note that convergence results presented in both Sections~\ref{sec.converg_results} and~\ref{sec.complex_results} rely on the existence of ``well-behaved'' merit parameters. Finally, we conclude this section by Section~\ref{sec.merit_parameter_behaviors}, where we argue that Algorithm~\ref{alg.main} is unlikely to generate ``poorly-behaved'' merit parameters  when the variances of stochastic estimates meet mild conditions.

\subsection{Fundamental Technical Results}\label{sec.fund_lemmas}
In this subsection, we present some fundamental technical lemmas that would be useful to asymptotic and non-asymptotic convergence results in Sections~\ref{sec.converg_results} and~\ref{sec.complex_results}.

Our first lemma is about the behavior of the merit parameter sequence $\{\bar\cT_k\}$. In particular, we show $\{\bar\cT_k\}\subset\RR_{>0}$ is a monotonically decreasing sequence and we reveal the relation between $\{\bar\cT_k\}$ and $\{\bar\cT_k^{\trial}\}$.
\blemma\label{lem.merit_parameter}
Suppose Assumptions~\ref{ass.prob}--\ref{ass.estimate} hold. At any iteration $k$ of Algorithm~\ref{alg.main}, we always have $\bar\cT_{k-1} \geq \bar\cT_k > 0$ and $\bar\cT_k \leq (1-\epsilon_{\tau})\cdot\bar\cT_k^{\trial}$.
\elemma
\bproof
We first prove $\bar\cT_k > 0$ by showing that $\bar\cT_k^{\trial} > 0$ for all $k\in\NN$. At any iteration $k$ such that $\bar{D}_k \neq 0$, if $\bar{C}_k = 0$ then we have from Assumption~\ref{ass.H} and \eqref{eq.linear_system} that
\bequationn
\bar{G}_k^T\bar{D}_k + \frac{1}{2}\bar{D}_k^TH_k\bar{D}_k \leq \bar{G}_k^T\bar{D}_k + \bar{D}_k^TH_k\bar{D}_k = \bar{D}_k^T(\bar{G}_k + H_k\bar{D}_k) = -\bar{D}_k^T\bar{J}_k^T\bar{Y}_k = \bar{C}_k^T\bar{Y}_k = 0,
\eequationn
which implies that $\bar\cT_k^{\trial} = +\infty$ by \eqref{eq.merit_parameter_update}. On the other hand, if $\bar{D}_k = 0$ or $\bar{C}_k \neq 0$, from \eqref{eq.merit_parameter_update} and Algorithm~\ref{alg.main}, we always have $\bar\cT_k^{\trial} > 0$. Therefore, $\bar\cT_k^{\trial} > 0$ at every iteration $k$. Moreover, combining $\bar\cT_k^{\trial} > 0$, $\bar\tau_0 > 0$, \eqref{eq.merit_parameter_update}, and Algorithm~\ref{alg.main}, we further know that $\bar\cT_k > 0$ at every iteration $k$. Furthermore, \eqref{eq.merit_parameter_update} implies that either $\bar\cT_k = \bar\cT_{k-1} \leq (1-\epsilon_{\tau})\cdot\bar\cT_k^{\trial}$ or $\bar\cT_k = (1-\epsilon_{\tau})\cdot\min\{\bar\cT_{k-1},\bar\cT_k^{\trial}\} \leq (1-\epsilon_{\tau})\cdot \bar\cT_k^{\trial} < \bar\cT_{k-1}$, which concludes the statement.
\eproof

The next lemma proves that \eqref{eq.model_reduction_cond} always holds at all iterations.

\blemma\label{lem.model_reduction_suff_large}
Suppose Assumptions~\ref{ass.prob}--\ref{ass.estimate} hold. For every iteration $k$ in Algorithm~\ref{alg.main}, $\Delta l(X_k,\bar\cT_k,\bar{G}_k,\bar{C}_k,\bar{J}_k,\bar{D}_k) \geq \frac{\zeta}{2}\bar\cT_k\|\bar{D}_k\|^2 + \sigma\|\bar{C}_k\|_1 \geq 0$ always holds, where $\zeta>0$ and $\sigma \in (0,1)$ are parameters from Assumption~\ref{ass.H} and~\eqref{eq.merit_parameter_update}. Moreover, if $\bar{D}_k \neq 0$, then $\Delta l(X_k,\bar\cT_k,\bar{G}_k,\bar{C}_k,\bar{J}_k,\bar{D}_k) > 0$.
\elemma
\bproof
By the update of $\{\bar\cT_k\}$ in \eqref{eq.merit_parameter_update} and Lemma~\ref{lem.merit_parameter}, we have that for all iterations $k$,
\bequation\label{eq.merit_parameter_relation}
\bar\cT_k(\bar{G}_k^T\bar{D}_k + \tfrac{1}{2}\bar{D}_k^TH_k\bar{D}_k) \leq (1-\epsilon_{\tau})(1-\sigma)\|\bar{C}_k\|_1 \leq (1-\sigma)\|\bar{C}_k\|_1.
\eequation
Moreover, at any iteration $k$, it follows Assumption~\ref{ass.H}, \eqref{eq.linear_model_reduction}, $\bar\cT_k > 0$ (see Lemma~\ref{lem.merit_parameter}), and \eqref{eq.merit_parameter_relation} that
\bequationn
\Delta l(X_k,\bar\cT_k,\bar{G}_k,\bar{C}_k,\bar{J}_k,\bar{D}_k) = -\bar\cT_k\bar{G}_k^T\bar{D}_k + \|\bar{C}_k\|_1 \geq \tfrac{1}{2}\bar\cT_k\bar{D}_k^TH_k\bar{D}_k + \sigma\|\bar{C}_k\|_1 \geq \tfrac{\zeta}{2}\bar\cT_k\|\bar{D}_k\|^2 + \sigma\|\bar{C}_k\|_1 \geq 0,
\eequationn
which concludes the first part of the statement. When $\bar{D}_k \neq 0$, from $\zeta > 0$ and $\bar\cT_k > 0$ (see Assumption~\ref{ass.H} and Lemma~\ref{lem.merit_parameter}) we know $\Delta l(X_k,\bar\cT_k,\bar{G}_k,\bar{C}_k,\bar{J}_k,\bar{D}_k) \geq \tfrac{\zeta}{2}\bar\cT_k\|\bar{D}_k\|^2 + \sigma\|\bar{C}_k\|_1 > 0$, which completes the proof.
\eproof

Similar to \eqref{eq.model_reduction_cond} and Lemma~\ref{lem.model_reduction_suff_large}, we next show that for all small enough $\cT > 0$, the deterministic model reduction function can also be lower bounded by a quantity related to $\|D_k\|^2$ and $\|c(X_k)\|_1$.
\blemma\label{lem.model_reduction_suff_large_true}
Suppose Assumptions~\ref{ass.prob}--\ref{ass.estimate} holds. For any $k\in\NN{}$, with $(\zeta,\sigma) \in \RR{}_{>0}\times (0,1)$ in Assumption~\ref{ass.H} and~\eqref{eq.merit_parameter_update}, we know that for any $\cT\in (0,\cT_k^{\trial}]$,
\bequation\label{eq.model_reduction_suff_large_true}
\Delta l(X_k,\cT,\nabla f(X_k), c(X_k), \nabla c(X_k)^T, D_k) \geq \tfrac{\zeta}{2}\cT\|D_k\|^2 + \sigma\|c(X_k)\|_1.
\eequation
\elemma
\bproof
This proof is the same as the proof of Lemma~\ref{lem.model_reduction_suff_large}, with stochastic quantities replaced by corresponding deterministic counterparts.
\eproof

The following lemma demonstrates that even though the ratio parameter sequence $\{\bar\Xi_k\}$ is updated based on stochastic quantities $\{(\bar{G}_k,\bar{C}_k,\bar{J}_k)\}$, the sequence $\{\bar\Xi_k\}$ generated by Algorithm~\ref{alg.main} is always bounded away from zero.
\blemma\label{lem.ratio_parameter}
Suppose Assumptions~\ref{ass.prob}--\ref{ass.estimate} hold. At any iteration $k$ of Algorithm~\ref{alg.main}, we have $\bar\Xi_k \leq \bar\Xi_{k-1}$ and $\bar\Xi_k \leq \bar\Xi_k^{\trial}$. Meanwhile, if $\bar\Xi_k < \bar\Xi_{k-1}$, then $\bar\Xi_k \leq (1-\epsilon_{\xi})\bar\Xi_{k-1}$. Moreover, there exists a constant $\xi_{\min} > 0$ such that $\bar\Xi_k \geq \xi_{\min}$ for all iterations $k\in\NN$.
\elemma
\bproof
Firstly, by $\bar\xi_0 > 0$, \eqref{eq.ratio_parameter_update}, Algorithm~\ref{alg.main} and Lemma~\ref{lem.model_reduction_suff_large}, we know that $\{\bar\Xi_k\}\subset\RR_{>0}$. Moreover,~\eqref{eq.ratio_parameter_update} and Algorithm~\ref{alg.main} imply that $\{\bar\Xi_k\}$ is a monotonically decreasing sequence that either $\bar\Xi_k = \bar\Xi_{k-1} \leq \bar\Xi_k^{\trial}$ or $\bar\Xi_k = \min\left\{(1-\epsilon_{\xi})\bar\Xi_{k-1},\bar\Xi_k^{\trial}\right\} \leq \bar\Xi_k^{\trial} < \bar\Xi_{k-1}$, which concludes the first part of the statement. Meanwhile, Algorithm~\ref{alg.main} directly implies that $\bar\Xi_k < \bar\Xi_{k-1}$ can only happen when $\bar{D}_k \neq 0$. By~\eqref{eq.ratio_parameter_update}, if $\bar\Xi_k < \bar\Xi_{k-1}$ then $\bar\Xi_k = \min\left\{(1-\epsilon_{\xi})\bar\Xi_{k-1},\bar\Xi_k^{\trial}\right\} \leq (1-\epsilon_{\xi})\bar\Xi_{k-1}$, which concludes the second part of the statement. Lastly, from \eqref{eq.ratio_parameter_update}, Algorithm~\ref{alg.main}, and Lemma~\ref{lem.model_reduction_suff_large}, we know that $\bar\Xi_k^{\trial} > 0$.
When $\bar\Xi_k < \bar\Xi_{k-1}$, we know from \eqref{eq.ratio_parameter_update} and $\bar\Xi_k^{\trial} > 0$ that $\bar\Xi_{k-1} > \bar\Xi_k^{\trial}$ and $\bar\Xi_k = \min\left\{(1-\epsilon_{\xi})\bar\Xi_{k-1},\bar\Xi_k^{\trial}\right\} > (1-\epsilon_{\xi})\bar\Xi_k^{\trial}$. Therefore, given $\bar\xi_0 > 0$ from Algorithm~\ref{alg.main}, to show the existence of such a lower bound constant $\xi_{\min} > 0$ for the sequence $\{\bar\Xi_k\}$, we only need to prove that the sequence $\{\bar\Xi_k^{\trial}\}$ is uniformly bounded away from zero.

At any iteration $k\in\NN$, if $\bar{D}_k = 0$, then we have $\bar\Xi_k^{\trial} = +\infty$ from Algorithm~\ref{alg.main}. Otherwise, if $\bar{D}_k \neq 0$, then by \eqref{eq.ratio_parameter_update} we have
\bequationn
\bar\Xi_k^{\trial} = \frac{\Delta l(X_k,\bar\cT_k,\bar{G}_k,\bar{C}_k,\bar{J}_k,\bar{D}_k)}{\bar\cT_k\|\bar{D}_k\|^2} \geq \frac{\zeta\bar\cT_k\|\bar{D}_k\|^2}{2\bar\cT_k\|\bar{D}_k\|^2} = \frac{\zeta}{2} > 0,
\eequationn
where the first inequality follows Lemma~\ref{lem.model_reduction_suff_large} and the last inequality is from Assumption~\ref{ass.H}.
\eproof

Now we are ready to show that the adaptive step size selection strategy is well-defined.
\blemma\label{lem.stepsize}
Suppose Assumptions~\ref{ass.prob}--\ref{ass.estimate} hold. For any iteration $k$ of Algorithm~\ref{alg.main}, we always have $0 < \bar\cA_k^{\min} \leq \bar\cA_k^{\max} \leq \bar\cA_k^{\varphi}$ and $\varphi_k(\alpha) \leq 0$ for all $\alpha\in (0,\bar\cA_k^{\varphi}]$.
\elemma
\bproof
By Algorithm~\ref{alg.main}, \eqref{eq.stepsizes}, $(\eta,L,\Gamma)\in(0,1)\times\RR_{> 0}\times\RR_{\geq 0}$, $\{\beta_k\}\subset(0,1]$, Lemma~\ref{lem.merit_parameter} and Lemma~\ref{lem.ratio_parameter}, for any iteration $k$ of Algorithm~\ref{alg.main}, we always have
\bequation\label{eq.stepsize_intermediate_1}
\bar\cA_k^{\max} \leq \bar\cA_k^{\varphi} \quad \text{and} \quad \bar\cA_k^{\min} = \frac{2(1-\eta)\beta_k\bar\Xi_k\bar\cT_k}{\bar\cT_kL+\Gamma} > 0.
\eequation
Meanwhile, we also have from Algorithm~\ref{alg.main} that $\bar\cA_k^{\min} \leq \bar\cA_k^{\max}$ when $\bar{D}_k = 0$. To prove $\bar\cA_k^{\min} \leq \bar\cA_k^{\max}$ for all $k\in\NN{}$ where $\bar{D}_k\neq 0$, we define $\bar\cA_k^{\suff} := \min\left\{1, \frac{2(1-\eta)\beta_k\Delta l(X_k,\bar\cT_k,\bar{G}_k,\bar{C}_k,\bar{J}_k,\bar{D}_k)}{(\bar\cT_kL+\Gamma)\|\bar{D}_k\|^2}\right\}$ and then show that $\bar\cA_k^{\min} \leq \bar\cA_k^{\suff} \leq \bar\cA_k^{\varphi}$. From the positivity and monotonicity of $\{\bar\cT_k\}$ and $\{\bar\Xi_k\}$ (see Lemmas~\ref{lem.merit_parameter} and~\ref{lem.ratio_parameter}), we know that $0 < \bar\cT_k \leq \bar\tau_0$ and $0 < \bar\Xi_k \leq \bar\xi_0$ for all iterations $k$. By~\eqref{eq.beta_requirement},~\eqref{eq.stepsizes}, $(\eta,L,\Gamma)\in(0,1)\times\RR_{> 0}\times\RR_{\geq 0}$ and $\{\beta_k\}\subset(0,1]$, we have
\bequation\label{eq.alphamin_bound_1}
\bar\cA_k^{\min} = \frac{2(1-\eta)\beta_k\bar\Xi_k\bar\cT_k}{(\bar\cT_kL+\Gamma)} \leq \frac{2(1-\eta)\beta_k\bar\xi_0\bar\tau_0}{(\bar\tau_0L+\Gamma)} \leq 1.
\eequation
Meanwhile, using \eqref{eq.stepsizes}, $\bar\Xi_k \leq \bar\Xi_k^{\trial}$ (see Lemma~\ref{lem.ratio_parameter}), $\bar\cT_k > 0$, $(\eta,L,\Gamma)\in(0,1)\times\RR_{> 0}\times\RR_{\geq 0}$ and $\{\beta_k\}\subset(0,1]$, we also have
\bequation\label{eq.alphamin_bound_2}
\bar\cA_k^{\min} = \frac{2(1-\eta)\beta_k\bar\Xi_k\bar\cT_k}{(\bar\cT_kL+\Gamma)} \leq \frac{2(1-\eta)\beta_k\bar\Xi_k^{\trial}\bar\cT_k}{(\bar\cT_kL+\Gamma)} = \frac{2(1-\eta)\beta_k\Delta l(X_k,\bar\cT_k,\bar{G}_k,\bar{C}_k,\bar{J}_k,\bar{D}_k)}{(\bar\cT_kL+\Gamma)\|\bar{D}_k\|^2},
\eequation
where the last equality follows \eqref{eq.ratio_parameter_update} and the condition of $\bar{D}_k \neq 0$. Combining \eqref{eq.alphamin_bound_1} and~\eqref{eq.alphamin_bound_2}, we conclude that $\bar\cA_k^{\min} \leq \min\left\{1, \frac{2(1-\eta)\beta_k\Delta l(X_k,\bar\cT_k,\bar{G}_k,\bar{C}_k,\bar{J}_k,\bar{D}_k)}{(\bar\cT_kL+\Gamma)\|\bar{D}_k\|^2}\right\} = \bar\cA_k^{\suff}$ for all iterations $k$ such that $\bar{D}_k \neq 0$.

Next, we are going to prove $\bar\cA_k^{\suff} \leq \bar\cA_k^{\varphi}$ by considering two cases on the value of $\bar\cA_k^{\suff}$.
\\
\noindent \textbf{Case (i):} when $\bar\cA_k^{\suff} = 1 \leq \frac{2(1-\eta)\beta_k\Delta l(X_k,\bar\cT_k,\bar{G}_k,\bar{C}_k,\bar{J}_k,\bar{D}_k)}{(\bar\cT_kL+\Gamma)\|\bar{D}_k\|^2}$ (or equivalently, $(\bar\cT_kL+\Gamma)\|\bar{D}_k\|^2 \leq 2(1-\eta)\beta_k\Delta l(X_k,\bar\cT_k,\bar{G}_k,\bar{C}_k,\bar{J}_k,\bar{D}_k)$), it follows \eqref{eq.alphaphi} that
\bequationn
\baligned
\varphi_k(\bar\cA_k^{\suff}) &= (\eta - 1)\bar\cA_k^{\suff}\beta_k\Delta l(X_k,\bar\cT_k,\bar{G}_k,\bar{C}_k,\bar{J}_k,\bar{D}_k) + (|1-\bar\cA_k^{\suff}| - (1-\bar\cA_k^{\suff}))\|\bar{C}_k\|_1 \\
&\quad\quad+ \frac{1}{2}(\bar\cT_k L+\Gamma)(\bar\cA_k^{\suff})^2\|\bar{D}_k\|^2 \\
&= (\eta - 1)\beta_k\Delta l(X_k,\bar\cT_k,\bar{G}_k,\bar{C}_k,\bar{J}_k,\bar{D}_k) + \frac{1}{2}(\bar\cT_kL+\Gamma)\|\bar{D}_k\|^2 \leq 0.
\ealigned
\eequationn
\noindent \textbf{Case (ii):} when $\bar\cA_k^{\suff} = \frac{2(1-\eta)\beta_k\Delta l(X_k,\bar\cT_k,\bar{G}_k,\bar{C}_k,\bar{J}_k,\bar{D}_k)}{(\bar\cT_kL+\Gamma)\|\bar{D}_k\|^2} < 1$ (or equivalently, $(\bar\cT_kL+\Gamma)\|\bar{D}_k\|^2 > 2(1-\eta)\beta_k\Delta l(X_k,\bar\cT_k,\bar{G}_k,\bar{C}_k,\bar{J}_k,\bar{D}_k)$), it follows \eqref{eq.alphaphi} that
\bequationn
\baligned
\varphi_k(\bar\cA_k^{\suff}) &= (\eta - 1)\bar\cA_k^{\suff}\beta_k\Delta l(X_k,\bar\cT_k,\bar{G}_k,\bar{C}_k,\bar{J}_k,\bar{D}_k) + (|1-\bar\cA_k^{\suff}| - (1-\bar\cA_k^{\suff}))\|\bar{C}_k\|_1 \\
&\quad\quad+ \frac{1}{2}(\bar\cT_k L+\Gamma)(\bar\cA_k^{\suff})^2\|\bar{D}_k\|^2 \\
&= (\eta - 1)\bar\cA_k^{\suff}\beta_k\Delta l(X_k,\bar\cT_k,\bar{G}_k,\bar{C}_k,\bar{J}_k,\bar{D}_k) + \bar\cA_k^{\suff}\cdot\frac{\bar\cA_k^{\suff}}{2}(\bar\cT_k L+\Gamma)\|\bar{D}_k\|^2 \\
&= (\eta - 1)\bar\cA_k^{\suff}\beta_k\Delta l(X_k,\bar\cT_k,\bar{G}_k,\bar{C}_k,\bar{J}_k,\bar{D}_k) + (1-\eta)\bar\cA_k^{\suff}\beta_k\Delta l(X_k,\bar\cT_k,\bar{G}_k,\bar{C}_k,\bar{J}_k,\bar{D}_k) = 0.
\ealigned
\eequationn
Combining \textbf{Cases (i)} and \textbf{(ii)}, we always have $\varphi_k(\bar\cA_k^{\suff}) \leq 0$. Then by using the definition of $\bar\cA_k^{\varphi}$ in \eqref{eq.stepsizes}, we have $\bar\cA_k^{\suff} \leq \bar\cA_k^{\varphi}$ for all iterations $k$ such that $\bar{D}_k \neq 0$. After summarizing results above, we know for all iterations $k$ where $\bar{D}_k \neq 0$, $\bar\cA_k^{\min} \leq \bar\cA_k^{\suff} \leq \bar\cA_k^{\varphi}$ always holds. Therefore, by Algorithm~\ref{alg.main},~\eqref{eq.stepsizes} and~\eqref{eq.stepsize_intermediate_1}, we conclude that $0 < \bar\cA_k^{\min} \leq \bar\cA_k^{\max} \leq \bar\cA_k^{\varphi}$ for all iterations $k$.

Next, we are going to prove $\varphi_k(\alpha) \leq 0$ for all step sizes $\alpha\in (0,\bar\cA_k^{\varphi}]$. When $\bar{D}_k = 0$, by \eqref{eq.linear_system}, \eqref{eq.linear_model_reduction}, and \eqref{eq.alphaphi}, we know $\varphi_k(\alpha) = 0$ for all $\alpha\in\RR$, so the statement trivially holds. On the other hand, when $\bar{D}_k \neq 0$, we notice that $\varphi_k:\RR\to\RR$
is a convex function in \(\alpha\) because it is a nonnegative weighted sum of \(\abs{1 - \alpha}\), \(\alpha^2\), and additional linear terms. Consequently, its sublevel set \(\{\alpha \in \reals \colon \varphi_k(\alpha) \leq 0\}\) is convex. Since \(\varphi_k(0) = 0\) and \(\varphi_k(\bar\cA_k^{\varphi}) = 0\) by construction, both points belong to the sublevel set, and accordingly the interval \([0, \bar\cA_k^{\varphi}]\) is contained in the sublevel set, which completes the proof.
\eproof

Our next lemma provides a critical upper bound on the merit function decrease.
\blemma\label{lem.merit_function_decrease}
Suppose Assumptions~\ref{ass.prob}--\ref{ass.estimate} hold. For all iterations $k$, we always have
\begin{equation}
\begin{aligned}
\phi(X_k + \bar\cA_k \bar{D}_k,\bar\cT_k) - \phi(X_k,\bar\cT_k)
\leq & -\bar\cA_k\Delta l(X_k,\bar\cT_k,\nabla f(X_k),c(X_k),\nabla c(X_k)^T,D_k) \\
& \quad + \bar\cA_k\bar\cT_k\nabla f(X_k)^T(\bar{D}_k - D_k)
  + \bar\cA_k\|\nabla c(X_k)^T(D_k - \bar{D}_k)\|_1 \\
& \quad \inrevise{+ 2 \bar\cA_k \norm{c(X_k) - \bar{C}_k}_1} + (1-\eta)\bar\cA_k\beta_k\Delta l(X_k,\bar\cT_k,\bar{G}_k,\bar{C}_k,\bar{J}_k,\bar{D}_k).
\end{aligned}
\end{equation}
\elemma
\bproof
By Assumption~\ref{ass.prob}, \eqref{eq.Lipschitz_continuity}, \eqref{eq.merit_func}, \eqref{eq.linear_model_reduction} and the triangle inequality, we know that for any $\bar\cA_k \in [\bar\cA_k^{\min},\bar\cA_k^{\max}]$,
\begin{align*}
&\phi(X_k + \bar\cA_k \bar{D}_k,\bar\cT_k) - \phi(X_k,\bar\cT_k) \\
= \ &\bar\cT_k(f(X_k + \bar\cA_k \bar{D}_k) - f(X_k)) + (\|c(X_k + \bar\cA_k \bar{D}_k)\|_1 - \|c(X_k)\|_1) \\
\leq \ &\bar\cA_k\bar\cT_k\nabla f(X_k)^T\bar{D}_k + \frac{L}{2}\bar\cT_k\bar\cA_k^2\|\bar{D}_k\|^2 + \|c(X_k) + \bar\cA_k\nabla c(X_k)^T\bar{D}_k\|_1 + \frac{\Gamma}{2}\bar\cA_k^2\|\bar{D}_k\|^2 - \|c(X_k)\|_1 \\
= \ &-\bar\cA_k\Delta l(X_k,\bar\cT_k,\nabla f(X_k),c(X_k),\nabla c(X_k)^T,D_k) + \bar\cA_k\bar\cT_k\nabla f(X_k)^T(\bar{D}_k - D_k) + \|c(X_k) + \bar\cA_k\nabla c(X_k)^T\bar{D}_k\|_1 \\
&\quad- (1-\bar\cA_k)\|c(X_k)\|_1 + \frac{1}{2}(\bar\cT_k L+\Gamma)\bar\cA_k^2\|\bar{D}_k\|^2 \\
=\ &-\bar\cA_k\Delta l(X_k,\bar\cT_k,\nabla f(X_k),c(X_k),\nabla c(X_k)^T,D_k) + \bar\cA_k\bar\cT_k\nabla f(X_k)^T(\bar{D}_k - D_k) \\
&\quad + \|(1-\bar\cA_k)c(X_k) + \bar\cA_k (c(X_k) + \nabla c(X_k)^T\bar{D}_k)\|_1 - (1-\bar\cA_k)\|c(X_k)\|_1 + \frac{1}{2}(\bar\cT_k L+\Gamma)\bar\cA_k^2\|\bar{D}_k\|^2 \\
\leq\ &-\bar\cA_k\Delta l(X_k,\bar\cT_k,\nabla f(X_k),c(X_k),\nabla c(X_k)^T,D_k) + \bar\cA_k\bar\cT_k\nabla f(X_k)^T(\bar{D}_k - D_k) + \bar\cA_k\|c(X_k) + \nabla c(X_k)^T\bar{D}_k\|_1 \\
&\quad + (|1-\bar\cA_k| - (1-\bar\cA_k))\|c(X_k)\|_1 + \frac{1}{2}(\bar\cT_k L+\Gamma)\bar\cA_k^2\|\bar{D}_k\|^2 \\
= \ & -\bar\cA_k\Delta l(X_k,\bar\cT_k,\nabla f(X_k),c(X_k),\nabla c(X_k)^T,D_k) + \bar\cA_k\bar\cT_k\nabla f(X_k)^T(\bar{D}_k - D_k)
+ \bar\cA_k\|\nabla c(X_k)^T(D_k - \bar{D}_k)\|_1
\\
& \quad \inrevise{ + (|1-\bar\cA_k| - (1-\bar\cA_k)) \left(\left( \|c(X_k)\|_1 - \|\bar{C}_k\|_1 \right) + \|\bar{C}_k\|_1\right)} + \frac{1}{2}(\bar\cT_k L+\Gamma)\bar\cA_k^2\|\bar{D}_k\|^2
\\
\leq\ & -\bar\cA_k\Delta l(X_k,\bar\cT_k,\nabla f(X_k),c(X_k),\nabla c(X_k)^T,D_k) + \bar\cA_k\bar\cT_k\nabla f(X_k)^T(\bar{D}_k - D_k) + \bar\cA_k\|\nabla c(X_k)^T(D_k - \bar{D}_k)\|_1 \\
&\quad \inrevise{+ 
(|1-\bar\cA_k| - (1-\bar\cA_k)) \|c(X_k) - \bar{C}_k\|_1} + (1-\eta) \bar\cA_k\beta_k\Delta l(X_k,\bar\cT_k,\bar{G}_k,\bar{C}_k,\bar{J}_k,\bar{D}_k)
\\
\inrevise{\leq} \ & \inrevise{ -\bar\cA_k\Delta l(X_k,\bar\cT_k,\nabla f(X_k),c(X_k),\nabla c(X_k)^T,D_k) + \bar\cA_k\bar\cT_k\nabla f(X_k)^T(\bar{D}_k - D_k) + \bar\cA_k\|\nabla c(X_k)^T(D_k - \bar{D}_k)\|_1 } \\
&\quad \inrevise{+
2 \bar\cA_k \|c(X_k) - \bar{C}_k\|_1 + (1-\eta) \bar\cA_k\beta_k\Delta l(X_k,\bar\cT_k,\bar{G}_k,\bar{C}_k,\bar{J}_k,\bar{D}_k),}
\end{align*}
where the last equality is from $c(X_k) + \nabla c(X_k)^TD_k = 0$, \inrevise{the second last inequality follows $[\bar\cA_k^{\min},\bar\cA_k^{\max}]\subset(0,\bar\cA_k^{\varphi}]$, \eqref{eq.stepsizes}, \eqref{eq.alphaphi}, Lemma~\ref{lem.stepsize}, triangle inequality, and $z \leq |z|$ for all $z \in \RR$, and the last inequality is due to $|1 - z| - (1 - z) \leq 2 z$ for any $z \in \RR_{\geq 0}$.}
\eproof

Lemma~\ref{lem.merit_function_decrease} is with a similar spirit as \cite[Lemma~3.7]{BeraCurtRobiZhou21}, however, there \inrevise{are additional terms, i.e., $\bar\cA_k\|\nabla c(X_k)^T(D_k - \bar{D}_k)\|_1$ and $2\bar\cA_k\|c(X_k) -\bar{C}_k\|_1$,} appearing in the upper bounding function of Lemma~\ref{lem.merit_function_decrease} because we are focusing on \textit{stochastic} constrained problems in this paper. In particular, for \textit{deterministic} constrained problems, where $\bar{C}_k = c(X_k)$ and $\bar{J}_k = \nabla c(X_k)^T$, \inrevise{ we directly have $2\bar\cA_k\|c(X_k) -\bar{C}_k\|_1 = 0$, and meanwhile}, it follows \eqref{eq.linear_system} and \eqref{eq.linear_system_true} that
\bequationn
\bar\cA_k\|\nabla c(X_k)^T(D_k - \bar{D}_k)\|_1 = \bar\cA_k\|\nabla c(X_k)^TD_k - \bar{J}_k\bar{D}_k\|_1 = \bar\cA_k\|\bar{C}_k - c(X_k)\|_1 = 0.
\eequationn
This observation demonstrates one aspect of the difference between the analysis of our paper and the prior work. More significant distinctions, including $\{\bar{D}_k\}$ performing as a biased estimator of $\{D_k\}$ and the behavior of stochastic merit parameters $\{\bar\cT_k\}$, are discussed in the following sections (Sections~\ref{sec.converg_results}--\ref{sec.merit_parameter_behaviors}) in more detail.

\subsection{Convergence Results}\label{sec.converg_results}

In this subsection, we are going to show asymptotic convergence performance of Algorithm~\ref{alg.main} under the condition that event $\cE_{\tau\xi} := \cE(\bar\tau_{\min},k_{\max},f_{\max})$ occurs, where $(\bar\tau_{\min},k_{\max},f_{\max}) \in \RR_{>0}\times \NN \times \RR$ and
\bequation\label{eq.event_asy}
\baligned
\cE(\bar\tau_{\min},k_{\max},f_{\max}) := &\left\{\inrevise{f(X_{k_{\max}})} \leq f_{\max}\text{ and there exist }(\bar\cT',\bar\Xi',k') \in \RR_{>0} \times \RR_{>0} \times [k_{\max}]\text{ satisfying } \right. \\
&\ \left.\cT_k^{\trial} \geq \bar\cT_k = \bar\cT' \geq \bar\tau_{\min} > 0\text{ and }\bar\Xi_k = \bar\Xi' > 0\text{ for all iterations }k\geq k'\right\}.
\ealigned
\eequation
We note that event $\cE_{\tau\xi} := \cE(\bar\tau_{\min},k_{\max},f_{\max})$ describes behaviors of \inrevise{$f(X_{k_{\max}})$}, stochastic merit parameters $\{\bar\cT_k\}$ and stochastic ratio parameters $\{\bar\Xi_k\}$. In particular, the event $\cE_{\tau\xi}$ requires that \inrevise{$f(X_{k_{\max}})$}, the objective value at iterate \inrevise{$X_{k_{\max}}$}, is not too large (no more than $f_{\max}$) and ratio parameters $\{\bar\Xi_k\}$ stay constant for all sufficiently large iterations, while these requirements are not restrictive given Assumptions~\ref{ass.prob}--\ref{ass.estimate} and Lemma~\ref{lem.ratio_parameter}. In addition, the event $\cE_{\tau\xi}$ asks for stochastic merit parameters $\{\bar\cT_k\}$ reaching a constant and sufficiently small value (compared to $\{\cT_k^{\trial}\}$) for all large enough iterations. However, such nice properties may not always hold. In fact, similar to discussions in~\cite[Section~3.2.2]{BeraCurtRobiZhou21} and~\cite[Section~3.3.1]{CurtRobiZhou24} on \textit{deterministic} constrained \textit{stochastic} optimization problems, our $\{\bar\cT_k\}$ sequence may behave in other two ways:
\bequation\label{eq.poor_merit_parameters_asymptotic}
\baligned
(i)\ &\{\bar\cT_k\}\text{ converges to zero, or} \\
(ii)\ &\{\bar\cT_k\}\text{ eventually stays at a constant but not sufficiently small value compared to }\{\cT_k^{\trial}\},\text{ i.e.,}\\
&\text{ there exists an infinite set }\cK_{\bar\tau,big}\subseteq\NN\text{ such that }\bar\cT_k > \cT_k^{\trial}\text{ for all }k\in\cK_{\bar\tau,big}.
\ealigned
\eequation
Note that there are two cases introduced in~\eqref{eq.poor_merit_parameters_asymptotic}. We will show in Lemmas~\ref{lem.tau_zero_asymptotic} and~\ref{lem.tau_large_asymptotic} that under additional reasonable assumptions, case $(i)$ never occurs while case $(ii)$ at most happens with probablity zero. These results, together with aforementioned discussions on $\inrevise{f(X_{k_{\max}})} \leq f_{\max}$ and $\{\bar\Xi_k\}$, illustrate the fact that the event $\cE_{\tau\xi}$ is not excessively restrictive and indeed occurs frequently in many realistic problems. We defer further detailed discussions on these behaviors of $\{\bar\cT_k\}$ to Section~\ref{sec.merit_parameter_behaviors}. For the rest of this subsection, we prove theoretical results under the condition that event $\cE_{\tau\xi}$ occurs, which is described in the following assumption.

\begin{ass}\label{ass.event_asymptotic}
During the runs of Algorithm~\ref{alg.main}, event $\cE_{\tau\xi}:=\cE(\bar\tau_{\min},k_{\max},f_{\max})$ (defined in~\eqref{eq.event_asy}) occurs with given constants $(\bar\tau_{\min},k_{\max},f_{\max})\in\RR_{>0}\times\NN\times\RR$.
\end{ass}
Due to the focus on event \(\cE_{\tau\xi}\), we accordingly restrict the filtration \(\{\cF_k\}\) by
\[
\cF_k := \{E \cap \cE_{\tau\xi} \colon E \in \cG_k\}
\quad
\text{ for all \(k \in \NN\) },
\]
and we define the following notation
\begin{align}
\PP_k[\cdot] := \PP[\cdot | \cF_k]
\quad \text{and} \quad
\EE_k[\cdot] := \EE[\cdot | \cF_k]. \label{eq:def-Pk-Ek}
\end{align}
In addition to Assumption~\ref{ass.event_asymptotic}, we make the following assumption on random variables $\{(\bar{G}_k,\bar{C}_k,\bar{J}_k)\}$.

\begin{ass}\label{ass.estimate_asymptotic}
For every iteration $k\in\NN$, conditioned on filtration $\cF_{k}$, the objective gradient estimate $\bar{G}_k$, the constraint function estimate $\bar{C}_k$, and the constraint jacobian estimate $\bar{J}_k$ are all unbiased estimators of their corresponding deterministic quantities evaluated at $X_k$, i.e., for all $k\in\NN$,
\bequation\label{eq.unbiased_estimates}
\EE_k[\bar{G}_k] = \nabla f(X_k), \quad \EE_k[\bar{C}_k] = c(X_k),\quad \text{and}\quad \EE_k[\bar{J}_k] = \nabla c(X_k)^T.
\eequation
Moreover, for any $k\in\NN$, conditioned on filtration $\cF_k$, random elements $(\bar{G}_k,\bar{C}_k)$ and $\bar{J}_k$ are independent, and their variances satisfy
\bequation\label{eq.variance}
\EE_k[\|\bar{G}_k - \nabla f(X_k)\|^2]  \leq \rho^g_k,\quad \EE_k[\|\bar{C}_k - c(X_k)\|^2] \leq \rho^c_k,\quad
\text{and}\quad \EE_k[\|\bar{J}_k - \nabla c(X_k)^T\|_F^2] \leq \rho^j_k,
\eequation
where non-negative parameters $\{\rho^g_k\}$, $\{\rho^c_k\}$, and $\{\rho^j_k\}$ are universally upper bounded by $\rhomax > 0$.
Additionally, \inrevise{conditioned on filtration \(\cF_k\)}, \(\nabla f(X_k)\), \(c(X_k)\), and \(\nabla c(X_k)^{\top}\) are always in the support of the distributions of $\bar{G}_k$, $\bar{C}_k$, and $\bar{J}_k$,
respectively; i.e., for all $k\in\NN$, it holds \inrevise{almost surely} that, conditioned on filtration \(\cF_k\),
\begin{align*}
\nabla f(X_k)\in\supp{\bar{G}_k}, \quad c(X_k)\in\supp{\bar{C}_k}, \ \ \text{and} \ \ \nabla c(X_k)^T \in \supp{\bar{J}_k},
\end{align*}
\inrevise{that is, for any \(r > 0\)}
\begin{align*}
\inrevise{\Prob_k[ || \nabla f(X_k) - \bar{G}_k || < r ] > 0, \quad
\Prob_k[ || c(X_k) - \bar{C}_k || < r ] > 0, \ \ \text{and} \ \
\Prob_k[ || \nabla c(X_k)^\top - \bar{J}_k || < r ] > 0.}
\end{align*}
\end{ass}

Assumption~\ref{ass.estimate_asymptotic} requires that random variables $\{(\bar{G}_k,\bar{C}_k,\bar{J}_k)\}$ are all unbiased estimators of their corresponding deterministic quantities, and meanwhile, their (conditional) variances are controlled by predetermined parameter sequences. \inrevise{Both of these conditions are common in constrained stochastic optimization literature~\cite{BeraCurtRobiZhou21,CurtJianWang24,CurtRobiOnei24,CurtRobiZhou23,FangNaMahoKola24}.} We also assume that $\bar{J}_k$ is (conditionally) independent to $(\bar{G}_k,\bar{C}_k)$, which facilitates the uncertainty quantification of the search direction  (see Lemma~\ref{lem.solution_bias}) by making the parameter matrix and the vector in~\eqref{eq.linear_system}, i.e., $\begin{bmatrix} H_k & \bar{J}_k^T \\ \bar{J}_k & 0 \end{bmatrix}$ and $\begin{bmatrix} \bar{G}_k \\ \bar{C}_k \end{bmatrix}$, (conditionally) independent to each other. Additionally, we presume a mild condition in Assumption~\ref{ass.estimate_asymptotic} that $\{(\nabla f(X_k), c(X_k), \nabla c(X_k)^T)\}$ are within the support of $\{(\bar{G}_k, \bar{C}_k, \bar{J}_k)\}$, which will be useful to ensure that the lower bound provided in Lemma~\ref{lem.prob-tau-good-k} is nontrivial and asymptotically tight when \(\{(\rho^g_k, \rho^c_k, \rho^j_k)\}\) diminishes. 
\inrevise{Such assumption on support is not restrictive. Specifically, in Appendix~\ref{sec:apx-supp-asm}, we show that it holds whenever the sampling errors of \(\nabla f(X_k)\), \(c(X_k)\), and \(\nabla c(X_k)^\top\) have zero mean and \emph{symmetric support}. Intuitively, this requires that the errors are centered around zero in expectation, and whenever the noise can take a realization \(z\) with positive probability, it is also possible to take \(-z\) (though not necessarily with the same probability). This property is satisfied by many commonly used noise models, including the white Gaussian noise~\cite[Chapter~1.2]{kay1993statistical}, zero-mean uniform noise supported on \(\ell_p\) norm ball~\cite{franceschi2018robustness}, as well as zero-mean discrete noise supported on a symmetric lattice~\cite{zamir1996lattice}.
}

Next, we start with proving a result that shows the model reduction $\Delta l(X_k,\bar\cT_k,\nabla f(X_k),c(X_k),\nabla c(X_k)^T,D_k)$ is always non-negative for any sufficiently large iteration $k$.

\blemma\label{lem.true_model_reduction_bound}
Suppose Assumptions~\ref{ass.prob}--\ref{ass.event_asymptotic} hold. For all iterations $k \geq k_{\max}$, we always have
\bequationn
\Delta l(X_k,\bar\cT_k,\nabla f(X_k),c(X_k),\nabla c(X_k)^T,D_k) \geq 0.
\eequationn
\elemma
\bproof
Using the definition of event $\cE_{\tau\xi}$ (which includes $0 < \bar\cT_k \leq \cT_k^{\trial}$ for any sufficiently large iteration~$k$), we have that for all iterations $k \geq k_{\max}$,
\bequationn
\Delta l(X_k,\bar\cT_k,\nabla f(X_k),c(X_k),\nabla c(X_k)^T,D_k) \geq \tfrac{\zeta}{2}\bar\cT_k\|D_k\|^2 + \sigma\|c(X_k)\|_1 \geq 0,
\eequationn
where the first inequality follows Lemma~\ref{lem.model_reduction_suff_large_true}.
\eproof

We then provide lower bounds and upper bounds for adaptive step size intervals. We note that both lower and upper bounds of the step size $\bar\cA_k$ are proportional to $\beta_k$.
\blemma\label{lem.stepsize_interval}
Suppose Assumptions~\ref{ass.prob}--\ref{ass.event_asymptotic} hold. For all iterations $k$,
\bequationn
\bar\cA_k \in [\bar\cA_k^{\min},\bar\cA_k^{\max}] \subseteq \left[\frac{2(1-\eta)\bar\Xi'\bar\cT'}{\bar\cT'L + \Gamma}\beta_k, \left(\frac{2(1-\eta)\bar\xi_0\bar\tau_0}{\bar\tau_0L + \Gamma}+ \theta\right)\beta_k\right].
\eequationn
Moreover, for all iterations $k \geq k_{\max}$,
\bequationn
\bar\cA_k \in [\bar\cA_k^{\min},\bar\cA_k^{\max}] \subseteq \left[\frac{2(1-\eta)\bar\Xi'\bar\cT'}{\bar\cT'L + \Gamma}\beta_k, \left(\frac{2(1-\eta)\bar\Xi'\bar\cT'}{\bar\cT'L + \Gamma}+ \theta\right)\beta_k\right].
\eequationn
\elemma
\bproof
By Algorithm~\ref{alg.main} and Lemma~\ref{lem.stepsize}, we know that $\bar\cA_k \in [\bar\cA_k^{\min},\bar\cA_k^{\max}]$ is well-defined for any $k\in\NN{}$. When event $\cE_{\tau\xi}$ occurs, from Algorithm~\ref{alg.main} and~\eqref{eq.stepsizes}, we further have
\bequationn
\bar\cA_k^{\min} = \frac{2(1-\eta)\beta_k\bar\Xi_k\bar\cT_k}{(\bar\cT_kL+\Gamma)} \geq \frac{2(1-\eta)\beta_k\bar\Xi'\bar\cT'}{(\bar\cT'L+\Gamma)}.
\eequationn
Meanwhile, by Algorithm~\ref{alg.main}, \eqref{eq.stepsizes}, Lemma~\ref{lem.merit_parameter}, and Lemma~\ref{lem.ratio_parameter}, for all iterations $k$,
\bequationn
\bar\cA_k^{\max} \leq \bar\cA_k^{\min} + \theta\beta_k = \frac{2(1-\eta)\beta_k\bar\Xi_k\bar\cT_k}{(\bar\cT_kL+\Gamma)} + \theta\beta_k \leq \left(\frac{2(1-\eta)\bar\xi_0\bar\tau_0}{\bar\tau_0L + \Gamma}+ \theta\right)\beta_k.
\eequationn
When $k \geq k_{\max}$, by additional information from Assumption~\ref{ass.event_asymptotic}, it holds that
\bequationn
\bar\cA_k^{\max} \leq \bar\cA_k^{\min} + \theta\beta_k = \left(\frac{2(1-\eta)\bar\Xi'\bar\cT'}{\bar\cT'L + \Gamma}+ \theta\right)\beta_k.
\eequationn
By combining all aforementioned results, we conclude the statement.
\eproof

The following lemma shows the relationship between matrices used in~\eqref{eq.linear_system} and~\eqref{eq.linear_system_true}, i.e., $\bbmatrix H_k & \bar{J}_k^T \\ \bar{J}_k & 0 \ebmatrix$ and $\bbmatrix H_k & \nabla c(X_k) \\ \nabla c(X_k)^T & 0 \ebmatrix$. In particular, we prove that the difference between the inverse of these two matrices is uniformly bounded. Meanwhile, when the difference between these two matrices is sufficiently small, we can use the norm of their difference to provide an upper bound on the difference between their inverses.

\begin{lem}\label{lem:inv-mat-err}
Suppose Assumptions~\ref{ass.prob}--\ref{ass.estimate} hold. For all iterations $k$, we always have
\bequationn
\left\|\bbmatrix H_k & \bar{J}_k^T \\ \bar{J}_k & 0 \ebmatrix^{-1} - \bbmatrix H_k & \nabla c(X_k) \\ \nabla c(X_k)^T & 0 \ebmatrix^{-1}\right\| \leq \frac{2}{q_{\min}},
\eequationn
where $q_{\min}$ is defined in Lemma~\ref{lem.singular_values}. Moreover, if $\left\|\bbmatrix H_k & \bar{J}_k^T \\ \bar{J}_k & 0 \ebmatrix - \bbmatrix H_k & \nabla c(X_k) \\ \nabla c(X_k)^T & 0 \ebmatrix\right\| < \frac{q_{\min}}{3}$, then
\bequationn
\left\|\bbmatrix H_k & \bar{J}_k^T \\ \bar{J}_k & 0 \ebmatrix^{-1} - \bbmatrix H_k & \nabla c(X_k) \\ \nabla c(X_k)^T & 0 \ebmatrix^{-1}\right\| \leq \frac{3}{2q_{\min}^2}\cdot\left\|\bbmatrix H_k & \bar{J}_k^T \\ \bar{J}_k & 0 \ebmatrix - \bbmatrix H_k & \nabla c(X_k) \\ \nabla c(X_k)^T & 0 \ebmatrix\right\|.
\eequationn
\end{lem}

\bproof
See Appendix~\ref{sec:apx-inv-mat-err}.
\eproof

Based on the result of Lemma~\ref{lem:inv-mat-err}, we use the next two lemmas to present some critical results that reveal the relationship between $\bar{D}_k$ and $D_k$, i.e., primal solutions of \eqref{eq.linear_system} and \eqref{eq.linear_system_true} at iteration $k$.

\blemma\label{lem.solution_bias}
Suppose Assumptions~\ref{ass.prob}--\ref{ass.estimate_asymptotic} hold. There exists a fixed constant $\omega_1 > 0$ such that  for all iterations $k\in\NN{}$ of Algorithm~\ref{alg.main},
\bequationn
\baligned
\left\|\EE_k\left[\bar{D}_k - D_k\right]\right\| &\leq \left\|\EE_k\left[\bbmatrix\bar{D}_k \\ \bar{Y}_k\ebmatrix - \bbmatrix D_k \\ Y_k \ebmatrix\right]\right\| \leq \omega_1\cdot \sqrt{\rho_k^j} \\
\text{and}\quad\EE_k \left[\left\|\bar{D}_k - D_k\right\|\right] &\leq \EE_k \left[\left\|\bbmatrix \bar{D}_k \\ \bar{Y}_k \ebmatrix - \bbmatrix D_k \\ Y_k \ebmatrix \right\|\right] \leq \frac{\sqrt{\rho_k^g + \rho_k^c}}{q_{\min}} + \omega_1\cdot \sqrt{\rho_k^j},
\ealigned
\eequationn
where \inrevise{$(q_{\min},\rho_k^g,\rho_k^c,\rho_k^j)$ are parameteres} defined in Lemma~\ref{lem.singular_values} and Assumption~\ref{ass.estimate_asymptotic}.
\elemma
\bproof
See Appendix~\ref{sec:apx-sol-est}.
\eproof

\blemma\label{lem.solution_variance}
Suppose Assumptions~\ref{ass.prob}--\ref{ass.estimate_asymptotic} hold. There exist fixed constants $\{\omega_2,\omega_3,\omega_4\}\subset\RR{}_{>0}$ such that for all iterations $k$ of Algorithm~\ref{alg.main}, we always have
\bequationn
\baligned
\EE_k \left[\left\|\bar{D}_k - D_k\right\|^2\right] &\leq \EE_k\left[\left\|\bbmatrix \bar{D}_k \\ \bar{Y}_k \ebmatrix - \bbmatrix D_k \\ Y_k \ebmatrix \right\|^2\right] \leq  \frac{\rho_k^g + \rho_k^c}{q_{\min}^2} + \omega_2\cdot\rho_k^j + \omega_3\cdot \sqrt{\rho_k^j(\rho_k^g+\rho_k^c)} \\
\text{and}\quad \left|\EE_k\left[\bar{G}_k^T\bar{D}_k - \nabla f(X_k)^TD_k\right]\right| &\leq \frac{\rho_k^g}{2} + \frac{\rho_k^g + \rho_k^c}{2q_{\min}^2} + \omega_4\cdot\sqrt{\rho_k^j} + \frac{\omega_3}{2}\cdot\sqrt{\rho_k^j(\rho_k^g+\rho_k^c)},
\ealigned
\eequationn
where \inrevise{$(q_{\min},\rho_k^g,\rho_k^c,\rho_k^j)$ are parameteres} defined in Lemma~\ref{lem.singular_values} and Assumption~\ref{ass.estimate_asymptotic}.
\elemma
\bproof
See Appendix~\ref{sec:apx-pf-sol-var}.
\eproof

We next present an upper bound on the expected value of the difference between stochastic model reduction estimate and its corresponding deterministic quantity (with $\cT_k$ replaced by $\bar\cT_k$), i.e., $\Delta l(X_k,\bar\cT_k,\bar{G}_k,\bar{C}_k,\bar{J}_k,\bar{D}_k)$ and $\Delta l(X_k,\bar\cT_k,\nabla f(X_k),c(X_k),\nabla c(X_k)^T,D_k)$, for all sufficiently large iterations.
\blemma\label{lem.model_reduction_comparison}
Suppose Assumptions~\ref{ass.prob}--\ref{ass.estimate_asymptotic} hold. There exist fixed constants $\{\omega_5,\omega_6\}\subset\RR{}_{>0}$ such that for all iterations $k \geq k_{\max}$ of Algorithm~\ref{alg.main}, we always have
\bequationn
\baligned
\ &\left|\EE_k\left[\Delta l(X_k,\bar\cT_k,\bar{G}_k,\bar{C}_k,\bar{J}_k,\bar{D}_k) - \Delta l(X_k,\bar\cT_k,\nabla f(X_k),c(X_k),\nabla c(X_k)^T,D_k)\right]\right| \\
\leq \ &\bar\tau_0\left(\omega_5\cdot \rho_k^g + \omega_6\cdot \sqrt{\rho_k^c} + \omega_4\cdot\sqrt{\rho_k^j} + \frac{\omega_3}{2}\cdot\sqrt{\rho_k^j(\rho_k^g+\rho_k^c)} \right),
\ealigned
\eequationn
where \inrevise{$(\rho_k^g,\rho_k^c,\rho_k^j,\omega_3,\omega_4)$ are parameters} defined in Assumption~\ref{ass.estimate_asymptotic} and Lemma~\ref{lem.solution_variance}.
\elemma
\bproof
From Assumption~\ref{ass.event_asymptotic}, we know $\bar\cT_k = \bar\cT'$ for all iterations $k \geq k_{\max}$. Moreover, by \eqref{eq.linear_model_reduction} and the triangle inequality, for all iterations $k \geq k_{\max}$,
\begin{align*}
\ &\left|\EE_k\left[\Delta l(X_k,\bar\cT_k,\bar{G}_k,\bar{C}_k,\bar{J}_k,\bar{D}_k) - \Delta l(X_k,\bar\cT_k,\nabla f(X_k),c(X_k),\nabla c(X_k)^T,D_k)\right]\right| \\
= \ &\left|\EE_k\left[\Delta l(X_k,\bar\cT',\bar{G}_k,\bar{C}_k,\bar{J}_k,\bar{D}_k) - \Delta l(X_k,\bar\cT',\nabla f(X_k),c(X_k),\nabla c(X_k)^T,D_k)\right]\right| \\
= \ &\left|\EE_k\left[\left(-\bar\cT'\bar{G}_k^T\bar{D}_k + \|\bar{C}_k\|_1\right) - \left(-\bar\cT'\nabla f(X_k)^TD_k + \|c(X_k)\|_1\right)\right]\right| \\
= \ &\left|\EE_k\left[\bar\cT'\left(\nabla f(X_k)^TD_k-\bar{G}_k^T\bar{D}_k\right) + \left(\|\bar{C}_k\|_1 - \|c(X_k)\|_1\right)\right]\right| \\
\leq \ &\bar\cT'\cdot \left|\EE_k\left[\nabla f(X_k)^TD_k-\bar{G}_k^T\bar{D}_k\right]\right| + \EE_k\left[\|\bar{C}_k - c(X_k)\|_1\right] \\
\leq \ &\bar\cT'\cdot \left(\frac{\rho_k^g}{2} + \frac{\rho_k^g + \rho_k^c}{2q_{\min}^2} + \omega_4\cdot\sqrt{\rho_k^j} + \frac{\omega_3}{2}\cdot\sqrt{\rho_k^j(\rho_k^g+\rho_k^c)}\right) + \sqrt{m}\cdot \EE_k\left[\|\bar{C}_k - c(X_k)\|\right] \\
\leq \ &\bar\cT'\cdot \left(\frac{\rho_k^g}{2} + \frac{\rho_k^g + \rho_k^c}{2q_{\min}^2} + \omega_4\cdot\sqrt{\rho_k^j} + \frac{\omega_3}{2}\cdot\sqrt{\rho_k^j(\rho_k^g+\rho_k^c)}\right) + \sqrt{m}\cdot \sqrt{\EE_k\left[\|\bar{C}_k - c(X_k)\|^2\right]} \\
\leq \ &\bar\tau_0\cdot \left(\frac{\rho_k^g}{2} + \frac{\rho_k^g + \rho_k^c}{2q_{\min}^2} + \omega_4\cdot\sqrt{\rho_k^j} + \frac{\omega_3}{2}\cdot\sqrt{\rho_k^j(\rho_k^g+\rho_k^c)}\right) + \sqrt{m}\cdot \sqrt{\rho_k^c} \\
\leq \ &\bar\tau_0\left(\omega_5\cdot \rho_k^g + \omega_6\cdot \sqrt{\rho_k^c} + \omega_4\cdot\sqrt{\rho_k^j} + \frac{\omega_3}{2}\cdot\sqrt{\rho_k^j(\rho_k^g+\rho_k^c)} \right),
\end{align*}
where the second inequality uses Lemma~\ref{lem.solution_variance}, the second to last inequality is from Lemma~\ref{lem.merit_parameter} that $\bar\tau_0 \geq \bar\cT'$, and the last inequality holds for suffciently large constants $\{\omega_5,\omega_6\}\subset\RR{}_{>0}$ because of $\rho_k^c \leq \sqrt{\rhomax}\cdot\sqrt{\rho_k^c}$ (see Assumption~\ref{ass.estimate_asymptotic}), e.g., $\omega_5 := \frac{1}{2}+\frac{1}{2q_{\min}^2}$ and $\omega_6 := \frac{\sqrt{\rho_{\max}}}{2q_{\min}^2} + \sqrt{m}$.
\eproof

The following lemma provides an upper bound on $\EE_k\left[\bar\cA_k\bar\cT_k\nabla f(X_k)^T(\bar{D}_k - D_k)\right]$ for all sufficiently large iterations, which plays a crucial role in our final asymptotic convergence theorem (see Theorem~\ref{thm.asymptotic}). Similar to Lemmas~\ref{lem.solution_bias}--\ref{lem.model_reduction_comparison}, the next lemma also implies that if function and gradient estimates $(\bar{G}_k,\bar{C}_k,\bar{J}_k)$ almost surely recover their corresponding true values $(\nabla f(X_k),c(X_k),\nabla c(X_k)^T)$ (or in other words if $\rho_k^g = \rho_k^c = \rho_k^j = 0$), then the expected difference between stochastic and deterministic quantities would diminish to zero.

\blemma\label{lem.direction_bias_product}
Suppose Assumptions~\ref{ass.prob}--\ref{ass.estimate_asymptotic} hold. For all iterations $k \geq k_{\max}$ of Algorithm~\ref{alg.main}, we always have
\bequationn
\EE_k\left[\bar\cA_k\bar\cT_k\nabla f(X_k)^T(\bar{D}_k - D_k)\right] \leq \bar\cT'\cdot \left(\frac{2(1-\eta)\bar\Xi'\bar\cT'}{\bar\cT'L + \Gamma}+ \theta\right)\beta_k \cdot \kappa_{\nabla f}\cdot \left(\frac{\sqrt{\rho_k^g + \rho_k^c}}{q_{\min}} + \omega_1\cdot \sqrt{\rho_k^j}\right),
\eequationn
where $(\kappa_{\nabla f},q_{\min},\rho_k^g,\rho_k^c,\rho_k^j,\omega_1)$ are parameteres defined in \eqref{eq.basic_condition}, Lemma~\ref{lem.singular_values}, Assumption~\ref{ass.estimate_asymptotic} and Lemma~\ref{lem.solution_bias}.
\elemma
\bproof
From Assumption~\ref{ass.event_asymptotic}, we know $\bar\cT_k = \bar\cT'$ for all iterations $k \geq k_{\max}$, which implies
\bequationn
\EE_k\left[\bar\cA_k\bar\cT_k\nabla f(X_k)^T(\bar{D}_k - D_k)\right] = \bar\cT'\cdot\EE_k\left[\bar\cA_k\nabla f(X_k)^T(\bar{D}_k - D_k)\right].
\eequationn
Moreover, it follows the Cauchy-Schwarz inequality, Lemma~\ref{lem.merit_parameter} and Lemma~\ref{lem.stepsize_interval} that for all iterations $k \geq k_{\max}$,
\bequationn
\baligned
\EE_k\left[\bar\cA_k\bar\cT_k\nabla f(X_k)^T(\bar{D}_k - D_k)\right]
&= \bar\cT'\cdot\EE_k\left[\bar\cA_k\nabla f(X_k)^T(\bar{D}_k - D_k)\right] \\
&\leq \bar\cT'\cdot \left(\frac{2(1-\eta)\bar\Xi'\bar\cT'}{\bar\cT'L + \Gamma}+ \theta\right)\beta_k \cdot \|\nabla f(X_k)\|\cdot \EE_k\left[\|\bar{D}_k - D_k\|\right] \\
&\leq \bar\cT'\cdot \left(\frac{2(1-\eta)\bar\Xi'\bar\cT'}{\bar\cT'L + \Gamma}+ \theta\right)\beta_k \cdot \kappa_{\nabla f}\cdot \left(\frac{\sqrt{\rho_k^g + \rho_k^c}}{q_{\min}} + \omega_1\cdot \sqrt{\rho_k^j}\right),
\ealigned
\eequationn
where the second inequality is from \eqref{eq.basic_condition} and Lemma~\ref{lem.solution_bias}.
\eproof

Now we are ready to present our final theorem describing the asymptotic convergence behavior of Algorithm~\ref{alg.main}.
\btheorem\label{thm.asymptotic}
Suppose Assumptions~\ref{ass.prob}--\ref{ass.estimate_asymptotic} hold. Let sequences $\{\rho_k^g\}$, $\{\rho_k^c\}$, and $\{\rho_k^j\}$ defined in Assumption~\ref{ass.estimate_asymptotic} be chosen such that there exists a fixed constant $\iota \geq 0$ satisfying for all iterations $k\in\NN{}$,
\bequationn
\rho_k^g \leq \iota\beta_k^2,\quad \rho_k^c \leq \iota\beta_k^2,\quad \text{and}\quad \rho_k^j \leq \iota\beta_k^2.
\eequationn
Then with $\EE[\cdot|\cE_{\tau\xi}]$ representing the total expectation over all realizations of Algorithm~\ref{alg.main} conditioned on event $\cE_{\tau\xi}$ occurring, by defining
$\omega_{\iota} = \left(\left(\bar\tau_0\kappa_{\nabla f} + \sqrt{m}\kappa_{\nabla c}\right)\cdot \left(\tfrac{\sqrt{2}}{q_{\min}} + \omega_1\right) \inrevise{+ 2 \sqrt{m}} + (1-\eta)\cdot\bar\tau_0\left(\omega_4 + \omega_6+ \tfrac{\omega_3}{\sqrt{2}} + \omega_5 \right)\right)\cdot \max\{\iota,\sqrt{\iota}\}$
and $\bar\cA' = \frac{2(1-\eta)\bar\Xi'\bar\cT'}{\bar\cT'L+\Gamma}$, the following statements hold: \\
\textbf{Case (i)}\quad if $\beta_k = \beta = \frac{\psi \bar\cA'}{2(1-\eta)(\bar\cA'+\theta)} \in (0,1]$ with some $\psi\in (0,1]$ for all iterations $k \geq k_{\max}$ (see Assumptions~\ref{ass.event_asymptotic} and~\ref{ass.estimate_asymptotic}), then Algorithm~\ref{alg.main} generates a sequence of iterates $\{X_k\}$ such that
\bequationn
\limsup_{K\to\infty}\ \EE\left[\frac{1}{K}\sum_{t=k_{\max}}^{k_{\max} + K - 1} \Delta l\left(X_{t},\bar\cT',\nabla f(X_{t}),c(X_{t}),\nabla c(X_{t})^T,D_{t}\right) \bigg| \cE_{\tau\xi}
\right] \leq \frac{\omega_{\iota}\psi}{(2-\psi)(1-\eta)};
\eequationn
\textbf{Case (ii)}\quad if $\sum_{k=k_{\max}}^{\infty}\beta_k = \infty$, $\sum_{k=k_{\max}}^{\infty}\beta_k^2 < \infty$ and $\beta_k \leq \frac{\psi \bar\cA'}{2(1-\eta)(\bar\cA'+\theta)} \in (0,1]$ with some $\psi\in (0,1]$ for all iterations $k \geq k_{\max}$ (see Assumptions~\ref{ass.event_asymptotic} and~\ref{ass.estimate_asymptotic}), then Algorithm~\ref{alg.main} generates a sequence of iterates $\{X_k\}$ such that
\bequationn
\lim_{K\to\infty}\ \EE\left[\frac{1}{\sum_{t = k_{\max}}^{k_{\max} + K - 1}\beta_t}\cdot \sum_{t=k_{\max}}^{k_{\max} + K - 1}\beta_t\Delta l\left(X_{t},\bar\cT',\nabla f(X_{t}),c(X_{t}),\nabla c(X_{t})^T,D_{t}\right) \bigg| \cE_{\tau\xi} \right] = 0.
\eequationn
\etheorem
\bproof
From Assumption~\ref{ass.event_asymptotic}, we know $\bar\cT_k = \bar\cT'$ for all iterations $k \geq k_{\max}$. Therefore, combining the Cauchy-Schwarz inequality, Algorithm~\ref{alg.main}, Assumptions~\ref{ass.event_asymptotic} and~\ref{ass.estimate_asymptotic}, Lemma~\ref{lem.model_reduction_suff_large}, Lemmas~\ref{lem.merit_function_decrease}--\ref{lem.stepsize_interval}, and Lemma~\ref{lem.direction_bias_product}, for all iterations $k \geq k_{\max}$,
\begin{align}\label{eq.main_before_telescoping}
\ &\EE_k[\phi(X_{k+1},\bar\cT_k) - \phi(X_k,\bar\cT_k)]
= \EE_k[\phi(X_k + \bar\cA_k\bar{D}_k,\bar\cT_k) - \phi(X_k,\bar\cT_k)] \nonumber{} \\
\leq \ &\EE_k[-\bar\cA_k\Delta l(X_k,\bar\cT_k,\nabla f(X_k),c(X_k),\nabla c(X_k)^T,D_k)] + \EE_k[\bar\cA_k\bar\cT_k\nabla f(X_k)^T(\bar{D}_k - D_k)] \nonumber{} \\
  &\ \  + \EE_k[\bar\cA_k\|\nabla c(X_k)^T(D_k - \bar{D}_k)\|_1]
    \inrevise{+ \EE_k [2\bar\cA_k \|c(X_k) - \bar{C}_k\|_1]}
    + \EE_k[(1-\eta)\bar\cA_k\beta_k\Delta l(X_k,\bar\cT_k,\bar{G}_k,\bar{C}_k,\bar{J}_k,\bar{D}_k)] \nonumber{} \\
\leq \ &-\frac{2(1-\eta)\beta_k\bar\Xi'\bar\cT'}{\bar\cT'L+\Gamma}\cdot \Delta l(X_k,\bar\cT',\nabla f(X_k),c(X_k),\nabla c(X_k)^T,D_k) \nonumber{} \\
&\ \ +\bar\cT'\cdot \left(\frac{2(1-\eta)\bar\Xi'\bar\cT'}{\bar\cT'L + \Gamma}+ \theta\right)\beta_k \cdot \kappa_{\nabla f}\cdot \left(\frac{\sqrt{\rho_k^g + \rho_k^c}}{q_{\min}} + \omega_1\cdot \sqrt{\rho_k^j}\right) \nonumber{} \\
&\ \ + \EE_k\left[\left(\frac{2(1-\eta)\beta_k\bar\Xi'\bar\cT'}{\bar\cT'L+\Gamma} + \theta\beta_k\right)\cdot\sqrt{m}\cdot\|\nabla c(X_k)\|\|D_k - \bar{D}_k\|\right] \nonumber{} \\
  & \ \ \inrevise{+ \EE_k \left[\left(\frac{2(1-\eta)\beta_k\bar\Xi'\bar\cT'}{\bar\cT'L+\Gamma} + \theta\beta_k\right) \cdot \sqrt{m} \cdot 2 \sqrt{\|c(X_k) - \bar{C}_k\|^2} \right]}
    \nonumber{} \\
&\ \ + \EE_k\left[(1-\eta)\cdot \left(\frac{2(1-\eta)\beta_k\bar\Xi'\bar\cT'}{\bar\cT'L+\Gamma} + \theta\beta_k\right)\cdot\beta_k\Delta l(X_k,\bar\cT_k,\bar{G}_k,\bar{C}_k,\bar{J}_k,\bar{D}_k)\right] \nonumber{} \\
\leq \ &-\frac{2(1-\eta)\beta_k\bar\Xi'\bar\cT'}{\bar\cT'L+\Gamma}\cdot \Delta l(X_k,\bar\cT',\nabla f(X_k),c(X_k),\nabla c(X_k)^T,D_k) \nonumber{} \\
&\ \ +\bar\cT'\cdot \left(\frac{2(1-\eta)\bar\Xi'\bar\cT'}{\bar\cT'L + \Gamma}+ \theta\right)\beta_k \cdot \kappa_{\nabla f}\cdot \left(\frac{\sqrt{\rho_k^g + \rho_k^c}}{q_{\min}} + \omega_1\cdot \sqrt{\rho_k^j}\right) \nonumber{} \\
&\ \ + \left(\frac{2(1-\eta)\beta_k\bar\Xi'\bar\cT'}{\bar\cT'L+\Gamma} + \theta\beta_k\right)\cdot\sqrt{m} \cdot \left( \kappa_{\nabla c} \cdot \left(\frac{\sqrt{\rho_k^g + \rho_k^c}}{q_{\min}} + \omega_1\cdot \sqrt{\rho_k^j} \right) \inrevise{+ 2\sqrt{\rho_k^c}}\right)  \nonumber{} \\
&\ \  + (1-\eta)\cdot \left(\frac{2(1-\eta)\beta_k\bar\Xi'\bar\cT'}{\bar\cT'L+\Gamma} + \theta\beta_k\right)\cdot\beta_k\Delta l(X_k,\bar\cT',\nabla f(X_k),c(X_k),\nabla c(X_k)^T,D_k) \nonumber{} \\
&\ \  + (1-\eta)\cdot \left(\frac{2(1-\eta)\beta_k\bar\Xi'\bar\cT'}{\bar\cT'L+\Gamma} + \theta\beta_k\right)\cdot\beta_k\bar\tau_0\left(\omega_5\cdot \rho_k^g + \omega_6\cdot \sqrt{\rho_k^c} + \omega_4\cdot\sqrt{\rho_k^j} + \frac{\omega_3}{2}\cdot\sqrt{\rho_k^j(\rho_k^g+\rho_k^c)} \right) \nonumber{} \\
\leq \ &-\left(\frac{2(1-\eta)\bar\Xi'\bar\cT'}{\bar\cT'L+\Gamma} - \left(\frac{2(1-\eta)\bar\Xi'\bar\cT'}{\bar\cT'L+\Gamma} + \theta\right)\cdot(1-\eta)\beta_k\right)\cdot \beta_k\Delta l(X_k,\bar\cT',\nabla f(X_k),c(X_k),\nabla c(X_k)^T,D_k) \nonumber{} \\
&\ \  + \left(\frac{2(1-\eta)\bar\Xi'\bar\cT'}{\bar\cT'L+\Gamma} + \theta\right)\cdot \left(\left(\bar\tau_0\kappa_{\nabla f} + \sqrt{m}\kappa_{\nabla c}\right)\cdot \left(\frac{\sqrt{2\iota}}{q_{\min}} + \omega_1\cdot \sqrt{\iota} \right) \inrevise{+ 2 \sqrt{m \iota}} \right) \cdot \beta_k^2 \nonumber{} \\
&\ \  + (1-\eta)\cdot \left(\frac{2(1-\eta)\bar\Xi'\bar\cT'}{\bar\cT'L+\Gamma} + \theta\right)\cdot\bar\tau_0\left((\omega_4 + \omega_6)\cdot\sqrt{\iota} + \left(\frac{\omega_3}{\sqrt{2}} + \omega_5\right)\cdot\iota \right)\cdot \beta_k^2 \nonumber{} \\
\leq \ &-\left(\bar\cA' - \left(\bar\cA' + \theta\right)\cdot(1-\eta)\beta_k\right)\cdot \beta_k\Delta l(X_k,\bar\cT',\nabla f(X_k),c(X_k),\nabla c(X_k)^T,D_k) + \omega_{\iota}\cdot\left(\bar\cA' + \theta\right)\cdot\beta_k^2,
\end{align}
where the third inequality follows \inrevise{Jenson's inequality,} \eqref{eq.basic_condition}, Lemma~\ref{lem.solution_bias} and Lemma~\ref{lem.model_reduction_comparison}, the fourth inequality is from $\{\beta_k\}\subset(0,1]$ and $\max\{\rho_k^g,\rho_k^c,\rho_k^j\} \leq \iota\beta_k^2$ in the theorem statement, while the last inequality results from the definitions of $\omega_{\iota}$ and $\bar\cA'$. \\
\textbf{Case (i):} combining the definition of $\beta_k$ in the theorem statement and \eqref{eq.main_before_telescoping}, for all iterations $k \geq k_{\max}$,
\bequationn
\baligned
\ &\EE_k[\phi(X_{k+1},\bar\cT_k) - \phi(X_k,\bar\cT_k)] \\
\leq \ &-\left(\bar\cA' - \left(\bar\cA' + \theta\right)\cdot(1-\eta)\beta_k\right)\cdot \beta_k\Delta l(X_k,\bar\cT',\nabla f(X_k),c(X_k),\nabla c(X_k)^T,D_k) + \omega_{\iota}\cdot\left(\bar\cA' + \theta\right)\cdot\beta_k^2 \\
= \ & -\left(1-\frac{\psi}{2}\right)\cdot\bar\cA'\beta\Delta l(X_k,\bar\cT',\nabla f(X_k),c(X_k),\nabla c(X_k)^T,D_k) + \omega_{\iota}\cdot\left(\bar\cA' + \theta\right)\cdot\beta^2.
\ealigned
\eequationn
By Assumption~\ref{ass.prob}, there exists $\phi_{\min}\in\RR{}$ such that $\phi_{\min} \leq \phi(X_k,\bar\cT')$ for all iterations $k \geq k_{\max}$. Moreover, from Assumptions~\ref{ass.event_asymptotic} and~\ref{ass.estimate_asymptotic}, we further have
\bequationn
\baligned
\ &\phi_{\min} - (\bar\tau_0|f_{\max}| + \kappa_c) \leq \phi_{\min} - \EE[\phi(X_{k_{\max}},\bar\cT')|\cE_{\tau\xi}] \leq \EE[\phi(X_{k_{\max} + K},\bar\cT') - \phi(X_{k_{\max}},\bar\cT')|\cE_{\tau\xi}] \\
\leq \ &\EE\left[\sum_{t=k_{\max}}^{k_{\max} + K - 1} \left(-\left(1-\frac{\psi}{2}\right)\cdot\bar\cA'\beta\Delta l\left(X_{t},\bar\cT',\nabla f(X_{t}),c(X_{t}),\nabla c(X_{t})^T,D_{t}\right) + \omega_{\iota}\cdot\left(\bar\cA' + \theta\right)\cdot\beta^2\right) \bigg| \cE_{\tau\xi}\right] \\
= \ &-\left(1-\frac{\psi}{2}\right)\bar\cA'\beta\cdot \EE\left[\sum_{t=k_{\max}}^{k_{\max} + K - 1}\Delta l\left(X_{t},\bar\cT',\nabla f(X_{t}),c(X_{t}),\nabla c(X_{t})^T,D_{t}\right) \bigg| \cE_{\tau\xi} \right] + K\cdot\left(\bar\cA' + \theta\right)\cdot\omega_{\iota}\beta^2,
\ealigned
\eequationn
which concludes the first part of the statement by rearranging terms, using definitions of $(\bar\cA',\beta)$,  dividing all the terms by $K$ and finally driving $K\to\infty$.

\textbf{Case (ii):} from the condition of $\beta_k \leq \frac{\psi \bar\cA'}{2(1-\eta)(\bar\cA'+\theta)} \in (0,1]$ and \eqref{eq.main_before_telescoping}, for all iterations $k \geq k_{\max}$, we have
\bequationn
\baligned
\ &\EE_k[\phi(X_{k+1},\bar\cT_k) - \phi(X_k,\bar\cT_k)] \\
\leq \ &-\left(\bar\cA' - \left(\bar\cA' + \theta\right)\cdot(1-\eta)\beta_k\right)\cdot \beta_k\Delta l(X_k,\bar\cT',\nabla f(X_k),c(X_k),\nabla c(X_k)^T,D_k) + \omega_{\iota}\cdot\left(\bar\cA' + \theta\right)\cdot\beta_k^2 \\
\leq \ & -\left(1-\frac{\psi}{2}\right)\cdot\bar\cA'\beta_k\Delta l(X_k,\bar\cT',\nabla f(X_k),c(X_k),\nabla c(X_k)^T,D_k) + \omega_{\iota}\cdot\left(\bar\cA' + \theta\right)\cdot\beta_k^2.
\ealigned
\eequationn
Using the same logic as the proof of \textbf{Case (i)}, we further have
\bequationn
\baligned
\ &\phi_{\min} - (\bar\tau_0|f_{\max}| + \kappa_c) \leq \phi_{\min} - \EE[\phi(X_{k_{\max}},\bar\cT')|\cE_{\tau\xi}] \leq \EE[\phi(X_{k_{\max} + K},\bar\cT') - \phi(X_{k_{\max}},\bar\cT')|\cE_{\tau\xi}] \\
\leq \ &\EE\left[\sum_{t=k_{\max}}^{k_{\max} + K - 1} \left(-\left(1-\frac{\psi}{2}\right)\cdot\bar\cA'\beta_t\Delta l\left(X_{t},\bar\cT',\nabla f(X_{t}),c(X_{t}),\nabla c(X_{t})^T,D_{t}\right) + \omega_{\iota}\left(\bar\cA' + \theta\right)\cdot\beta_t^2\right) \bigg| \cE_{\tau\xi}
\right] \\
= \ &-\left(1-\frac{\psi}{2}\right)\bar\cA'\cdot \EE\left[\sum_{t=k_{\max}}^{k_{\max} + K - 1}\beta_t\Delta l\left(X_{t},\bar\cT',\nabla f(X_{t}),c(X_{t}),\nabla c(X_{t})^T,D_{t}\right) \bigg| \cE_{\tau\xi} \right] + \omega_{\iota}\left(\bar\cA' + \theta\right)\cdot\sum_{t=k_{\max}}^{k_{\max} + K - 1}\beta_t^2,
\ealigned
\eequationn
which further implies that
\bequationn
\frac{(\bar\tau_0|f_{\max}| + \kappa_c) - \phi_{\min}}{(1-\frac{\psi}{2})\bar\cA'} \geq \EE\left[\sum_{t=k_{\max}}^{k_{\max} + K - 1}\beta_t\Delta l\left(X_{t},\bar\cT',\nabla f(X_{t}),c(X_{t}),\nabla c(X_{t})^T,D_{t}\right) \bigg| \cE_{\tau\xi} \right] - \frac{\omega_{\iota}\left(\bar\cA' + \theta\right)}{(1-\frac{\psi}{2})\bar\cA'}\cdot \sum_{t=k_{\max}}^{k_{\max} + K - 1}\beta_t^2.
\eequationn
After dividing both sides by $\sum_{k=k_{\max}}^{k_{\max} + K - 1}\beta_k$, we drive $K\to\infty$ and utilize conditions of $\sum_{k=k_{\max}}^{\infty}\beta_k = \infty$ and $\sum_{k=k_{\max}}^{\infty}\beta_k^2 < \infty$, then the conclusion follows.
\eproof

As Theorem~\ref{thm.asymptotic} demonstrates the asymptotic convergence behavior of \textit{deterministic} model reduction sequence $\left\{\Delta l\left(X_k,\bar\cT',\nabla f(X_k),c(X_k),\nabla c(X_k)^T,D_k\right)\right\}$ by running Algorithm~\ref{alg.main}, with the help of~\eqref{eq.linear_system_true} and Lemma~\ref{lem.model_reduction_suff_large_true}, we further use the following corollary to present the asymptotic convergence behavior of a quantity related to stationary measurements in~\eqref{eq.KKT}, i.e., $\left\{\|\nabla f(X_k) + \nabla c(X_k)Y_k\|^2 + \|c(X_k)\|_1\right\}$.

\begin{cor}\label{cor.asymptotic}
Under the same conditions and using the same notations as Theorem~\ref{thm.asymptotic}, the following statements hold: \\
\textbf{Case (i)}\quad if $\beta_k = \beta = \frac{\psi \bar\cA'}{2(1-\eta)(\bar\cA'+\theta)} \in (0,1]$ with some $\psi\in (0,1]$ for all iterations $k \geq k_{\max}$ (see Assumptions~\ref{ass.event_asymptotic} and~\ref{ass.estimate_asymptotic}), then Algorithm~\ref{alg.main} generates a sequence of iterates $\{X_k\}$ such that
\bequationn
\limsup_{K\to\infty}\ \EE\left[\frac{1}{K}\sum_{t=k_{\max}}^{k_{\max} + K - 1} \left(\|\nabla f(X_t) + \nabla c(X_t)Y_t\|^2 + \|c(X_t)\|_1\right) \bigg| \cE_{\tau\xi} \right] \leq \frac{\omega_{\iota}\psi}{(2-\psi)(1-\eta)\cdot\min\left\{\frac{\zeta\bar\cT'}{2\kappa_H^2},\sigma\right\}};
\eequationn
\textbf{Case (ii)}\quad if $\sum_{k=k_{\max}}^{\infty}\beta_k = \infty$, $\sum_{k=k_{\max}}^{\infty}\beta_k^2 < \infty$ and $\beta_k \leq \frac{\psi \bar\cA'}{2(1-\eta)(\bar\cA'+\theta)} \in (0,1]$ with some $\psi\in (0,1]$ for all iterations $k \geq k_{\max}$ (see Assumptions~\ref{ass.event_asymptotic} and~\ref{ass.estimate_asymptotic}), then Algorithm~\ref{alg.main} generates a sequence of iterates $\{X_k\}$ such that
\bequationn
\lim_{K\to\infty}\ \EE\left[\frac{1}{\sum_{t = k_{\max}}^{k_{\max} + K - 1}\beta_t}\cdot \sum_{t=k_{\max}}^{k_{\max} + K - 1}\beta_t\left(\|\nabla f(X_t) + \nabla c(X_t)Y_t\|^2 + \|c(X_t)\|_1\right) \bigg| \cE_{\tau\xi} \right] = 0,
\eequationn
which further implies that
\bequationn
\liminf_{k\to\infty}\ \EE\left[\|\nabla f(X_k) + \nabla c(X_k)Y_k\|^2 + \|c(X_k)\|_1 \bigg| \cE_{\tau\xi} \right] = 0.
\eequationn
\end{cor}
\bproof
It follows Assumption~\ref{ass.H}, \eqref{eq.linear_system_true}, and the Cauchy-Schwarz inequality that
\bequationn
\|\nabla f(X_k) + \nabla c(X_k)Y_k\| = \|H_kD_k\| \leq \|H_k\|\|D_k\| \leq \kappa_H\|D_k\|,
\eequationn
from which we further have
\bequation\label{eq.Dk_true_lb}
\|D_k\| \geq \frac{\|\nabla f(X_k) + \nabla c(X_k)Y_k\|}{\kappa_H}.
\eequation
Combining Assumption~\ref{ass.event_asymptotic}, Lemma~\ref{lem.model_reduction_suff_large_true}, and~\eqref{eq.Dk_true_lb}, we know that for any iteration $k\geq k_{\max}$,
\bequationn
\begin{aligned}
&\Delta l\left(X_k,\bar\cT',\nabla f(X_k),c(X_k),\nabla c(X_k)^T,D_k\right) \geq \frac{\zeta}{2}\bar\cT'\|D_k\|^2 + \sigma\|c(X_k)\|_1 \\
\geq \ &\frac{\zeta\bar\cT'}{2\kappa_H^2}\|\nabla f(X_k) + \nabla c(X_k)Y_k\|^2 + \sigma\|c(X_k)\|_1 \geq \min\left\{\frac{\zeta\bar\cT'}{2\kappa_H^2},\sigma\right\}\cdot\left(\|\nabla f(X_k) + \nabla c(X_k)Y_k\|^2 + \|c(X_k)\|_1\right).
\end{aligned}
\eequationn
Finally, using the results of Theorem~\ref{thm.asymptotic}, we may conclude the statement.
\eproof

We close this subsection with the following remark on Corollary~\ref{cor.asymptotic}.
\begin{remark}
For the constant step size policy (\textbf{Case (i)}) of Algorithm~\ref{alg.main}, the radius of asymptotic convergence neighborhood is upper bounded by a quantity proportional to $\frac{\psi}{2-\psi}$, which implies that users may choose a smaller $\psi$ (indicating a shorter step size $\beta$) to improve the algorithm's final asymptotic convergence behavior in expectation. Moreover, such an upper bound on the radius of convergence neighborhood is also proportional to $\max\{\iota,\sqrt{\iota}\} \geq 0$, which further describes the impacts from $\{\rho_k^g\}$, $\{\rho_k^c\}$, and $\{\rho_k^j\}$, i.e., the variances of stochastic estimates $\{\bar{G}_k\}$, $\{\bar{C}_k\}$, and $\{\bar{J}_k\}$, respectively. In particular, when estimates $\{(\bar{G}_k,\bar{C}_k,\bar{J}_k)\}$ are more accurate and with smaller variances, we could have a smaller value of $\max\{\iota,\sqrt{\iota}\}$, and improved near-stationary iterates would be achieved by Algorithm~\ref{alg.main} asymptotically in expectation. Finally, under the diminishing step size policy (\textbf{Case (ii)}), Algorithm~\ref{alg.main} achieves exact convergence in expectation. Similar results have been established for the unconstrained setting (e.g.,~\cite{bottou2018optimization}) and the \textit{deterministic} constrained setting (see~\cite{BeraCurtRobiZhou21,CurtRobiZhou23}). However, unlike~\cite{BeraCurtRobiZhou21,bottou2018optimization}, our algorithm requires variance-reduced stochastic estimates in addition to diminishing step sizes to guarantee exact convergence in expectation. This is mainly because in our expectation-constrained setting, the computed \textit{stochastic} search directions serve as biased estimators of the \textit{deterministic} counterpart; i.e., $\mathbb{E}_k[\bar{D}_k] \neq D_k$, which makes our asymptotic convergence theoretical results more close to~\cite{CurtRobiZhou23} rather than~\cite{BeraCurtRobiZhou21,bottou2018optimization}.
\end{remark}

\subsection{Complexity Results}\label{sec.complex_results}

\newcounter{cnstone}
\setcounter{cnstone}{6}

\newcounter{cnsttwo}
\setcounter{cnsttwo}{0}

\newcommand*{\nc}{
  \refstepcounter{cnstone}
  \ensuremath{\omega_{\thecnstone}}
}
\newcommand{\rc}[1]{\ensuremath{\omega_{\ref{#1}}}}

\newcommand{\nctwo}{
  \refstepcounter{cnsttwo}
  \ensuremath{M_{\thecnsttwo}}
}

\newcommand{\rctwo}[1]{\ensuremath{M_{\ref{#1}}}}

The analysis in Section~\ref{sec.converg_results} provides valuable insights into the asymptotic performance of Algorithm~\ref{alg.main} in the regime of infinite iterations, and we now examine its non-asymptotic behavior to offer a detailed investigation of its performance when the total number of iterations is finite. To begin with, we define the following event for any $(\taumin, \kmax) \in \reals_{>0} \times \NN$:
\inrevise{
\begin{equation}\label{eq.event_nonasy}
\begin{aligned}
\cE'(\tau_{\min},k_{\max}) :=
& \left\{
\min\{\cT_k^{\trial}, \bar\cT_k\} \geq \taumin > 0 \text{ for all } k \in [\kmax]
\right\}
\end{aligned}
\end{equation}
}and denote it as \(\nonE\) for brevity. \inrevise{Specifically, it includes all runs where \(\taumin\) serves as a lower bound for both the unknown true trial merit parameters \(\{\cT^{\trial}_k\}_{k\in[\kmax]}\) and the stochastic merit parameters \(\{\bar\cT_k\}_{k\in[\kmax]}\).
}With these notations in place, we make the following assumption.

\begin{assprime}{\ref*{ass.event_asymptotic}$'$}
\label{ass.event_nonasymptotic}
During the runs of Algorithm~\ref{alg.main}, event \(\nonE:=\cE'(\taumin,\kmax)\) (defined in~\eqref{eq.event_nonasy}) occurs with given constants \((\taumin,\kmax) \in \reals_{>0}\times\mathbb{N}\).
\end{assprime}
\renewcommand{\theass}{\ref*{ass.estimate_asymptotic}$'$}
In addition, we slightly abuse the notations of \(\Prob_k[\cdot]\) and \(\Expect_k[\cdot]\) by referring them to \(\Prob [ \cdot | \mathcal{F}_k' ]\) and \(\Expect [ \cdot | \mathcal{F}_k' ]\), respectively, where $\mathcal{F}_k':= \{E\cap \nonE: E\in\cG_k\}$ for all $k\in\mathbb{N}$, and we consider another assumption similar to Assumption~\ref{ass.estimate_asymptotic}.
\begin{assprime}{\ref*{ass.estimate_asymptotic}$'$}
\label{ass.estimate_nonasymptotic}
\inrevise{Assumption~\ref{ass.estimate_asymptotic} holds with $\mathcal{F}_k$ replaced by $\mathcal{F}_k'$, or equivalently, with \(\mathcal{E}_{\tau\xi}\) replaced by \(\nonE\).}
\end{assprime}
Furthermore, the technical lemmas in Section~\ref{sec.fund_lemmas} remain valid under this adjustment, and Assumptions~\ref{ass.event_nonasymptotic} and~\ref{ass.estimate_nonasymptotic} are mild for the following reasons:
\inrevise{
\begin{enumerate}[label=(\roman*)]
\item the sequence $\{\cT_k^{\trial}\}_{k\in\mathbb{N}}$ is bounded away from zero under Assumptions \ref{ass.prob}--\ref{ass.H} (see~\cite[Lemma~2.16]{BeraCurtRobiZhou21}), 
\item the limit of stochastic merit parameters (i.e., $\lim_{k\to\infty}\bar\cT_k$) is strictly positive with probability one if \(\left\{\rho^g_k + \rho^c_k + \rho^j_k\right\}_{k\in\mathbb{N}}\) is summable (Lemma~\ref{lem.tau_zero_wpzero}), and
\item with high probability, the stochastic merit parameters \(\{\bar{\cT}_k\}_{k\in[\kmax]}\) are bounded below by a positive constant independent of $\kmax$  if \(\left(\sqrt{\rho^g_k} + \sqrt{\rho^c_k} + \sqrt{\rho^j_k}\right)\) is bounded above by \(\mathcal{O}(1/\sqrt{\kmax})\) for all $k\in[\kmax]$ (Lemma~\ref{lem:tau-nonasymptotic-bdd-away-from-zero}).
\end{enumerate}
}\paragraph{An overview of our complexity result} The main objective of this section is to establish a worst-case complexity bound for the total number of iterations needed to reach a near stationary point in expectation,~\ie{}, determining an iteration $k$ such that
\inrevise{
\begin{align}
\Expect \left[ \|\nabla f(X_k) + \nabla c(X_k)Y_k\|^2 + \|c(X_k)\|_1 | \nonE \right] < \varepsilon^2,  \label{eq:p-sketch-goal}
\end{align}
}for a given tolerance $\varepsilon > 0$. A natural approach to achieve this is to leverage~\eqref{eq.linear_system_true} and Lemma~\ref{lem.model_reduction_suff_large_true}, which connect the stationary measure in~\eqref{eq:p-sketch-goal} with $\Delta l(X_k,\cT,\nabla f(X_k), c(X_k), \nabla c(X_k)^T, D_k)$ whenever $\cT \in (0, \cT_k^{\trial}]$. However, since $\cT_k^{\trial}$ is always inaccessible, we instead rely on the stochastic merit parameter $\bar\cT_k$, at the cost of introducing an additional \textit{``noise''} term. By further connecting the model reduction quantity with the bounded expected cumulative changes in the merit function, we will show that the following chain of inequalities are satisfied at all iterates:
\begin{align}
&\Expect_k \left[ \|\nabla f(X_k) + \nabla c(X_k)Y_k\|^2 + \|c(X_k)\|_1  \right] \leq \rc{cnst:kkt-delta-l-demo}
\cdot \Expect_k \left[ \Delta l(X_k, \bar{\cT}_k ,\nabla f(X_k),c(X_k),\nabla c(X_k)^T,D_k)  \right] \label{eq:p-sketch-l-bar} \\
\leq \ &\rc{cnst:kkt-delta-l-demo} \cdot \Expect_k \left[ \phi(X_{k+1}, \bar{\cT}_k) - \phi(X_k, \bar{\cT}_k) \right]
+ \text{``noise one''} \cdot \Prob_k \left[ \bar{\cT}_k > \cT_k^{\trial} \right]
+ \text{``noise two''} \cdot \left( \sqrt{\rho^g_k} + \sqrt{\rho^c_k} + \sqrt{\rho^j_k} \right), \nonumber{}
\end{align}
where $\nc\label{cnst:kkt-delta-l-demo}$ is a positive constant that only depends on parameters $(\kappa_H, \sigma, \zeta, \taumin)$ from \eqref{eq.merit_parameter_update} and Assumptions~\ref{ass.H} and~\ref{ass.event_nonasymptotic}, and the~\dbquot{noise two} is incurred due to \inrevise{the bias of the search direction $\mathbb{E}_k[\bar{D}_k - D_k]$ and the estimation error in  \((\bar{G}_k, \bar{C}_k, \bar{J}_k)\)}. Moreover, we will show that the entire first noise term (\ie{}, \(\dbquot{noise one} \cdot \Prob_k \left[ \bar{\cT}_k > \cT_k^{\trial} \right]\)) is upper bounded by \inrevise{\(\mathcal{O}(\rho^g_k + \rho^c_k + \rho^j_k)/\varepsilon^2\)} (see the proof of Lemma~\ref{lem.expect-diff-phi-k-bad}), so the negative impact of the two noise terms in~\eqref{eq:p-sketch-l-bar} is minimal when the pre-specified variances \(\{(\rho^g_k, \rho^c_k, \rho^j_k)\}\) are small.
\inrevise{Specifically, we show in Theorem~\ref{thm.complex} and Corollary~\ref{cor.complex} that if the variances are \(\mathcal{O}\left( 1 / \kmax^2 \right)\), then Algorithm~\ref{alg.main} can achieve
\begin{align*}
\Expect \left[ \|\nabla f(X_k) + \nabla c(X_k)Y_k\|^2 + \|c(X_k)\|_1 | \nonE \right] \leq
\mathcal{O} \left(\frac{1}{\kmax}\right) + \mathcal{O} \left(\frac{1}{\varepsilon^2 \cdot \kmax^2}\right) + \mathcal{O} \left(\varepsilon^2\right),
\end{align*}
establishing an iteration complexity of \(\mathcal{O}(1/\varepsilon^2)\) and sample complexity of \(\mathcal{O}(1/\varepsilon^6)\) for identifying an $\varepsilon$-stationary iterate in expectation (see~\eqref{eq:p-sketch-goal}).
}

\paragraph{Comparision with the complexity result in~\cite{CurtRobiOnei24}}
Compared to the proof techniques employed in the complexity result of~\cite{CurtRobiOnei24}, we adopt significantly relaxed assumptions and emphasize the role of the variances in the algorithm's performance, resulting in a different mechanism for establishing convergence. We next detail the similarities and differences between our proof and the one provided in~\cite{CurtRobiOnei24}.
\begin{enumerate}
\item[\emph{Similar approach to relate \(\cT_k^{\trial}\) with \(\bar{\cT}_k\)}.] When bounding the ``noise one'' term on the right-hand-side of~\eqref{eq:p-sketch-l-bar}, one needs to relate \(\bar{\cT}_k\) with \(\cT_k^{\trial}\). To this end, in Section~\ref{sec:complex-noise} we follow~\cite{CurtRobiOnei24}'s approach by defining an auxiliary sequence
\(\hat{\cT}_k := \min \{ \bar{\cT}_k, \cT_k^{\trial}\}\)
and analyzing the relationship between \(\{\hat{\cT}_k\}\) and \(\{\bar{\cT}_k\}\) instead.
\item[\emph{Different assumptions on the variance sequence}.] The major distinction between our complexity result and that of~\cite{CurtRobiOnei24} lies in the presence of stochastic constraints. Specifically, when the constraint function and its derivative can be evaluated deterministically, search directions $\{\bar{D}_k\}$ are unbiased estimators for their deterministic counterpart \(\{D_k\}\), and we may expect the algorithm to reach an exact stationary point in expectation, even if the variance sequence \(\{\rho^g_k\}\) remains non-diminishing. However, incorporating stochasticity into constraints typically leads to biased search directions in SQP methods~\cite{FaccKung23}, making it unrealistic to expect the algorithm to perform effectively as before. Therefore, by allowing the variance sequence \(\{(\rho^g_k, \rho^c_k, \rho^j_k)\}\) to be dependent on \(\kmax\), we show that Algorithm~\ref{alg.main} achieves an iteration complexity of \(\mathcal{O}(\varepsilon^{-2})\), matching the best known results for the deterministic case and better than the \(\mathcal{O}(\varepsilon^{-4})\) iteration complexity established in~\cite{CurtRobiOnei24}. However, such improvement is unsurprising, as~\cite{CurtRobiOnei24} employs \inrevise{a weaker condition, namely, that the sequence of stochastic estimates has constant variance $\{\rho_k^g\}$ independent of $\kmax$.}

\item[\emph{A new approach for analyzing the relationship between \(\{\bar{\cT}_k\}\) and \(\{\cT_k^{\trial}\}\)}.] Because of stochastic constraints, our complexity analysis emphasizes the influence of the variance sequence on the stationarity measure. In particular, we make much more relaxed assumptions on the relationship between \(\{\bar{\cT}_k\}\) and \(\{\cT_k^{\trial}\}\) compared to~\cite{CurtRobiOnei24}, which directly assumes the existence of a nonzero probability \(p_{\tau} \in (0, 1]\) such that
\[
\Prob_k \left[
\bar{G}_k^{\top} \bar{D}_k + \max \{\bar{D}_k^{\top} H_k \bar{D}_k, 0\}
\geq \nabla f(X_k)^{\top} D_k + \max\{D_k^{\top} H_k D_k, 0\}
\right] \geq p_{\tau}
\text{ for~\emph{all} iterations \(k\in\NN\).}
\]
In view of the definitions of \(\{\bar{\cT}_k\}\) and \(\{\cT_k^{\trial}\}\) in~\cite{CurtRobiOnei24}, this condition is directly assuming that \(\Prob_k \left[ \bar{\cT}_k \leq \cT_k^{\trial} \right] \geq p_{\tau}\) for~\emph{all} iterations \(k\in\NN\) when the constraint function is deterministic. Although they have shown that this assumption indeed holds when the estimation error on \(\nabla f(x)\) is~\emph{symmetric} and~\emph{sub-Gaussian}, such ideal conditions are rarely encountered in real-world applications~\cite{mandelbrot1997variation}. For example, when $F(x, \omega_f) = (x + \omega_f)^3$ (see~\eqref{eq.prob}) and $x = 0$, one has \(\nabla f(x)|_{x=0} = \Expect \left[ \nabla F(0, \omega_f) \right] = \Expect \left[ 3\omega_f^2 \right]\). In this case, even when \(\omega_f\) is sub-Gaussian and symmetric around zero, the objective gradient estimates $3\omega_f^2$ are \inrevise{neither symmetric around their mean value nor sub-Gaussian}. In contrast, our complexity analysis imposes no additional assumptions on the estimation errors of $\{(\nabla f(X_k), c(X_k), \nabla c(X_k)^T)\}$. Moreover, we derive an asymptotically tight, iteration-dependent lower bound for \(\Prob_k [ \bar{\cT}_k \leq \cT_k^{\trial}]\) in Lemma~\ref{lem.prob-tau-good-k} and apply it in Lemma~\ref{lem.expect-diff-phi-k-bad} to show that the error incurred by the ``bad'' case (\ie{}, \(\bar{\cT}_k > \cT_k^{\trial}\)) is \inrevise{controlled by \(\mathcal{O}(\rho^g_k + \rho^c_k + \rho^j_k)/\varepsilon^2\)} as long as \((X_k,Y_k)\) is~\emph{not} near-stationary. While this result only applies to iterates where \((X_k,Y_k)\) is~\emph{not} near-stationary, rather than to~\emph{all} iterates, it is sufficient to establish the desired convergence rate because any iterate \((X_k,Y_k)\) where Lemma~\ref{lem.expect-diff-phi-k-bad} does not apply is already near-stationary, thus contributes minimally to the expected stationarity error.
\item[\emph{A new justification for \(\nonE\)}.] Both~\cite{CurtRobiOnei24} and our analysis assume the occurance of event \(\nonE\). To justify the occurance of event \(\nonE\), \cite{CurtRobiOnei24} imposes a restrictive sufficient condition that the estimation errors of \(\nabla f(x)\) are symmetric and sub-Gaussian. In contrast, we show in Lemma~\ref{lem:tau-nonasymptotic-bdd-away-from-zero} that as long as the variance sequence $\left\{\sqrt{\rho^g_k} + \sqrt{\rho^j_k} + \sqrt{\rho^c_k}\right\}$ is upper bounded by $\mathcal{O}(1/\sqrt{\kmax})$, there exists a \(\kmax\)-independent constant $\delta_{\tau} > 0$ such that $\bar{\cT}_k > \delta_{\tau}$ occurs with high probability for all \(k \in [\kmax]\).
\end{enumerate}

We are now ready to present our non-asymptotic convergence analysis. In Section~\ref{sec:complex-prob-bound}, we derive a tight lower bound on the probability of the event \(\bar{\cT}_k \leq \cT_k^{\trial}\) happening. Then, in Section~\ref{sec:complex-noise}, we formally establish~\eqref{eq:p-sketch-l-bar} and present our final complexity results in Section~\ref{sec:complex-result}.

\subsubsection{Analysis of the relation between $\bar{\cT}_k$ and $\cT_k^{\trial}$}\label{sec:complex-prob-bound}
To analyze the relationship between \(\{\bar{\cT}_k\}\) and \(\{\cT^{\trial}_k\}\), we leverage the relation \(\bar{\cT}_k \leq (1 - \epsilon_{\tau}) \bar{\cT}_k^{\trial}\) (see Lemma~\ref{lem.merit_parameter}) and analyze \(\{\cT^{\trial}_k\}\) and \(\{\bar{\cT}^{\trial}_k\}\) instead. To this end, we rewrite \(\bar{\cT}_k^{\trial}\) as a function of \((\bar{G}_k, \bar{C}_k, \bar{J}_k)\) and quantify how inaccurate estimates of \((\nabla f(X_k), c(X_k), \nabla c(X_k)^{\top})\) contribute to the estimation error of \(\cT_k^{\trial}\). Specifically, we introduce functions $T_n:\RR^{n+m+mn}\to\RR$ and $T_d:\RR^{n+m+mn}\times\RR^{n\times n}\to\RR$ to represent the e\emph{n}umerator and \emph{d}enominator of \(\bar{\cT}_k^{\trial} / (1 - \sigma)\) when $\bar\cT_k^{\trial} < +\infty$ (see~\eqref{eq.merit_parameter_update}), i.e.,
\begin{align*}
T_n\left(
\begin{bmatrix} G ; C; \Vectorize{}(J) \end{bmatrix}\right) = \norm{C}_1 \quad
\text{and} \quad
T_d\left(\begin{bmatrix} G ; C; \Vectorize{}(J) \end{bmatrix}, H\right) = - \frac{1}{2}
\begin{bmatrix}
G \\ C
\end{bmatrix}^T
\begin{bmatrix}
H & J^T \\ J & 0
\end{bmatrix}^{-1}
\begin{bmatrix}
G \\ C
\end{bmatrix},
\end{align*}
where \(\begin{bmatrix} G ; C; \Vectorize{}(J) \end{bmatrix}\) represents the vertical concatenation of column vectors \(G, C,\) and \(\Vectorize{}(J)\). Accordingly, it holds from \eqref{eq.merit_parameter_update} that when $\bar\cT_k^{\trial} < +\infty$,
\begin{align}
\frac{\bar{\cT}^{\trial}_k}{(1-\sigma)} :=
 \frac{\|\bar{C}_k\|_1}{\bar{G}_k^T \bar{D}_k + \frac{1}{2} \bar{D}^T_k H_k \bar{D}_k }
=
 \frac{\|\bar{C}_k\|_1}{\frac{1}{2} \bar{G}^T_k\bar{D}_k + \frac{1}{2} \bar{C}_k^T \bar{Y}_k}
=
\frac{T_n\left(
\begin{bmatrix} \bar{G}_k; \bar{C}_k; \Vectorize{}\left(\bar{J}_k\right)\end{bmatrix}
\right)
}{
T_d\left(
\begin{bmatrix} \bar{G}_k; \bar{C}_k; \Vectorize{}\left(\bar{J}_k\right)\end{bmatrix}, H_k\right)}, \label{eq:tau-trial-TnTd}
\end{align}
where the last two equalities are both from~\eqref{eq.linear_system}. Similarly, one can also use \(T_n\) and \(T_d\) to represent \(\cT_k^{\trial}\) as
\begin{align}
\cT_k^{\trial} = \frac{ (1-\sigma) \cdot
T_n\left(
\bbmatrix \nabla f(X_k) ; c(X_k) ; \Vectorize{}\left(\nabla c(X_k)^T\right) \ebmatrix
\right)
}{
T_d\left(
\bbmatrix \nabla f(X_k) ; c(X_k) ; \Vectorize{}\left(\nabla c(X_k)^T\right)\ebmatrix, H_k\right)
} \label{eq:tau-k-TnTd}
\end{align}
when $\cT_k^{\trial} < +\infty$.
 Additionally, we slightly abuse the notation by denoting
\[
Z_k :=
\bbmatrix \nabla f(X_k) ; c(X_k) ; \Vectorize{}(\nabla c(X_k)^T)\ebmatrix,
\quad
\bar{Z}_k :=
\bbmatrix \bar{G}_k ; \bar{C}_k ; \Vectorize{}(\bar{J}_k) \ebmatrix,\ \ \text{and} \
 \bar{\Delta}_k := \bar{Z}_k - Z_k
\]
as the true \(Z_k\), the estimated \(Z_k\), and the associated estimation error, respectively, which help to rewrite \(\bar{\cT}_k^{\trial} = \tfrac{(1-\sigma)T_n(\bar{Z}_k)}{T_d(\bar{Z}_k, H_k)}\) when \(\bar{\cT}_k^{\trial} < +\infty\). Moreover, by Assumption~\ref{ass.estimate_nonasymptotic},  \inrevise{conditioned on \(\cF_k'\)}, it holds for all \(k\in\NN\) that
\begin{align*}
\Expect_k \left[ \bar{\Delta}_k \right] = 0,
\quad 0 \in \supp{\bar{\Delta}_k},
\ \ \text{and} \ \ \Expect_k \left[ \norm{\bar{\Delta}_k}^2 \right] \leq \rho^g_k + \rho^c_k + \rho^j_k.
\end{align*}
Now, with~\eqref{eq:tau-trial-TnTd} and~\eqref{eq:tau-k-TnTd}, we compare \(\bar{\cT}_k^\trial\) and \(\cT^{\trial}_k\) by quantifying the differences \(T_n(Z_k + \bar{\Delta}_k) - T_n(Z_k)\) and \(T_d(Z_k + \bar{\Delta}_k, H_k) - T_d(Z_k, H_k)\), and then obtain the following iterate dependent lower bound on the probability of \(\bar{\cT}_k \leq \cT_k^{\trial}\).

\blemma\label{lem.prob-tau-good-k}
Suppose that Assumptions~\ref{ass.prob}--\ref{ass.estimate},~\ref{ass.event_nonasymptotic}, and~\ref{ass.estimate_nonasymptotic} hold, and define $\epsilon_k := \epsilon_{\tau} \norm{c(X_k)}_1$, where $\epsilon_{\tau}\in (0,1)$ is a parameter in~\eqref{eq.merit_parameter_update}.
Then, there exists a constant \(\nc\label{cnst:lem.prob-tau-good-k} > 0\) that only depends on \((q_{\text{min}}, \tauzero, \sigma, \epsilon_{\tau})\) defined in Lemma~\ref{lem.singular_values} and Algorithm~\ref{alg.main} such that
{\small
\begin{align}
\Prob_k \left[\bar{\cT}_k \leq \cT^{\trial}_k \right]
\geq
\begin{cases}
\displaystyle \frac{
\Expect_k \left[ \left(\epsilon_k - \rc{cnst:lem.prob-tau-good-k} \norm{\bar{\Delta}_k}  \right) \cdot \ones \left( \norm{\bar{\Delta}_k} < \frac{\epsilon_k}{2 \rc{cnst:lem.prob-tau-good-k}} \right)  \right]
}{
\epsilon_k - \rc{cnst:lem.prob-tau-good-k} \Expect_k \left[ \norm{\bar{\Delta}_k} \cdot \ones \left( \norm{\bar{\Delta}_k} < \frac{\epsilon_k}{2 \rc{cnst:lem.prob-tau-good-k}} \right) \right] + \left( \rc{cnst:lem.prob-tau-good-k}^2 / \epsilon_k \right) \Expect_k \left[ \norm{\bar{\Delta}_k}^2 \right]
}
& \text{ if \(\cT_k^{\trial} < \tauzero\) and \(\norm{c(X_k)}_1 > 0\),} \\
1 & \text{ otherwise. }
\end{cases}
\label{eq:prob-tau-good-main-k}
\end{align}
}
\elemma

\bproof
See Appendix~\ref{sec:proof-lem-prob-tau-good-k}.
\eproof

Although this lemma assumes Assumptions~\ref{ass.event_nonasymptotic} and~\ref{ass.estimate_nonasymptotic}, it does not rely on the occurance of \(\nonE\). Therefore, Lemma~\ref{lem.prob-tau-good-k} remains valid without Assumption~\ref{ass.event_nonasymptotic} and if Assumption~\ref{ass.estimate_nonasymptotic} holds with respective to the natural filtration \(\{\mathcal{G}_k\}\).

When \(\cT_k^{\trial} < \tauzero\) and \(\norm{c(X_k)}_1 > 0\),~\eqref{eq:prob-tau-good-main-k} is nontrivial, because \(0 \in \supp{\bar{\Delta}_k}\) for all $k$ (by Assumption~\ref{ass.estimate_nonasymptotic}), and then it holds that
\begin{align*}
\Expect_k \left[
(\epsilon_k - \rc{cnst:lem.prob-tau-good-k} \norm{\bar{\Delta}_k}) \cdot
\ones \left(
\norm{\bar{\Delta}_k} \leq \frac{\epsilon_k}{2 \rc{cnst:lem.prob-tau-good-k}}
\right)
\right]
\geq
\Expect_k \left[ (\epsilon_k - \epsilon_k / 2) \cdot \ones \left(
\norm{\bar{\Delta}_k} \leq \frac{\epsilon_k}{2 \rc{cnst:lem.prob-tau-good-k}}
\right)
\right] > 0.
\end{align*}
Meanwhile, \eqref{eq:prob-tau-good-main-k} is asymptotically tight when the variance \(\Expect_k \left[ \norm{\bar{\Delta}_k}^2 \right]\) diminishes to zero because
\begin{align*}
&\frac{
\Expect_k \left[ \left(\epsilon_k - \rc{cnst:lem.prob-tau-good-k} \norm{\bar{\Delta}_k}  \right) \cdot \ones \left( \norm{\bar{\Delta}_k} < \frac{\epsilon_k}{2 \rc{cnst:lem.prob-tau-good-k}} \right)  \right]
}{
\epsilon_k - \rc{cnst:lem.prob-tau-good-k} \Expect_k \left[ \norm{\bar{\Delta}_k} \cdot \ones \left( \norm{\bar{\Delta}_k} < \frac{\epsilon_k}{2 \rc{cnst:lem.prob-tau-good-k}} \right) \right] + \left( \rc{cnst:lem.prob-tau-good-k}^2 / \epsilon_k \right) \Expect_k \left[ \norm{\bar{\Delta}_k}^2 \right]
} \\
= \
&\frac{
\Expect_k \left[ (\epsilon_k - \rc{cnst:lem.prob-tau-good-k} \norm{\bar{\Delta}_k}) \cdot \ones \left( \norm{\bar{\Delta}_k} < \frac{\epsilon_k}{2\rc{cnst:lem.prob-tau-good-k}}\right) \right]
}{
\Expect_k \left[ (\epsilon_k - \rc{cnst:lem.prob-tau-good-k} \norm{\bar{\Delta}_k}) \cdot \ones \left( \norm{\bar{\Delta}_k} < \frac{\epsilon_k}{2\rc{cnst:lem.prob-tau-good-k}}\right) \right] + \epsilon_k \cdot \Prob_k \left[ \norm{\bar{\Delta}_k} \geq\frac{\epsilon_k}{2\rc{cnst:lem.prob-tau-good-k}}\right] + \left( \rc{cnst:lem.prob-tau-good-k}^2 / \epsilon_k \right) \Expect_k \left[ \norm{\bar{\Delta}_k}^2 \right]
}\\
\geq \
&\frac{
\Expect_k \left[ (\epsilon_k - \rc{cnst:lem.prob-tau-good-k} \norm{\bar{\Delta}_k}) \cdot \ones \left( \norm{\bar{\Delta}_k} < \frac{\epsilon_k}{2\rc{cnst:lem.prob-tau-good-k}}\right) \right]
}{
\Expect_k \left[ (\epsilon_k - \rc{cnst:lem.prob-tau-good-k} \norm{\bar{\Delta}_k}) \cdot \ones \left( \norm{\bar{\Delta}_k} < \frac{\epsilon_k}{2\rc{cnst:lem.prob-tau-good-k}}\right) \right] + \left(
 4 \rc{cnst:lem.prob-tau-good-k}^2 / \epsilon_k + \rc{cnst:lem.prob-tau-good-k}^2 / \epsilon_k \right) \Expect_k \left[ \norm{\bar{\Delta}_k}^2 \right]
} \xrightarrow[]{\Expect_k \left[ \norm{\bar{\Delta}_k}^2 \right]\searrow 0} 1,
\end{align*}
where the inequality follows from conditional Markov's inequality.

However, when applying Lemma~\ref{lem.prob-tau-good-k} to a sequence of stochastic iterates, it only provides nontrival lower bounds when the associated \inrevise{sequence of infeasibility errors are uniformly bounded away from zero.} To this end, we show in the following lemma that this is indeed true when the iterate \((X_k, Y_k)\) is~\emph{not} near-stationary and \(\cT^{\trial}_k < +\infty\). Specifically, for any \(\delta > 0\), we use the following event to represent the collection of realizations whose $k$th iterate \((X_k, Y_k)\) is~\emph{not} $\delta$-stationary, i.e.,
\begin{align}\label{eq.nonopt_event}
\mathcal{B}_{\delta,k} := \left\{ \;
\norm{\nabla f(X_k) + \nabla c(X_k) Y_k}^2 + \norm{c(X_k)}_1 > \delta \;
\right\}.
\end{align}

\blemma\label{lem.ck-bdd-away}
Suppose that Assumptions~\ref{ass.prob}--\ref{ass.estimate} hold. Then, there exists \(\kappa_y > 0\) that only depends on \((\kappa_{\nabla f}, \kappa_c, q_{\text{min}})\) defined in~\eqref{eq.basic_condition} and Lemma~\ref{lem.singular_values} such that \(\norm{Y_k}_{\infty} \leq \kappa_y\) and \(\norm{D_k} \leq \kappa_y\) hold for all iterations \(k\in\NN\), and
\begin{align*}
\norm{D_k}^2 < \frac{2 \kappa_y}{\zeta} \norm{c(X_k)}_1
\quad \text{and} \quad
\frac{\zeta \delta}{2 \kappa^2_H \kappa_y + \zeta} < \norm{c(X_k)}_1
\end{align*}
whenever \(\cT^{\trial}_k < +\infty\) and \(\mathcal{B}_{\delta,k}\) occurs (see~\eqref{eq.nonopt_event}).
\elemma

\bproof
By Assumption~\ref{ass.prob},~\eqref{eq.basic_condition},~\eqref{eq.linear_system_true}, Lemma~\ref{lem.singular_values} and the Cauchy-Schwarz inequality, it follows that for any $k\in\NN$, by choosing a sufficiently large $\kappa_y \geq \frac{\kappa_{\nabla f} + \kappa_c}{q_{\min}}$, we have
\begin{equation}\label{eq.yk_bound}
\max \Set{\norm{Y_k}_{\infty}, \norm{D_k}} \leq \left\|\bbmatrix D_k \\ Y_k\ebmatrix\right\| \leq \left\|\bbmatrix H_k & \nabla c(X_k) \\ \nabla c(X_k)^T & 0 \ebmatrix^{-1}\right\| \left\|\bbmatrix \nabla f(X_k) \\ c(X_k) \ebmatrix\right\| \leq \frac{\kappa_{\nabla f} + \kappa_c}{q_{\min}} \leq \kappa_y.
\end{equation}
Moreover, by the definition of $\cT_k^{\trial}$, $\cT_k^{\trial} < +\infty$, Assumption~\ref{ass.H},~\eqref{eq.linear_system_true}, the Cauchy-Schwarz inequality and~\eqref{eq.yk_bound}, it holds for any $k\in\NN$ that
\begin{align}
0 & < \nabla f(X_k)^TD_k + \frac{1}{2}D_k^TH_kD_k
= D_k^T(\nabla f(X_k) + H_kD_k) - \frac{1}{2}D_k^TH_kD_k = -D_k^T\nabla c(X_k)Y_k - \frac{1}{2}D_k^TH_kD_k \nonumber{} \\
& \leq c(X_k)^TY_k - \frac{\zeta}{2}\|D_k\|^2
\leq
\|c(X_k)\|_1\|Y_k\|_{\infty} - \frac{\zeta}{2}\|D_k\|^2 \leq \kappa_y\|c(X_k)\|_1 - \frac{\zeta}{2}\|D_k\|^2, \label{eq.ck_lb_by_dk}
\end{align}
implying that \(\norm{D_k}^2 < 2 \kappa_y \norm{c(X_k)}_1 / \zeta\).
Furthermore, by the occurrence of \(\mathcal{B}_{\delta,k}\), Assumption~\ref{ass.H},~\eqref{eq.linear_system_true}, and~\eqref{eq.ck_lb_by_dk}, it follows that
\begin{align*}
\delta
& < \norm{\nabla f(X_k) + \nabla c(X_k)Y_k}^2 + \norm{c(X_k)}_1
= \|H_kD_k\|^2 + \norm{c(X_k)}_1  \nonumber{} \\
& \leq \kappa_H^2\|D_k\|^2 + \norm{c(X_k)}_1
< \left(\frac{2\kappa_H^2\kappa_y}{\zeta} + 1\right)\cdot \norm{c(X_k)}_1,
\end{align*}
yielding $\|c(X_k)\|_1 > \frac{\zeta\delta}{2\kappa_H^2\kappa_y + \zeta}$.
\eproof

\subsubsection{Analysis of the noise term}\label{sec:complex-noise}

In this subsection, we introduce several preparatory lemmas that establish connections between
\begin{enumerate}[label=(\roman*)]
\item\label{it:complex-noise-one} the stationarity measure and the model reduction function, as well as
\item\label{it:complex-noise-two} the model reduction function and the changes in merit functions across iterations.
\end{enumerate}
To facilitate the analysis, we consider the following auxiliary merit parameter sequence (similar to~\cite{CurtRobiOnei24}),
\begin{align}\label{eq.tau_hat_def}
\hat{\cT}_k := \min \{ \bar{\cT}_k, \cT_k^{\trial}\} \quad \forall k \in \NN,
\end{align}
and we will focus on the relation between \(\hat{\cT}_k\) and \(\bar{\cT}_k\). To begin with, we achieve~\ref{it:complex-noise-one} in the following lemma.

\blemma\label{lem:kkt-delta-l}
Suppose that Assumptions~\ref{ass.prob}--\ref{ass.estimate} and~\ref{ass.event_nonasymptotic} hold. Then, for any iteration \(k \in \mathbb{N}\), it holds that
\begin{align*}
\norm{\nabla f(X_k) + \nabla c(X_k) Y_k}^2 + \norm{c(X_k)}_1
\leq  \frac{2 \kappa_H}{\min\{\taumin,1\} \zeta \sigma}
\Delta l(X_k, \hat{\cT}_k, \nabla f(X_k), c(X_k), \nabla c(X_k)^T, D_k).
\end{align*}
\elemma

\bproof
By~\eqref{eq.linear_system_true} and Assumption~\ref{ass.H}, it holds that
\begin{align*}
& \norm{\nabla f(X_k) + \nabla c(X_k) Y_k}^2 + \norm{c(X_k)}_1
= \norm{H_k D_k}^2 + \norm{c(X_k)}_1
\leq \kappa_H \norm{D_k}^2 + \frac{1}{\sigma} \cdot \sigma \norm{c(X_k)}_1\\
\leq ~ & \frac{2 \kappa_H}{\zeta} \frac{\hat{\cT}_k}{\taumin} \frac{\zeta}{2 \sigma} \norm{D_k}^2
  + \frac{1}{\sigma} \cdot \sigma \norm{c(X_k)}_1 \cdot \frac{2 \kappa_H}{\min\{\taumin,1\} \zeta}
\leq \frac{2 \kappa_H}{\min\{\taumin,1\} \zeta \sigma}
\left( \frac{\zeta}{2} \hat{\cT}_k \norm{D_k}^2 + \sigma \norm{c(X_k)}_1 \right)\\
\leq ~ & \frac{2 \kappa_H}{\min\{\taumin,1\} \zeta \sigma}
  \Delta l \left( X_k, \hat{\cT}_k, \nabla f(X_k), c(X_k), \nabla c(X_k)^T, D_k \right),
\end{align*}
where the second inequality is due to \(\taumin \leq \min \{\cT_k^{\trial}, \bar{\cT}_k\}=\hat{\cT}_k\) (see Assumption~\ref{ass.event_nonasymptotic}), \(\zeta \leq \kappa_H\), and \(\sigma \in (0, 1)\), and the last inequality follows Lemma~\ref{lem.model_reduction_suff_large_true}.
\eproof

To achieve~\ref{it:complex-noise-two}, one naturally considers employing Lemma~\ref{lem.merit_function_decrease}:
\begin{equation}\label{eq:complex-change-in-phi-demo}
\begin{aligned}
& \Expect_k \left[ \phi (X_k + \bar{\cA}_k \bar{D}_k, \bar{\cT}_k) \right] - \Expect_k \left[ \phi(X_k, \bar{\cT}_k) \right] \\
& \leq \Expect_k \left[
-\bar\cA_k\Delta l(X_k,\bar\cT_k,\nabla f(X_k),c(X_k),\nabla c(X_k)^T,D_k) +
(1-\eta)\bar\cA_k\beta_k\Delta l(X_k,\bar\cT_k,\bar{G}_k,\bar{C}_k,\bar{J}_k,\bar{D}_k)
 \right. \\
& \phantom{\leq \Expect_k \left[ \right]}
\left. + \bar\cA_k\bar\cT_k\nabla f(X_k)^T(\bar{D}_k - D_k) + \bar\cA_k\|\nabla c(X_k)^T(D_k - \bar{D}_k)\|_1 \inrevise{+ 2\bar\cA_k\|c(X_k) - \bar{C}_k\|_1} \right].
\end{aligned}
\end{equation}
However, according to Lemma~\ref{lem:kkt-delta-l}, we need \(\Delta l(X_k, \hat{\cT}_k, \nabla f(X_k), c(X_k), \nabla c(X_k)^T, D_k)\) on the right-hand-side of~\eqref{eq:complex-change-in-phi-demo}.
To this end, we relate three quantities, i.e., $\Delta l(X_k, \hat{\cT}_k, \bar{G}_k, \bar{C}_k, \bar{J}_k, \bar{D}_k)$, $\Delta l(X_k, \hat{\cT}_k, \nabla f(X_k), c(X_k), \nabla c(X_k)^T, D_k)$ and $\Delta l(X_k, \bar{\cT}_k, \nabla f(X_k), c(X_k), \nabla c(X_k)^T, D_k)$, in Lemmas~\ref{lem.l-hat-diff} and~\ref{lem.tau-hat-bar} to formally achieve~\ref{it:complex-noise-two} in Lemma~\ref{lem.expect-diff-phi-k}:
\begin{equation}
\label{eq:complex-change-in-phi-demo-2}
\begin{aligned}
& \Expect_k \left[ \phi (X_k + \bar{\cA}_k \bar{D}_k, \bar{\cT}_k) - \phi(X_k, \bar{\cT}_k) \right]
\leq
- \cAmin \beta_k \Expect_k \left[ \Delta l(X_k,\hat\cT_k,\nabla f(X_k),c(X_k),\nabla c(X_k)^T,D_k) \right] + \rc{cnst:lem.expect-diff-phi-k-2} \beta_k^2  \\
& \qquad + \text{``noise one''}\cdot
\Prob_k \left[ \bar{\cT}_k > \hat{\cT}_k \right]
+ \text{``noise two''} \cdot \left( \sqrt{\rho^g_k} + \sqrt{\rho^c_k} + \sqrt{\rho^j_k} \right),
\end{aligned}
\end{equation}
where ``noise one'' represents errors incurred due to \(\bar{\cT}_k > \hat{\cT}_k\) and ``noise two'' appears because of $\bar{D}_k$ being a biased estimator to $D_k$. Since we allow the variance sequence \(\{(\rho^g_k, \rho^c_k, \rho^j_k)\}\) to be dependent on \(\kmax\), the impact of the ``noise two'' diminishes when \(\kmax\) is large. The impact of ``noise one'' depends on \(\Prob_k \left[ \bar{\cT}_k > \hat{\cT}_k \right]\), which, as formalized in Lemma~\ref{lem.expect-diff-phi-k-bad}, can be controlled by \((\rho^g_k, \rho^c_k, \rho^j_k)\) based on the results in Section~\ref{sec:complex-prob-bound}, and consequently decreases when \(\kmax\) is large. Then, we combine all the results together to formally establish the complexity result in Section~\ref{sec:complex-result}.

In the following lemmas, we consider the following step size values:
\begin{gather}
\cAmin := \frac{2(1 - \eta) \ximin \taumin}{\taumin L + \Gamma} \quad \text{and} \quad
\cAmax := \frac{2(1 - \eta) \bar{\xi}_0 \bar\tau_0}{\bar\tau_0 L + \Gamma}, \label{eq:step-size-bds}
\end{gather}
which, by Assumption~\ref{ass.event_nonasymptotic}, Lemma~\ref{lem.ratio_parameter} and Lemma~\ref{lem.stepsize_interval}, satisfy \(
\cAmin \beta_k \leq \bar{\cA}^{\text{min}}_k \leq \bar{\cA}^{\text{max}}_k \leq \left( \cAmax + \theta \right) \beta_k
\) for all iterations \(k \in \naturals\).

\blemma\label{lem.l-hat-diff}
Suppose that Assumptions~\ref{ass.prob}--\ref{ass.estimate}, \ref{ass.event_nonasymptotic} and \ref{ass.estimate_nonasymptotic} hold. Then, there exists a constant \(\nc\label{cnst:lem.l-hat-diff} > 0\) that only depends on constant parameters \((\omega_3, \omega_4)\) in Lemma~\ref{lem.solution_variance}, \((\omega_5, \omega_6)\) in Lemma~\ref{lem.model_reduction_comparison}, and \(\rhomax\) in Assumption~\ref{ass.estimate_nonasymptotic} such that
\begin{align}
\Expect_k & \left[ \Delta l (X_k, \hat{\cT}_k, \bar{G}_k, \bar{C}_k, \bar{J}_k, \bar{D}_k) \right] \nonumber{} \\
& \leq \Expect_k \left[ \Delta l(X_k, \hat{\cT}_k, \nabla f(X_k), c(X_k), \nabla c(X_k)^T, D_k) \right] +
\rc{cnst:lem.l-hat-diff} \left( \sqrt{\rho^g_k} + \sqrt{\rho^c_k} + \sqrt{\rho^j_k} \right)
\label{eq:l-hat-diff}
\end{align}
holds for all iterations \(k \in \NN\).
\elemma

\bproof
By replacing \(\bar{\cT}_k\) with \(\hat{\cT}_k\) and following a similar logic as in the proof of Lemma~\ref{lem.model_reduction_comparison}, we obtain
\bequation\label{eq.key_relation_0}
\begin{aligned}
\Expect_k \left[ \Delta l (X_k, \hat{\cT}_k, \bar{G}_k, \bar{C}_k, \bar{J}_k, \bar{D}_k) \right]
&\leq \Expect_k \left[ \Delta l(X_k, \hat{\cT}_k, \nabla f(X_k), c(X_k), \nabla c(X_k)^T, D_k) \right] \\
&\quad\quad + \tauzero \left( \omega_5 \cdot \rho^g_k + \omega_6 \cdot \sqrt{\rho^c_k} + \omega_4 \cdot \sqrt{\rho^j_k} + \frac{\omega_3}{2} \sqrt{\rho^j_k \left( \rho^g_k + \rho^c_k \right)} \right).
\end{aligned}
\eequation
Since \(ab \leq \left( a^2 + b^2 \right) / 2\) and \(\sqrt{a + b} \leq \sqrt{a} + \sqrt{b}\) for any $\{a,b\}\subset\RR_{\geq 0}$, the last term in~\eqref{eq.key_relation_0} satisfies
\begin{align*}
& \omega_5 \cdot \rho^g_k + \omega_6 \cdot \sqrt{\rho^c_k} + \omega_4 \cdot \sqrt{\rho^j_k} + \frac{\omega_3}{2} \sqrt{\rho^j_k \left( \rho^g_k + \rho^c_k \right)}
\leq
\omega_5 \cdot \rho^g_k + \omega_6 \cdot \sqrt{\rho^c_k} + \omega_4 \cdot \sqrt{\rho^j_k} + \frac{\omega_3}{4}
\left( \rho^j_k + \rho^g_k + \rho^c_k \right) \\
& = \left(\omega_5 + \frac{\omega_3}{4}  \right) \rhomax \cdot \frac{\rho^g_k}{\rhomax} + \omega_6 \sqrt{\rho^c_k}
+ \frac{\omega_3 \rhomax}{4} \cdot \frac{\rho^c_k}{\rhomax}
+ \omega_4 \sqrt{\rho^j_k}
+ \frac{\omega_3 \rhomax}{4} \cdot \frac{\rho^j_k}{\rhomax} \\
& \leq \left(\omega_5 + \frac{\omega_3}{4}  \right) \sqrt{\rhomax}  \cdot \sqrt{\rho^g_k}
+ \left( \omega_6 + \frac{\omega_3 \sqrt{\rhomax}}{4} \right) \cdot \sqrt{\rho^c_k}
+ \left( \omega_4 + \frac{\omega_3 \sqrt{\rhomax}}{4} \right) \cdot \sqrt{\rho^j_k},
\end{align*}
where the last inequality follows Assumption~\ref{ass.estimate_nonasymptotic}. Consequently,~\eqref{eq:l-hat-diff} follows by taking a large enough positive constant
\(\rc{cnst:lem.l-hat-diff} \geq \bar\tau_0\cdot\max\left\{ \left(\omega_5 + \frac{\omega_3}{4}  \right) \sqrt{\rhomax}, \omega_6 + \frac{\omega_3 \sqrt{\rhomax}}{4}, \omega_4 + \frac{\omega_3 \sqrt{\rhomax}}{4}\right\}\).
\eproof

\inrevise{For ease of exposition, we use \(\Expect_k \left[ Z \bdsemicolon E \right] := \Expect_k \left[ Z \cdot \Ind{E} \right]\) to denote the corresponding (conditional) expectation of a random vector \(Z\) when an event $E$ occurs.}

\blemma\label{lem.tau-hat-bar}
Suppose that Assumptions~\ref{ass.prob}--\ref{ass.estimate}, \ref{ass.event_nonasymptotic} and \ref{ass.estimate_nonasymptotic} hold.
Then, for any iteration \(k \in \NN\),
\begin{align*}
& \Expect_k \left[ \bar{\cA}_k \Delta l(X_k, \hat{\cT}_k, \nabla f(X_k), c(X_k), \nabla c(X_k)^T, D_k) - \bar{\cA}_k \Delta l(X_k, \bar{\cT}_k, \nabla f(X_k), c(X_k), \nabla c(X_k)^T, D_k)
\bdsemicolon \bar{\cT}_k > \hat{\cT}_k
\right]
\\
& \qquad \leq \left( \cAmax + \theta \right) \beta_k \left( \tauzero - \taumin \right) \left( \kappa_H \norm{D_k}^2 + \norm{c(X_k)}_1 \norm{Y_k}_{\infty} \right) \cdot \Prob_k \left[ \bar{\cT}_k > \hat{\cT}_k \right]
\end{align*}
holds, and there exists \(\nc\label{cnst:lem.tau-hat-bar} > 0\) that only depends on the \((\omega_2, \omega_3)\) in Lemma~\ref{lem.solution_variance}, \(q_{\text{min}}\) in Lemma~\ref{lem.singular_values}, \(\kappa_y\) in Lemma~\ref{lem.ck-bdd-away}, \inrevise{\(\kappa_{\nabla f}\) in Assumption~\ref{ass.prob},} and \(\rhomax\) in Assumption~\ref{ass.estimate_nonasymptotic}  such that
\begin{align*}
& \Expect_k \left[
 \bar{\cA}_k \Delta l(X_k, \bar{\cT}_k, \bar{G}_k, \bar{C}_k, \bar{J}_k, \bar{D}_k) - \bar{\cA}_k \Delta l( X_k, \hat{\cT}_k, \bar{G}_k, \bar{C}_k, \bar{J}_k, \bar{D}_k)
\bdsemicolon \bar{\cT}_k > \hat{\cT}_k
  \right] \\
& \qquad  \leq \left( \cAmax + \theta \right) \beta_k (\tauzero - \taumin) \inrevise{\left(
  \left( \kappa_H \norm{D_k}^2 + \norm{c(X_k)}_1 \norm{Y_k}_{\infty} \right) \cdot \Prob_k \left[ \bar{\cT}_k > \hat{\cT}_k \right]
  + \rc{cnst:lem.tau-hat-bar} \left( \sqrt{\rho^g_k} + \sqrt{\rho^c_k} + \sqrt{\rho^j_k} \right) \right)}
\end{align*}
for all iterations \(k \in \NN\).
\elemma

\bproof
By~\eqref{eq.linear_model_reduction}, one finds
\begin{align}
& \Expect_k \left[
\bar{\cA}_k  \Delta l(X_k, \hat{\cT}_k, \nabla f(X_k), c(X_k), \nabla c(X_k)^T, D_k) - \bar{\cA}_k \Delta l(X_k, \bar{\cT}_k, \nabla f(X_k), c(X_k), \nabla c(X_k)^T, D_k)
\bdsemicolon \bar{\cT}_k > \hat{\cT}_k
\right] \nonumber{} \\
& = \Expect_k \left[
\bar{\cA}_k \left( - \hat{\cT}_k \nabla f(X_k)^{\top} D_k + \norm{c(X_k)}_1 \right) -
\bar{\cA}_k \left( - \bar{\cT}_k \nabla f(X_k)^{\top} D_k + \norm{c(X_k)}_1 \right)
\bdsemicolon \bar{\cT}_k > \hat{\cT}_k
\right] \nonumber{} \\
& = \Expect_k \left[\bar{\cA}_k \left(\bar{\cT}_k - \hat{\cT}_k \right) \nabla f(X_k)^{\top} D_k
\bdsemicolon \bar{\cT}_k > \hat{\cT}_k
\right]
\leq \Expect_k \left[\bar{\cA}^{\text{max}}_k \left(\bar{\cT}_k - \hat{\cT}_k \right) \abs{\nabla f(X_k)^{\top} D_k}
\bdsemicolon \bar{\cT}_k > \hat{\cT}_k
\right], \label{eq:tau-hat-bar-1}
\end{align}
where the inequality follows \(\hat{\cT}_k \leq \bar{\cT}_k\), \(\bar{\cA}_k \leq \bar{\cA}_k^{\text{max}}\) and the Cauchy-Schwarz inequality. Moreover, by~\eqref{eq.linear_system_true},
\begin{align*}
\abs{\nabla f(X_k)^{\top} D_k}
& = \abs{D^{\top}_k H_k D_k + D_k^{\top} \nabla c(X_k) Y_k}
= \abs{D^{\top}_k H_k D_k - c(X_k)^{\top} Y_k}
\leq |D^{\top}_k H_k D_k|  + \norm{c(X_k)}_1 \norm{Y_k}_{\infty} \\
& \leq \kappa_H \norm{D_k}^2 + \norm{c(X_k)}_1 \norm{Y_k}_{\infty},
\end{align*}
where the last two inequalities are due to triangle inequality and the Cauchy-Schwarz inequality, as well as Assumption~\ref{ass.H}. Combining Assumption~\ref{ass.event_nonasymptotic}, Lemmas~\ref{lem.merit_parameter} and~\ref{lem.stepsize_interval}, \eqref{eq:step-size-bds} and the inequality above, we have
\begin{align}
\eqref{eq:tau-hat-bar-1}
&\leq
\Expect_k \left[
\left(\cAmax + \theta \right) \beta_k (\tauzero - \taumin) \left( \kappa_H \norm{D_k}^2 + \norm{c(X_k)}_1 \norm{Y_k}_{\infty} \right)
\bdsemicolon \bar{\cT}_k > \hat{\cT}_k
\right] \nonumber{} \\
&= \left(\cAmax + \theta \right) \beta_k (\tauzero - \taumin) \left( \kappa_H \norm{D_k}^2 + \norm{c(X_k)}_1 \norm{Y_k}_{\infty} \right) \cdot \Prob_k \left[ \bar{\cT}_k > \hat{\cT}_k
 \right] \label{eq:tau-hat-bar-2}
\end{align}
where the equality is because \((D_k, X_k, Y_k)\) only depend on \((X_k, Y_k)\), hence \inrevise{\(\mathcal{F}'_k\)-measurable}.
It follows the similar logic that
\begin{align}
& \Expect_k \left[ \bar{\cA}_k  \Delta l(X_k, \bar{\cT}_k, \bar{G}_k, \bar{C}_k, \bar{J}_k, \bar{D}_k) - \bar{\cA}_k \Delta l(X_k, \hat{\cT}_k, \bar{G}_k, \bar{C}_k, \bar{J}_k, \bar{D}_k)
\bdsemicolon \bar{\cT}_k > \hat{\cT}_k
\right] \nonumber{}\\
& = \Expect_k \left[
  \bar{\cA}_k \left( - \bar{\cT}_k \bar{G}_k^{\top} \bar{D}_k + \norm{\bar{C}_k}_1 \right)
- \bar{\cA}_k  \left( - \hat{\cT}_k \bar{G}_k^{\top} \bar{D}_k + \norm{\bar{C}_k}_1 \right)
\bdsemicolon \bar{\cT}_k > \hat{\cT}_k
 \right] \nonumber{}\\
& = \Expect_k \left[
  - \bar{\cA}_k (\bar{\cT}_k - \hat{\cT}_k) \cdot \bar{G}_k^{\top} \bar{D}_k
\bdsemicolon \bar{\cT}_k > \hat{\cT}_k
 \right]
\leq \inrevise{\Expect_k \left[
  \bar{\cA}^{\text{max}}_k (\bar{\cT}_k - \hat{\cT}_k) \cdot \abs{ \bar{G}_k^{\top} \bar{D}_k }
\bdsemicolon \bar{\cT}_k > \hat{\cT}_k
 \right]} \label{eq:tau-hat-bar-3},
\end{align}
\inrevise{where the inequality is by the Cauchy-Schwarz inequality. Furthermore, by the Cauchy-Schwarz inequality, Assumption~\ref{ass.prob}, Lemma~\ref{lem.ck-bdd-away}, and \(ab \leq  a^2 + b^2 \) for any \(\{a,b\}\subset\mathbb{R}\), one has
\begin{align*}
  \abs{\bar{G}_k^{\top} \bar{D}_k}
  & = \abs{\left( \bar{G}_k - \nabla f(X_k) + \nabla f(X_k) \right)^{\top} \left( \bar{D}_k - D_k + D_k \right)} \\
  & = \abs{\left( \bar{G}_k - \nabla f(X_k) \right)^{\top} \left( \bar{D}_k - D_k \right)
    + \nabla f(X_k)^{\top} \left( \bar{D}_k - D_k \right)
    + \left( \bar{G}_k - \nabla f(X_k) \right)^{\top} D_k
    + \nabla f(X_k)^{\top} D_k} \\
  & \leq \abs{\nabla f(X_k)^{\top} D_k}
    + \norm{\bar{G}_k - \nabla f(X_k)} \norm{\bar{D}_k - D_k}
    + \kappa_{\nabla f} \norm{\bar{D}_k - D_k}
    + \kappa_y \norm{\bar{G}_k - \nabla f(X_k)} \\
  & \leq \abs{\nabla f(X_k)^{\top} D_k}
    + \norm{\bar{G}_k - \nabla f(X_k)}^2 +  \norm{\bar{D}_k - D_k}^2
    + \kappa_{\nabla f} \norm{\bar{D}_k - D_k}
    + \kappa_y \norm{\bar{G}_k - \nabla f(X_k)}.
\end{align*}
\inrevise{Accordingly, it holds by 
\(0 \leq \bar{\cA}^{\text{max}}_k \leq (\cA^{\text{max}} + \theta) \beta_k\), \(\bar{\cT}_k \geq \hat{\cT}_k\), Assumptions~\ref{ass.event_nonasymptotic} and~\ref{ass.estimate_nonasymptotic}, Lemmas~\ref{lem.solution_bias}--\ref{lem.solution_variance}, and Jensen's inequality that}
\begin{align*}
  \eqref{eq:tau-hat-bar-3}
  & \leq \Expect_k \left[ \bar{\cA}^{\text{max}}_k \left( \bar{\cT}_k - \hat{\cT}_k \right) \abs{\nabla f(X_k)^{\top} D_k} \bdsemicolon \bar{\cT}_k > \hat{\cT}_k \right]
    + \Expect_k \left[ \bar{\cA}^{\text{max}}_k \left( \bar{\cT}_k - \hat{\cT}_k \right) \kappa_{\nabla f} \norm{\bar{D}_k - D_k} \right] \\
  & \qquad + \Expect_k \left[ \bar{\cA}^{\text{max}}_k \left( \bar{\cT}_k - \hat{\cT}_k \right) \kappa_y \norm{\bar{G}_k - \nabla f(X_k)} \right]
    + \Expect_k \left[ \bar{\cA}^{\text{max}}_k \left( \bar{\cT}_k - \hat{\cT}_k \right) \norm{\bar{G}_k - \nabla f(X_k)}^2  \right] \\
  & \qquad + \Expect_k \left[ \bar{\cA}^{\text{max}}_k \left( \bar{\cT}_k - \hat{\cT}_k \right) \norm{\bar{D}_k - D_k}^2  \right] \\
  & \leq \Expect_k \left[ \bar{\cA}^{\text{max}}_k \left( \bar{\cT}_k - \hat{\cT}_k \right) \abs{\nabla f(X_k)^{\top} D_k} \bdsemicolon \bar{\cT}_k > \hat{\cT}_k \right]
    + \left( \cA^{\text{max}} + \theta  \right) \beta_k \left(\tauzero - \taumin\right) \left( \left(\sqrt{\rhomax} + \kappa_y\right) \sqrt{\rho_k^g} \right. \\
  & \qquad + \left. \left( \frac{2}{q_{\text{min}}} +  \omega_1 \right) \kappa_{\nabla f} \left( \sqrt{\rho^g_k} + \sqrt{\rho^c_k} + \sqrt{\rho^j_k}\right) + \left( \frac{1}{q_{\text{min}}^2} + \omega_2 + \omega_3 \right) \sqrt{\rhomax} \left( \sqrt{\rho^g_k} + \sqrt{\rho^c_k} + \sqrt{\rho^j_k} \right) \right).
\end{align*}
By choosing \(\rc{cnst:lem.tau-hat-bar}\) such that
\begin{align*}
\rc{cnst:lem.tau-hat-bar}
\geq \sqrt{\rhomax} + \kappa_y + \left( \frac{2}{q_{\text{min}}} +  \omega_1 \right) \kappa_{\nabla f} + \left( \frac{1}{q_{\text{min}}^2} + \omega_2 + \omega_3 \right) \sqrt{\rhomax},
\end{align*}
one has
\begin{align*}
\eqref{eq:tau-hat-bar-3}
  & \leq \Expect_k \left[ \bar{\cA}^{\text{max}}_k \left( \bar{\cT}_k - \hat{\cT}_k \right) \abs{\nabla f(X_k)^{\top} D_k} \bdsemicolon \bar{\cT}_k > \hat{\cT}_k \right]
    + \left( \cA^{\text{max}} + \theta  \right) \beta_k \left(\tauzero - \taumin\right) \rc{cnst:lem.tau-hat-bar} \left( \sqrt{\rho^g_k} + \sqrt{\rho^c_k} + \sqrt{\rho^j_k}  \right),
\end{align*}
which together with~\eqref{eq:tau-hat-bar-2} finishes the proof.}
\eproof

Building on the results of Lemmas~\ref{lem.l-hat-diff} and~\ref{lem.tau-hat-bar}, we formally derive~\eqref{eq:complex-change-in-phi-demo-2} in the following lemma.

\blemma\label{lem.expect-diff-phi-k}
Suppose that Assumptions~\ref{ass.prob}--\ref{ass.estimate}, \ref{ass.event_nonasymptotic} and \ref{ass.estimate_nonasymptotic} hold.
Then, there exist positive constants \(\{\nc\label{cnst:lem.expect-diff-phi-k-1}, \nc\label{cnst:lem.expect-diff-phi-k-2}\} \subset \reals_{>0}\) that only depend on \((\eta, \theta, \tauzero)\) in Algorithm~\ref{alg.main}, \((\kappa_{\nabla f}, \kappa_c)\) in Assumption~\ref{ass.prob}, \(\taumin\) in Assumption~\ref{ass.event_nonasymptotic}, \(\rhomax\) in Assumption~\ref{ass.estimate_nonasymptotic}, \(\cAmax\) in~\eqref{eq:step-size-bds}, \(q_{\text{min}}\) in Lemma~\ref{lem.singular_values}, \(\kappa_y\) in Lemma~\ref{lem.ck-bdd-away}, \(\rc{cnst:lem.l-hat-diff}\) in Lemma~\ref{lem.l-hat-diff}, and \(\rc{cnst:lem.tau-hat-bar}\) in Lemma~\ref{lem.tau-hat-bar} such that
\begin{align}
\label{eq:expect-diff-phi-k-main}
& \Expect_k \left[ \phi (X_k + \bar{\cA}_k \bar{D}_k, \bar{\cT}_k) \right] - \Expect_k \left[ \phi(X_k, \bar{\cT}_k) \right] \nonumber{} \\
&\qquad \leq \inrevise{((1 - \eta) (\cAmax{} + \theta) \beta^2_k - \cAmin{} \beta_k) \Expect_k \left[ \Delta l(X_k,\hat\cT_k,\nabla f(X_k),c(X_k),\nabla c(X_k)^T,D_k) \right]} \\
&\qquad \qquad +  \inrevise{2}\left( \cAmax + \theta \right) \beta_k (\tauzero - \taumin)
\left( \kappa_H \norm{D_k}^2 + \norm{c(X_k)}_1 \norm{Y_k}_{\infty} \right) \cdot
\Prob_k \left[ \bar{\cT}_k > \hat{\cT}_k \right] \nonumber{} \\
&\qquad\qquad + \inrevise{\left(  \beta_k \rc{cnst:lem.expect-diff-phi-k-1} + \beta^2_k   \rc{cnst:lem.expect-diff-phi-k-2} \right)}\cdot \left( \sqrt{\rho^g_k} + \sqrt{\rho^c_k} + \sqrt{\rho^j_k} \right),
\end{align}
holds for all iterations \(k \in \naturals\).
\elemma

\bproof
By Lemma~\ref{lem.merit_function_decrease}, it holds for every iteration \(k\) that
\bequation\label{eq.key_relation_1}
\begin{aligned}
& \Expect_k \left[ \phi (X_k + \bar{\cA}_k \bar{D}_k, \bar{\cT}_k) \right] - \Expect_k \left[ \phi(X_k, \bar{\cT}_k) \right] \\
& \leq \Expect_k \left[
-\bar\cA_k\Delta l(X_k,\bar\cT_k,\nabla f(X_k),c(X_k),\nabla c(X_k)^T,D_k) +
(1-\eta)\bar\cA_k\beta_k\Delta l(X_k,\bar\cT_k,\bar{G}_k,\bar{C}_k,\bar{J}_k,\bar{D}_k)
 \right. \\
& \phantom{\leq \Expect_k \left[ \right]}
\left. + \bar\cA_k\bar\cT_k\nabla f(X_k)^T(\bar{D}_k - D_k) + \bar\cA_k\|\nabla c(X_k)^T(D_k - \bar{D}_k)\|_1  \inrevise{+ 2 \bar\cA_k \norm{c(X_k) - \bar{C}_k}_1}\right].
\end{aligned}
\eequation
Moreover, we can bound the last \inrevise{three} terms in~\eqref{eq.key_relation_1} by choosing a large positive constant satisfying $\rc{cnst:lem.expect-diff-phi-k-1} \geq \inrevise{(\cAmax + \theta) (( \tauzero \kappa_{\nabla f} + \sqrt{m} \kappa_{\nabla c}) \max\{1 / q_{\text{min}}, \omega_1\} + 2 \sqrt{m})}$ such that
\begin{equation}
\label{eq.key_relation_2}
\begin{aligned}
& \Expect_k \left[
  \bar{\cA}_k \bar{\cT}_k \nabla f(X_k)^{\top}(\bar{D}_k - D_k) + \bar{\cA}_k \norm{\nabla c(X_k)^{\top} (D_k - \bar{D}_k)}_1 \inrevise{+ 2 \bar\cA_k \norm{c(X_k) - \bar{C}_k}_1}
 \right] \\
\leq \ & \Expect_k \left[
  \tauzero \bar{\cA}_k \abs{\nabla f(X_k)^{\top}(\bar{D}_k - D_k)} + \bar{\cA}_k \sqrt{m} \norm{\nabla c(X_k)^{\top} (D_k - \bar{D}_k)} \inrevise{+ 2 \bar\cA_k \sqrt{m} \sqrt{\norm{c(X_k) - \bar{C}_k}^2}}
 \right] \\
\leq \ &\left(\tauzero \kappa_{\nabla f} + \sqrt{m} \norm{\nabla c(X_k)} \right) \Expect_k \left[ \bar{\cA}^{\text{max}}_k \norm{\bar{D}_k - D_k}\right] \inrevise{+ 2 \sqrt{m} \left( \cAmax + \theta \right) \beta_k \sqrt{\Expect_k \left[ \norm{c(X_k) - \bar{C}_k}^2 \right]}}  \\
\leq \ & \left( \cAmax + \theta \right) \beta_k \inrevise{ \left[ \left(\tauzero \kappa_{\nabla f} + \sqrt{m} \kappa_{\nabla c} \right) \left( \frac{\sqrt{\rho^g_k + \rho^c_k}}{q_{\text{min}}} + \omega_1 \sqrt{\rho^j_k} \right) + 2 \sqrt{m \rho^c_k}\right] } \\
\leq \ &  \inrevise{\beta_k} \rc{cnst:lem.expect-diff-phi-k-1} \left( \sqrt{\rho^g_k} + \sqrt{\rho^c_k} + \sqrt{\rho^j_k}\right),
\end{aligned}
\end{equation}
where the first inequality exploits the relation between \(\norm{\cdot}_1\) and \(\norm{\cdot}_2\), the second inequality employs the step size selection policy, the Cauchy-Schwarz inequality, \inrevise{Jensen's inequality,} \inrevise{\eqref{eq:step-size-bds},} Lemma~\ref{lem.stepsize_interval}, and Assumption~\ref{ass.prob}, the third inequality takes Assumption~\ref{ass.prob}, similar logic as Lemma~\ref{lem.solution_bias} and the definition of $\cAmax$ (see~\eqref{eq:step-size-bds}), and the last inequality follows from the selection of $\rc{cnst:lem.expect-diff-phi-k-1}$ and \(\sqrt{a + b} \leq \sqrt{a} + \sqrt{b}\) for all \(\{a, b\} \subset \reals_{\geq 0}\).

To analyze the remaining two terms in~\eqref{eq.key_relation_1}, we first consider the case where \(\hat{\cT}_k = \bar{\cT}_k\). Then, we will derive a bound for the other case \((\hat{\cT}_k < \bar{\cT}_k)\) and combine the two cases together.

When \(\hat{\cT}_k = \bar{\cT}_k\), by Lemmas~\ref{lem.model_reduction_suff_large} and~\ref{lem.model_reduction_suff_large_true}, it holds that
\begin{align}
& \Expect_k \left[
-\bar\cA_k\Delta l(X_k,\bar\cT_k,\nabla f(X_k),c(X_k),\nabla c(X_k)^T,D_k) +
(1-\eta)\bar\cA_k\beta_k\Delta l(X_k,\bar\cT_k,\bar{G}_k,\bar{C}_k,\bar{J}_k,\bar{D}_k)
\bdsemicolon \bar{\cT}_k = \hat{\cT}_k
 \right] \nonumber{} \\
& \leq \Expect_k \left[
-\bar\cA^{\text{min}}_k\Delta l(X_k,\hat\cT_k,\nabla f(X_k),c(X_k),\nabla c(X_k)^T,D_k) +
(1-\eta)\bar\cA^{\text{max}}_k\beta_k\Delta l(X_k,\hat\cT_k,\bar{G}_k,\bar{C}_k,\bar{J}_k,\bar{D}_k)
\bdsemicolon \bar{\cT}_k = \hat{\cT}_k
 \right] \nonumber{} \\
& = \Expect_k \left[
\left((1 - \eta) \bar{\cA}^{\text{max}}_k \beta_k -\bar\cA^{\text{min}}_k\right) \Delta l(X_k,\hat\cT_k,\nabla f(X_k),c(X_k),\nabla c(X_k)^T,D_k)
 \right. \nonumber{} \\
& \hspace{3em} + \left. (1 - \eta) \bar{\cA}^{\text{max}}_k \beta_k \left(
  \Delta l(X_k, \hat{\cT}_k, \bar{G}_k, \bar{C}_k, \bar{J}_k, \bar{D}_k)
- \Delta l(X_k, \hat{\cT}_k, \nabla f(X_k), c(X_k), \nabla c(X_k)^{\top}, D_k)
 \right)
\bdsemicolon \bar{\cT}_k = \hat{\cT}_k
\right] \nonumber{}
\\
& \leq
\Expect_k \left[
\left((1 - \eta) \left( \cAmax + \theta \right) \beta^2_k - \cAmin \beta_k \right) \Delta l(X_k,\hat\cT_k,\nabla f(X_k),c(X_k),\nabla c(X_k)^T,D_k) \right. \label{eq:expect-diff-phi-k-equal} \\
& \hspace{3em} \left.
+ (1 - \eta) \bar{\cA}^{\text{max}}_k \beta_k  \left(
  \Delta l(X_k, \hat{\cT}_k, \bar{G}_k, \bar{C}_k, \bar{J}_k, \bar{D}_k)
- \Delta l(X_k, \hat{\cT}_k, \nabla f(X_k), c(X_k), \nabla c(X_k)^{\top}, D_k)
 \right)
\bdsemicolon \bar{\cT}_k = \hat{\cT}_k
 \right] , \nonumber{}
\end{align}
where the first inequality is because \(\bar{\cT}_k = \hat{\cT}_k \leq \cT^{\trial}_k\) and the last inequality is because \([\bar{\cA}_k^{\text{min}},  \bar{\cA}_k^{\text{max}}] \subseteq [\cAmin \beta_k, (\cAmax + \theta) \beta_k] \) for all \(k\) (see~\eqref{eq:step-size-bds} and Lemma~\ref{lem.stepsize_interval}).

When \(\hat{\cT}_k < \bar{\cT}_k\), we use Lemmas~\ref{lem.l-hat-diff} and~\ref{lem.tau-hat-bar} to bound the error incurred by large \(\bar{\cT}_k\), i.e.,
\begin{align}
& \Expect_k \left[
- \bar{\cA}_k \Delta l(X_k, \bar{\cT}_k, \nabla f(X_k), c(X_k), \nabla c(X_k)^T, D_k)
+ (1 - \eta) \bar{\cA}_k \beta_k \Delta l(X_k, \bar{\cT}_k, \bar{G}_k, \bar{C}_k, \bar{J}_k, \bar{D}_k)
\bdsemicolon \bar{\cT}_k > \hat{\cT}_k
\right] \nonumber{}\\
& = \Expect_k \left[
- \bar{\cA}_k \left( \Delta l(X_k, \hat{\cT}_k, \nabla f(X_k), c(X_k), \nabla c(X_k)^T, D_k) +
 \Delta l(X_k, \bar{\cT}_k, \nabla f(X_k), c(X_k), \nabla c(X_k)^T, D_k) \right. \right. \nonumber{}\\
& \quad \qquad \left. \left.
- \Delta l(X_k, \hat{\cT}_k, \nabla f(X_k), c(X_k), \nabla c(X_k)^T, D_k) \right)
+ (1 - \eta) \bar{\cA}_k \beta_k \left(
\Delta l(X_k, \hat{\cT}_k, \bar{G}_k, \bar{C}_k, \bar{J}_k, \bar{D}_k) \right. \right. \nonumber{}\\
& \qquad \quad \left. \left.
+ \Delta l(X_k, \bar{\cT}_k, \bar{G}_k, \bar{C}_k, \bar{J}_k, \bar{D}_k)
- \Delta l(X_k, \hat{\cT}_k, \bar{G}_k, \bar{C}_k, \bar{J}_k, \bar{D}_k)
\right)
\bdsemicolon \bar{\cT}_k > \hat{\cT}_k
\right] \nonumber{} \\
& \leq \Expect_k \left[
- \bar{\cA}^{\text{min}}_k \Delta l(X_k, \hat{\cT}_k, \nabla f(X_k), c(X_k), \nabla c(X_k)^T, D_k)
+ (1 - \eta) \bar{\cA}^{\text{max}}_k \beta_k \Delta l(X_k, \hat{\cT}_k, \bar{G}_k, \bar{C}_k, \bar{J}_k, \bar{D}_k)
\bdsemicolon \bar{\cT}_k > \hat{\cT}_k
\right] \nonumber{}\\
& \quad \qquad + \Expect_k \left[
  \bar{\cA}_k \left(-\Delta l(X_k, \bar{\cT}_k, \nabla f(X_k), c(X_k), \nabla c(X_k)^T, D_k)
+ \Delta l(X_k, \hat{\cT}_k, \nabla f(X_k), c(X_k), \nabla c(X_k)^T, D_k)
  \right)
\bdsemicolon \bar{\cT}_k > \hat{\cT}_k
\right] \nonumber{}\\
& \quad \qquad + \Expect_k \left[
(1 - \eta) \bar{\cA}_k \beta_k \left(
\Delta l(X_k, \bar{\cT}_k, \bar{G}_k, \bar{C}_k, \bar{J}_k, \bar{D}_k)
- \Delta l(X_k, \hat{\cT}_k, \bar{G}_k, \bar{C}_k, \bar{J}_k, \bar{D}_k)
 \right)
\bdsemicolon \bar{\cT}_k > \hat{\cT}_k
 \right] \nonumber{} \\
& = \Expect_k \left[
  \left((1 - \eta)\bar{\cA}^{\text{max}}_k \beta_k - \bar{\cA}^{\text{min}}_k\right) \Delta l(X_k, \hat{\cT}_k, \nabla f(X_k), c(X_k), \nabla c(X_k)^{\top}, D_k)
 \bdsemicolon \bar{\cT}_k > \hat{\cT}_k \right] \nonumber{}
\\
& \quad \qquad + \Expect_k \left[
  (1 - \eta) \bar{\cA}^{\text{max}}_k \beta_k \left(
  \Delta l(X_k, \hat{\cT}_k, \bar{G}_k, \bar{C}_k, \bar{J}_k, \bar{D}_k)
- \Delta l(X_k, \hat{\cT}_k, \nabla f(X_k), c(X_k), \nabla c(X_k)^{\top}, D_k)
 \right)
 \bdsemicolon  \bar{\cT}_k > \hat{\cT}_k\right] \nonumber{}
\\
& \quad \qquad + \Expect_k \left[
  \bar{\cA}_k \left(-\Delta l(X_k, \bar{\cT}_k, \nabla f(X_k), c(X_k), \nabla c(X_k)^T, D_k)
+ \Delta l(X_k, \hat{\cT}_k, \nabla f(X_k), c(X_k), \nabla c(X_k)^T, D_k)
  \right)
\bdsemicolon \bar{\cT}_k > \hat{\cT}_k
\right] \nonumber{} \\
& \quad \qquad + \Expect_k \left[
(1 - \eta) \bar{\cA}_k \beta_k \left(
\Delta l(X_k, \bar{\cT}_k, \bar{G}_k, \bar{C}_k, \bar{J}_k, \bar{D}_k)
- \Delta l(X_k, \hat{\cT}_k, \bar{G}_k, \bar{C}_k, \bar{J}_k, \bar{D}_k)
 \right)
\bdsemicolon \bar{\cT}_k > \hat{\cT}_k
 \right] \nonumber{} \\
& \leq \Expect_k
\left[
\left((1 - \eta) \left( \cAmax + \theta \right) \beta^2_k - \cAmin \beta_k \right) \Delta l(X_k,\hat\cT_k,\nabla f(X_k),c(X_k),\nabla c(X_k)^T,D_k)
\bdsemicolon \bar{\cT}_k > \hat{\cT}_k
\right] \nonumber{} \\
& \qquad \quad + \Expect_k \left[ (1 - \eta) \bar{\cA}^{\text{max}}_k \beta_k \left(
  \Delta l(X_k, \hat{\cT}_k, \bar{G}_k, \bar{C}_k, \bar{J}_k, \bar{D}_k)
- \Delta l(X_k, \hat{\cT}_k, \nabla f(X_k), c(X_k), \nabla c(X_k)^{\top}, D_k)
\right)
\bdsemicolon \bar{\cT}_k > \hat{\cT}_k
\right] \nonumber{}
 \\
& \qquad \quad +  \inrevise{\left( \cAmax + \theta \right) \left(\beta_k + \beta_k^2 \right)}\left( \tauzero - \taumin  \right)
  \left( \kappa_H \norm{D_k}^2 + \norm{c(X_k)}_1 \norm{Y_k}_{\infty} \right) \cdot
\Prob_k \left[ \bar{\cT}_k > \hat{\cT}_k
 \right] \nonumber{} \\
& \qquad \quad + (1 - \eta)\left(\cAmax + \theta\right) \beta^2_k \left( \tauzero - \taumin \right) \rc{cnst:lem.tau-hat-bar} \inrevise{\left( \sqrt{\rho^g_k} + \sqrt{\rho^c_k} + \sqrt{\rho^j_k} \right)} \label{eq:expect-diff-phi-k-large}
\end{align}
where the first inequality follows from Lemma~\ref{lem.model_reduction_suff_large}, Lemma~\ref{lem.model_reduction_suff_large_true}, and \(\hat{\cT}_k \leq \cT_k^{\trial}\), and the last inequality is due to \inrevise{\(\eta \in (0,1)\),} \eqref{eq:step-size-bds}, Lemma~\ref{lem.stepsize_interval} and Lemma~\ref{lem.tau-hat-bar}.

We combine~\eqref{eq:expect-diff-phi-k-equal} and~\eqref{eq:expect-diff-phi-k-large} then apply Lemma~\ref{lem.l-hat-diff} to obtain
\begin{equation}
\label{eq.key_relation_3}
\begin{aligned}
& \Expect_k \left[
-\bar\cA_k\Delta l(X_k,\bar\cT_k,\nabla f(X_k),c(X_k),\nabla c(X_k)^T,D_k) +
(1-\eta)\bar\cA_k\beta_k\Delta l(X_k,\bar\cT_k,\bar{G}_k,\bar{C}_k,\bar{J}_k,\bar{D}_k)
 \right] \\
& \leq \Expect_k
\left[
\left((1 - \eta) \left( \cAmax + \theta \right) \beta^2_k - \cAmin \beta_k \right) \Delta l(X_k,\hat\cT_k,\nabla f(X_k),c(X_k),\nabla c(X_k)^T,D_k) \right. \\
& \qquad \left.
+ (1 - \eta) \left( \cAmax + \theta \right) \beta^2_k \, \rc{cnst:lem.l-hat-diff} \left( \sqrt{\rho^g_k} + \sqrt{\rho^c_k} + \sqrt{\rho^j_k} \right)
\right]
+ (1 - \eta) \left( \cAmax + \theta \right) \beta^2_k
\left( \tauzero - \taumin \right) \rc{cnst:lem.tau-hat-bar} \inrevise{\left( \sqrt{\rho^g_k} + \sqrt{\rho^c_k} + \sqrt{\rho^j_k} \right)}
\\
& \qquad +
\inrevise{\left( \cAmax + \theta \right) \left(\beta_k + \beta_k^2 \right)} \left( \tauzero - \taumin \right)
  \left( \kappa_H \norm{D_k}^2 + \norm{c(X_k)}_1 \norm{Y_k}_{\infty} \right)
\cdot  \Prob_k \left[ \bar{\cT}_k > \hat{\cT}_k \right] \\
& \leq ((1 - \eta) (\cAmax + \theta) \beta^2_k - \cAmin \beta_k) \Expect_k \left[ \Delta l(X_k,\hat\cT_k,\nabla f(X_k),c(X_k),\nabla c(X_k)^T,D_k) \right] \\
& \qquad + \inrevise{\rc{cnst:lem.expect-diff-phi-k-2} \beta_k^2 \left( \sqrt{\rho^g_k} + \sqrt{\rho^c_k} + \sqrt{\rho^j_k} \right)} \\
& \hspace{2em} + \inrevise{2} (\cAmax + \theta) \beta_k (\tauzero - \taumin)
\left( \kappa_H \norm{D_k}^2 + \norm{c(X_k)}_1 \norm{Y_k}_{\infty} \right)
\cdot \Prob_k \left[ \bar{\cT}_k > \hat{\cT}_k \right],
\end{aligned}
\end{equation}
where the first inequality utilizes $\bar{\cA}_k^{\text{max}} \leq (\cAmax + \theta) \beta_k$ (by~\eqref{eq:step-size-bds} and Lemma~\ref{lem.stepsize_interval}), and \inrevise{the last inequality follows \(\beta_k \in (0,1]\) and $\rc{cnst:lem.expect-diff-phi-k-2}$ being} a sufficiently large constant satisfying
\inrevise{\(\rc{cnst:lem.expect-diff-phi-k-2}
\geq
(1 - \eta)(\cAmax + \theta)((\tauzero - \taumin)\rc{cnst:lem.tau-hat-bar} + \rc{cnst:lem.l-hat-diff})\)}.
Consequently, by~\eqref{eq.key_relation_1},~\eqref{eq.key_relation_2} and~\eqref{eq.key_relation_3}, we bound the expected difference in merit functions with \(\{(\beta_k, \rho^g_k, \rho^c_k, \rho^j_k)\}\) and obtain the desired result that for any iteration $k\in\mathbb{N}$,
\begin{align*}
& \Expect_k \left[ \phi (X_k + \bar{\cA}_k \bar{D}_k, \bar{\cT}_k) \right] - \Expect_k \left[ \phi(X_k, \bar{\cT}_k) \right] \\
& \leq
  ((1 - \eta) (\cAmax{} + \theta) \beta^2_k - \cAmin{} \beta_k) \Expect_k \left[ \Delta l(X_k,\hat\cT_k,\nabla f(X_k),c(X_k),\nabla c(X_k)^T,D_k) \right] \\
& \qquad + \inrevise{ \rc{cnst:lem.expect-diff-phi-k-2}  \beta_k^2 \left( \sqrt{\rho^g_k} + \sqrt{\rho^c_k} + \sqrt{\rho^j_k} \right) } \\
& \qquad + \inrevise{2}(\cAmax{} + \theta) \beta_k (\tauzero - \taumin)
\left( \kappa_H \norm{D_k}^2 + \norm{c(X_k)}_1 \norm{Y_k}_{\infty} \right)
\Prob_k \left[ \bar{\cT}_k > \hat{\cT}_k \right] \\
& \qquad + \inrevise{\beta_k} \rc{cnst:lem.expect-diff-phi-k-1} \left( \sqrt{\rho^g_k} + \sqrt{\rho^c_k} + \sqrt{\rho^j_k} \right), \end{align*}
\inrevise{concluding the statement.}
\eproof

In the next lemma, we leverage Lemmas~\ref{lem.prob-tau-good-k} and~\ref{lem.ck-bdd-away} to relate \(\Prob_k \left[ \bar{\cT}_k > \hat{\cT}_k \right]\) with \((\rho^g_k, \rho^c_k, \rho^j_k)\) when
\((X_k, Y_k)\)
is not \(\delta\)-stationary. Conversely, when $(X_k,Y_k)$ is \(\delta\)-stationary, we show in Lemma~\ref{lem.expect-diff-phi-k-stationary} that the expected change in the merit function is small.

\blemma\label{lem.expect-diff-phi-k-bad}
Suppose that Assumptions~\ref{ass.prob}--\ref{ass.estimate}, \ref{ass.event_nonasymptotic} and \ref{ass.estimate_nonasymptotic} hold.
Then, there exists a positive constant \(\nc\label{cnst:lem.expect-diff-phi-k-bad} > 0\) that only depends on \((\rc{cnst:lem.prob-tau-good-k}, \kappa_H, \kappa_y, \zeta, \epsilon_{\tau})\) in Assumption~\ref{ass.H}, Lemmas~\ref{lem.prob-tau-good-k} and~\ref{lem.ck-bdd-away}, and~\eqref{eq.merit_parameter_update}
such that
\begin{align*}
& \Expect_k \left[ \phi (X_k + \bar{\cA}_k \bar{D}_k, \bar{\cT}_k) \right] - \Expect_k \left[ \phi(X_k, \bar{\cT}_k) \right] \\
& \qquad \leq
  \inrevise{((1 - \eta) (\cAmax{} + \theta) \beta^2_k - \cAmin{} \beta_k) \Expect_k \left[ \Delta l(X_k,\hat\cT_k,\nabla f(X_k),c(X_k),\nabla c(X_k)^T,D_k) \right]}
\\
& \qquad\qquad + \inrevise{2} (\cAmax + \theta) \beta_k (\tauzero - \taumin) \frac{\rc{cnst:lem.expect-diff-phi-k-bad}}{\delta} \left(  \rho^g_k + \rho^c_k + \rho^j_k \right)
  \\
& \qquad \qquad + \inrevise{\left( \beta_k \rc{cnst:lem.expect-diff-phi-k-1} + \beta^2_k  \rc{cnst:lem.expect-diff-phi-k-2} \right)} \left( \sqrt{\rho^g_k} + \sqrt{\rho^c_k} + \sqrt{\rho^j_k} \right)
\end{align*}
holds for any iteration \(k \in \naturals\) where \(\mathcal{B}_{\delta,k}\) occurs (see~\eqref{eq.nonopt_event}).
\elemma

\bproof
Based on Lemma~\ref{lem.expect-diff-phi-k}, to conclude the statement, it is sufficient to show that
\begin{align*}
\left( \kappa_H \norm{D_k}^2 + \norm{c(X_k)}_1 \norm{Y_k}_{\infty} \right) \cdot
\Prob_k \left[ \bar{\cT}_k > \hat{\cT}_k \right] \leq \frac{\rc{cnst:lem.expect-diff-phi-k-bad}}{\delta} \left(  \rho^g_k + \rho^c_k + \rho^j_k \right)
\end{align*}
holds for any iteration \(k \in \naturals\) where \(\mathcal{B}_{\delta,k}\) occurs (see~\eqref{eq.nonopt_event}).
In view of Lemma~\ref{lem.expect-diff-phi-k}, we first bound the (conditional) probability of event \(\bar{\cT}_k > \hat{\cT}_k\) happening, i.e.,
\begin{align}
\Prob_k \left[ \bar{\cT}_k > \hat{\cT}_k \right]
& = \Prob_k \left[ \bar{\cT}_k > \min \{\bar{\cT}_k, \cT_k^{\trial}\} \right]
= \Prob_k \left[ \bar{\cT}_k > \cT_k^{\trial} \right]
= 1 - \Prob_k \left[ \bar{\cT}_k \leq \cT_k^{\trial} \right]. \label{eq:eiff-phi-k-bad-1}
\end{align}
When \(\cT_k^{\trial} \geq \tauzero\), we have \(\bar{\cT}_k \leq \cT^{\trial}_k\) by Lemma~\ref{lem.merit_parameter}, making \(\Prob_k \left[ \bar{\cT}_k > \hat{\cT}_k \right] = 0\). Therefore, it is sufficient to focus on the case when \(\cT_k^{\trial} < \tauzero\). By the occurance of \(\mathcal{B}_{\delta,k}\), Lemma~\ref{lem.ck-bdd-away} ensures that \(\norm{c(X_k)}_1 > \frac{\zeta \delta}{2 \kappa^2_H \kappa_y + \zeta} > 0\). Therefore, we have $\cT_k^{\trial} <\bar\tau_0$ and $\|c(X_k)\|_1 > 0$. Then, applying Lemma~\ref{lem.prob-tau-good-k} to~\eqref{eq:eiff-phi-k-bad-1} results in
\begin{align}\label{eq:key_phi_diff_subcase}
\eqref{eq:eiff-phi-k-bad-1}
& \leq 1 - \frac{
\Expect_k \left[ \left(\epsilon_k - \rc{cnst:lem.prob-tau-good-k} \norm{\bar{\Delta}_k}  \right) \cdot \ones \left( \norm{\bar{\Delta}_k} < \frac{\epsilon_k}{2 \rc{cnst:lem.prob-tau-good-k}} \right)  \right]
}{
\epsilon_k - \rc{cnst:lem.prob-tau-good-k} \Expect_k \left[ \norm{\bar{\Delta}_k} \cdot \ones \left( \norm{\bar{\Delta}_k} < \frac{\epsilon_k}{2 \rc{cnst:lem.prob-tau-good-k}} \right) \right] + \left( \rc{cnst:lem.prob-tau-good-k}^2 / \epsilon_k \right) \Expect_k \left[ \norm{\bar{\Delta}_k}^2 \right]
} \nonumber{} \\
& = \frac{ \epsilon_k \cdot \Prob_k \left[ \norm{\bar{\Delta}_k} \geq \frac{\epsilon_k}{2 \rc{cnst:lem.prob-tau-good-k}} \right] +
\left( \rc{cnst:lem.prob-tau-good-k}^2 / \epsilon_k \right) \Expect_k \left[ \norm{\bar{\Delta}_k}^2 \right]
}{
\epsilon_k - \rc{cnst:lem.prob-tau-good-k} \Expect_k \left[ \norm{\bar{\Delta}_k} \cdot \ones \left( \norm{\bar{\Delta}_k} < \frac{\epsilon_k}{2 \rc{cnst:lem.prob-tau-good-k}} \right) \right] + \left( \rc{cnst:lem.prob-tau-good-k}^2 / \epsilon_k \right) \Expect_k \left[ \norm{\bar{\Delta}_k}^2 \right]
} \nonumber{} \\
& \leq \frac{ \epsilon_k \cdot \Prob_k \left[ \norm{\bar{\Delta}_k} \geq \frac{\epsilon_k}{2 \rc{cnst:lem.prob-tau-good-k}} \right] +
\left( \rc{cnst:lem.prob-tau-good-k}^2 / \epsilon_k \right) \Expect_k \left[ \norm{\bar{\Delta}_k}^2 \right]
}{
0.5 \epsilon_k + \left( \rc{cnst:lem.prob-tau-good-k}^2 / \epsilon_k \right) \Expect_k \left[ \norm{\bar{\Delta}_k}^2 \right]
} \nonumber{} \\
& \leq \frac{ \left( 4 \rc{cnst:lem.prob-tau-good-k}^2  / \epsilon_k +
 \rc{cnst:lem.prob-tau-good-k}^2 / \epsilon_k \right) \Expect_k \left[ \norm{\bar{\Delta}_k}^2 \right]
}{
0.5 \epsilon_k
} = \frac{10 \rc{cnst:lem.prob-tau-good-k}^{2} }{ \epsilon^2_k} \Expect_k \left[ \norm{\bar{\Delta}_k}^2 \right] \nonumber{} \\
& \leq \frac{10 \rc{cnst:lem.prob-tau-good-k}^2 }{ \epsilon^2_k} \left( \rho^g_k + \rho^c_k + \rho^j_k \right) = \frac{10 \rc{cnst:lem.prob-tau-good-k}^2}{\epsilon_{\tau}^2 \norm{c(X_k)}_1^2}
\left( \rho^g_k + \rho^c_k + \rho^j_k \right),
\end{align}
where \(\epsilon_k := \epsilon_{\tau} \norm{c(X_k)}_1\), the second inequality is because \(\rc{cnst:lem.prob-tau-good-k} \Expect_k \left[ \norm{\bar{\Delta}_k} \cdot \ones \left( \norm{\bar{\Delta}_k} < \epsilon_k / (2 \rc{cnst:lem.prob-tau-good-k}) \right) \right] \leq 0.5 \epsilon_k\), the third inequality is by conditional Markov's inequality, and the last inequality is from Assumption~\ref{ass.estimate_nonasymptotic}.

Furthermore, Lemma~\ref{lem.ck-bdd-away} implies that when $\cT_k^{\trial} < +\infty$ and \(\mathcal{B}_{\delta,k}\) occurs (see~\eqref{eq.nonopt_event}),
\begin{align*}
&\kappa_H \norm{D_k}^2 + \norm{c(X_k)}_1 \norm{Y_k}_{\infty} < \frac{2 \kappa_H \kappa_y}{\zeta} \norm{c(X_k)}_1 + \kappa_y \norm{c(X_k)}_1 = \frac{ 2 \kappa_H \kappa_y + \kappa_y \zeta}{\zeta} \norm{c(X_k)}_1 \\
\text{and}\quad &\frac{\zeta\delta}{2 \kappa^2_H \kappa_y + \zeta} < \norm{c(X_k)}_1,
\end{align*}
from which, we further have
\begin{align*}
\left(\kappa_H \norm{D_k}^2 + \norm{c(X_k)}_1 \norm{Y_k}_{\infty}\right) \Prob_k \left[ \bar{\cT}_k > \hat{\cT}_k \right]
& \leq
\frac{ 2 \kappa_H \kappa_y + \kappa_y \zeta}{\zeta} \norm{c(X_k)}_1 \cdot
\frac{10 \rc{cnst:lem.prob-tau-good-k}^2}{\epsilon_{\tau}^2 \norm{c(X_k)}_1^2}
\left( \rho^g_k + \rho^c_k + \rho^j_k \right) \\
& = \frac{10 \rc{cnst:lem.prob-tau-good-k}^2 (2 \kappa_H \kappa_y + \kappa_y \zeta)}{\zeta \epsilon^2_{\tau} \norm{c(X_k)}_1}
\left( \rho^g_k + \rho^c_k + \rho^j_k \right)
\leq \frac{\rc{cnst:lem.expect-diff-phi-k-bad}}{\delta} \left( \rho^g_k + \rho^c_k + \rho^j_k \right),
\end{align*}
where the first inequality follows from~\eqref{eq:key_phi_diff_subcase} and the last inequality holds by choosing a large enough constant \(\rc{cnst:lem.expect-diff-phi-k-bad} \geq 10 \rc{cnst:lem.prob-tau-good-k}^2 (2 \kappa_H \kappa_y + \kappa_y \zeta) (2 \kappa^2_H \kappa_y + \zeta)/(\zeta^2\epsilon_{\tau}^2)\). The proof is completed.
\eproof

\blemma\label{lem.expect-diff-phi-k-stationary}
Suppose that Assumptions~\ref{ass.prob}--\ref{ass.estimate}, \ref{ass.event_nonasymptotic} and \ref{ass.estimate_nonasymptotic} hold.
Then, there exist positive constants \(\{ \nc\label{cnst:lem.expect-diff-phi-k-stationary-1}, \nc\label{cnst:lem.expect-diff-phi-k-stationary-2}\} \subset \reals_{>0}\) that only depend on \((L, \Gamma, \cAmax, \theta, \delta, \zeta, \rhomax, q_{\text{min}}, \omega_1, \omega_2, \omega_3, \kappa_{\nabla f}, \kappa_{\nabla c})\) in
Assumptions~\ref{ass.prob},~\ref{ass.H}, and~\ref{ass.estimate_nonasymptotic},~\eqref{eq:step-size-bds}, and Lemmas~\ref{lem.singular_values},~\ref{lem.solution_bias}, and~\ref{lem.solution_variance} such that the following holds for any iteration \(k \in \naturals\) where \(\mathcal{B}_{\delta,k}\) (see~\eqref{eq.nonopt_event}) does not occur:
\begin{align}
\Expect_k \left[ \phi (X_k + \bar{\cA}_k \bar{D}_k, \bar{\cT}_k) - \phi(X_k, \bar{\cT}_k) \right] \leq \inrevise{\rc{cnst:lem.expect-diff-phi-k-stationary-2} \beta_k \delta} +
\rc{cnst:lem.expect-diff-phi-k-stationary-1} \beta_k \left( \sqrt{\rho^g_k} + \sqrt{\rho^c_k} + \sqrt{\rho^j_k} \right).
\label{eq:expect-diff-phi-k-stationary}
\end{align}
\elemma

\bproof
At any \(\delta\)-stationary iterate \((X_k, Y_k)\), it holds that
\(\norm{\nabla f(X_k) + \nabla c(X_k) Y_k}^2 + \norm{c(X_k)}_1 \leq \delta\) and by~\eqref{eq.linear_system_true} one has
\begin{align}
&-\nabla f(X_k)^{\top} D_k
= D^{\top}_k H_k D_k + D^{\top}_k \nabla c(X_k) Y_k
= D^{\top}_k H_k D_k - c(X_k)^{\top} Y_k \label{eq:kkt-grad-decrease}
\\
\text{and} \quad &\zeta^2 \norm{D_k}^2
\leq \norm{H_k D_k}^2
= \norm{-\nabla f(X_k) - \nabla c(X_k) Y_k}^2 \leq \delta. \label{eq:dk-bdd-by-delta}
\end{align}
Expanding the difference between merit functions in the left-hand-side of~\eqref{eq:expect-diff-phi-k-stationary} yields
\begin{align}
& \phi (X_k + \bar{\cA}_k \bar{D}_k, \bar{\cT}_k)  -  \phi(X_k, \bar{\cT}_k) \nonumber{}\\
= \ &\bar{\cT}_k \left( f(X_k + \bar{\cA}_k \bar{D}_k) - f(X_k) \right)
+ \norm{c(X_k + \bar{\cA}_k \bar{D}_k)}_1 - \norm{c(X_k)}_1 \nonumber{} \\
\leq \ &\bar{\cT}_k \left(\bar{\cA}_k \nabla f(X_k)^{\top} \bar{D}_k + \frac{L}{2} \bar{\cA}_k^2 \norm{\bar{D}_k}^2\right) + \left(\norm{\bar{\cA}_k \nabla c(X_k)^{\top} \bar{D}_k}_1 + \frac{\Gamma}{2} \bar{\cA}_k^2 \norm{\bar{D}_k}^2  \right), \label{eq:diff-phi-decompo}
\end{align}
where the inequality is due to~\eqref{eq.Lipschitz_continuity} and triangle inequality. Next, we bound the first term in the right-hand-side of~\eqref{eq:diff-phi-decompo} by
\begin{equation}\label{eq.merit_func_decrease_key2}
\begin{aligned}
& \bar{\cT}_k \left(
  \bar{\cA}_k \nabla f(X_k)^{\top} \bar{D}_k + \frac{L}{2} \bar{\cA}_k^2 \norm{\bar{D}_k}^2
 \right) \\
\leq \ &\bar{\cT}_k \bar{\cA}_k \left(
  \nabla f(X_k)^{\top} D_k + \nabla f(X_k)^{\top} \left(\bar{D}_k - D_k\right)
+ \frac{L}{2} \bar{\cA}_k \left( \norm{D_k} + \norm{\bar{D}_k - D_k} \right)^2
 \right) \\
= \ &\bar{\cT}_k \bar{\cA}_k \left(
  c(X_k)^{\top} Y_k - D_k^{\top} H_k D_k + \nabla f(X_k)^{\top} \left(\bar{D}_k - D_k\right)
+ \frac{L}{2} \bar{\cA}_k \left( \norm{D_k} + \norm{\bar{D}_k - D_k} \right)^2
 \right) \\
\leq \ &\tauzero \bar{\cA}^{\text{max}}_k \left(
  \norm{c(X_k)}_1 \norm{Y_k}_{\infty} + \kappa_{\nabla f} \norm{\bar{D}_k - D_k} + L \bar{\cA}_k \left( \norm{D_k}^2 + \norm{\bar{D}_k - D_k}^2 \right)
 \right),
\end{aligned}
\end{equation}
where the first equality is due to~\eqref{eq:kkt-grad-decrease}, and the second inequality follows from the Cauchy-Schwarz inequality, Assumption~\ref{ass.H}, and the monotonicity of $\{\bar\cT_k\}$ (see Lemma~\ref{lem.merit_parameter}). By taking conditional expectation on both sides of~\eqref{eq.merit_func_decrease_key2} and choosing large enough positive constants \(\nctwo\label{cnst:aux-1}\) and \(\nctwo\label{cnst:aux-2}\) such that
\begin{gather*}
\rctwo{cnst:aux-1} \geq
L \left( \cAmax + \theta \right)
  \inrevise{\left(\frac{1}{q^2_{\text{min}}} + \omega_2 + \omega_3 \right) \sqrt{\rhomax}}
\quad \text{and} \quad
\rctwo{cnst:aux-2} \geq
\kappa_{\nabla f} \cdot \max \left\{ \frac{1}{q_{\text{min}}}, \omega_1\right\},
\end{gather*}
one has
\begin{align}
& \Expect_k \left[
\bar{\cT}_k \left(
  \bar{\cA}_k \nabla f(X_k)^{\top} \bar{D}_k + \frac{L}{2} \bar{\cA}_k^2 \norm{\bar{D}_k}^2
 \right)
\right] \nonumber{} \\
\leq \ &\tauzero \inrevise{(\cAmax + \theta) \beta_k} \left(
  \kappa_y \delta + \kappa_{\nabla f} \Expect_k \left[ \norm{\bar{D}_k - D_k} \right] + \inrevise{L (\cAmax + \theta) \beta_k \left( \delta / \zeta^2 + \frac{\rho_k^g + \rho_k^c}{q_{\min}^2} + \omega_2\cdot\rho_k^j + \omega_3\cdot \sqrt{\rho_k^j(\rho_k^g+\rho_k^c)}  \right)}
 \right) \nonumber{} \\
\leq \ &\tauzero \inrevise{(\cAmax + \theta) \beta_k} \left(
  \inrevise{\left(\kappa_y + L \left( \cAmax + \theta \right) \beta_k / \zeta^2 \right)} \delta + \kappa_{\nabla f} \left( \frac{\sqrt{\rho_k^g + \rho_k^c}}{q_{\min}} + \omega_1\cdot \sqrt{\rho_k^j} \right)
+ \beta_k \rctwo{cnst:aux-1} \inrevise{\left( \sqrt{\rho^g_k} + \sqrt{\rho^c_k} + \sqrt{\rho^j_k} \right)}
 \right) \nonumber{} \\
\leq \ & \tauzero (\cAmax + \theta) \beta_k \left(
      \inrevise{\left(\kappa_y + L \left( \cAmax + \theta \right)\beta_k / \zeta^2 \right)} \delta
      + \inrevise{\left(\rctwo{cnst:aux-2} + \beta_k \rctwo{cnst:aux-1} \right) \left( \sqrt{\rho^g_k} + \sqrt{\rho^c_k} + \sqrt{\rho^j_k} \right)} \right) \label{eq:diff-phi-decompo-1}
\end{align}
where the first inequality is due to~\eqref{eq.basic_condition},~\inrevise{\eqref{eq:step-size-bds}},~\eqref{eq:dk-bdd-by-delta}, \eqref{eq.merit_func_decrease_key2} and Lemmas~\ref{lem.solution_variance} and~\ref{lem.ck-bdd-away}, \inrevise{and the second inequality follows Assumption~\ref{ass.estimate_nonasymptotic}, Lemma~\ref{lem.solution_bias} and \(\sqrt{a + b} \leq \sqrt{a} + \sqrt{b}\) for \(\set{a, b} \subset \reals_{\geq 0}\)}.
Following a similar approach, for the second term in the right-hand-side of~\eqref{eq:diff-phi-decompo}, we have that
\begin{equation}\label{eq.merit_func_decrease_key3}
\begin{aligned}
& \bar{\cA}_k \left( \norm{\nabla c(X_k)^{\top} \bar{D}_k}_1 + \frac{\Gamma}{2} \bar{\cA}_k \norm{\bar{D}_k}^2 \right) \\
\leq \ &\bar{\cA}_k \left(
\norm{ \nabla c(X_k)^{\top} D_k }_1 + \norm{\nabla c(X_k)^{\top} \left( \bar{D}_k - D_k \right)}_1
+ \frac{\Gamma}{2} \bar{\cA}_k \left( \norm{D_k} + \norm{\bar{D}_k - D_k} \right)^2
 \right) \\
\leq \ &\bar{\cA}_k^{\text{max}} \left(
  \norm{-c(X_k)}_1
  + \sqrt{m} \norm{\nabla c(X_k)^{\top}} \norm{ \bar{D}_k - D_k}
  + \Gamma \bar{\cA}_k \left( \norm{D_k}^2 + \norm{\bar{D}_k - D_k}^2 \right)
 \right),
\end{aligned}
\end{equation}
where the both of the inequalities apply triangle inequality, and the second inequality also leverages \eqref{eq.linear_system_true} and the Cauchy-Schwarz inequality.
Consequently, by choosing large enough positive constants \(\nctwo\label{cnst:aux-3}\) and \(\nctwo\label{cnst:aux-4}\) such that
\begin{gather*}
\rctwo{cnst:aux-3} \geq
\Gamma \left( \cAmax + \theta \right) \inrevise{\left(
  \frac{1}{q^2_{\text{min}}} + \omega_2 + \omega_3 \right) \sqrt{\rhomax}}
\quad \text{and} \quad
\rctwo{cnst:aux-4} \geq \sqrt{m}
\kappa_{\nabla c} \cdot \max \left\{ \frac{1}{q_{\text{min}}}, \omega_1\right\},
\end{gather*}
one has
\begin{align}
& \Expect_k \left[
\bar{\cA}_k \left( \norm{\nabla c(X_k)^{\top} \bar{D}_k}_1 + \frac{\Gamma}{2} \bar{\cA}_k \norm{\bar{D}_k}^2 \right)
\right] \nonumber{} \\
\leq \ &\inrevise{(\cAmax + \theta) \beta_k} \left(
\delta + \sqrt{m}\kappa_{\nabla c} \Expect_k \left[ \norm{\bar{D}_k - D_k} \right]
  + \Gamma \inrevise{(\cAmax + \theta) \beta_k} \left( \delta / \zeta^2 + \frac{\rho_k^g + \rho_k^c}{q_{\min}^2} + \omega_2\cdot\rho_k^j + \omega_3\cdot \sqrt{\rho_k^j(\rho_k^g+\rho_k^c)}  \right)
\right) \nonumber{} \\
\leq \ &(\cAmax + \theta) \beta_k  \left(
      \inrevise{\left( 1 + \Gamma (\cAmax + \theta) \beta_k / \zeta^2 \right)} \delta + \sqrt{m}\kappa_{\nabla c} \left( \frac{\sqrt{\rho_k^g + \rho_k^c}}{q_{\min}} + \omega_1\cdot \sqrt{\rho_k^j} \right)
      + \beta_k \rctwo{cnst:aux-3} \inrevise{\left( \sqrt{\rho^g_k} + \sqrt{\rho^c_k} + \sqrt{\rho^j_k} \right)}
 \right) \nonumber{} \\
\leq \ &(\cAmax + \theta) \beta_k  \left( \inrevise{\left( 1 + \Gamma \left( \cAmax + \theta \right) \beta_k / \zeta^2 \right)} \delta +  \inrevise{\left(\rctwo{cnst:aux-4} + \beta_k \rctwo{cnst:aux-3}  \right) \left( \sqrt{\rho^g_k} + \sqrt{\rho^c_k} + \sqrt{\rho^j_k} \right)} \right),
\label{eq:diff-phi-decompo-2}
\end{align}
where the first inequality is due to~\eqref{eq.basic_condition},\inrevise{\eqref{eq:step-size-bds},}~\eqref{eq:dk-bdd-by-delta}, \eqref{eq.merit_func_decrease_key3} and Lemma~\ref{lem.solution_variance}, \inrevise{and the second inequality follows Assumption~\ref{ass.estimate_nonasymptotic}, Lemma~\ref{lem.solution_bias} and \(\sqrt{a + b} \leq \sqrt{a} + \sqrt{b}\) for \(\set{a, b} \subset \reals_{\geq 0}\)}.
Combining~\eqref{eq:diff-phi-decompo},~\eqref{eq:diff-phi-decompo-1} and~\eqref{eq:diff-phi-decompo-2} yields
\begin{align*}
&\Expect_k \left[ \phi (X_k + \bar{\cA}_k \bar{D}_k, \bar{\cT}_k)  -  \phi(X_k, \bar{\cT}_k)  \right] \\
\leq \ &\tauzero (\cAmax + \theta) \beta_k \left(
  \inrevise{\left( \kappa_y + L \left( \cAmax + \theta \right) \beta_k / \zeta^2  \right)} \delta + \inrevise{\left(  \rctwo{cnst:aux-2} + \beta_k \rctwo{cnst:aux-1} \right) \left( \sqrt{\rho^g_k} + \sqrt{\rho^c_k} + \sqrt{\rho^j_k} \right)}
\right) \\
& \qquad + (\cAmax + \theta) \beta_k  \left( \inrevise{\left( 1 + \Gamma \left( \cAmax + \theta \right) \beta_k / \zeta^2 \right)} \delta + \inrevise{\left(\rctwo{cnst:aux-4} + \beta_k \rctwo{cnst:aux-3}  \right)
\left( \sqrt{\rho^g_k} + \sqrt{\rho^c_k} + \sqrt{\rho^j_k} \right)} \right)
 \\
  = \ & \beta_k \delta \inrevise{(\cAmax + \theta) \left( \tauzero  \left(\kappa_y + L(\cAmax + \theta) \beta_k / \zeta^2 \right) + \left( 1 + \Gamma (\cAmax + \theta) \beta_k / \zeta^2 \right) \right)} \\
& \qquad
  + \inrevise{\left( \cAmax + \theta \right) \beta_k \left( \tauzero \left( \rctwo{cnst:aux-2} + \beta_k \rctwo{cnst:aux-1} \right) + \rctwo{cnst:aux-4} + \beta_k \rctwo{cnst:aux-3} \right) \left( \sqrt{\rho^g_k} + \sqrt{\rho^c_k} + \sqrt{\rho^j_k} \right).} 
\end{align*}
\inrevise{Then, because $\beta_k\in(0,1]$,}~\eqref{eq:expect-diff-phi-k-stationary} follows by choosing large enough constants \(\{\rc{cnst:lem.expect-diff-phi-k-stationary-1}, \rc{cnst:lem.expect-diff-phi-k-stationary-2}\} \subset \RR_{>0}\) such that
\begin{gather*}
\rc{cnst:lem.expect-diff-phi-k-stationary-1}
\geq  \left( \cAmax + \theta \right) \inrevise{\left( \tauzero \left( \rctwo{cnst:aux-2} + \rctwo{cnst:aux-1} \right) + \rctwo{cnst:aux-4} + \rctwo{cnst:aux-3}\right)} \\
\text{and}\quad \rc{cnst:lem.expect-diff-phi-k-stationary-2}
\geq
(\cAmax + \theta) \inrevise{\left( \tauzero  \left(\kappa_y + L(\cAmax + \theta) / \zeta^2 \right) + \left( 1 + \Gamma (\cAmax + \theta) / \zeta^2 \right) \right),}
\end{gather*}
\inrevise{concluding the proof}.
\eproof

\subsubsection{Complexity of Algorithm~\ref{alg.main}}\label{sec:complex-result}

Combining theoretical results in Sections~\ref{sec:complex-prob-bound} and~\ref{sec:complex-noise} yields our final non-asymptotic convergence result.

\begin{theorem}\label{thm.complex}
Suppose that Assumptions~\ref{ass.prob}--\ref{ass.estimate}, \ref{ass.event_nonasymptotic} and \ref{ass.estimate_nonasymptotic} hold.
Let
\(\omega_{\delta} \geq \left(1 + \frac{\inrevise{4\kappa_H \rc{cnst:lem.expect-diff-phi-k-stationary-2}}}{\cAmin \min\{\taumin,1\} \zeta \sigma}\right)\) and \(\{\omega_{\beta}, \omega_{\rho}\} \subset \RR_{>0}\) be some positive constants, and
\(K\) be a uniform random variable with probability mass function \(p_K (k) := 1 / \kmax\) for all \(k \in [\kmax]\).
If \((\beta_k,\rho^{g}_k,\rho^c_k,\rho^j_k) = (\beta,\rho^g,\rho^c,\rho^j)\) for all \(k \in [k_{\text{max}}]\) and for any $\varepsilon > 0$ we set
\begin{equation}\label{eq:complex-cond}
\begin{aligned}
\inrevise{\beta \in \left(0, \min \left\{\frac{\cAmin}{2(1 - \eta)(\cAmax + \theta)}, 1  \right\}\right]}, \quad
  \left(\sqrt{\rho^g} + \sqrt{\rho^c} + \sqrt{\rho^j}  \right)
  \leq  \frac{\omega_{\rho}}{\inrevise{k_{\text{max}}}}, \quad\text{and}\quad
\delta = \frac{\inrevise{\varepsilon^2}}{10 \; \omega_{\delta}},
\end{aligned}
\end{equation}
then it follows with \(K\) having an independent discrete uniform distribution over \([k_{\text{max}}]\) that
\begin{align}
\Expect \left[ \,
\norm{\nabla f(X_K) + \nabla c(X_K) Y_K}^2 + \norm{c(X_K)}_1
\mid \nonE \,
\right]
  \leq
\inrevise{\frac{\rc{cnst:thm.complex-1} + \rc{cnst:thm.complex-4} \omega_{\rho}}{\beta \kmax}}
+ \frac{10 \omega_{\delta} \rc{cnst:thm.complex-3} \omega^2_{\rho}}{\inrevise{\varepsilon^2 \kmax^2}}
  + 0.1 \inrevise{\varepsilon^2},
  \label{eq:complex-main-rst}
\end{align}
where constants \inrevise{\(\{\nc\label{cnst:thm.complex-1}, \nc\label{cnst:thm.complex-3}, \nc\label{cnst:thm.complex-4}\} \subset\mathbb{R}_{>0}\)} are
\inrevise{{\small
\begin{equation}\label{eq:complex-param}
\begin{aligned}
\rc{cnst:thm.complex-1}
:= \frac{4\kappa_H \left( \tauzero \left( f(x_1) - f_{\inf} \right) + \norm{c(x_1)}_1 \right)}{\cAmin \min\{\taumin,1\} \zeta \sigma},~
\rc{cnst:thm.complex-3}
:= \frac{8 \kappa_H(\cAmax + \theta) (\tauzero - \taumin) \rc{cnst:lem.expect-diff-phi-k-bad}}{\cAmin \min\{\taumin,1\} \zeta \sigma},
\text{ and }
\rc{cnst:thm.complex-4}
:= \frac{4\kappa_H \left(\rc{cnst:lem.expect-diff-phi-k-stationary-1} + \rc{cnst:lem.expect-diff-phi-k-1} + \rc{cnst:lem.expect-diff-phi-k-2}\right)}{\cAmin \min\{\taumin,1\} \zeta \sigma}.
\end{aligned}
\end{equation}
}}
\end{theorem}
\bproof
\inrevise{Let $\IndSet{\mathcal{B}_{\delta,k}}$ and $\IndSet{\mathcal{B}^c_{\delta,k}}$ denote the indicator functions of the events that \(\mathcal{B}_{\delta,k}\) (see~\eqref{eq.nonopt_event}) occurs and does not occur, respectively.
We first relate the left-hand-side of~\eqref{eq:complex-main-rst} with the model reduction function by conditioning on whether event \(\mathcal{B}_{\delta,k}\) happens or not,} i.e.,
\begin{align}
& \Expect \left[ \,
\norm{\nabla f(X_K) + \nabla c(X_K) Y_K}^2 + \norm{c(X_K)}_1
\mid \nonE \,
\right] \nonumber{} \\
& = \sum_{k = 1}^{\kmax} p_K(k) \cdot \Expect \left[ \norm{\nabla f(X_k) + \nabla c(X_k) Y_k}^2 + \norm{c(X_k)}_1 \mid \nonE \right] \nonumber{} \\
& = \sum_{k = 1}^{\kmax} p_K(k) \cdot \left(
\Expect \left[ \left(\norm{\nabla f(X_k) + \nabla c(X_k) Y_k}^2 + \norm{c(X_k)}_1\right) \cdot \IndSet{\mathcal{B}_{\delta,k}}
\mid \nonE \right] \right. \nonumber{}
 \\
& \hspace{10em} \left.
  + \Expect \left[ \left(\norm{\nabla f(X_k) + \nabla c(X_k) Y_k}^2 + \norm{c(X_k)}_1 \right) \cdot \IndSet{\mathcal{B}^c_{\delta,k}}
\mid \nonE \right]
\right) \nonumber{} \\
& \leq \sum_{k = 1}^{\kmax} p_K(k) \cdot \left(
\Expect \left[ \left(\norm{\nabla f(X_k) + \nabla c(X_k) Y_k}^2 + \norm{c(X_k)}_1\right) \cdot \IndSet{\mathcal{B}_{\delta,k}}
\mid \nonE \right] + \delta \cdot \Prob \left[ \mathcal{B}^c_{\delta,k} \mid \nonE \right] \right) \nonumber{} \\
& \leq \sum_{k = 1}^{\kmax} p_K(k) \cdot \left( \frac{2 \kappa_H}{\min\{\taumin,1\} \zeta \sigma}
\Expect \left[ \Delta l(X_k, \hat{\cT}_k, \nabla f(X_k), c(X_k), \nabla c(X_k)^T, D_k) \cdot \IndSet{\mathcal{B}_{\delta,k}}
\mid \nonE\right] + \delta  \right), \label{eq:complex-main-1}
\end{align}
where the second equality follows from the property of expectation, the first inequality is due to the definition of event $\mathcal{B}_{\delta,k}$ (see~\eqref{eq.nonopt_event}), and the last inequality is by Lemma~\ref{lem:kkt-delta-l}. Next, we relate the model reduction function with the changes in merit functions, i.e.,
\begin{align}
& \sum_{k = 1}^{\kmax} \Expect \left[
\phi(X_k + \bar{\cA}_k \bar{D}_k, \bar{\cT}_k) - \phi(X_k, \bar{\cT}_k)
 \mid \nonE \right] \label{eq:sum-expect-phi-diff} \\
& = \sum_{k = 1}^{\kmax} \left(\Expect \left[
\left(\phi(X_k + \bar{\cA}_k \bar{D}_k, \bar{\cT}_k) - \phi(X_k, \bar{\cT}_k)\right) \cdot \IndSet{\mathcal{B}_{\delta,k}}
\mid \nonE \right] \right. \nonumber{} \\
& \left.\hspace{6em} + \Expect \left[
\left(\phi(X_k + \bar{\cA}_k \bar{D}_k, \bar{\cT}_k) - \phi(X_k, \bar{\cT}_k)  \right) \cdot \IndSet{\mathcal{B}^c_{\delta,\kappa}}
\mid \nonE \right] \right) \nonumber{} \\
& = \sum_{k = 1}^{\kmax} \left(\Expect \left[ \Expect_k \left[ \left(
\phi(X_k + \bar{\cA}_k \bar{D}_k, \bar{\cT}_k) - \phi(X_k, \bar{\cT}_k) \right)
 \cdot \IndSet{\mathcal{B}_{\delta,k}}\right]
\mid \nonE \right] \right. \nonumber{}\\
& \left.\hspace{6em} + \Expect \left[ \Expect_k \left[ \left(
\phi(X_k + \bar{\cA}_k \bar{D}_k, \bar{\cT}_k) - \phi(X_k, \bar{\cT}_k) \right)
  \cdot \IndSet{\mathcal{B}^c_{\delta,k}}\right]
\mid \nonE \right] \right), \nonumber{} \\
& = \sum_{k = 1}^{\kmax} \left(\Expect \left[ \Expect_k \left[
\phi(X_k + \bar{\cA}_k \bar{D}_k, \bar{\cT}_k) - \phi(X_k, \bar{\cT}_k)
\right] \cdot \IndSet{\mathcal{B}_{\delta,k}}
\mid \nonE \right] \right. \label{eq:cond-phi-diff-1}\\
& \left.\hspace{6em} + \Expect \left[ \Expect_k \left[
\phi(X_k + \bar{\cA}_k \bar{D}_k, \bar{\cT}_k) - \phi(X_k, \bar{\cT}_k)
  \right] \cdot \IndSet{\mathcal{B}^c_{\delta,k}}
\mid \nonE \right] \right), \label{eq:cond-phi-diff-2}
\end{align}
where the second equality follows from the tower property of conditional expectation, and the last equality is because \(\IndSet{\mathcal{B}_{\delta,k}}\) and \(\IndSet{\mathcal{B}^c_{\delta,k}}\) only depend on \((X_k,Y_k)\), making them \inrevise{\(\mathcal{F}'_k\)-measurable}.
By Lemma~\ref{lem.expect-diff-phi-k-bad}, the inner conditional expectation in~\eqref{eq:cond-phi-diff-1} admits the following upper bound
\begin{align*}
& \Expect_k \left[
\phi(X_k + \bar{\cA}_k \bar{D}_k, \bar{\cT}_k) - \phi(X_k, \bar{\cT}_k)
\right] \inrevise{\cdot \IndSet{\mathcal{B}_{\delta,k}}}  \\
&\quad \leq
  \inrevise{((1 - \eta) (\cAmax{} + \theta) \beta^2_k - \cAmin{} \beta_k) \Expect_k \left[ \Delta l(X_k,\hat\cT_k,\nabla f(X_k),c(X_k),\nabla c(X_k)^T,D_k) \right] \cdot \IndSet{\mathcal{B}_{\delta,k}}} \\
& \qquad + \inrevise{2} (\cAmax + \theta) \beta_k (\tauzero - \taumin) \frac{\rc{cnst:lem.expect-diff-phi-k-bad}}{\delta} \left(  \rho^g_k + \rho^c_k + \rho^j_k \right)
+ \inrevise{\left(\beta_k \rc{cnst:lem.expect-diff-phi-k-1}  + \beta^2_k \rc{cnst:lem.expect-diff-phi-k-2} \right)} \left( \sqrt{\rho^g_k} + \sqrt{\rho^c_k} + \sqrt{\rho^j_k} \right).
\end{align*}
Similarly, by Lemma~\ref{lem.expect-diff-phi-k-stationary}, one can obtain the following upper bound for the inner conditional expectation in~\eqref{eq:cond-phi-diff-2}
\begin{align*}
\Expect_k \left[ \phi (X_k + \bar{\cA}_k \bar{D}_k, \bar{\cT}_k) - \phi(X_k, \bar{\cT}_k) \right] \inrevise{\cdot \IndSet{\mathcal{B}^c_{\delta,k}}} 
\leq \inrevise{\rc{cnst:lem.expect-diff-phi-k-stationary-2} \beta_k \delta} + \rc{cnst:lem.expect-diff-phi-k-stationary-1}
 \beta_k \left( \sqrt{\rho^g_k} + \sqrt{\rho^c_k} + \sqrt{\rho^j_k} \right).
\end{align*}

Consequently,~\eqref{eq:sum-expect-phi-diff} has the following upper bound:
\begin{align}
& \sum_{k = 1}^{\kmax} \Expect \left[
\phi(X_k + \bar{\cA}_k \bar{D}_k, \bar{\cT}_k) - \phi(X_k, \bar{\cT}_k)
 \mid \nonE \right] \\
&\quad \leq \sum_{k=1}^{\kmax} \left( \Expect \left[
  \inrevise{((1 - \eta) (\cAmax{} + \theta) \beta^2_k - \cAmin{} \beta_k)} \Expect_k \left[
  \Delta l(X_k,\hat\cT_k,\nabla f(X_k),c(X_k),\nabla c(X_k)^T,D_k)
 \right] \cdot \IndSet{\mathcal{B}_{\delta,k}} \mid \nonE \right] \right. \nonumber{}\\
& \qquad \qquad + \inrevise{2} (\cAmax + \theta) \beta_k (\tauzero - \taumin) \frac{\rc{cnst:lem.expect-diff-phi-k-bad}}{\delta} \left(  \rho^g_k + \rho^c_k + \rho^j_k \right)
+ \inrevise{\left( \beta_k  \rc{cnst:lem.expect-diff-phi-k-1} + \beta_k^2 \rc{cnst:lem.expect-diff-phi-k-2} \right)} \left( \sqrt{\rho^g_k} + \sqrt{\rho^c_k} + \sqrt{\rho^j_k} \right) \nonumber{}
\\
&  \left.
 \qquad \qquad + \inrevise{\rc{cnst:lem.expect-diff-phi-k-stationary-2} \beta_k \delta} + \rc{cnst:lem.expect-diff-phi-k-stationary-1}
 \beta_k \left( \sqrt{\rho^g_k} + \sqrt{\rho^c_k} + \sqrt{\rho^j_k} \right)
 \right) \nonumber{} \\
&\quad \leq \sum_{k = 1}^{\kmax} \left(
  - \inrevise{\frac{1}{2}} \cAmin{} \beta_k \Expect \left[
  \Delta l (X_k, \hat{\cT}_k, \nabla f(X_k), c(X_k), \nabla c(X_k)^T, D_k) \cdot \IndSet{\mathcal{B}_{\delta,k}}
 \mid \nonE \right] \right.\nonumber{} \\
& \hspace{4em}
+ \inrevise{2} (\cAmax + \theta) \beta_k (\tauzero - \taumin) \frac{\rc{cnst:lem.expect-diff-phi-k-bad}}{\delta}
\left(  \rho^g_k + \rho^c_k + \rho^j_k \right)
+ \inrevise{\rc{cnst:lem.expect-diff-phi-k-stationary-2} \beta_k \delta} \nonumber{} \\
& \hspace{4em} \left.
+ \inrevise{\left( \rc{cnst:lem.expect-diff-phi-k-stationary-1} + \rc{cnst:lem.expect-diff-phi-k-1} \right)} \beta_k
  \left( \sqrt{\rho^g_k} + \sqrt{\rho^c_k} + \sqrt{\rho^j_k} \right)
+ \inrevise{\beta^2_k \rc{cnst:lem.expect-diff-phi-k-2}} \left( \sqrt{\rho^g_k} + \sqrt{\rho^c_k} + \sqrt{\rho^j_k} \right)
  \right),
\label{eq:sum-expect-phi-diff-ubd}
\end{align}
\inrevise{where the last inequality follows~\eqref{eq:complex-cond}.}
Next, we provide a lower bound of~\eqref{eq:sum-expect-phi-diff}. It follows Lemma~\ref{lem.merit_parameter} and~\eqref{eq.basic_condition} that \(\{\bar{\cT}_k\}_{k=1}^{\kmax}\) is monotonically decreasing and \(f(X_k) \geq f_{\inf}\) always holds, and we further have from~\eqref{eq.merit_func} that for any $k\in\{2,\ldots,k_{\text{max}}\}$
\begin{align*}
&\Expect \left[\phi(X_k,\bar{\cT}_{k-1}) - \bar{\cT}_{k-1}f_{\inf} \mid \nonE \right] = \Expect \left[
 \bar{\cT}_{k-1} \left( f(X_k) - f_{\inf} \right) + \norm{c(X_k)}_1 \mid \nonE
\right] \\
\geq \ &\Expect \left[
 \bar{\cT}_k \left( f(X_k) - f_{\inf} \right) + \norm{c(X_k)}_1 \mid  \nonE
\right]
= \Expect \left[\phi(X_k,\bar{\cT}_k) - \bar{\cT}_k f_{\inf} \mid \nonE \right],
\end{align*}
from which it follows that
\begin{align}
& \sum_{k = 1}^{k_{\text{max}}} \left(\Expect \left[
 \phi (X_k + \bar{\cA}_k \bar{D}_k, \bar{\cT}_k) \mid \nonE \right] - \Expect \left[ \phi(X_k, \bar{\cT}_k) \mid \nonE
\right] \right) \nonumber{} \\
= ~& \sum_{k = 1}^{k_{\text{max}}} \left(\Expect \left[
 \phi (X_{k+1}, \bar{\cT}_k) \mid \nonE \right] - \Expect \left[ \phi(X_k, \bar{\cT}_k) \mid \nonE
\right] \right) \nonumber{} \\
= ~& \Expect \left[ \phi(X_{k_{\text{max}}+1},\bar{\cT}_{k_{\text{max}}}) \mid \nonE \right] - \Expect \left[ \phi(x_1, \bar{\cT}_1) \mid \nonE
\right] + \sum_{k=2}^{k_{\text{max}}}\left( \Expect \left[ \phi(X_k, \bar{\cT}_{k-1}) - \phi(X_k, \bar{\cT}_k) \mid \nonE
\right] \right) \nonumber{} \\
\geq ~& \Expect \left[ \phi(X_{k_{\text{max}}+1},\bar{\cT}_{k_{\text{max}}}) \mid \nonE \right] - \Expect \left[ \phi(x_1, \bar{\cT}_1) \mid \nonE
\right] + \inrevise{\sum_{k=2}^{k_{\text{max}}} \Expect \left[ (\bar{\cT}_{k-1} - \bar{\cT}_k)f_{\inf} \mid \nonE
\right]} \nonumber{} \\
= ~ &
\Expect \left[ \bar{\cT}_{k_{\text{max}}} \left(
f(X_{k_{\text{max}} + 1}) - f_{\inf} \right) + \norm{c(X_{k_{\text{max}}+1})}_1 \mid \nonE \right]
- \inrevise{\Expect \left[ \bar{\cT}_1(f(x_1) - f_{\inf}) + \|c(x_1)\|_1 \mid \nonE \right]} \nonumber{} \\
\geq ~ &- \bar\tau_0 (f(x_1) - f_{\inf}) - \norm{c(x_1)}_1.
\label{eq:sum-expect-phi-diff-lbd}
\end{align}
Since \(\beta_k = \beta\) and \((\rho^g_k, \rho^c_k, \rho^j_k) = (\rho^g, \rho^c, \rho^j)\) for all $k\in[k_{\max}]$, combining~\eqref{eq:sum-expect-phi-diff-ubd} and~\eqref{eq:sum-expect-phi-diff-lbd} yields
\begin{equation}\label{eq.merit_func_decrease_key4}
\begin{aligned}
& \sum_{k=1}^{\kmax} \inrevise{\frac{1}{2}} \Expect \left[
  \Delta l (X_k, \hat{\cT}_k, \nabla f(X_k), c(X_k), \nabla c(X_k)^T, D_k) \cdot \IndSet{\mathcal{B}_{\delta,k}}
 \mid \nonE \right] \\
\leq \ & \frac{\bar\tau_0 (f(x_1) - f_{\inf}) + \norm{c(x_1)}_1}{\cAmin \beta}
      + \frac{\kmax}{\cAmin} \left[
\inrevise{2} (\cAmax + \theta) (\tauzero - \taumin) \frac{\rc{cnst:lem.expect-diff-phi-k-bad}}{\delta} \left( \rho^g + \rho^c + \rho^j \right) \right.  \\
& \qquad \left. + \inrevise{\rc{cnst:lem.expect-diff-phi-k-stationary-2}} \delta + \inrevise{\left( \rc{cnst:lem.expect-diff-phi-k-stationary-1} + \rc{cnst:lem.expect-diff-phi-k-1} + \beta \rc{cnst:lem.expect-diff-phi-k-2}\right)}
\left( \sqrt{\rho^g} + \sqrt{\rho^c} + \sqrt{\rho^j} \right) \right].
\end{aligned}
\end{equation}
Therefore, by the definition of function $p_K(\cdot)$, \eqref{eq:complex-cond}, \eqref{eq:complex-param}, \eqref{eq:complex-main-1}, \eqref{eq.merit_func_decrease_key4}, \inrevise{and Lemma~\ref{lem:kkt-delta-l}}, after rearrangement one has
\begin{align*}
&\Expect \left[ \,
\norm{\nabla f(X_K) + \nabla c(X_K) Y_K}^2 + \norm{c(X_K)}_1
\mid \nonE \,
\right] \\
\leq \ & \frac{\inrevise{4}\kappa_H}{\cAmin \min\{\taumin,1\} \zeta \sigma} \left[
\frac{\tauzero (f(x_1) - f_{\inf}) + \norm{c(x_1)}_1}{\beta \kmax}
+ \inrevise{2} (\cAmax + \theta) (\tauzero - \taumin) \frac{\rc{cnst:lem.expect-diff-phi-k-bad}}{\delta} \left(  \rho^g + \rho^c + \rho^j \right)
 \right. \\
&  \hspace{4em} \left. +
\inrevise{\left(
  \rc{cnst:lem.expect-diff-phi-k-stationary-1} + \rc{cnst:lem.expect-diff-phi-k-1}
+ \rc{cnst:lem.expect-diff-phi-k-2}
\right)}  \left( \sqrt{\rho^g} + \sqrt{\rho^c} + \sqrt{\rho^j} \right)
\right] + \left(\frac{\inrevise{4\kappa_H \rc{cnst:lem.expect-diff-phi-k-stationary-2}}}{\cAmin \min\{\taumin,1\} \zeta \sigma}  + 1\right) \delta \\
= \ & \inrevise{\frac{\rc{cnst:thm.complex-1}}{\beta \kmax}}
+ \frac{\rc{cnst:thm.complex-3}}{\delta} \left(  \rho^g + \rho^c + \rho^j  \right)
+ \rc{cnst:thm.complex-4} \left( \sqrt{\rho^g} + \sqrt{\rho^c} + \sqrt{\rho^j} \right)
+ \left(1 + \frac{\inrevise{4\kappa_H \rc{cnst:lem.expect-diff-phi-k-stationary-2}}}{\cAmin \min\{\taumin,1\} \zeta \sigma}\right) \delta \\
\leq \ &\inrevise{\frac{\rc{cnst:thm.complex-1} + \rc{cnst:thm.complex-4} \omega_{\rho}}{\beta \kmax}}
+ \frac{10 \omega_{\delta} \rc{cnst:thm.complex-3}}{\inrevise{\varepsilon^2}} \cdot \frac{\omega^2_{\rho}}{\inrevise{\kmax^2}}
+ 0.1 \inrevise{\varepsilon^2},
\end{align*}
which concludes the statement.
\eproof

\begin{cor}\label{cor.complex}
Suppose that one uses the empirical mean for estimating $(\nabla f(X_k), c(X_k), \nabla c(X_k)^T)$, and their population variances are all bounded by \(\rhomax\) (defined in Assumption~\ref{ass.estimate_nonasymptotic}). Then, under the conditions of Theorem~\ref{thm.complex}, Algorithm~\ref{alg.main} takes at most
\[
K_{\varepsilon} = \inrevise{\left\lceil\max \Set{
\frac{5(\rc{cnst:thm.complex-1} + \rc{cnst:thm.complex-4} \omega_{\rho})}{2\beta \varepsilon^2},
\frac{  5 \omega_{\rho} \sqrt{\omega_{\delta} \rc{cnst:thm.complex-3}}}{\varepsilon^2}
} \right\rceil} = \mathcal{O}(\varepsilon^{-2})
\] iterations and \(W_{\varepsilon} = \left\lceil\frac{27 \rhomax}{\omega^2_{\rho}} \inrevise{K^3_{\varepsilon}}\right\rceil = \mathcal{O}(\inrevise{\varepsilon^{-6}})\) samples for Algorithm~\ref{alg.main} to reach an iterate \((X_k, Y_k)\) that is \(\varepsilon\)-stationary in expectation, that is,
\begin{align*}
\Expect \left[
\norm{\nabla f(X_k) + \nabla c(X_k) Y_k}^2 + \norm{c(X_k)}_1
\mid \nonE \right] < \inrevise{\varepsilon^2}.
\end{align*}
\end{cor}

\begin{proof}
By setting $\kmax \leftarrow K_{\varepsilon}$ and conditioning on the occurance of \(\nonE\), it holds by Theorem~\ref{thm.complex} that
\begin{align*}
& \min_{k \in [K_{\varepsilon}]} \left\{ \Expect \left[
\norm{\nabla f(X_k) + \nabla c(X_k) Y_k}^2 + \norm{c(X_k)}_1
\mid \nonE \right]
\right\}
\leq
\Expect \left[ \,
\norm{\nabla f(X_K) + \nabla c(X_K) Y_K}^2 + \norm{c(X_K)}_1
\mid \nonE \,
\right] \\
& \qquad \leq
\inrevise{\frac{\rc{cnst:thm.complex-1} + \rc{cnst:thm.complex-4} \omega_{\rho}}{\beta K_{\varepsilon}}}
+ \frac{10 \omega_{\delta} \rc{cnst:thm.complex-3}\omega^2_{\rho}}{\inrevise{\varepsilon^2K_{\varepsilon}^2}}
+ 0.1 \inrevise{\varepsilon^2}
\leq 0.4 \inrevise{\varepsilon^2} + 0.4 \inrevise{\varepsilon^2} + 0.1 \inrevise{\varepsilon^2} < \inrevise{\varepsilon^2},
\end{align*}
where $K$ is the uniform random variable defined in Theorem~\ref{thm.complex} and the third inequality follows the selection of \(K_{\varepsilon}\) because
\begin{align*}
K_{\varepsilon} &\geq \inrevise{\frac{5(\rc{cnst:thm.complex-1} + \rc{cnst:thm.complex-4} \omega_{\rho})}{2\beta \varepsilon^2}}
\quad\text{implies} \quad
\inrevise{\frac{\rc{cnst:thm.complex-1} + \rc{cnst:thm.complex-4} \omega_{\rho}}{\beta K_{\varepsilon}}}
\leq 0.4\inrevise{\varepsilon^2}, \\
\text{and} \quad K_{\varepsilon} &\geq \inrevise{\frac{5 \omega_{\rho} \sqrt{\omega_{\delta} \rc{cnst:thm.complex-3}}}{\varepsilon^2}}
\quad\text{implies} \quad
\frac{10 \omega_{\delta} \rc{cnst:thm.complex-3}\omega^2_{\rho}}{\inrevise{\varepsilon^2K_{\varepsilon}^2}}
\leq 0.4\inrevise{\varepsilon^2}.
\end{align*}
Next, we analyze the sample complexity. Suppose that one obtains \(N^g_k, N^c_k\) and \(N^j_k\)~\iid{} samples for estimating \(\nabla f(X_k)\), \(c(X_k)\) and \(\nabla c(X_k)\) at iteration \(k\), then
\begin{align*}
\Expect_k \left[ \norm{\bar{G}_k - \nabla f(X_k)}^2 \right] = \frac{(\sigma^g_k)^2}{N^g_k},
\quad
\Expect_k \left[ \norm{\bar{C}_k - c(X_k)}^2 \right] = \frac{(\sigma^c_k)^2}{N^c_k},\quad \text{and}
\quad
\Expect_k \left[ \norm{\inrevise{\bar{J}_k - \nabla c(X_k)^{\top}}}_{F}^2 \right] = \frac{(\sigma^j_k)^2}{N^j_k},
\end{align*}
where $(\sigma^g_k)^2$, $(\sigma^c_k)^2$, and  $(\sigma^j_k)^2$ are the population variances associated to $\nabla f(X_k)$, $c(X_k)$, and \(\nabla c(X_k)^T\), respectively. Therefore, it is sufficient to sample
\begin{align*}
N^g_k = \left\lceil \frac{9 (\sigma^g_k)^2 \inrevise{K^2_{\varepsilon}}}{\omega^2_{\rho}} \right\rceil,
\quad
N^c_k = \left\lceil \frac{9 (\sigma^c_k)^2 \inrevise{K^2_{\varepsilon}}}{\omega^2_{\rho}} \right\rceil,
\quad \text{and}\quad
N^j_k = \left\lceil \frac{9 (\sigma^j_k)^2 \inrevise{K^2_{\varepsilon}}}{\omega^2_{\rho}} \right\rceil
\end{align*}
to satisfy
\(\frac{\sigma^g_k}{\sqrt{N^g_k}} + \frac{\sigma^c_k}{\sqrt{N^c_k}}
+ \frac{\sigma^j_k}{\sqrt{N^j_k}}
\leq \frac{\omega_{\rho}}{\inrevise{K_{\varepsilon}}}\). In total, one needs
\begin{align*}
W_{\varepsilon} = \sum_{k=1}^{K_{\varepsilon}} \left( N^g_k + N^c_k + N^j_k \right)
\leq 3K_{\varepsilon} +  \frac{9 \inrevise{K^2_{\varepsilon}}}{\omega^2_{\rho}} \sum_{k=1}^{K_{\varepsilon}} \left( (\sigma^g_k)^2 + (\sigma^c_k)^2 + (\sigma^j_k)^2 \right)
\leq 3K_{\varepsilon} + \frac{27 \rhomax}{\omega^2_{\rho}} \inrevise{K_{\varepsilon}^3}
= \mathcal{O}(\inrevise{\varepsilon^{-6}})
\end{align*}
samples, where the last inequality follows the property of $\rhomax \geq \max\left\{(\sigma^g_k)^2, (\sigma^c_k)^2, (\sigma^j_k)^2\right\}$ (see Assumption~\ref{ass.estimate_nonasymptotic}). Therefore, we conclude the statement.
\end{proof}

\begin{bulkrevise}
\begin{remark}\label{rem.complexity_comparison}
The complexity results reported in Corollary~\ref{cor.complex} are competitive and well aligned with existing iteration and sample complexity bounds in the literature~\cite{alacaoglu2024complexity,BeraXieZhou23,CurtRobiOnei24,fang2025high}. 
Among these works,~\cite{BeraXieZhou23,CurtRobiOnei24,fang2025high} focus on \textit{deterministic} equality-constrained stochastic optimization problems, whereas~\cite{alacaoglu2024complexity} studies \textit{expectation} equality-constrained problems, which align more closely with the setting considered in our paper. 

Specifically,~\cite{CurtRobiOnei24} analyzes the complexity of the objective-function-free stochastic SQP method originally proposed in~\cite{BeraCurtRobiZhou21} and establishes an $\mathcal{O}(\varepsilon^{-4})$ iteration and sample complexity in the fully stochastic regime, where only one or very few stochastic estimates are required per iteration. In contrast,~\cite{BeraXieZhou23,fang2025high} rely on objective function estimates in their iterative updates, consider irreducible noise, and require refined stochastic estimates until the noise level in estimates is sufficiently small (with certain probability). More precisely,~\cite{BeraXieZhou23} assumes light-tailed noise in objective function estimates (e.g., ``one-sided'' subexponential distributions) and proposes a stochastic line-search SQP method that identifies first-order $\varepsilon$-stationary points within $\mathcal{O}(\varepsilon^{-2})$ iterations with high probability. \citet{fang2025high} addresses heavy-tailed noise and analyzes a stochastic trust-region SQP algorithm, establishing high-probability iteration complexities of $\mathcal{O}(\varepsilon^{-2})$ and $\mathcal{O}(\varepsilon^{-3})$ for first- and second-order $\varepsilon$-stationary points, respectively. Notably, neither~\cite{BeraXieZhou23} nor \cite{fang2025high} operates in the fully stochastic regime, and both require substantially more samples per iteration that the algorithm in~\cite{CurtRobiOnei24}, which can lead to inferior overall sample complexity. For example, the algorithm in~\cite{fang2025high} requires $\mathcal{O}(\varepsilon^{-6})$ and $\mathcal{O}(\varepsilon^{-9})$ samples to identify first- and second-order $\varepsilon$-stationary points, respectively, which are worse than the sample complexity reported in~\cite{CurtRobiOnei24}.

Finally,~\cite{alacaoglu2024complexity} considers the same problem setting as our work but imposes a strictly stronger assumption, namely the \textit{mean-square smoothness} condition, and proposes a variance-reduced algorithm that achieves an $\mathcal{O}(\varepsilon^{-5})$ sample complexity for identifying $\varepsilon$-stationary points. In contrast, our analysis neither relies on the \textit{mean-square smoothness} condition nor employs variance-reduction techniques. It is well known that, in unconstrained stochastic optimization, classical variance-reduction methods (e.g.,~\cite{fang2018spider}) typically improve the complexity of vanilla stochastic gradient methods by one order in $\varepsilon$ (from $\mathcal{O}(\varepsilon^{-4})$ to $\mathcal{O}(\varepsilon^{-3})$). From this perspective, we believe that our $\mathcal{O}(\varepsilon^{-6})$ sample complexity bound remains competitive, as the variance-reduced algorithm in~\cite{alacaoglu2024complexity} also improves the dependence on $\varepsilon$ by only one order.
\end{remark}

\begin{remark}\label{rmk:stop-rule}
A common practical stopping rule in stochastic optimization is to run the algorithm for a prescribed number of iterations \(\kmax\) and return one of the generated iterates. Motivated by the non-asymptotic result in Theorem~\ref{thm.complex}, one may select the output iterate uniformly at random from the \(\kmax\) iterates. In this case, Theorem~\ref{thm.complex} guarantees that the expected stationarity error of the selected iterate is bounded by the right-hand side of~\eqref{eq:complex-main-rst}.

In practice, however, it is often desirable to obtain a probabilistic guarantee on the stationarity error of the output iterate. To this end, under the assumptions of Theorem~\ref{thm.complex}, for any
$(\upvarepsilon, \updelta) \in \reals_{>0} \times (0,1]$, to achieve
\begin{align*}
\Prob \Big[
\norm{\nabla f(X_K) + \nabla c(X_K) Y_K}^2 + \norm{c(X_K)}_1 \geq \upvarepsilon^2
\,\Big|\, \nonE \Big] \leq \updelta,
\end{align*}
it is sufficient to select \(\kmax\) such that 
\begin{align*}
\kmax \geq \max \Set{
\frac{3 \left( \rc{cnst:thm.complex-1} + \rc{cnst:thm.complex-4} \omega_{\rho} \right)}{\beta \updelta \cdot \upvarepsilon^2},
\frac{3 \sqrt{\omega_{\updelta} \rc{cnst:thm.complex-3} \omega^2_{\rho}}}{\updelta \cdot \upvarepsilon^2}
} = \mathcal{O} \left( \frac{1}{\updelta \cdot \upvarepsilon^2} \right).
\end{align*}
To see this, let \(\varepsilon := \sqrt{10 \updelta \cdot \upvarepsilon^2 / 3}\) and by Theorem~\ref{thm.complex} and conditional Markov's inequality, it holds almost surely that
\begin{align*}
\Prob \Big[
\norm{\nabla f(X_K) + \nabla c(X_K) Y_K}^2 + \norm{c(X_K)}_1 \geq \upvarepsilon^2
\,\Big|\, \nonE \Big]
&\leq \frac{1}{\upvarepsilon^2}\mathbb{E}\Big[ \norm{\nabla f(X_K) + \nabla c(X_K) Y_K}^2 + \norm{c(X_K)}_1 \,\Big|\, \nonE \Big] \\
&\leq  
\frac{1}{\upvarepsilon^2} \left( 
\frac{\rc{cnst:thm.complex-1} + \rc{cnst:thm.complex-4} \omega_{\rho}}{\beta \kmax} 
+ \frac{3 \omega_{\updelta} \rc{cnst:thm.complex-3} \omega^2_{\rho}}{\updelta \upvarepsilon^2 \kmax^2}
  + \frac{1}{3} \updelta \upvarepsilon^2
\right)
\leq \updelta.
\end{align*}

In practice, selecting the last iterate encountered also serves as a reasonable option. In this regard, we refer the reader to our numerical experiments on CUTEst~\cite{gould2015cutest} problems in Section \ref{sec.ssm}, in particular Figure~\ref{fig:8plots}, which reports data profiles over 20 instances on 8 different problems. These results indicate that the last iterate is usually not far from the best iterate attained among all iterates generated by the algorithm. When our algorithm is applied to solve constrained machine learning problems, one may terminate it in a manner similar to stochastic gradient methods. For example, one can stop the algorithm based on practical heuristics such as stabilization of the training or validation loss or diminishing improvements across iterations.
\end{remark}
\end{bulkrevise}

\subsection{Poor Merit Parameter Behaviors}\label{sec.merit_parameter_behaviors}

\inrevise{In this section, we further illustrate the poor behaviors of $\{\bar\cT_k\}$ discussed in~\eqref{eq.poor_merit_parameters_asymptotic} and show that, under additional mild assumptions, case~$(i)$ of~\eqref{eq.poor_merit_parameters_asymptotic} never happens while case~$(ii)$ of~\eqref{eq.poor_merit_parameters_asymptotic} can occur only with probability zero.

Specifically, we will prove in Lemma~\ref{lem.tau_zero_asymptotic} that case~$(i)$ of~\eqref{eq.poor_merit_parameters_asymptotic} can occur only when the stochastic objective gradient estimates $\{\bar{G}_k\}$ and/or the stochastic constraint function estimates $\{\bar{C}_k\}$ are not  uniformly bounded (or, more precisely, there exists an infinite set $\mathcal{K}_{\searrow 0}\subseteq\mathbb{N}$ such that $\left\{\max\{\|\bar{G}_k\|,\|\bar{C}_k\|\}\right\}_{k\in\mathcal{K}_{\searrow 0}}$ goes to $+\infty$). When case~$(i)$ of~\eqref{eq.poor_merit_parameters_asymptotic} occurs (i.e., $\lim_{k\to\infty}\bar\cT_k = 0$), the algorithm asymptotically ignores the objective function and focuses solely on minimizing constraint violation, which in turn prevents convergence of the KKT stationarity measures in expectation. 

We will further analyze the likelihood of case~$(ii)$ of~\eqref{eq.poor_merit_parameters_asymptotic} under standard conditions. By Lemma~\ref{lem.merit_parameter}, case~$(ii)$ of~\eqref{eq.poor_merit_parameters_asymptotic} implies that there exist infinitely many iterations $k\in\mathcal{K}_{\bar\tau,big}$ such that 
\begin{equation}\label{eq.tau_trial}
\bar\cT_k^{\text{trial}} \geq (1-\epsilon_{\tau})\cdot\bar\cT_k^{\text{trial}} \geq \bar\cT_k > \cT_k^{\text{trial}} > 0.
\end{equation}
For any iteration $t\in\mathcal{K}_{\bar\tau,big}$, we directly have from~\eqref{eq.tau_trial} that
\begin{equation*}
\frac{|\bar\cT_t^{\text{trial}} - \cT_t^{\text{trial}}|}{\cT_t^{\text{trial}}} \geq \frac{\bar\cT_t^{\text{trial}} - \cT_t^{\text{trial}}}{\cT_t^{\text{trial}}} > \frac{\frac{\cT_t^{\text{trial}}}{1-\epsilon_{\tau}} - \cT_t^{\text{trial}}}{\cT_t^{\text{trial}}} = \frac{1}{1-\epsilon_{\tau}} - 1 = \frac{\epsilon_{\tau}}{1-\epsilon_{\tau}} > \epsilon_{\tau},
\end{equation*}
which implies that the stochastic estimates $(\bar{G}_t,\bar{C}_t,\bar{J}_t)$ cannot be sufficiently accurate. Hence, each iteration $t\in\mathcal{K}_{\bar\tau,big}$ necessarily corresponds to a significant inaccuracy in the stochastic estimates $(\bar{G}_t,\bar{C}_t,\bar{J}_t)$. 
Consequently, case~$(ii)$ of~\eqref{eq.poor_merit_parameters_asymptotic} would require such inaccuracies to occur infinitely often. Therefore, if the stochastic estimates $(\bar{G}_k,\bar{C}_k,\bar{J}_k)$ could be sufficiently accurate with non-zero probability at any iteration $k\in\mathbb{N}$, then case~$(ii)$ of~\eqref{eq.poor_merit_parameters_asymptotic} can occur only with probability zero. This observation constitutes the core reasoning underlying Lemma~\ref{lem.tau_large_asymptotic}. From an algorithmic perspective, whenever $\bar\cT_k > \cT_k^{\text{trial}}$, the crucial link between the model reduction $\Delta l(X_k,\bar\cT_k,\nabla f(X_k),c(X_k),\nabla c(X_k)^T,D_k)$ and the quantity $\|D_k\|^2 + \|c(X_k)\|_1$ (see Lemma~\ref{lem.model_reduction_suff_large_true}) is lost, which is essential for establishing convergence of the KKT stationarity measures.}

\blemma\label{lem.tau_zero_asymptotic}
Suppose Assumptions~\ref{ass.prob}--\ref{ass.estimate} hold. For any $k\in\mathbb{N}$, if $\bar{G}_k^T\bar{D}_k + \frac{1}{2}\bar{D}_k^TH_k\bar{D}_k > 0$, then $\bar\cT_k^{\trial} \geq \frac{1-\sigma}{\|\bar{Y}_k\|_{\infty}}$.
If $\{\bar{G}_k\}$ and $\{\bar{C}_k\}$ are uniformly bounded, then case~$(i)$ of~\eqref{eq.poor_merit_parameters_asymptotic} never happens, i.e., there must exists a constant $\bar\tau_{lb} > 0$ such that $\bar\cT_k \geq \bar\tau_{lb}$ for all $k\in\mathbb{N}$.
\elemma
\bproof
When $\bar{G}_k^T\bar{D}_k + \frac{1}{2}\bar{D}_k^TH_k\bar{D}_k > 0$, we have from Assumption~\ref{ass.H},~\eqref{eq.linear_system},~\eqref{eq.merit_parameter_update} and the Cauchy-Schwarz inequality that
\bequation
\label{eq:lower-bd-on-tau-trial}
\begin{aligned}
\bar\cT_k^{\trial} &= \frac{(1-\sigma)\|\bar{C}_k\|_1}{\bar{G}_k^T\bar{D}_k + \frac{1}{2}\bar{D}_k^TH_k\bar{D}_k} \geq \frac{(1-\sigma)\|\bar{C}_k\|_1}{\bar{G}_k^T\bar{D}_k + \bar{D}_k^TH_k\bar{D}_k} = \frac{(1-\sigma)\|\bar{C}_k\|_1}{-\bar{D}_k^T\bar{J}_k^T\bar{Y}_k} \\
&= \frac{(1-\sigma)\|\bar{C}_k\|_1}{\bar{C}_k^T\bar{Y}_k} \geq \frac{(1-\sigma)\|\bar{C}_k\|_1}{\|\bar{C}_k\|_1\|\bar{Y}_k\|_{\infty}} = \frac{1-\sigma}{\|\bar{Y}_k\|_{\infty}},
\end{aligned}
\eequation
which proves the first part of the statement.

\inrevise{Next, we show}
that $\{\bar\cT_k^{\trial}\}$ is uniformly bounded away from zero. When $\bar{G}_k^T\bar{D}_k + \frac{1}{2}\bar{D}_k^TH_k\bar{D}_k \leq 0$, we have $\bar\cT_k^{\trial} = +\infty$ from~\eqref{eq.merit_parameter_update}. On the other hand, if $\bar{G}_k^T\bar{D}_k + \frac{1}{2}\bar{D}_k^TH_k\bar{D}_k > 0$, then we have \eqref{eq:lower-bd-on-tau-trial} holds.
Moreover, by~\eqref{eq.linear_system}, Lemma~\ref{lem.singular_values}, and the boundedness of $\{\bar{G}_k\}$ and $\{\bar{C}_k\}$, there exists a constant $\kappa_{\bar{y}} > 0$ such that $\|\bar{Y}_k\|_{\infty} \leq \kappa_{\bar{y}}$ for all $k\in\mathbb{N}$, which further provides a lower bound of $\{\bar\cT_k^{\trial}\}$ that
\bequationn
\bar\cT_k^{\trial} \geq \frac{1-\sigma}{\|\bar{Y}_k\|_{\infty}} \geq \frac{1-\sigma}{\kappa_{\bar{y}}}
\eequationn
for all $k\in\mathbb{N}$. Finally, by the help of~\eqref{eq.merit_parameter_update}, we have $\bar\cT_k < \bar\cT_{k-1}$ only when $\bar\cT_{k-1} > (1-\epsilon_{\tau})\bar\cT_k^{\trial} \geq \frac{(1-\epsilon_{\tau})(1-\sigma)}{\kappa_{\bar{y}}}$. By taking $\bar\tau_{lb} = (1-\epsilon_{\tau})\cdot\min\left\{\bar\tau_{0}, \frac{(1-\epsilon_{\tau})(1-\sigma)}{\kappa_{\bar{y}}}\right\}$, we would conclude the statement.
\eproof

\blemma\label{lem.tau_large_asymptotic}
Suppose Assumptions~\ref{ass.prob}--\ref{ass.estimate} hold. If there exists a probability $p\in(0,1]$ that for any $k\in\mathbb{N}$
\bequationn
\mathbb{P}[(\bar{G}_k,\bar{C}_k,\bar{J}_k) = (\nabla f(X_k),c(X_k),\nabla c(X_k)^T) | \inrevise{\cG_k}] \geq p,
\eequationn
then case~$(ii)$ of~\eqref{eq.poor_merit_parameters_asymptotic} at most happens with probability zero.
\elemma
\bproof
From the lemma statement and Lemma~\ref{lem.merit_parameter}, we know that for any $k\in\mathbb{N}$,
\bequation\label{eq.low_prob_event}
\mathbb{P}[\bar\cT_k > \cT_k^{\trial} | \inrevise{\cG_k}] = 1 - \mathbb{P}[\bar\cT_k \leq \cT_k^{\trial} | \inrevise{\cG_k}] \leq 1 - \mathbb{P}[\bar\cT_k^{\trial} = \cT_k^{\trial} | \inrevise{\cG_k}] \leq 1 - p.
\eequation
Note that case~$(ii)$ of~\eqref{eq.poor_merit_parameters_asymptotic} can only happen when the event $\bar\cT_k > \cT_k^{\trial}$ shows up for infinite iterations $k\in\cK_{\bar\tau,big}$, which can only happen with probability zero; see~\eqref{eq.low_prob_event}.
\eproof

We note that in addition to Assumptions~\ref{ass.prob}--\ref{ass.estimate}, Lemma~\ref{lem.tau_zero_asymptotic} requires $\{\bar{G}_k\}$ and $\{\bar{C}_k\}$ being uniformly bounded, which is not a restrictive condition because when random vector $\omega$ (see~\eqref{eq.prob})
has finite support it is implied by the condition that all objective and constraint component functions as well as their associated derivatives are bounded. Similarly, Lemma~\ref{lem.tau_large_asymptotic} also applies when \(\omega\) has finite support, because each singleton in the support is an atom. Additionally, we show that \inrevise{case $(i)$ of}~\eqref{eq.poor_merit_parameters_asymptotic} occurs with low probability if the variance sequence \(\left\{\sqrt{\rho^g_k} + \sqrt{\rho^j_k} + \sqrt{\rho^c_k}\right\}\) is bounded by \(\mathcal{O}(1/\sqrt{\kmax})\) \inrevise{(see Lemma~\ref{lem:tau-nonasymptotic-bdd-away-from-zero})} and with zero probability if the variance sequence \(\{\rho^g_k + \rho^j_k + \rho^c_k\}\) is summable \inrevise{(see Lemma~\ref{lem.tau_zero_wpzero})}. The main motivation behind these two results is that \(1/\bar{\cT}_k^\trial\) is finite with high probability, which is formalized in the following lemma.

\blemma\label{lem.tau_trail_markov}
Suppose that Assumptions~\ref{ass.prob}--\ref{ass.estimate} hold and  that Assumption~\ref{ass.estimate_asymptotic} holds with the natural filtration $\{\natF_k\}$. Then, for any iterate \(k \in \naturals{}\), it holds for any \(\nctwo\label{cnst:tau-trail-markov} > \kappa_y\) (see Lemma~\ref{lem.ck-bdd-away}) that
\bequationn
\Prob \left[
 \frac{(1 - \sigma)}{\bar{\cT}_k^{\trial}} \geq \rctwo{cnst:tau-trail-markov}
\bigg| \natF_k
\right] \leq
\frac{\rc{cnst:diminishing-tau-1}}{(\rctwo{cnst:tau-trail-markov} - \kappa_y)^2} \left( \rho^j_k + \rho^g_k + \rho^c_k \right),
\eequationn
where \(\nc\label{cnst:diminishing-tau-1}\) is a positive constant that only depends on \((q_{\text{min}}, \omega_2, \omega_3)\) defined in Lemmas~\ref{lem.singular_values} and~\ref{lem.solution_variance}.
\elemma

\bproof
By Lemma~\ref{lem.tau_zero_asymptotic} and the definition of \(\bar\cT_k^\trial\), when $\bar{G}_k^T\bar{D}_k + \frac{1}{2}\bar{D}_k^TH_k\bar{D}_k > 0$, we have
\begin{align*}
\frac{(1 - \sigma)}{\bar{\cT}_k^{\trial}}
\leq \norm{\bar{Y}_k}_\infty
\leq \norm{\bar{Y}_k} \leq \norm[\bigg]{
  \begin{bmatrix}
    \bar{D}_k \\ \bar{Y}_k
  \end{bmatrix}
} \leq \norm[\bigg]{
  \begin{bmatrix}
    D_k \\ Y_k
  \end{bmatrix}} + \norm[\bigg]{
  \begin{bmatrix}
    \bar{D}_k - D_k \\ \bar{Y}_k - Y_k
  \end{bmatrix}}
\leq \kappa_y + \norm[\bigg]{
  \begin{bmatrix}
    \bar{D}_k - D_k \\ \bar{Y}_k - Y_k
  \end{bmatrix}},
\end{align*}
where the fourth inequality is due to the triangle inequality and the last inequality follows from \eqref{eq.yk_bound}.
In addition, when $\bar{G}_k^T\bar{D}_k + \frac{1}{2}\bar{D}_k^TH_k\bar{D}_k \leq 0$, by~\eqref{eq.merit_parameter_update} and \eqref{eq.yk_bound}, we directly have $\frac{(1 - \sigma)}{\bar{\cT}_k^{\trial}} \leq \kappa_y + \left\|\begin{bmatrix}
    \bar{D}_k - D_k \\ \bar{Y}_k - Y_k
  \end{bmatrix}\right\|$ as well.
Accordingly, by the conditional Markov's inequality and Lemma~\ref{lem.solution_variance} applied to the natural filtration, one has
\begin{align*}
\Prob \left[
 \frac{(1 - \sigma)}{\bar{\cT}_k^{\trial}} \geq \rctwo{cnst:tau-trail-markov}
 \bigg| \natF_k
\right]
& \leq
\Prob \left[
 \kappa_y +
\norm[\bigg]{
  \begin{bmatrix}
    \bar{D}_k - D_k \\ \bar{Y}_k - Y_k
  \end{bmatrix}}
\geq \rctwo{cnst:tau-trail-markov}
\bigg| \natF_k
\right]
=
\Prob \left[
\norm[\bigg]{
  \begin{bmatrix}
    \bar{D}_k - D_k \\ \bar{Y}_k - Y_k
  \end{bmatrix}}
\geq \rctwo{cnst:tau-trail-markov} - \kappa_y
\bigg| \natF_k
\right] \\
& \leq \frac{1}{(\rctwo{cnst:tau-trail-markov} - \kappa_y)^2} \
\Expect \left[
\norm[\bigg]{
  \begin{bmatrix}
    \bar{D}_k - D_k \\ \bar{Y}_k - Y_k
  \end{bmatrix}}^2
  \bigg| \natF_k
 \right] \\
& \leq \frac{1}{(\rctwo{cnst:tau-trail-markov} - \kappa_y)^2}
\left(
\frac{\rho_k^g + \rho_k^c}{q_{\min}^2} + \omega_2\cdot\rho_k^j + \omega_3\cdot \sqrt{\rho_k^j(\rho_k^g+\rho_k^c)}
\right).
\end{align*}
Subsequently, by \(ab \leq \left( a^2 + b^2 \right) / 2\) for any $\{a, b\} \subset\mathbb{R}_{\geq 0}$, we have \(\sqrt{\rho^j_k \left( \rho^g_k + \rho^c_k \right)} \leq \frac{1}{2} \left(\rho^j_k + \rho^g_k + \rho^c_k\right)\) and
\begin{align*}
\Prob \left[
 \frac{(1 - \sigma)}{\bar{\cT}_k^{\trial}} \geq \rctwo{cnst:tau-trail-markov}
 \bigg| \natF_k
\right] \leq
\frac{\rc{cnst:diminishing-tau-1}}{(\rctwo{cnst:tau-trail-markov} - \kappa_y)^2} \left( \rho^j_k + \rho^g_k + \rho^c_k \right)
\end{align*}
follows by choosing \(\rc{cnst:diminishing-tau-1} = \max\left\{1/q^2_{\text{min}}, \omega_2\right\} + \omega_3/2\).
\eproof

\blemma\label{lem.tau_zero_wpzero}
Suppose that Assumptions~\ref{ass.prob}--\ref{ass.estimate} hold and  that Assumption~\ref{ass.estimate_asymptotic} holds with the natural filtration. Then \(\displaystyle \Prob \left[ \lim_{k \to \infty} \bar{\cT}_k = 0 \right] = 0\) as long as \(\sum_{k=1}^{\infty} (\rho^g_k + \rho^j_k + \rho^c_k) < +\infty\).
\elemma
\bproof
By~\eqref{eq.merit_parameter_update}, we have $\bar\cT_k < \bar\cT_{k-1}$ only when \(\bar{\cT}_k^{\trial} <  \bar{\cT}_{k-1} / (1 - \epsilon_{\tau})\), and $\lim_{k\to\infty}\bar{\cT}_k = 0$ implies that \(\bar{\cT}_k^{\trial} < \epsilon\) infinitely often (\io{}) for any given \(\epsilon > 0\). In particular, it holds for an arbitrary \(\epsilon_0 \in (0, (1 - \sigma) / (2\kappa_y)) \) that
\begin{align*}
\Prob \left[ \lim_{k \to \infty} \bar{\cT}_k = 0 \right] \leq \Prob \left[ \bar{\cT}_k^\trial < \epsilon_0 \ \text{i.o.} \right]
= \Prob \left[ 1/\bar{\cT}_k^\trial > 1/\epsilon_0 \ \text{i.o.} \right].
\end{align*}
If we can show that \(\sum_{k=1}^\infty \Prob \left[ 1 / \bar{\cT}_k^\trial > 1 / \epsilon_0 \right]\) exists and is finite, then \(\Prob \left[ 1 / \bar{\cT}_k^\trial > 1 / \epsilon_0 \ \text{i.o.} \right] = 0\) by Borel-Cantelli Lemma and we may complete the proof. To this end, we apply Lemma~\ref{lem.tau_trail_markov}, yielding
\begin{align*}
\Prob \left[
 \frac{(1 - \sigma)}{\bar{\cT}_k^{\trial}} > 2 \kappa_y + \delta_{\cT}
 \bigg| \natF_k
\right] \leq
\frac{\rc{cnst:diminishing-tau-1}}{(\kappa_y + \delta_{\cT})^2} \left( \rho^j_k + \rho^g_k + \rho^c_k \right)
\end{align*}
for any \(\delta_\cT > 0\). Then, it holds that
\begin{align*}
\sum_{k=1}^{\infty} \Prob \left[ \frac{1}{\bar{\cT}^\trial_k} > \frac{1}{\epsilon_0} \right]
& = \sum_{k=1}^{\infty} \Expect \left[ \Prob \left[
\frac{(1 - \sigma)}{\bar{\cT}^\trial_k} > 2 \kappa_y + \left( \frac{1 - \sigma}{\epsilon_0} - 2 \kappa_y \right) \bigg| \natF_k
\right]
\right] \\
& \leq \frac{\rc{cnst:diminishing-tau-1}}{(\frac{1 - \sigma}{\epsilon_0} - \kappa_y)^2}
\sum_{k=1}^{\infty} \left(  \rho^j_k + \rho^g_k + \rho^c_k \right) < +\infty,
\end{align*}
which concludes the statement.
\eproof

\blemma\label{lem:tau-nonasymptotic-bdd-away-from-zero}
Suppose that Assumptions~\ref{ass.prob}--\ref{ass.estimate} hold, Assumption~\ref{ass.estimate_asymptotic} holds with the natural filtration \(\{\cG_k\}\), and there exists a positive constant \(\omega_{\rho} > 0\) such that \(\sqrt{\rho^j_k} + \sqrt{\rho^g_k} + \sqrt{\rho^c_k} \leq \frac{\omega_{\rho}}{\sqrt{\kmax}}\) for all \(k \in [\kmax]\). Then, for any \(\epsilon \in (0, 1)\) there exists a positive constant \(\delta_{\tau} \) that is independent of \(\kmax\) and only relies on \((\bar\tau_0, \sigma, \epsilon_{\tau}, \omega_{\rho}, \kappa_y, \rc{cnst:diminishing-tau-1}, \epsilon)\) defined in Algorithm~\ref{alg.main},~\eqref{eq.merit_parameter_update}, Lemma~\ref{lem.ck-bdd-away}, and Lemma~\ref{lem.tau_trail_markov} such that
\begin{align*}
  \Prob \left[ \bar{\cT}_k > \delta_{\tau} \right] \geq 1 - \epsilon,
\end{align*}
for all \(k \in [\kmax]\).
\elemma

\bproof
We take \(\delta_{\tau} := 0.5 \cdot \min \Set{\frac{(1-\sigma)(1 - \epsilon_{\tau})^2}{\kappa_y + \sqrt{(\omega^2_{\rho} \rc{cnst:diminishing-tau-1}) / \epsilon}}, (1 - \epsilon_\tau) \tauzero}\) and decompose the proof into three steps, where in the steps~1 and~2 we construct an auxiliary random sequence to help connect \(\bar{\cT}_k\) with \(\max_{i \in [k]} \{1/\bar{\cT}_i^{\trial}\}\). In step~3, we show that for any \(k \in [\kmax]\), \(\max_{i \in [k]} \{1/\bar{\cT}_i^{\trial}\}\) is finite, or equivalently \(\min_{i \in [k]} \bar{\cT}_i^{\trial}\) is bounded away from zero, with high probability. \\[0.5em]
\textbf{Step 1.} Consider the following auxiliary sequence
\begin{align}\label{eq.key_relation_Tk}
\bar{T}_0 := \tauzero
\quad \text{and} \quad
\bar{T}_k :=
\begin{cases}
(1 - \epsilon_{\tau}) \cdot \bar{\cT}^{\trial}_k & \text{if \(\bar{\cT}_{k-1} > ( 1 - \epsilon_{\tau}) \bar{\cT}_k^{\trial}\),} \\
\bar{T}_{k-1} & \text{otherwise,}
\end{cases}
\quad \forall k \in \naturals,
\end{align}
where we know $\{\bar{T}_k\}\subset\mathbb{R}_{>0}$ by Lemma~\ref{lem.merit_parameter}. Next, we show that \((1 - \epsilon_{\tau}) \bar{T}_k \leq \bar{\cT}_k\) for all \(k \in \naturals\) using mathematical induction. When \(k = 1\), we consider two cases depending on $\bar\tau_0 \leq (1-\epsilon_{\tau})\cdot\bar\cT_1^{\trial}$ or not:
\begin{enumerate}[label=(\arabic*)]
\item When \(\bar\tau_0 \leq (1 - \epsilon_{\tau}) \cdot \bar{\cT}_1^{\trial}\), by \eqref{eq.merit_parameter_update} and \eqref{eq.key_relation_Tk}, we have
\(\bar{\cT}_1 = \bar\tau_0 = \bar{T}_0 = \bar{T}_1 \geq (1 - \epsilon_{\tau})\bar{T}_1\).
\item When \(\bar\tau_0 > (1 - \epsilon_{\tau}) \cdot \bar{\cT}_1^{\trial}\), it holds from \eqref{eq.merit_parameter_update} and \eqref{eq.key_relation_Tk} that
\[
\bar{\cT}_1
= (1 - \epsilon_{\tau}) \min \{\bar\tau_0, \bar{\cT}_1^{\trial}\}
\geq (1 - \epsilon_{\tau}) \min \{(1- \epsilon_{\tau}) \bar{\cT}_1^{\trial}, \bar{\cT}^{\trial}_1\}
= (1 - \epsilon_{\tau})^2 \cdot \bar{\cT}_1^{\trial} = (1 - \epsilon_{\tau}) \bar{T}_1.
\]
\end{enumerate}
For \(k \geq 2\), we assume that \((1 - \epsilon_{\tau}) \bar{T}_{k-1} \leq \bar{\cT}_{k-1}\) and aim to show \((1 - \epsilon_{\tau}) \bar{T}_k \leq \bar{\cT}_k\). To this end, using the same logic as above, we consider the following two cases:
\begin{enumerate}[label=(\arabic*)]
\item When \(\bar{\cT}_{k-1} \leq (1 - \epsilon_{\tau}) \cdot \bar{\cT}_k^{\trial}\), by \eqref{eq.merit_parameter_update} and \eqref{eq.key_relation_Tk}, we have
\(\bar{\cT}_k = \bar{\cT}_{k-1} \geq (1 - \epsilon_{\tau})\bar{T}_{k-1} = (1 - \epsilon_{\tau})\bar{T}_k\).
\item When \(\bar{\cT}_{k-1} > (1 - \epsilon_{\tau}) \cdot \bar{\cT}_k^{\trial}\), it holds from \eqref{eq.merit_parameter_update} and \eqref{eq.key_relation_Tk} that
\[
\bar{\cT}_k
= (1 - \epsilon_{\tau}) \min \{\bar{\cT}_{k-1}, \bar{\cT}_k^{\trial}\}
\geq (1 - \epsilon_{\tau}) \min \{(1- \epsilon_{\tau}) \bar{\cT}_k^{\trial}, \bar{\cT}^{\trial}_k\}
= (1 - \epsilon_{\tau})^2 \cdot \bar{\cT}_k^{\trial} = (1 - \epsilon_{\tau}) \bar{T}_k.
\]
\end{enumerate}
Therefore, we complete \textbf{Step 1}.

\textbf{Step~2.} We show that \(\displaystyle \bar{T}_k \geq \min\left\{\tauzero, (1 - \epsilon_{\tau}) \min_{i \in [k]} \bar{\cT}_i^{\trial}\right\}\) for all $k\in\NN$. When \(k = 1\), this relation holds by considering the following two cases:
\begin{enumerate}[label=(\arabic*)]
\item When \(\bar\tau_0 \leq (1 - \epsilon_{\tau}) \cdot \bar{\cT}_1^{\trial}\), by \eqref{eq.merit_parameter_update} and \eqref{eq.key_relation_Tk}, we have $\bar{T}_1 = \bar{T}_0 = \bar\tau_0
\geq \min\left\{\tauzero, (1 - \epsilon_{\tau}) \min_{i \in [1]} \bar{\cT}_i^{\trial}\right\}$.
\item When \(\bar\tau_0 > (1 - \epsilon_{\tau}) \cdot \bar{\cT}_1^{\trial}\), by \eqref{eq.merit_parameter_update} and \eqref{eq.key_relation_Tk}, we have
\(\bar{T}_1 = (1 - \epsilon_{\tau}) \bar{\cT}_1^{\trial}
\geq \min\left\{\tauzero, (1 - \epsilon_{\tau}) \min_{i \in [1]} \bar{\cT}_i^{\trial}\right\}\).
\end{enumerate}
When \(k \geq 2\), let's assume that \(\bar{T}_{k-1} \geq \min\{\tauzero, (1 - \epsilon_{\tau}) \min_{i \in [k-1]} \bar{\cT}_i^{\trial}\}\) and consider the following two cases:
\begin{enumerate}[label=(\arabic*)]
\item When \(\bar{\cT}_{k-1} \leq (1 - \epsilon_{\tau}) \cdot \bar{\cT}_k^{\trial}\), by \eqref{eq.merit_parameter_update} and \eqref{eq.key_relation_Tk}, we have
\[
\bar{T}_k = \bar{T}_{k-1}
\geq \min\left\{\tauzero, (1 - \epsilon_{\tau}) \min_{i \in [k-1]} \bar{\cT}_i^{\trial}\right\}
\geq \min\left\{\tauzero, (1 - \epsilon_{\tau}) \min_{i \in [k]} \bar{\cT}_i^{\trial}\right\}.
\]
\item When \(\bar{\cT}_{k-1} > (1 - \epsilon_{\tau}) \cdot \bar{\cT}_k^{\trial}\), by \eqref{eq.merit_parameter_update} and \eqref{eq.key_relation_Tk}, we have
\(\bar{T}_k = (1 - \epsilon_{\tau}) \bar{\cT}_k^{\trial}
\geq \min\left\{\tauzero, (1 - \epsilon_{\tau}) \min_{i \in [k]} \bar{\cT}_i^{\trial}\right\}\).
\end{enumerate}
In this way, we complete \textbf{Step 2}. Combining steps~1 and~2, we conclude that
\begin{align*}
\min\left\{
  (1 - \epsilon_{\tau}) \tauzero, (1 - \epsilon_{\tau})^2 \min_{i \in [k]} \bar{\cT}_i^{\trial}
\right\} \leq (1 - \epsilon_{\tau})\bar{T}_k \leq \bar{\cT}_k,
\quad
\forall k \in \naturals{}.
\end{align*}
As a result, it is sufficient to analyze \(\{\bar{\cT}_k^{\trial}\}\) because for any \(k \in [\kmax]\),
\begin{align}
\Prob \left[ \bar{\cT}_k \leq \delta_{\tau} \right]
& \leq \Prob \left[
\min\left\{
  (1 - \epsilon_{\tau}) \tauzero, (1 - \epsilon_{\tau})^2 \min_{i \in [k]} \bar{\cT}_i^{\trial}
\right\} \leq \delta_{\tau}
 \right] \nonumber{}\\
& = \Prob \left[
  (1 - \epsilon_{\tau})^2 \min_{i \in [k]} \bar{\cT}_i^{\trial}
\leq \delta_{\tau}
\right]
= \Prob \left[
  \max_{i \in [k]} \frac{(1 - \sigma)}{\bar{\cT}_i^{\trial}} \geq \frac{(1 - \sigma)(1 - \epsilon_{\tau})^2}{\delta_{\tau}}
 \right], \label{eq:tau-step23}
\end{align}
where the first equality is because \(\delta_{\tau} < (1 - \epsilon_{\tau}) \tauzero{}\) due to the definition of $\delta_{\tau}$, and the last equality is because \(\bar{\cT}_k^{\trial}\) is always positive.

\textbf{Step~3.} For clarity, we define \(\nctwo{}\label{cnst:delta-tau}:=(1 - \sigma)(1 - \epsilon_{\tau})^2 / \delta_{\tau}\) and \(S_k := \max_{i \in [k]} \{(1 - \sigma) / \bar{\cT}_i^{\trial}\}\) for all $k\in\NN$. Then, by the choice of \(\delta_{\tau}\), we have
\begin{align*}
\rctwo{cnst:delta-tau}
\geq 2 \left(\kappa_y + \sqrt{\frac{\omega^2_{\rho} \rc{cnst:diminishing-tau-1}}{\epsilon}} \right)
\implies \left( \rctwo{cnst:delta-tau} - \kappa_y \right)^2 \geq \frac{\omega^2_{\rho} \rc{cnst:diminishing-tau-1}}{\epsilon}
\implies \frac{\omega^2_{\rho} \rc{cnst:diminishing-tau-1}}{\left( \rctwo{cnst:delta-tau} - \kappa_y \right)^2} \leq \epsilon,
\end{align*}
and that \(S_{k-1}\) is \(\cG_k\)-measurable, because \(S_{k-1}\) only depends on
\(\left\{\bar{\cT}_i^{\trial}\right\}_{i \in [k-1]}\). Accordingly, we have that for all \(k \in [\kmax]\),
\bequation\label{eq.final_tau_bad_prob}
\begin{aligned}
\Prob \left[ S_k \geq \rctwo{cnst:delta-tau} \right]
& = \Prob \left[ S_k \geq \rctwo{cnst:delta-tau}, \ \frac{1 - \sigma}{\bar{\cT}_k^{\trial}} \geq \rctwo{cnst:delta-tau} \right]
+ \Prob \left[ S_k \geq \rctwo{cnst:delta-tau}, \ \frac{1 - \sigma}{\bar{\cT}_k^{\trial}} < \rctwo{cnst:delta-tau} \right] \\
& = \Prob \left[ \frac{1 - \sigma}{\bar{\cT}_k^{\trial}} \geq \rctwo{cnst:delta-tau} \right]
+ \Expect \left[ \Prob \left[ S_{k-1} \geq \rctwo{cnst:delta-tau}, \ \frac{1 - \sigma}{\bar{\cT}_k^{\trial}} < \rctwo{cnst:delta-tau} \bigg| \cG_k \right] \right] \\
& = \Expect \left[ \Prob \left[ \frac{1 - \sigma}{\bar{\cT}_k^{\trial}} \geq \rctwo{cnst:delta-tau} \bigg| \cG_k \right]
+ \Prob \left[
\frac{1 - \sigma}{\bar{\cT}_k^{\trial}} < \rctwo{cnst:delta-tau} \bigg| \cG_k \right] \cdot
\ones (S_{k-1} \geq \rctwo{cnst:delta-tau}) \right].
\end{aligned}
\eequation
For ease of exposition, we define \(a_k := \Prob \left[ \frac{1 - \sigma}{\bar{\cT}_k^{\trial}} \geq \rctwo{cnst:delta-tau} \mid \cG_k \right]\) and let \(\bar{a}_k \in [0, 1]\) be any deterministic upper bound of \(a_k\), whose existence is guaranteed because \(a_k \leq 1\). Then, it holds pointwise that
\begin{align*}
a_k + (1 - a_k) \cdot \ones (S_{k-1} \geq \rctwo{cnst:delta-tau}) \leq \bar{a}_k + (1 - \bar{a}_k) \cdot \ones (S_{k-1} \geq \rctwo{cnst:delta-tau}),
\end{align*}
and by~\eqref{eq.final_tau_bad_prob} we have the following recursive relation
\begin{align*}
\Prob \left[ S_k \geq \rctwo{cnst:delta-tau} \right] &= \mathbb{E}[a_k + (1 - a_k) \cdot \ones (S_{k-1} \geq \rctwo{cnst:delta-tau})] \\
& \leq \Expect \left[ \bar{a}_k + (1 - \bar{a}_k) \cdot \ones (S_{k-1} \geq \rctwo{cnst:delta-tau})\right]
  = \bar{a}_k + (1 - \bar{a}_k) \Prob \left[ S_{k-1} \geq \rctwo{cnst:delta-tau} \right],
\end{align*}
which recursively expands into
\begin{equation}\label{eq.key_cond_ub_P}
\begin{aligned}
\Prob \left[ S_k \geq \rctwo{cnst:delta-tau} \right]
& \leq \bar{a}_k + (1 - \bar{a}_k) \Prob \left[ S_{k-1} \geq \rctwo{cnst:delta-tau} \right] \\
& \leq \bar{a}_k + (1 - \bar{a}_k) (\bar{a}_{k-1} + (1 - \bar{a}_{k-1}) \Prob \left[ S_{k-2} \geq \rctwo{cnst:delta-tau} \right]) \\
& = \bar{a}_k + (1 - \bar{a}_k) \bar{a}_{k-1} + (1 - \bar{a}_k)(1 - \bar{a}_{k-1}) \Prob \left[ S_{k-2} \geq \rctwo{cnst:delta-tau} \right] \\
& \leq \sum_{i = 2}^k \left[ \bar{a}_i \prod_{j = i+1}^k (1 - \bar{a}_j) \right] + \Prob \left[ S_1 \geq \rctwo{cnst:delta-tau} \right] \prod_{i=2}^k (1 - \bar{a}_i)
\leq \sum_{i = 1}^k \left[ \bar{a}_i \prod_{j = i+1}^k (1 - \bar{a}_j) \right].
\end{aligned}
\end{equation}
Since \(\sqrt{\rho^j_k} + \sqrt{\rho^g_k} + \sqrt{\rho^c_k} \leq \frac{\omega_{\rho}}{\sqrt{\kmax}}\) for all \(k \in [\kmax]\), we have
\begin{gather*}
\left(\rho^j_k + \rho^g_k + \rho^c_k\right)
\leq \left( \sqrt{\rho^j_k} + \sqrt{\rho^g_k} + \sqrt{\rho^c_k} \right)^2
\leq \frac{\omega^2_{\rho}}{\kmax}, \quad \text{and} \\
a_k := \Prob \left[ \frac{1-\sigma}{\bar{\cT}_k^{\trial}} \geq \rctwo{cnst:delta-tau} \bigg| \cG_k \right]
  \leq \frac{\rc{cnst:diminishing-tau-1}}{\left(\rctwo{cnst:delta-tau} - \kappa_y\right)^2} \left( \rho^j_k + \rho^g_k + \rho^c_k \right)
\leq \frac{\omega^2_{\rho}\rc{cnst:diminishing-tau-1}}{\left(\rctwo{cnst:delta-tau} - \kappa_y\right)^2 \cdot \kmax}
\leq
\frac{\omega^2_{\rho}\rc{cnst:diminishing-tau-1}}{\left(\rctwo{cnst:delta-tau} - \kappa_y\right)^2}
\leq \epsilon < 1,
\end{gather*}
by Lemma~\ref{lem.tau_trail_markov} and the definitions of \(\delta_{\tau}\) and $\rctwo{cnst:delta-tau}$, implying that \(\frac{\omega^2_{\rho}\rc{cnst:diminishing-tau-1}}{\left(\rctwo{cnst:delta-tau} - \kappa_y\right)^2 \cdot \kmax}
\) is a valid choice of \(\bar{a}_k\) for all \(k\). Denoting \(\bar{a} := \frac{\omega^2_{\rho}\rc{cnst:diminishing-tau-1}}{\left(\rctwo{cnst:delta-tau} - \kappa_y\right)^2 \cdot \kmax} \in (0,1)\), then by~\eqref{eq.key_cond_ub_P}, we obtain an upper bound for \(\Prob \left[ S_k \geq \rctwo{cnst:delta-tau} \right]\) for all \(k \in [\kmax]\) as
\begin{align*}
\Prob \left[ S_k \geq \rctwo{cnst:delta-tau} \right]
& \leq \sum_{i = 1}^k \left[ \bar{a} \prod_{j = i+1}^k (1 - \bar{a}) \right]
= \bar{a} \cdot \sum_{i=1}^k (1 - \bar{a})^{k-i}
= \bar{a} \cdot \sum_{j=0}^{k-1} (1 - \bar{a})^j
= \bar{a} \cdot \frac{1 - (1 - \bar{a})^k}{1 - (1 - \bar{a})} \\
& = 1 - (1 - \bar{a})^k
= 1 - \left( 1 - \frac{\omega^2_{\rho}\rc{cnst:diminishing-tau-1}}{\left(\rctwo{cnst:delta-tau} - \kappa_y\right)^2 \cdot \kmax} \right)^k,
\end{align*}
and we have for all \(k \in [\kmax]\) that
\begin{align*}
\Prob \left[ \bar{\cT}_k \leq \delta_{\tau} \right]
\leq \Prob \left[ S_{\kmax} \geq \rctwo{cnst:delta-tau} \right]
\leq 1 - \left( 1 - \frac{\omega^2_{\rho}\rc{cnst:diminishing-tau-1}}{\left(\rctwo{cnst:delta-tau} - \kappa_y\right)^2 \cdot \kmax}\right)^{\kmax}
\leq \frac{\omega^2_{\rho} \rc{cnst:diminishing-tau-1}}{(\rctwo{cnst:delta-tau} - \kappa_y)^2}
\leq \epsilon,
\end{align*}
where the first inequality is due to~\eqref{eq:tau-step23}, the third inequality is because \(\left( 1 - \frac{\bar{a}k_{\max}}{x} \right)^x\) is monotonically increasing with $x$ over \([\bar{a}k_{\max}, +\infty) \supset [1,k_{\max}]\), and the last inequality is by the choice of $\delta_{\tau}$ and $\rctwo{cnst:delta-tau}$.
\eproof

\section{Numerical Results}\label{sec.numerical}

In this section, we compare Algorithm~\ref{alg.main} to a stochastic subgradient method and a stochastic momentum-based algorithm on test problems from the CUTEst collection~\cite{gould2015cutest} as well as the LIBSVM collection~\cite{chang2011libsvm}. The purpose of these experiments is demonstrating the numerical performance of our proposed Algorithm~\ref{alg.main} and the other two alternative algorithms' for solving expectation equality constrained stochastic optimization problems. We first present numerical comparisons of Algorithm~\ref{alg.main} and a stochastic subgradient method on test problems from the CUTEst collection~\cite{gould2015cutest} in Section~\ref{sec.ssm}, while in Section~\ref{sec.tstom}, we further consider the LIBSVM~\cite{chang2011libsvm} test problems and the stochastic momentum-based method from~\cite{CuiWangXiao24}.

\subsection{Experiments on CUTEst~\cite{gould2015cutest} \inrevise{Problems}}\label{sec.ssm}

The first set of experiments compared Algorithm~\ref{alg.main} against a stochastic subgradient method, which aims for minimizing the exact penalty function in the form of~\eqref{eq.merit_func}. This set of experiments focused on a total of 44 test problems selected from the CUTEst collection~\cite{gould2015cutest}, which met the criteria of {\sl (i)} only having equality constraints, {\sl (ii)} the total number of variables and constraints not exceeding 1000, {\sl (iii)} not being with constant objectives, and {\sl (iv)} the LICQ condition being satisfied at all iterations in all runs of both algorithms.

\inrevise{The experiments were conducted at different sample variances. Specifically, we considered two noise-level settings: \textit{(a)} $\epsilon_g \in \{10^{-8}, 10^{-4}, 10^{-2}\}$ with $\epsilon_c = \epsilon_J^2 \in \{10^{-8}, 10^{-4}, 10^{-2}\}$; and \textit{(b)} $(\epsilon_g,\epsilon_c,\epsilon_J)\in\{(10^{-2},0.25,0.5),(0.25,10^{-2},10^{-1}),(0.25,0.25,0.5)\}$, where setting \textit{(b)} corresponds to higher noise levels than setting \textit{(a)}. Altogether, this yields 12 distinct combinations of sample variances. Specifically, for test problems with $n$-dimensional variables and $m$ equality constraints, we generated stochastic estimates $\bar{g}_k\sim\mathcal{N}\left(\nabla f(x_k),\frac{\epsilon_g\beta_k^2}{n}I\right)$, $\bar{c}_k\sim\mathcal{N}\left(c(x_k),\frac{\epsilon_c\beta_k^2}{m}I\right)$, and $\bar{j}_k^i\sim\left(\nabla c^i(x_k),\frac{\epsilon_J\beta_k^2}{mn}I\right)$ for all $i\in[m]$, where $\bar{j}_k^i$ represents the $i$-th column of matrix $\bar{j}_k^T$ and $\{\beta_k\}\subset(0,1]$ is a predetermined step size parameter sequence. We tested the performance of both Algorithm~\ref{alg.main} and the stochastic subgradient method with two different selections of $\{\beta_k\}$: {\sl (i)} a constant sequence $\beta_k = \beta = 0.1$ for every iteration $k$, and {\sl (ii)} a diminishing sequence of $\beta_k = \left(\left(\left\lceil \frac{k}{500}\right\rceil - 1\right)\times 500 + 1\right)^{-0.6}$ at each iteration $k$. We considered all 9 aforementioned sample variances from setting \textit{(a)}, i.e., $(\epsilon_g,\epsilon_c) \in \{10^{-8},10^{-4},10^{-2}\}\times \{10^{-8},10^{-4},10^{-2}\}$ and $\epsilon_J = \sqrt{\epsilon_c}$, for the constant sequence $\{\beta_k\}$. 
Meanwhile, for the diminishing sequence $\{\beta_k\}$, we focused on variance combinations from setting \textit{(b)}, as well as the ones with the highest noise level of $\epsilon_c = \epsilon_J^2$ from setting \textit{(a)}, leading to 6 different combinations such as $(\epsilon_g,\epsilon_c,\epsilon_J) \in \{(10^{-2},0.25,0.5),(0.25,10^{-2},10^{-1}),(0.25,0.25,0.5)\} \cup \{10^{-8},10^{-4},10^{-2}\}\times \{10^{-2}\}\times \{10^{-1}\}$. We conducted 5 independent runs of Algorithm~\ref{alg.main} over 44 problems and 15 different sample variances (9 for constant $\{\beta_k\}$ and 6 for diminishing $\{\beta_k\}$), which resulted in 3300 instances. For each instance, we first ran Algorithm~\ref{alg.main} with a budget of 5000 iterations and recorded the CPU time has been used. Then we provided the same CPU time budget for each run of the stochastic subgradient method with 7 different choices of $\tau \in \{10^{-6}, 10^{-5}, 10^{-4}, 10^{-3}, 10^{-2}, 10^{-1}, 10^{0}\}$. We further compared the best iterate found by Algorithm~\ref{alg.main} and the stochastic subgradient method and reported their infeasibility errors and stationarity errors in Figures~\ref{fig:first_experiment}--\ref{fig:diminishing}.
However, it's worth mentioning that the best iterate for the stochastic subgradient method was selected among different values of $\tau$, effectively providing the stochastic subgradient algorithm with a CPU time budget of as many as 7 times compared to what Algorithm~\ref{alg.main} used.
The criterion used to determine the best results is described in~\eqref{eq.best_iterate}.

To ensure a fair comparison, we provided the same sample variances $(\epsilon_g,\epsilon_c,\epsilon_J)$ and step size parameter sequences $\{\beta_k\}$ for both Algorithm~\ref{alg.main} and the stochastic subgradient method. For each test problem, both algorithms were using the same initial iterate suggested by the CUTEst collection~\cite{gould2015cutest} and the same Lipschitz constant parameters $(L,\Gamma)$, which were inputs estimated before running any instances and stayed unchanged in all runs of both algorithms. The other input parameters of Algorithm~\ref{alg.main} were chosen as follows: $\bar\tau_0 = 0.1$, $\bar\xi_0 = 1$, $\eta = 0.5$, $\sigma = 0.1$, $\epsilon_\tau = 0.01$, $\epsilon_\xi = 0.01$, $\theta=10$, and $H_k=I$ for all $k$. At each iteration $k$, after achieving $\bar\alpha_k^{\min}$ by~\eqref{eq.stepsizes} and if $\bar{d}_k\neq 0$, we computed an adaptive step size $\bar\alpha_k$ by setting
\bequationn
\bar\alpha_k = \max\left\{(1.1)^t\cdot\bar\alpha_k^{\min} | t\in\{0\}\cup\mathbb{N},\ \varphi_k((1.1)^t\cdot\bar\alpha_k^{\min}) \leq 0, \text{ and } (1.1)^t\cdot\bar\alpha_k^{\min} \leq \bar\alpha_k^{\min} + \theta\beta_k \right\},
\eequationn
and it is guaranteed by~\eqref{eq.stepsizes} and Lemma~\ref{lem.stepsize} that such $\bar\alpha_k\in[\bar\alpha_k^{\min},\bar\alpha_k^{\max}]$. On the other hand, the stochastic subgradient method always took non-adaptive step sizes $\alpha_k = \frac{\beta_k\tau}{\tau L+\Gamma}$ at all iterations $k$. Given a sequence of iterates $\mathcal{S}$, which was generated by an instance of Algorithm~\ref{alg.main} or the stochastic subgradient method, we selected the best iterate, $x_{\texttt{best}}(\mathcal{S})$, by setting
\bequation\label{eq.best_iterate}
x_{\texttt{best}}(\mathcal{S})\leftarrow \bcases \argmin_{x\in\mathcal{S}}\left\{\left\|\bbmatrix\nabla f(x) + \nabla c(x)y_{\texttt{LS}(x)} \\ c(x) \ebmatrix\right\|_{\infty}\right\} & \text{if }\min\left\{\|c(x)\|_{\infty}:x\in\mathcal{S}\right\} \leq 10^{-4}, \\
\argmin_{x\in\mathcal{S}}\left\{\|c(x)\|_{\infty}\right\} & \text{otherwise,}\ecases
\eequation
where $y_{\texttt{LS}}(x)= \argmin_{y\in\mathbb{R}^m}\|\nabla f(x) + \nabla c(x)y\|$ is a least-squares multiplier. In Figures~\ref{fig:first_experiment}--\ref{fig:diminishing}, we present infeasibility errors and stationarity errors of the best iterates ever found by Algorithm~\ref{alg.main} and the stochastic subgradient method (over seven different values of $\tau \in \{10^{-6}, 10^{-5}, 10^{-4}, 10^{-3}, 10^{-2}, 10^{-1}, 10^{0}\}$), i.e., $\|c(x_{\texttt{best}}(\mathcal{S}))\|_{\infty}$ and $\|\nabla f(x) + \nabla c(x_{\texttt{best}}(\mathcal{S}))y_{\texttt{LS}(x_{\texttt{best}}(\mathcal{S}))}\|_{\infty}$ with $\mathcal{S}$ representing the iterate sequence generated by either Algorithm~\ref{alg.main} or the stochastic subgradient method with seven different choices of merit parameters.

\begin{figure}[ht!]
    \centering
    \includegraphics[width=0.8\textwidth]{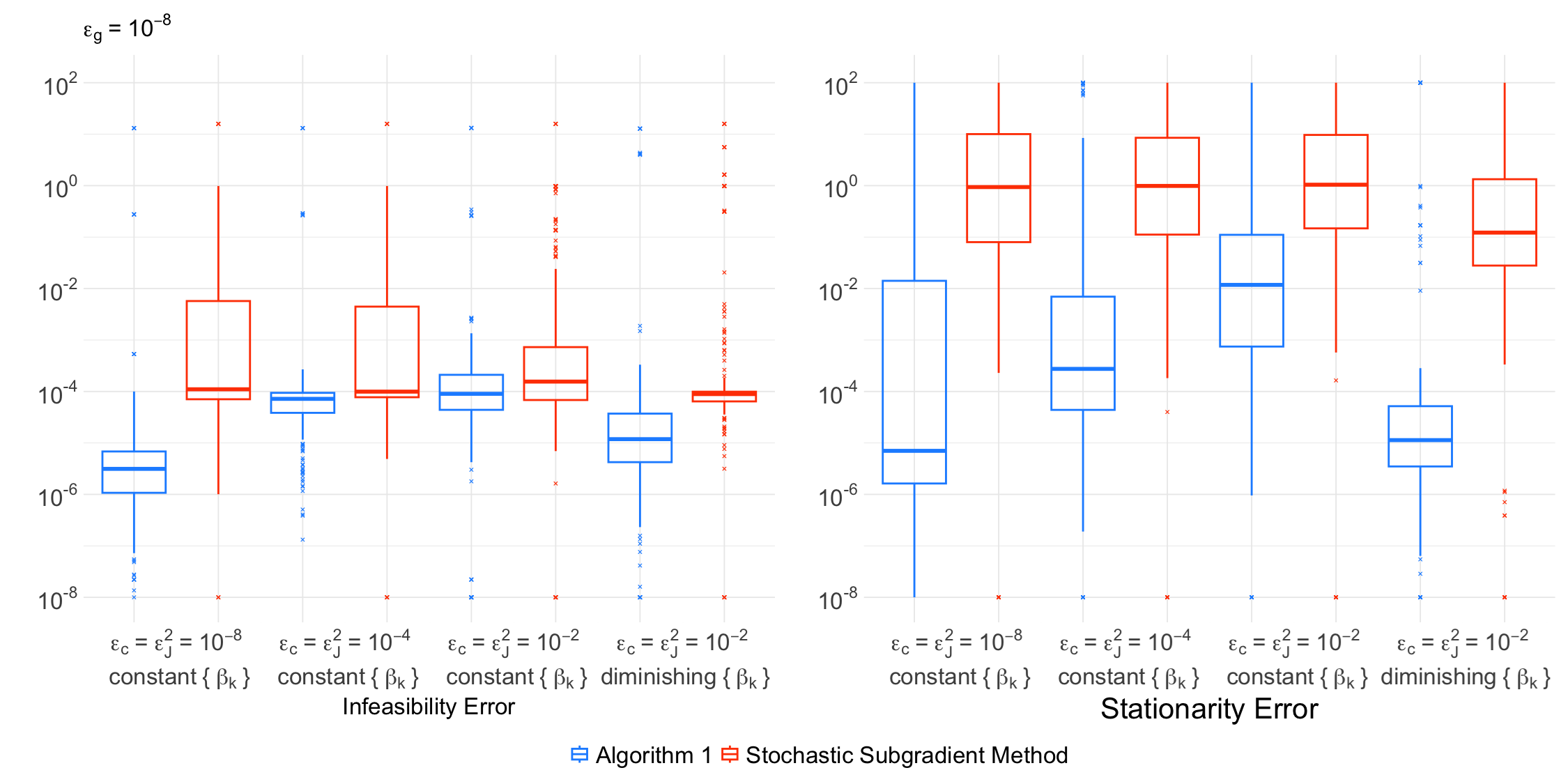}
    \includegraphics[width=0.8\textwidth]{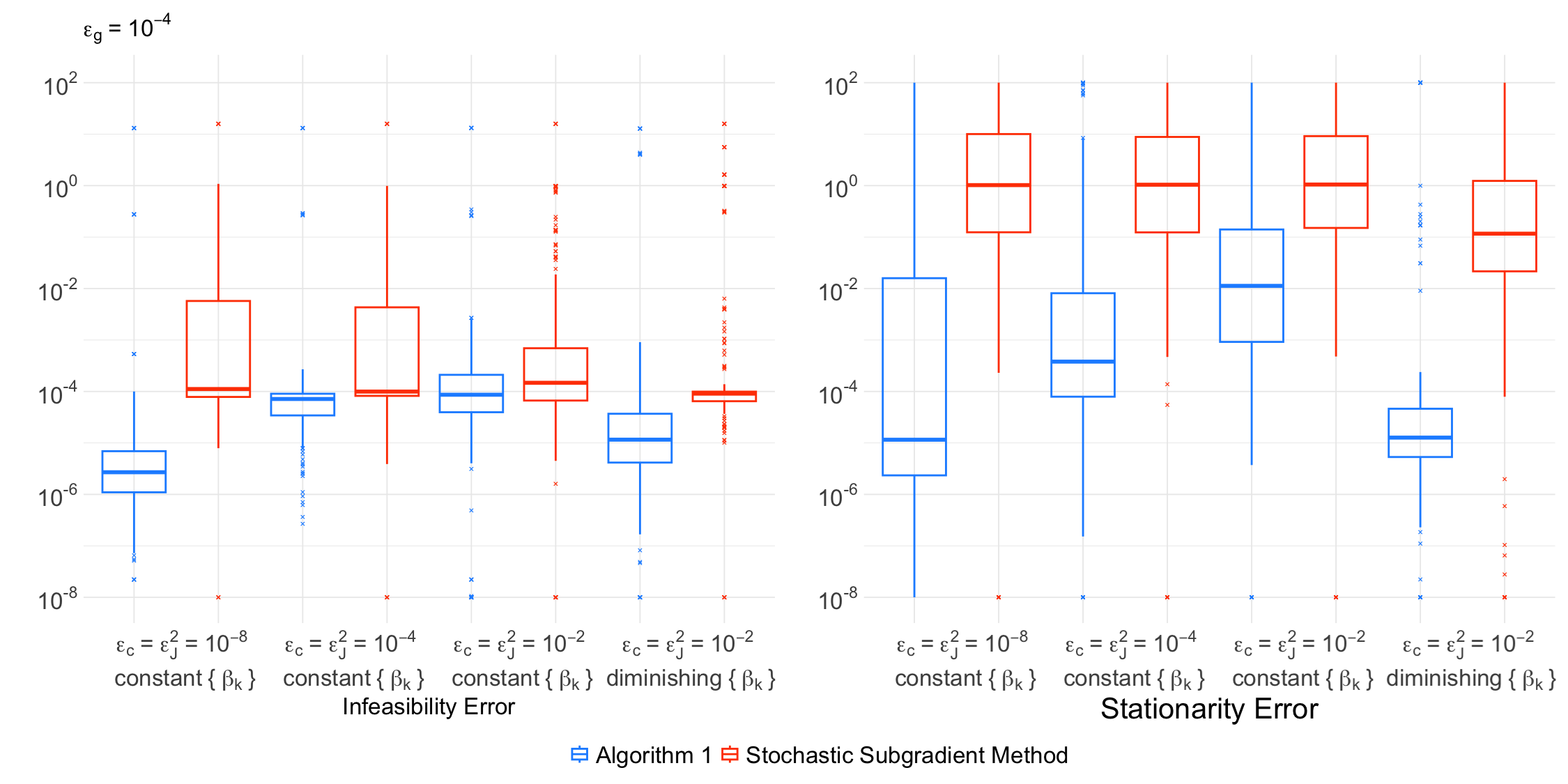}
    \includegraphics[width=0.8\textwidth]{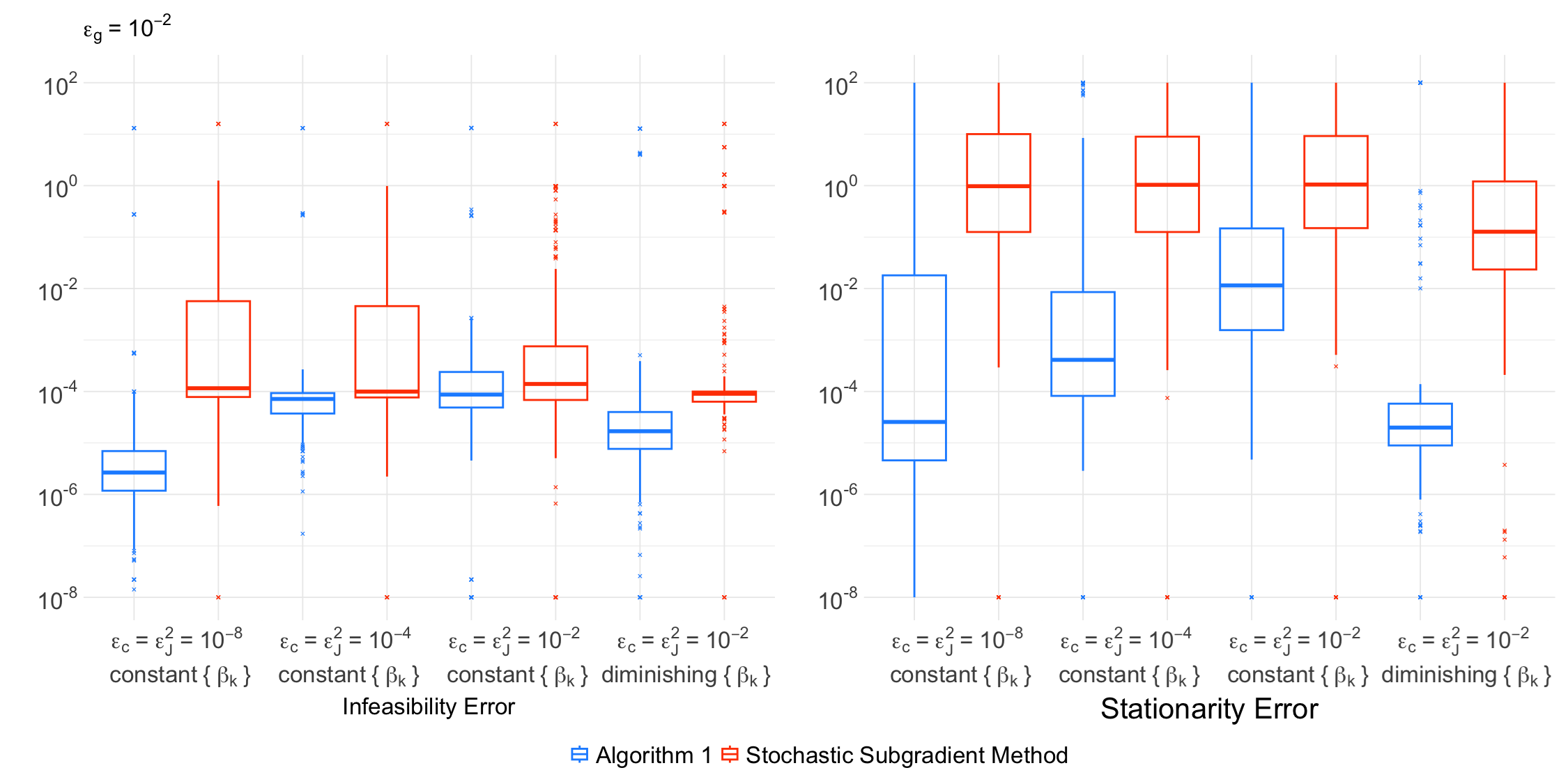}
    \caption{\inrevise{Box plots of infeasibility errors (the left column) and stationarity errors (the right column) on a total of 44 CUTEst problems with $\epsilon_g \in \{10^{-8},10^{-4},10^{-2}\}$ (from top to bottom).}}
    \label{fig:first_experiment}
\end{figure}
\begin{figure}[ht!]
    \centering
    \includegraphics[width=0.8\textwidth]{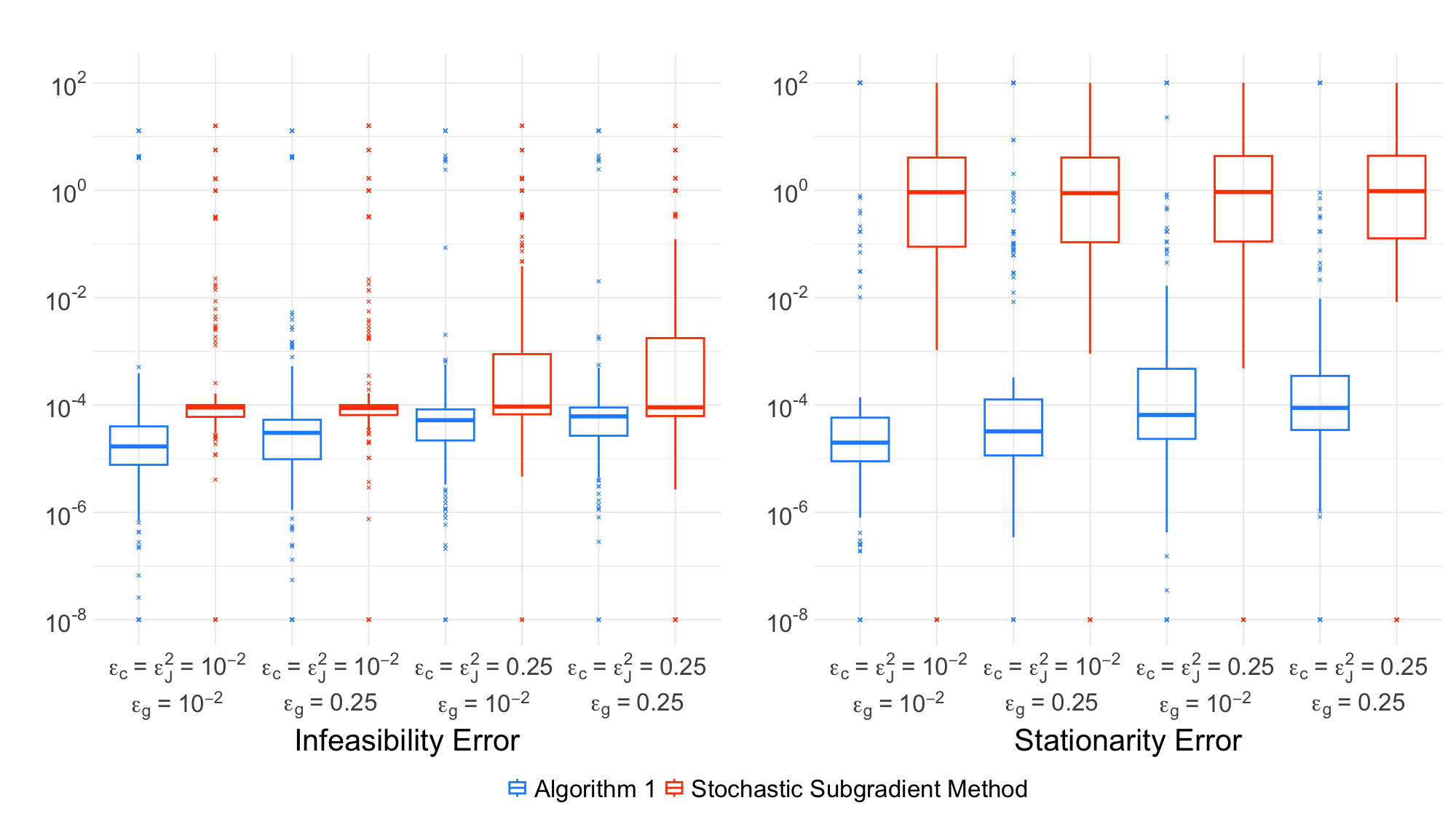}
    \caption{\inrevise{Box plots of infeasibility errors (the left column) and stationarity errors (the right column) on a total of 44 CUTEst problems with diminishing $\{\beta_k\}$ sequences under large-noise settings.}}
    \label{fig:diminishing}
\end{figure}

Figures~\ref{fig:first_experiment}--\ref{fig:diminishing} report box plots that include information of infeasibility errors and stationarity errors of the best iterates ever found by Algorithm~\ref{alg.main} and the stochastic subgradient method on 44 CUTEst problems~\cite{gould2015cutest}.} From these box plots, we observe that although provided with less CPU time, Algorithm~\ref{alg.main} constantly outperforms the stochastic subgradient method in terms of both infeasibility error and stationarity error over all options of variance level $(\epsilon_g,\epsilon_c,\epsilon_J)$ and step size parameter sequence $\{\beta_k\}$. In particular, Algorithm~\ref{alg.main} is usually able to identify the best iterate with infeasibility error smaller than $10^{-4}$ and stationarity error smaller than $10^{-2}$, while the stochastic subgradient method suffers from improving stationarity errors. 
\begin{wrapfigure}{r}{0.38\textwidth}
  \vspace{-15pt}
  
  \begin{center}
    \includegraphics[width=0.36\textwidth]{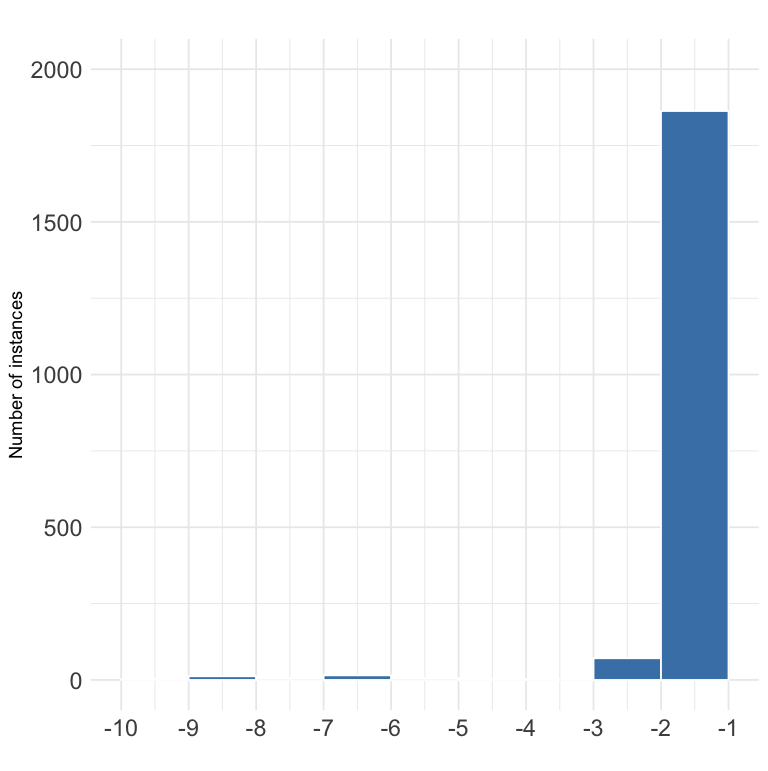}
  \end{center}
  \vspace{-10pt}
  
  \caption{\inrevise{Histogram of the final values of \(\log_{10}(\bar{\mathcal{T}}')\).}}
  \label{fig.histogram_tau}
  \vspace{-10pt}
\end{wrapfigure}
In addition, by comparing results over different sample variances, we notice that Algorithm~\ref{alg.main} can achieve better iterates, in terms of both infeasibility errors and stationarity errors, when the variance level $(\epsilon_g,\epsilon_c,\epsilon_J)$ is smaller, which matches the result of \textbf{Case (i)} of Corollary~\ref{cor.asymptotic} that the upper bound of convergence neighborhood decays as the variances diminish. Last but not least, we know from these box plots that both Algorithm~\ref{alg.main} and the stochastic subgradient method may benefit from utilizing a diminishing $\{\beta_k\}$ sequence, which controls step sizes and variances of stochastic objective gradient, constraint function, and constraint jacobian estimates. 
This result has also been described by Corollary~\ref{cor.asymptotic} that a constant $\{\beta_k\}$ sequence can make Algorithm~\ref{alg.main} converge to a neighborhood of stationarity in expectation (see \textbf{Case (i)} of Corollary~\ref{cor.asymptotic}) while a diminishing sequence of $\{\beta_k\}$ would enhance the performance of Algorithm~\ref{alg.main} to exact convergence in expectation (\textbf{Case (ii)} of Corollary~\ref{cor.asymptotic}).

\inrevise{
\paragraph{On the practical validity of Assumption~\ref{ass.event_asymptotic}.} \ \  Using the instances with constant $\{\beta_k\}$ sequences that have been reported in Figure~\ref{fig:first_experiment}, we further numerically validate our Assumption~\ref{ass.event_asymptotic}. 
Specifically, we set \((\bar{\tau}_{\min}, k_{\max}, f_{\max}) = (10^{-3}, 4000, 10^6)\) in Assumption~\ref{ass.event_asymptotic} and examined whether each condition in Assumption~\ref{ass.event_asymptotic} was satisfied across a total of 1980 numerical instances, summarizing 44 problems, 5 random seeds, and 9 noise levels of $(\epsilon_g,\epsilon_c,\epsilon_J)$, with constant \(\{\beta_k\}\) on the CUTEst benchmark problems, whose box plots have been reported in Figure~\ref{fig:first_experiment}. The validity of the individual conditions in Assumption~\ref{ass.event_asymptotic} across these 1980 instances is as follows: $(i)$ \(f(x_{k_{\max}})\leq f_{\max}\): all 1980 intances satisfied this inequality with $k_{\max}=4000$ and $f_{\max} = 10^6$; $(ii)$ \(\mathcal{T}_k^{\text{trial}} \geq \bar{\mathcal{T}}_k\): with \(k_{\max}=4000\) and an iteration budget of 5000, we found that more than \(95\%\) of the iterations with $k\geq k_{\max}$ satisfied this condition across the 1980 instances;
$(iii)$ \(\bar{\mathcal{T}}_k = \bar{\mathcal{T}}'\): with \(k_{\max}=4000\) and an iteration budget of 5000, nearly \(98\%\) of the 1980 instances satisfied this condition for all $k\geq k_{\max}$; $(iv)$ \(\bar{\mathcal{T}}' \geq \bar{\tau}_{\min}\): with $\bar\tau_{\min} = 10^{-3}$ and an iteration budget of 5000, there were over \(97\%\) of the 1980 instances having their final stochastic merit parameter value satisfy this condition. We also provide a histogram of the final \(\log_{10}(\bar{\mathcal{T}}')\) over the 1980 instances (see Figure~\ref{fig.histogram_tau}); and $(v)$ \(\bar{\Xi}_k = \bar{\Xi}' > 0\): with \(k_{\max}=4000\) and an iteration budget of 5000,  approximately \(99.5\%\) of the 1980 instances had a stochatic ratio parameter at iteration $k_{\max}$ that was equal to its final value. In summary, these observations provide concrete support for the validity of Assumption~\ref{ass.event_asymptotic}.
}

\begin{bulkrevise}
\paragraph{Data profiles for Algorithm~\ref{alg.main}.} \ \  
In addition to the box plots shown in Figures~\ref{fig:first_experiment}--\ref{fig:diminishing}, we present data profiles of the true KKT error per iteration for eight problems from the CUTEst collection~\cite{gould2015cutest}: \texttt{BT5}, \texttt{BT12}, \texttt{BYRDSPHR}, \texttt{GENHS28}, \texttt{HS27}, \texttt{HS77}, \texttt{MWRIGHT}, and \texttt{ORTHREGB}; see Figure~\ref{fig:8plots}. For each problem, we ran 20 instances corresponding to 10 different random seeds and two noise settings, both employing diminishing $\{\beta_k\}$ sequences defined as $\beta_k=\left(\left(\left\lceil \frac{k}{500}\right\rceil-1\right)\times 500+1\right)^{-0.6}$ at each iteration $k$. Figure~\ref{fig:8plots} reports data profiles (``True KKT Error vs. Iteration'') over these eight problems. In each plot, the curves show mean values, while the shaded regions indicate 95\% confidence intervals. The red curves correspond to the larger noise level \((\epsilon_g,\epsilon_c,\epsilon_J)=(0.25,0.25,0.5)\), and the cyan curves use a smaller noise level \((\epsilon_g,\epsilon_c,\epsilon_J)=(10^{-2},10^{-2},10^{-1})\). The results indicate that $(i)$ instances with smaller noise consistently achieve lower true KKT errors than their larger-noise counterparts, and $(ii)$ the true KKT error generally decreases as the algorithm progresses, with fluctuations attributable to stochastic noise in the estimates.

\begin{figure}[h]
\centering

\begin{subfigure}{0.48\textwidth}
  \centering
  \includegraphics[width=\textwidth]{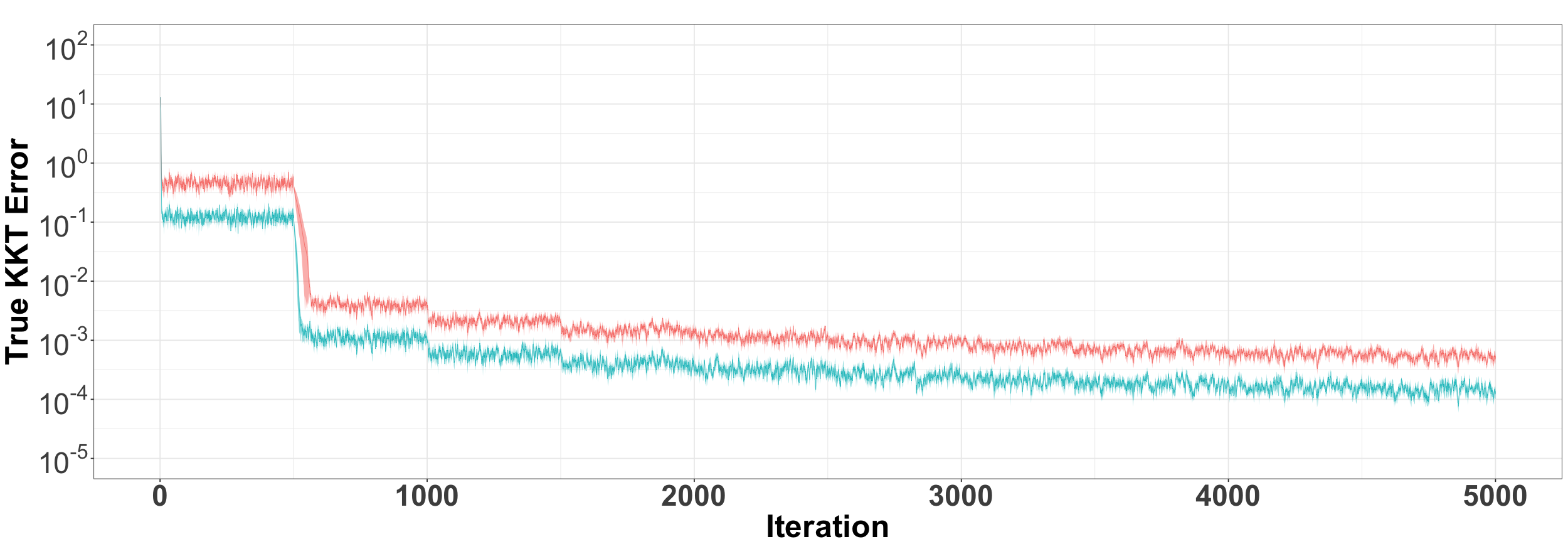}
  \caption{\texttt{BT5}}
\end{subfigure}
\hfill
\begin{subfigure}{0.48\textwidth}
  \centering
  \includegraphics[width=\textwidth]{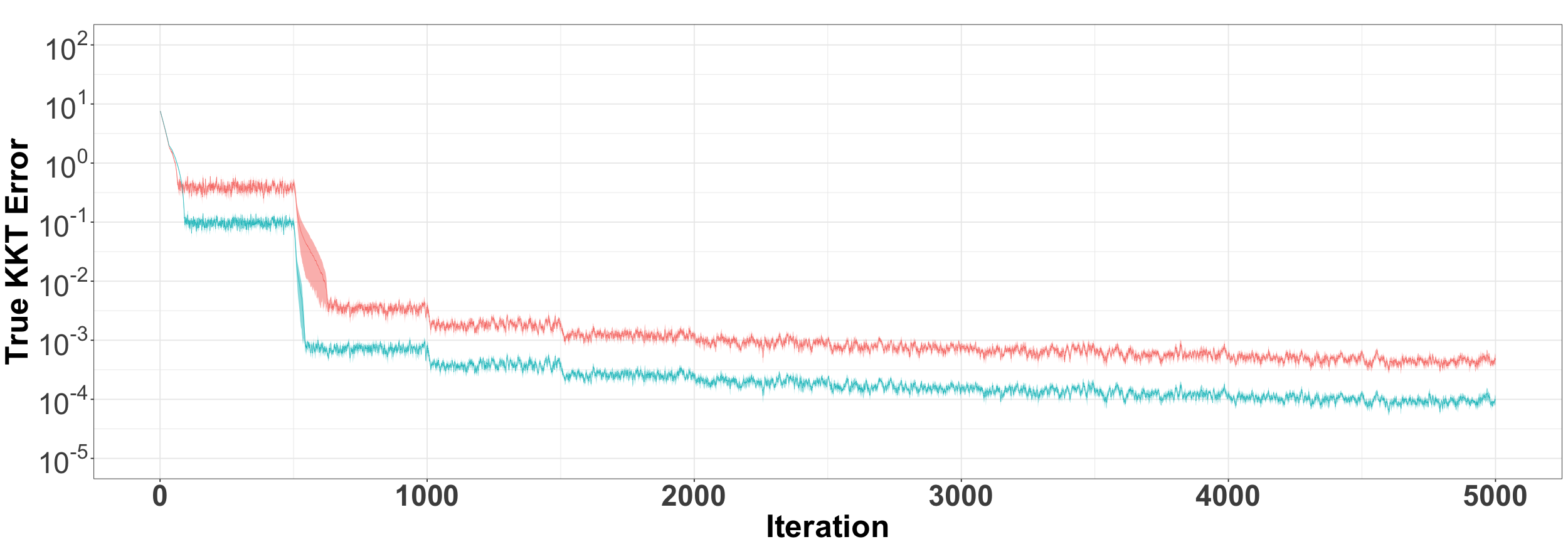}
  \caption{\texttt{BT12}}
\end{subfigure}

\vspace{-0.1cm}

\begin{subfigure}{0.48\textwidth}
  \centering
  \includegraphics[width=\textwidth]{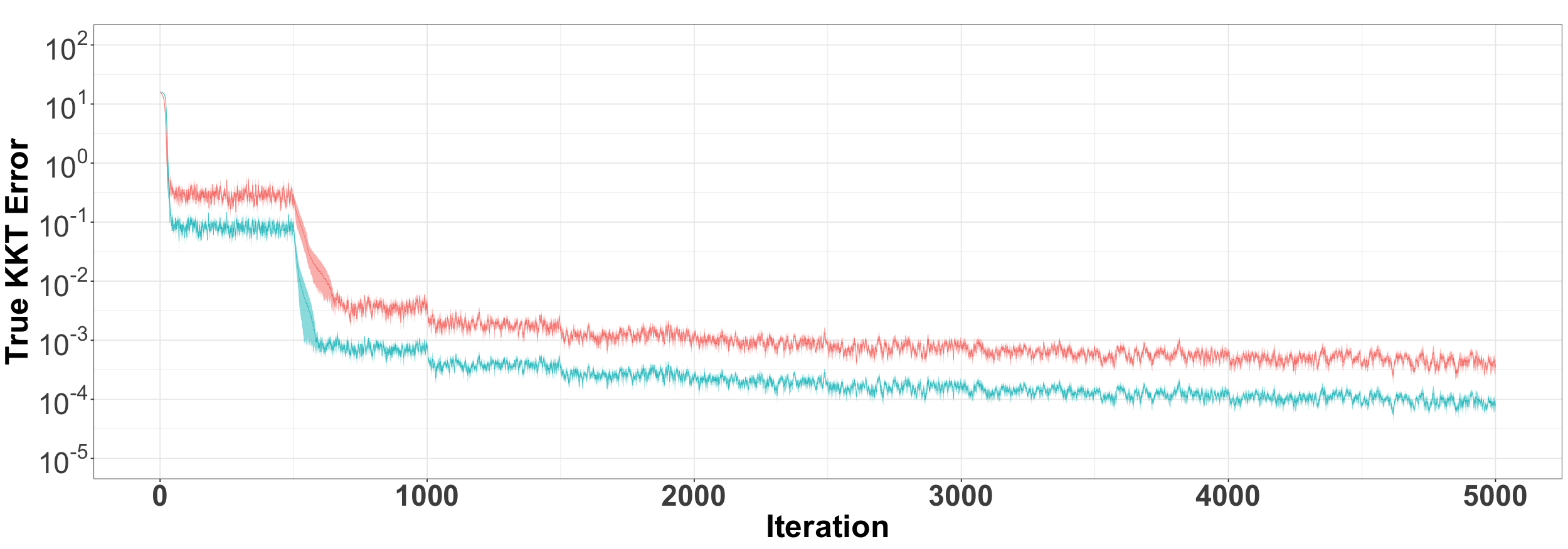}
  \caption{\texttt{BYRDSPHR}}
\end{subfigure}
\hfill
\begin{subfigure}{0.48\textwidth}
  \centering
  \includegraphics[width=\textwidth]{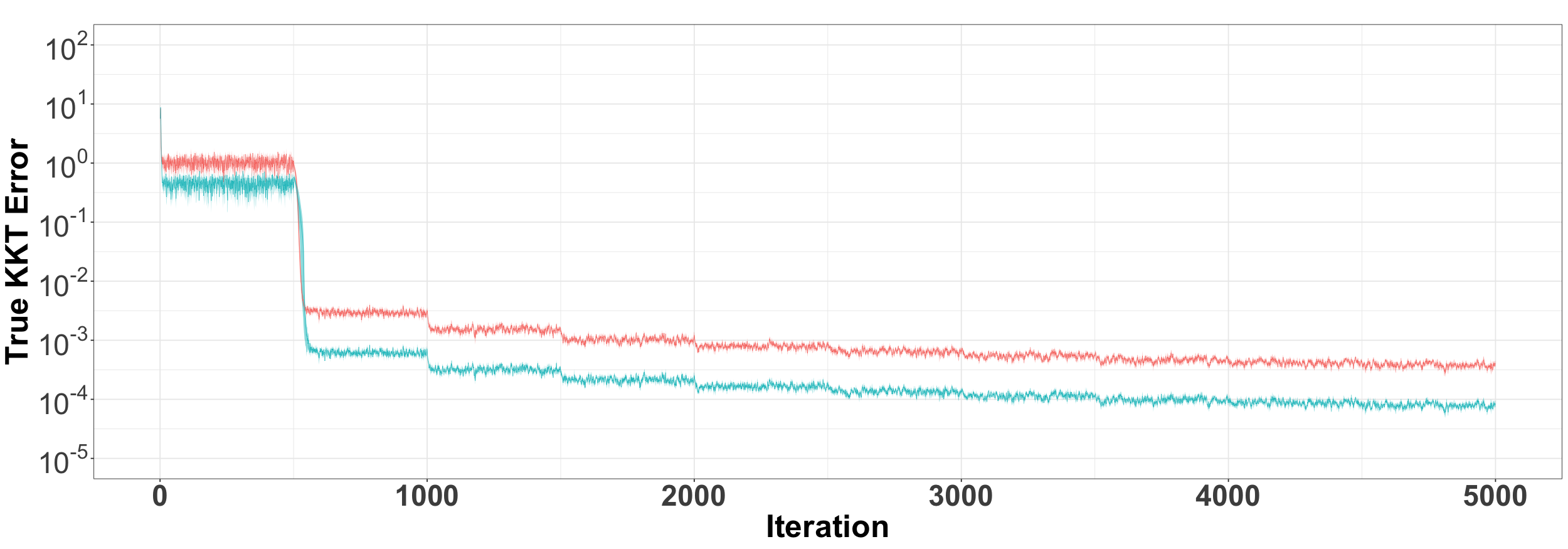}
  \caption{\texttt{GENHS28}}
\end{subfigure}

\vspace{-0.1cm}

\begin{subfigure}{0.48\textwidth}
  \centering
  \includegraphics[width=\textwidth]{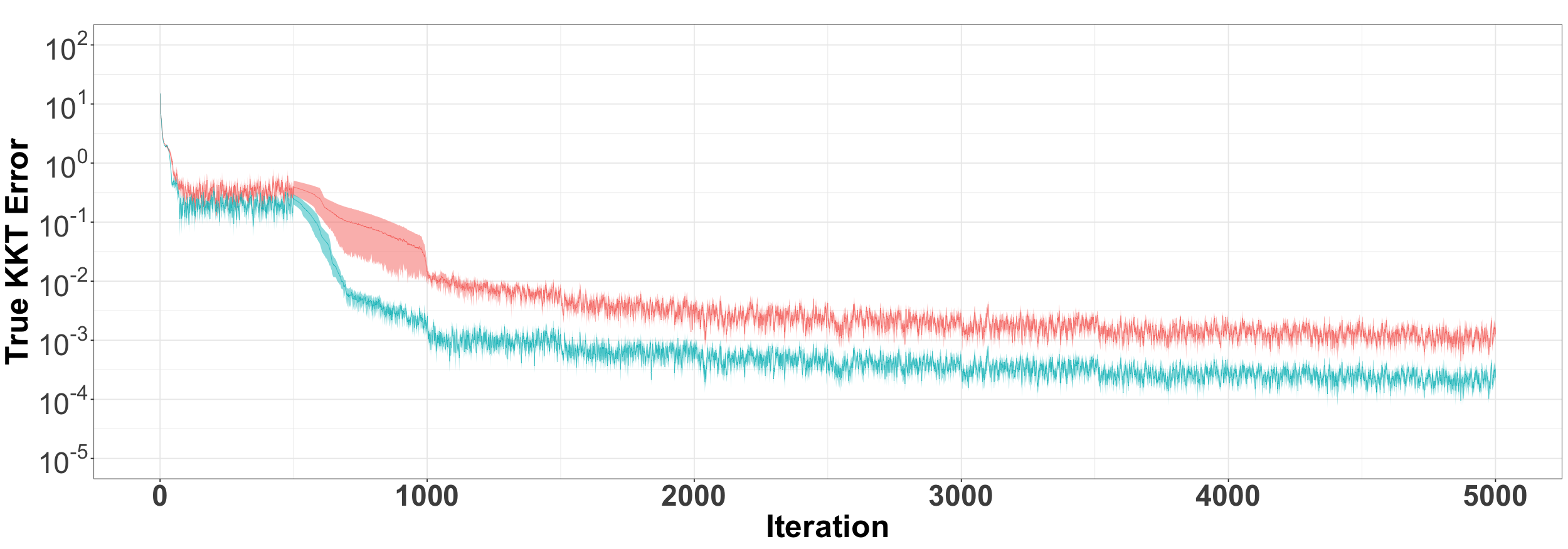}
  \caption{\texttt{HS27}}
\end{subfigure}
\hfill
\begin{subfigure}{0.48\textwidth}
  \centering
  \includegraphics[width=\textwidth]{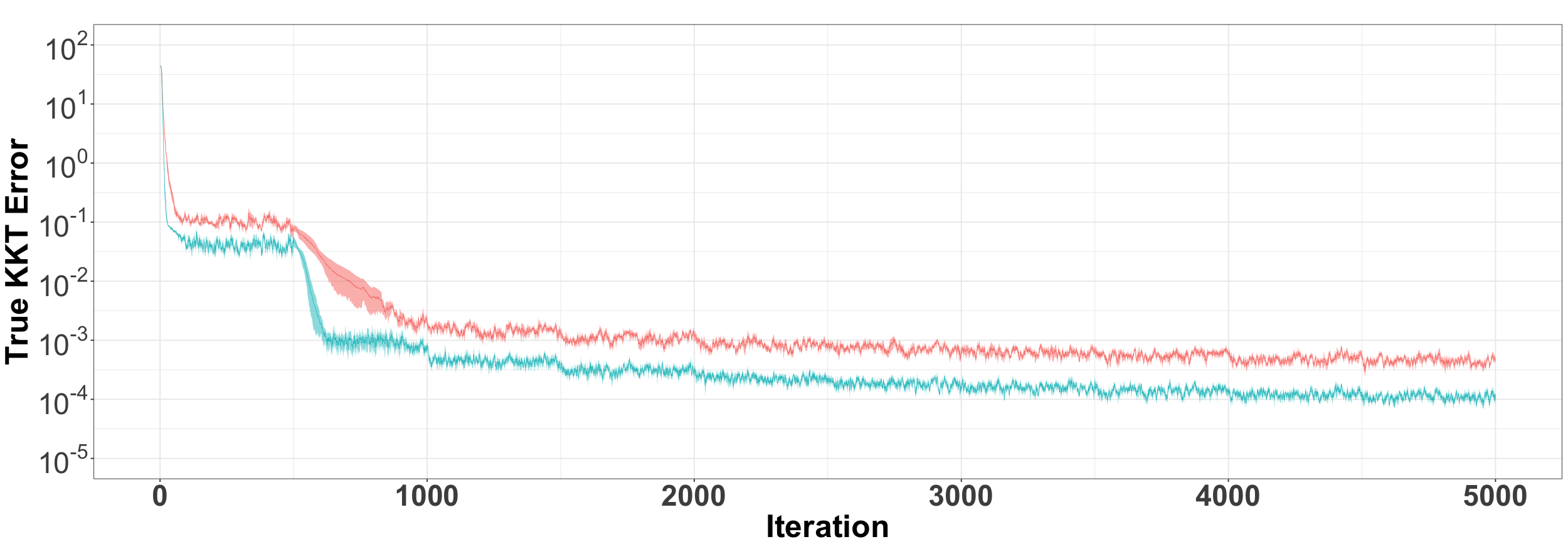}
  \caption{\texttt{HS77}}
\end{subfigure}

\vspace{-0.1cm}

\begin{subfigure}{0.48\textwidth}
  \centering
  \includegraphics[width=\textwidth]{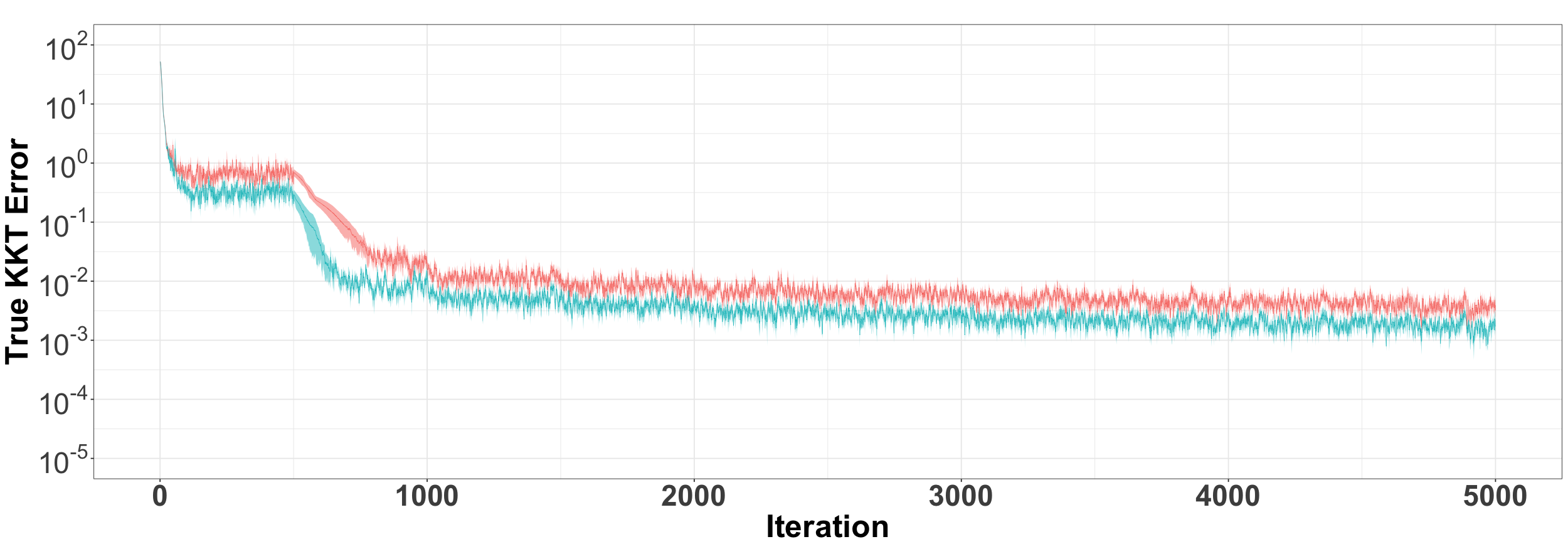}
  \caption{\texttt{MWRIGHT}}
\end{subfigure}
\hfill
\begin{subfigure}{0.48\textwidth}
  \centering
  \includegraphics[width=\textwidth]{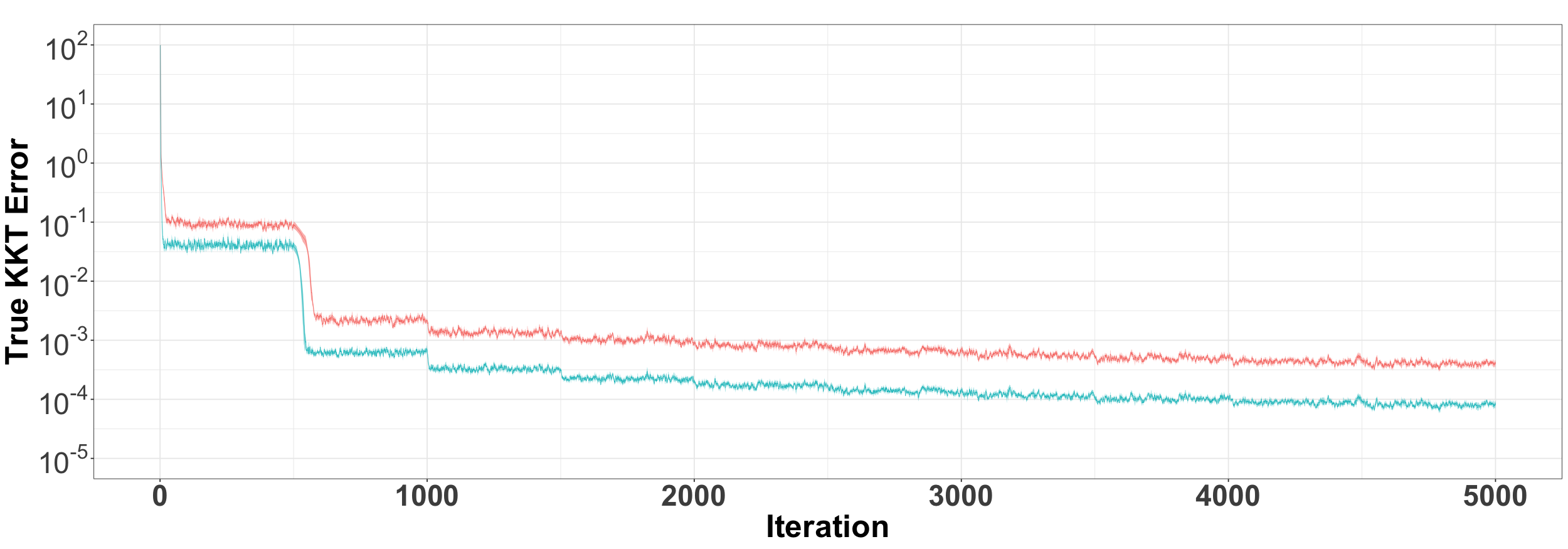}
  \caption{\texttt{ORTHREGB}}
\end{subfigure}

\caption{\inrevise{Data profiles, including mean values and 95\% confidence intervals, of true KKT errors vs. iterations on 8 problems from the CUTEst collection~\cite{gould2015cutest}. Red and cyan curves represent noise levels $(\epsilon_g,\epsilon_c,\epsilon_J) = (0.25,0.25,0.5)$ and $(\epsilon_g,\epsilon_c,\epsilon_J) = (10^{-2},10^{-2},10^{-1})$, respectively.}}
\label{fig:8plots}
\end{figure}
    
\end{bulkrevise}

\subsection{Experiments on LIBSVM~\cite{chang2011libsvm} \inrevise{Problems}}\label{sec.tstom}
In the second set of experiments, we tested the performance of our proposed Algorithm~\ref{alg.main} against with TStoM, a stochastic momentum-based optimization algorithm designed by~\cite{CuiWangXiao24}. We especially considered the following constrained binary classification problem:
\begin{equation}\label{eq.libsvm}
\mathop{\text{min}}\limits_{x\in\mathbb{R}^n} f(x)=\frac{1}{N}\sum\limits_{i=1}^N\text{log}(1+e^{-y_i(X_i^Tx)}) \text{ s.t. } \frac{1}{K}\sum\limits_{k=1}^KA_{1,k}x=\frac{1}{K}\sum\limits_{k=1}^Ka_{1,k}, ||x||_2^2=a_2,
\end{equation}
where $X_i\in\mathbb{R}^n$ and $y_i\in\{-1,1\}$ are data representing the feature vector and the label, respectively, for each $i\in[N]$; $K = 1000$, $a_2 = 1$, and $(A_{1,k},a_{1,k})\in\mathbb{R}^{10\times n}\times\mathbb{R}^{10}$ are random matrices and random vectors for each $k\in[K]$. \inrevise{Each random matrix $A_{1,k}$ was generated based on a fixed matrix $A_1\in\mathbb{R}^{10\times n}$ that $A_{1,k}^{(i,j)}$, the element at $i$th row and $j$th column of $A_{1,k}$, satisfied $A_{1,k}^{(i,j)}\sim\mathcal{N}\left(A_1^{(i,j)},\frac{10^{-3}}{n}\right)$ with $A_1^{(i,j)}\sim\mathcal{N}(1,10^4)$ for any $(i,j,k)\in[10]\times[n]\times[1000]$, i.e., each element of the fixed matrix $A_1$ was generated by the distribution of $\mathcal{N}(1,10^4)$ and each element of $A_{1,k}$ followed a normal distribution with the mean value as the corresponding element of $A_1$ and the variance as $\frac{10^{-3}}{n}$. Similarly, we first generated a fixed vector $a_1\in\mathbb{R}^{10}$, whose elements all followed the distribution of $\mathcal{N}(1,10^4)$, then each random vector $a_{1,k}\in\mathbb{R}^{10}$ was sampled by using $a_{1,k}\sim\mathcal{N}(a_1,10^{-3}I)$ for any $k\in[1000]$.} We tested problem~\eqref{eq.libsvm} over five datasets from the LIBSVM collection~\cite{chang2011libsvm} with five independent runs, while the dataset information is listed in Table~\ref{table:libsvm_info}.

\begin{table}[h]
\centering
\caption{Binary classification datasets details. For more information, see \cite{chang2011libsvm}}
\begin{tabular}{|c|c|c|}
\hline
\textbf{Dataset} & \textbf{Dimension ($n$)} & \textbf{$\#$ of Data ($N$)} \\ \hline
\texttt{a9a}             & 123                       & 32,561                    \\ \hline
\texttt{ionosphere}      & 34                        & 351                       \\ \hline
\texttt{mushrooms}       & 112                       & 5,500                     \\ \hline
\texttt{phishing}         & 68                        & 11,055                    \\ \hline
\texttt{sonar}           & 60                        & 208                       \\ \hline
\end{tabular}
\label{table:libsvm_info}
\end{table}

At the beginning of each run of Algorithm~\ref{alg.main} and TStoM~\cite{CuiWangXiao24}, we selected the objective batch size $b_1$ and the constraint batch size $b_2$, where $(b_1, b_2) \in \{16, 128\}\times\{16, 128\}$. For all instances, we initialized \( x_1 \in \mathbb{R}^n \) by sampling from the standard Gaussian distribution and normalizing it to ensure \( \|x_1\| = 0.1 \). Based on the structure of problem~\eqref{eq.libsvm}, we set $\Gamma=2$ and estimated $L$ using differences of gradients near $x_1$. With exceptions of Lipschitz constants $(L,\Gamma)$ described above and step size parameters $\beta_k=1$ for all iterations $k\in\mathbb{N}$, all other parameter selections of Algorithm~\ref{alg.main} were the same with experiments conducted in Section~\ref{sec.ssm}.

\begin{bulkrevise}
For TStoM, we set $T=50$, $\gamma=0.3$, and $V=2\times10^7$ for Phase I iterations and $\alpha_k=0.6$, $\tau_k=0.3$, $\beta_k=15$, $\eta_k=5\times 10^{-8}$, $\rho_k=0.4$ for Phase II iterations, where these parameter values are based on the recommendations from the authors of~\cite{CuiWangXiao24} and were further tuned for our experiments. We first ran each instance of Algorithm~\ref{alg.main} with 10000 iterations, and then provided the same amount of CPU time that had been used by Algorithm~\ref{alg.main} to all TStoM~\cite{CuiWangXiao24} instances. After finishing all runs of both algorithms, following the same policies in Section~\ref{sec.ssm}, we selected the best iterates ever found by Algorithm~\ref{alg.main} and TStoM~\cite{CuiWangXiao24} and reported their infeasibility errors and stationarity errors in Table~\ref{table:libsvm_result}.

\begin{table}[h]
\centering
\footnotesize
\caption{\inrevise{Mean values of infeasibility and stationarity errors of Algorithm~\ref{alg.main} and TStoM~\cite{CuiWangXiao24}}}
\begin{tabular}{|c|c|c|c|c|c|c|c|c|}
\hline
{} & \multicolumn{4}{c|}{$b_1=16$, $b_2=16$} & \multicolumn{4}{c|}{$b_1=16$, $b_2=128$}\\ \hline
{} & \multicolumn{2}{c|}{Infeasibility Error} & \multicolumn{2}{c|}{Stationarity Error} & \multicolumn{2}{c|}{Infeasibility Error} & \multicolumn{2}{c|}{Stationarity Error}\\ \hline
Dataset & Algorithm 1 & TStoM & Algorithm 1 & TStoM & Algorithm 1 & TStoM & Algorithm 1 & TStoM\\ \hline
\texttt{a9a} &
\inrevise{$2.00\text{e-}04$} & \inrevise{$8.63\text{e-}01$} &
\inrevise{$4.96\text{e-}03$} & \inrevise{$2.42\text{e-}01$} &
\inrevise{$1.87\text{e-}04$} & \inrevise{$8.74\text{e-}01$} &
\inrevise{$4.22\text{e-}03$} & \inrevise{$2.43\text{e-}01$} \\ \hline
\texttt{ionosphere} &
\inrevise{$8.70\text{e-}04$} & \inrevise{$4.34\text{e-}01$} &
\inrevise{$8.74\text{e-}03$} & \inrevise{$1.97\text{e-}01$} &
\inrevise{$7.87\text{e-}04$} & \inrevise{$4.39\text{e-}01$} &
\inrevise{$7.51\text{e-}03$} & \inrevise{$1.96\text{e-}01$} \\ \hline
\texttt{mushrooms} &
\inrevise{$1.13\text{e-}04$} & \inrevise{$8.36\text{e-}01$} &
\inrevise{$9.51\text{e-}04$} & \inrevise{$7.18\text{e-}01$} &
\inrevise{$3.27\text{e-}05$} & \inrevise{$8.62\text{e-}01$} &
\inrevise{$2.42\text{e-}04$} & \inrevise{$7.07\text{e-}01$} \\ \hline
\texttt{phishing} &
\inrevise{$1.27\text{e-}04$} & \inrevise{$8.15\text{e-}01$} &
\inrevise{$9.08\text{e-}04$} & \inrevise{$5.59\text{e-}02$} &
\inrevise{$3.93\text{e-}05$} & \inrevise{$8.32\text{e-}01$} &
\inrevise{$3.06\text{e-}04$} & \inrevise{$5.58\text{e-}02$} \\ \hline
\texttt{sonar} &
\inrevise{$3.69\text{e-}04$} & \inrevise{$8.18\text{e-}01$} &
\inrevise{$7.92\text{e-}03$} & \inrevise{$9.06\text{e-}02$} &
\inrevise{$3.60\text{e-}04$} & \inrevise{$8.31\text{e-}01$} &
\inrevise{$5.71\text{e-}03$} & \inrevise{$8.80\text{e-}02$} \\ \hline

{} & \multicolumn{4}{c|}{$b_1=128$, $b_2=16$} & \multicolumn{4}{c|}{$b_1=128$, $b_2=128$}\\ \hline
{} & \multicolumn{2}{c|}{Infeasibility Error} & \multicolumn{2}{c|}{Stationarity Error} & \multicolumn{2}{c|}{Infeasibility Error} & \multicolumn{2}{c|}{Stationarity Error}\\ \hline
Dataset & Algorithm 1 & TStoM & Algorithm 1 & TStoM & Algorithm 1 & TStoM & Algorithm 1 & TStoM\\ \hline
\texttt{a9a} &
\inrevise{$4.72\text{e-}05$} & \inrevise{$8.70\text{e-}01$} &
\inrevise{$4.63\text{e-}04$} & \inrevise{$2.42\text{e-}01$} &
\inrevise{$2.80\text{e-}05$} & \inrevise{$8.74\text{e-}01$} &
\inrevise{$7.40\text{e-}04$} & \inrevise{$2.42\text{e-}01$} \\ \hline
\texttt{ionosphere} &
\inrevise{$2.44\text{e-}04$} & \inrevise{$4.36\text{e-}01$} &
\inrevise{$3.34\text{e-}03$} & \inrevise{$1.97\text{e-}01$} &
\inrevise{$1.25\text{e-}04$} & \inrevise{$4.38\text{e-}01$} &
\inrevise{$2.22\text{e-}03$} & \inrevise{$1.96\text{e-}01$} \\ \hline
\texttt{mushrooms} &
\inrevise{$1.11\text{e-}04$} & \inrevise{$8.47\text{e-}01$} &
\inrevise{$6.82\text{e-}04$} & \inrevise{$7.14\text{e-}01$} &
\inrevise{$3.26\text{e-}05$} & \inrevise{$8.53\text{e-}01$} &
\inrevise{$1.75\text{e-}04$} & \inrevise{$7.11\text{e-}01$} \\ \hline
\texttt{phishing} &
\inrevise{$1.15\text{e-}04$} & \inrevise{$8.27\text{e-}01$} &
\inrevise{$8.79\text{e-}04$} & \inrevise{$5.59\text{e-}02$} &
\inrevise{$3.69\text{e-}05$} & \inrevise{$8.31\text{e-}01$} &
\inrevise{$2.34\text{e-}04$} & \inrevise{$5.59\text{e-}02$} \\ \hline
\texttt{sonar} &
\inrevise{$7.01\text{e-}05$} & \inrevise{$8.25\text{e-}01$} &
\inrevise{$6.20\text{e-}04$} & \inrevise{$8.91\text{e-}02$} &
\inrevise{$4.99\text{e-}05$} & \inrevise{$8.29\text{e-}01$} &
\inrevise{$1.28\text{e-}03$} & \inrevise{$8.83 \text{e-}02$} \\ \hline
\end{tabular}
\label{table:libsvm_result}
\end{table}

From Table~\ref{table:libsvm_result}, we observe that, under the same CPU time budget, Algorithm~\ref{alg.main} is able to consistently identify near-feasible and near-stationary iterates for all test instances across different datasets and various batch-size combinations for the objective and constraint functions. In constract, under the same experimental settings, TStoM~\cite{CuiWangXiao24} tends to produce iterates with comparatively larger infeasibility errors, which limits the meaningful assessment of stationarity. For example, when $(b_1,b_2) = (16,16)$, Table~\ref{table:libsvm_result} shows that Algorithm~\ref{alg.main} achieved best iterates with average infeasibility and stationarity errors on the order of \texttt{1e-04} and \texttt{1e-03}, respectively, while TStoM~\cite{CuiWangXiao24} only attained an average infeasibility error on the order of \texttt{1e-01}. Additionally, Table~\ref{table:libsvm_result} further indicates that increasing batch sizes generally improves the performance of Algorithm~\ref{alg.main}. For TStoM~\cite{CuiWangXiao24}, however, the effect of batch size is less pronounced, as the best iterates obtained across different batch-size settings exhibit similar infeasibility levels. This behavior aligns with the algorithmic design of TStoM~\cite{CuiWangXiao24}, which adopts a momentum-based framework with fixed parameter choices during execution. As a result, its numerical performance can be sensitive to the initial parameter selection and may benefit from additional, problem-dependent tuning. By contrast, our Algorithm~\ref{alg.main} incorporates adaptive updates of key parameters (such as merit parameters and ratio parameters) throughout the iterations, which appears to enhance its robustness and effectiveness across a diverse collection of test problems and batch-size configurations.
\end{bulkrevise}

\section{Conclusion}\label{sec.conclusion}
In this paper, we discuss the design, analysis, and implementation of a stochastic SQP algorithm for solving expectation-equality-constrained stochastic optimization problems. Our algorithm is objective-function-free that its iterative update only relies on the estimates of objective gradient, constraint function, and constraint jacobian information. Meanwhile, we consider the ``fully-stochastic'' regime that only relatively loose conditions on the quality of stochastic estimates need to be satisfied at each iteration. Our proposed algorithm does not require a heavy-tuning process, as it utilizes adaptive parameter and step size update strategies. Under common assumptions, we have shown that our algorithm achieves theoretical convergence guarantees and both iteration and sample complexity results in expectation. 
\inrevise{These theoretical results remain competitive with the performance of corresponding state-of-the-art optimization algorithms for the same problem setting.} 
The results of numerical experiments have shown that our proposed stochastic SQP algorithm is more efficient and more reliable than a stochastic subgradient method and a stochastic momentum-based algorithm on classic constrained optimization test problems.

\subsection*{Acknowledgements}
The authors would like to thank Dr. Xiao Wang for providing consultation about the implementation of~\cite[Algorithm~2.1]{CuiWangXiao24}.

\subsection*{Competing Interests}
All authors certify that they have no affiliations with or involvement in any organization or entity with any financial interest or non-financial interest in the subject matter or materials discussed in this manuscript.

\printbibliography{}

\appendix

\begin{bulkrevise}
\section{Justification of the Support Assumption in Assumption~\ref{ass.estimate_asymptotic}}\label{sec:apx-supp-asm}

In this section, we show that if the sampling noise \(\Delta\) for a random vector has zero mean and~\emph{symmetric support}, then the estimation error of the empirical mean estimator satisfies the support assumption in Assumption~\ref{ass.estimate_asymptotic}.

\blemma
Let \(\Delta\) be a zero mean random vector with distribution \(\Prob\) and symmetric support, that is, \(z \in \supp{\Delta}\) implies that \( - z \in \supp{\Delta}\). Additionally, let \(N \in \naturals{}_{\geq 2}\) be an even number, \(\set{\Delta_i}_{i=1}^N\) be \(N\)~\iid{} samples of \(\Delta\), and \(\bar{\Delta} := \frac{1}{N} \sum_{i=1}^N \Delta_i\) be the empirical mean estimator of \(\Delta\).
Then,
\begin{align*}
0 \in \supp{\bar{\Delta}}.
\end{align*}
\elemma

\begin{proof}
We denote by \(\Prob^N\) the joint distribution of \(\set{\Delta_i}_{i=1}^N\)  and by \(\ball{z}\) the open ball centered around \(z\) with radius \(r > 0\). Then, it is sufficient to show that for all \(r > 0\),
\(\Prob^N \left[ \norm{\bar{\Delta}} < r  \right] > 0\).
To this end, we first pick \(z_i \in \supp{\Delta}\) for all \(i \in [N / 2]\), and then set \(z_i \gets -z_{i - N / 2}\) for \(i \in [N] \backslash [N/2]\). Then, it holds that
\begin{align*}
\Prob^N \left[ \norm[\bigg] {\frac{1}{N}\sum_{i=1}^N \Delta_i} < r \right]
& = \Prob^N \left[ \sum_{i=1}^N \Delta_i \in \ball[rN]{0} \right]
\geq \Prob^N \left[ \Delta_i \in \ball[r/2]{z_i}, \forall i \in [N] \right] \\
& = \prod_{i = 1}^N \Prob \left[ \Delta \in \ball[r/2]{z_i} \right]
> 0,
\end{align*}
where the first inequality uses the Minkowski sum, \(\sum_{i=1}^N \ball[r/2]{z_i} = \ball[rN/2]{\sum_{i=1}^N z_i} = \ball[rN/2]{0}\) being a subset of \(\ball[rN]{0}\), the second equality follows from the independence among \(\set{\Delta_i}_{i=1}^N\), and the last strict inequality is due to \(\Delta\) having symmetric support and \(z_i \in \supp{\Delta}\) for all \(i \in [N]\).
\end{proof}
\end{bulkrevise}

\section{Proof of Lemma~\ref{lem:inv-mat-err}}\label{sec:apx-inv-mat-err}

\begin{proof}[Proof of Lemma~\ref{lem:inv-mat-err}]
To simplify notations, we denote $A_k := \bbmatrix H_k & \nabla c(X_k) \\ \nabla c(X_k)^T & 0 \ebmatrix$ and $\bar{A}_k := \bbmatrix H_k & \bar{J}_k^T \\ \bar{J}_k & 0 \ebmatrix$ for each iteration~$k$. By triangle inequality and Lemma~\ref{lem.singular_values}, for any $k\in\NN$,
\begin{align*}
\norm{\bar{A}_k^{-1} - A_k^{-1}} \leq \|\bar{A}_k^{-1}\| + \|A_k^{-1}\| \leq \frac{2}{q_{\min}},
\end{align*}
which proves the first part of the statement. Moreover, when $\|\bar{A}_k - A_k\| < \frac{q_{\min}}{3}$, using the Cauchy-Schwarz inequality and Lemma~\ref{lem.singular_values}, we know
\bequation\label{eq.key_inverse_diff_1}
\|A_k^{-1}(\bar{A}_k - A_k)\| \leq \|A_k^{-1}\|\|\bar{A}_k - A_k\| \leq \frac{1}{q_{\min}}\cdot \frac{q_{\min}}{3} = \frac{1}{3}.
\eequation
Furthermore, for any $k\in\NN$,
\bequationn
\baligned
\bar{A}_k^{-1} &= (A_k + (\bar{A}_k - A_k))^{-1} = (A_k(I + A_k^{-1}(\bar{A}_k - A_k)))^{-1} \\
&= (I + A_k^{-1}(\bar{A}_k - A_k))^{-1}A_k^{-1} = \left(I + \sum_{j=1}^{\infty} (-1)^j(A_k^{-1}(\bar{A}_k - A_k))^j\right)A_k^{-1},
\ealigned
\eequationn
where the last equality follows \cite[Proposition 3.3.1]{AllaKabe08}. Meanwhile, by triangle inequality and Lemma~\ref{lem.singular_values}, we have that
\bequationn
\baligned
\|\bar{A}_k^{-1} - A_k^{-1}\| &= \left\|\left(\sum_{j=1}^{\infty} (-1)^j(A_k^{-1}(\bar{A}_k - A_k))^j\right)A_k^{-1}\right\| \leq \left\|\sum_{j=1}^{\infty} (-1)^j(A_k^{-1}(\bar{A}_k - A_k))^j\right\|\cdot\|A_k^{-1}\| \\
&\leq \frac{1}{q_{\min}}\cdot \sum_{j=1}^{\infty}\left\|A_k^{-1}(\bar{A}_k - A_k)\right\|^j
= \frac{1}{q_{\min}}\cdot \frac{\|A_k^{-1}(\bar{A}_k - A_k)\|}{1 - \|A_k^{-1}(\bar{A}_k - A_k)\|}
\leq \frac{3\|A_k^{-1}(\bar{A}_k - A_k)\|}{2q_{\min}} \\
&\leq \frac{3\|A_k^{-1}\|\|\bar{A}_k - A_k\|}{2q_{\min}} \leq \frac{3}{2q_{\min}^2}\cdot \|\bar{A}_k - A_k\|,
\ealigned
\eequationn
where the second equality and the third inequality are both from \eqref{eq.key_inverse_diff_1}.
\end{proof}

\section{Proof of Lemma~\ref{lem.solution_bias}}\label{sec:apx-sol-est}

\bproof[Proof of Lemma~\ref{lem.solution_bias}.]  Similar to the proof of Lemma~\ref{lem:inv-mat-err}, we first define matrices $A_k := \bbmatrix H_k & \nabla c(X_k) \\ \nabla c(X_k)^T & 0 \ebmatrix$ and $\bar{A}_k := \bbmatrix H_k & \bar{J}_k^T \\ \bar{J}_k & 0 \ebmatrix$ for all iterations $k$ to simplify notations. By \eqref{eq.linear_system} and \eqref{eq.linear_system_true}, it holds that for any $k\in\NN{}$,
\bequationn
\bbmatrix \bar{D}_k \\ \bar{Y}_k \ebmatrix - \bbmatrix D_k \\ Y_k \ebmatrix = -\bar{A}_k^{-1}\bbmatrix \bar{G}_k \\ \bar{C}_k \ebmatrix + A_k^{-1}\bbmatrix \nabla f(X_k) \\ c(X_k) \ebmatrix = (A_k^{-1} - \bar{A}_k^{-1})\bbmatrix \bar{G}_k \\ \bar{C}_k \ebmatrix + A_k^{-1}\bbmatrix \nabla f(X_k) - \bar{G}_k \\ c(X_k) - \bar{C}_k \ebmatrix.
\eequationn
Next, for every iteration $k\in\NN{}$, we define the following events depending on whether $\bar{A}_k$ and $A_k$ are close enough, i.e., let $E_k$ be the event that $\|\bar{A}_k - A_k\| < \frac{q_{\min}}{3}$ and let $E_k^c$ be the event that $\|\bar{A}_k - A_k\| \geq \frac{q_{\min}}{3}$. Furthermore, using conditional Markov's inequality, we have that
\begin{equation}\label{eq.probability_Ekc}
   \begin{aligned}
    \PP_k[E_k^c] &= \PP_k\left[\|\bar{A}_k - A_k\| \geq \frac{q_{\min}}{3}\right]\leq \frac{9\EE_k\left[\|\bar{A}_k - A_k\|^2\right]}{q_{\min}^2}\leq \frac{9\EE_k\left[\|\bar{A}_k - A_k\|_F^2\right]}{q_{\min}^2}\\
    &= \frac{18\EE_k\left[\|\bar{J}_k - \nabla c(X_k)^T\|_F^2\right]}{q_{\min}^2} \leq \frac{18\rho_k^j}{q_{\min}^2},
\end{aligned}
\end{equation}
where the last inequality follows \eqref{eq.variance}. Consequently, by Lemma~\ref{lem:inv-mat-err} and triangle inequality, we further have
\bequation\label{eq.inverse_diff}
\baligned
\EE_k\left[\|\bar{A}_k^{-1} - A_k^{-1}\|\right] &= \EE_k\left[\|\bar{A}_k^{-1} - A_k^{-1}\| | E_k\right]\cdot\PP_k[E_k] + \EE_k\left[\|\bar{A}_k^{-1} - A_k^{-1}\| | E_k^c\right]\cdot\PP_k[E_k^c] \\
&\leq \frac{3}{2q_{\min}^2}\cdot\EE_k\left[\|\bar{A}_k - A_k\| | E_k\right]\cdot\PP_k[E_k] + \EE_k\left[\|\bar{A}_k^{-1}\| + \|A_k^{-1}\| | E_k^c\right]\cdot\PP_k[E_k^c] \\
&\leq \frac{3}{2q_{\min}^2}\cdot \EE_k\left[\|\bar{A}_k - A_k\|\right] + \frac{2}{q_{\min}}\cdot\PP_k[E_k^c] \leq \frac{3}{2q_{\min}^2}\cdot\sqrt{\EE_k\left[\|\bar{A}_k - A_k\|_{F}^2\right]} + \frac{36\rho_k^j}{q_{\min}^3} \\
&= \frac{3}{\sqrt{2}q_{\min}^2}\cdot\sqrt{\EE_k\left[\|\bar{J}_k - \nabla c(X_k)^T\|_{F}^2\right]} + \frac{36\rho_k^j}{q_{\min}^3} \leq \frac{3\sqrt{\rho_k^j}}{\sqrt{2}q_{\min}^2} + \frac{36\rho_k^j}{q_{\min}^3},
\ealigned
\eequation
where the second inequality follows Lemma~\ref{lem.singular_values}, the third inequality is using \eqref{eq.probability_Ekc}, the second equality comes from the structure of $\bar{A}_k$ and $A_k$, and the last inequality is from Assumption~\ref{ass.estimate_asymptotic}. Now we are ready to show the main results in the statement.

First, for any iteration $k\in\NN{}$, it follows \eqref{eq.linear_system}, \eqref{eq.linear_system_true}, Assumption~\ref{ass.estimate_asymptotic}, and the Cauchy-Schwarz inequality that
\bequationn
\baligned
\left\|\EE_k\left[\bar{D}_k - D_k\right]\right\| &\leq \left\|\EE_k\left[\bbmatrix \bar{D}_k \\ \bar{Y}_k \ebmatrix - \bbmatrix D_k \\ Y_k \ebmatrix\right]\right\| = \left\|\EE_k\left[\bar{A}_k^{-1}\bbmatrix \bar{G}_k \\ \bar{C}_k \ebmatrix - A_k^{-1}\bbmatrix \nabla f(X_k) \\ c(X_k) \ebmatrix\right]\right\| \\
&= \left\|\EE_k\left[(\bar{A}_k^{-1} - A_k^{-1})\bbmatrix \bar{G}_k \\ \bar{C}_k \ebmatrix\right]\right\| = \left\|\EE_k\left[\bar{A}_k^{-1} - A_k^{-1}\right] \cdot \EE_k\bbmatrix\bbmatrix \bar{G}_k \\ \bar{C}_k \ebmatrix\ebmatrix\right\| \\
&\leq \left\|\EE_k\left[\bar{A}_k^{-1} - A_k^{-1}\right]\right\| \cdot \left\|\EE_k\bbmatrix\bbmatrix \bar{G}_k \\ \bar{C}_k \ebmatrix\ebmatrix\right\| = \left\|\EE_k\left[\bar{A}_k^{-1} - A_k^{-1}\right]\right\| \cdot \left\|\bbmatrix \nabla f(X_k) \\ c(X_k) \ebmatrix\right\| \\
&\leq \EE_k\left[\left\|\bar{A}_k^{-1} - A_k^{-1}\right\|\right]\cdot(\kappa_{\nabla f} + \kappa_c) \leq \frac{\kappa_{\nabla f} + \kappa_c}{q_{\min}^3}\cdot \left(\frac{3}{\sqrt{2}}q_{\min}\sqrt{\rho_k^j} + 36\rho_k^j\right),
\ealigned
\eequationn
where the third equality follows from the assumed independence between \(\bar{J}_k\) and \((\bar{G}_k, \bar{C}_k)\) in Assumption~\ref{ass.estimate_asymptotic}, the second last inequality is from \eqref{eq.basic_condition}, and the last inequality follows \eqref{eq.inverse_diff}. Because $\rho_k^j \leq \sqrt{\rhomax} \cdot \sqrt{\rho_k^j}$ for all $k\in\NN{}$ (see Assumption~\ref{ass.estimate_asymptotic}), by setting the fixed constant $\omega_1 := \frac{\kappa_{\nabla f} + \kappa_c}{q_{\min}^3}\cdot \left(\frac{3}{\sqrt{2}}q_{\min} + 36\sqrt{\rhomax}\right)$, we may conclude the first part of the statement.

Meanwhile, for any iteration $k\in\NN{}$, using \eqref{eq.linear_system}, \eqref{eq.linear_system_true}, Assumption~\ref{ass.estimate_asymptotic}, and the Cauchy-Schwarz inequality, we have
\begin{align*}
\EE_k\left[\|\bar{D}_k - D_k\|\right] &\leq \EE_k\left[\left\|\bbmatrix \bar{D}_k \\ \bar{Y}_k \ebmatrix - \bbmatrix D_k \\ Y_k \ebmatrix\right\|\right] = \EE_k\left[\left\|\bar{A}_k^{-1}\bbmatrix \bar{G}_k \\ \bar{C}_k \ebmatrix - A_k^{-1}\bbmatrix \nabla f(X_k) \\ c(X_k) \ebmatrix\right\|\right] \\
&\leq \EE_k\left[\left\|\bar{A}_k^{-1}\bbmatrix \bar{G}_k - \nabla f(X_k) \\ \bar{C}_k - c(X_k) \ebmatrix\right\|\right] + \EE_k\left[\left\|(\bar{A}_k^{-1} - A_k^{-1})\bbmatrix \nabla f(X_k) \\ c(X_k) \ebmatrix\right\|\right] \\
&\leq \EE_k\left[\left\|\bar{A}_k^{-1}\right\|\cdot\left\|\bbmatrix \bar{G}_k - \nabla f(X_k) \\ \bar{C}_k - c(X_k) \ebmatrix\right\|\right] + \EE_k\left[\left\|\bar{A}_k^{-1} - A_k^{-1}\right\|\right]\cdot \left\|\bbmatrix \nabla f(X_k) \\ c(X_k) \ebmatrix\right\| \\
&\leq \frac{1}{q_{\min}}\cdot \EE_k\left[\left\|\bbmatrix \bar{G}_k - \nabla f(X_k) \\ \bar{C}_k - c(X_k) \ebmatrix\right\|\right] + (\kappa_{\nabla f} + \kappa_c)\cdot\EE_k\left[\left\|\bar{A}_k^{-1} - A_k^{-1}\right\|\right] \\
&\leq \frac{1}{q_{\min}}\cdot \sqrt{\EE_k\left[\left\|\bbmatrix \bar{G}_k - \nabla f(X_k) \\ \bar{C}_k - c(X_k) \ebmatrix\right\|^2\right]} + (\kappa_{\nabla f} + \kappa_c)\cdot\EE_k\left[\left\|\bar{A}_k^{-1} - A_k^{-1}\right\|\right] \\
&\leq \frac{\sqrt{\rho_k^g + \rho_k^c}}{q_{\min}} + \frac{\kappa_{\nabla f} + \kappa_c}{q_{\min}^3}\cdot\left(\frac{3}{\sqrt{2}}q_{\min}\sqrt{\rho_k^j} + 36\rho_k^j\right) \leq \frac{\sqrt{\rho_k^g + \rho_k^c}}{q_{\min}} + \omega_1\cdot \sqrt{\rho_k^j},
\end{align*}
where the fourth inequality is from Lemma~\ref{lem.singular_values} and \eqref{eq.basic_condition}, the second last inequality follows Assumption~\ref{ass.estimate_asymptotic} and \eqref{eq.inverse_diff}, and the last inequality relies on the definition of $\omega_1$, which is defined in the first part of the proof, and $\rho_k^j \leq \sqrt{\rhomax} \cdot \sqrt{\rho_k^j}$ at any iteration $k\in\NN{}$ (see Assumption~\ref{ass.estimate_asymptotic}). Therefore, we conclude the statement.
\eproof

\section{Proof of Lemma~\ref{lem.solution_variance}}\label{sec:apx-pf-sol-var}

\bproof[Proof of Lemma~\ref{lem.solution_variance}.]
Similar to the proof of Lemma~\ref{lem:inv-mat-err}, we first define matrices $A_k := \bbmatrix H_k & \nabla c(X_k) \\ \nabla c(X_k)^T & 0 \ebmatrix$ and $\bar{A}_k := \bbmatrix H_k & \bar{J}_k^T \\ \bar{J}_k & 0 \ebmatrix$ for all iterations $k$ to simplify notations. By \eqref{eq.linear_system} and \eqref{eq.linear_system_true}, it holds that for any $k\in\NN{}$,
\begin{align}
\EE_k\left[\left\|\bar{D}_k - D_k\right\|^2\right] &\leq \EE_k\left[\left\|\bbmatrix \bar{D}_k \\ \bar{Y}_k \ebmatrix - \bbmatrix D_k \\ Y_k \ebmatrix \right\|^2\right] = \EE_k\left[\left\|\bar{A}_k^{-1}\bbmatrix \bar{G}_k \\ \bar{C}_k \ebmatrix - A_k^{-1}\bbmatrix \nabla f(X_k) \\ c(X_k) \ebmatrix\right\|^2\right] \nonumber{} \\
&= \EE_k\left[\left\|\bar{A}_k^{-1}\bbmatrix \bar{G}_k - \nabla f(X_k) \\ \bar{C}_k - c(X_k) \ebmatrix + (\bar{A}_k^{-1} - A_k^{-1})\bbmatrix \nabla f(X_k) \\ c(X_k) \ebmatrix\right\|^2\right] \nonumber{} \\
&= \EE_k\left[\left\|\bar{A}_k^{-1}\bbmatrix \bar{G}_k - \nabla f(X_k) \\ \bar{C}_k - c(X_k) \ebmatrix\right\|^2\right] + \EE_k\left[\left\|(\bar{A}_k^{-1} - A_k^{-1})\bbmatrix \nabla f(X_k) \\ c(X_k) \ebmatrix\right\|^2\right] \label{eq.sol_variance} \\
&\quad  \quad + \EE_k\left[2\bbmatrix \nabla f(X_k) \\ c(X_k) \ebmatrix^T(\bar{A}_k^{-1} - A_k^{-1})\bar{A}_k^{-1}\bbmatrix \bar{G}_k - \nabla f(X_k) \\ \bar{C}_k - c(X_k) \ebmatrix\right]. \nonumber{}
\end{align}
Then we provide seperate upper bounds for the three terms on the right-hand side of \eqref{eq.sol_variance}.

First, by the Cauchy-Schwarz inequality, Lemma~\ref{lem.singular_values} and Assumption~\ref{ass.estimate_asymptotic}, we have
\bequation\label{eq.variance_ub_1}
\baligned
\EE_k\left[\left\|\bar{A}_k^{-1}\bbmatrix \bar{G}_k - \nabla f(X_k) \\ \bar{C}_k - c(X_k) \ebmatrix\right\|^2\right] &\leq \EE_k\left[\left\|\bar{A}_k^{-1}\right\|^2\cdot\left\|\bbmatrix \bar{G}_k - \nabla f(X_k) \\ \bar{C}_k - c(X_k) \ebmatrix\right\|^2\right] \\
&\leq \frac{1}{q_{\min}^2}\cdot\EE_k\left[\left\|\bbmatrix \bar{G}_k - \nabla f(X_k) \\ \bar{C}_k - c(X_k) \ebmatrix\right\|^2\right] \leq \frac{\rho_k^g + \rho_k^c}{q_{\min}^2}.
\ealigned
\eequation
Before providing an upper bound for the second term, we try to bound $\EE_k\left[\|\bar{A}_k^{-1} - A_k^{-1}\|^2\right]$ as follows. Using the same definition of events $E_k$ and $E_k^c$ from the proof of Lemma~\ref{lem.solution_bias} and following the same logic as \eqref{eq.inverse_diff}, by Lemma~\ref{lem:inv-mat-err} and the triangle inequality, we have
\begin{align}
\EE_k\left[\|\bar{A}_k^{-1} - A_k^{-1}\|^2\right] &= \EE_k\left[\|\bar{A}_k^{-1} - A_k^{-1}\|^2 | E_k\right]\cdot\PP_k[E_k] + \EE_k\left[\|\bar{A}_k^{-1} - A_k^{-1}\|^2 | E_k^c\right]\cdot\PP_k[E_k^c] \nonumber{} \\
&\leq \frac{9}{4q_{\min}^4} \cdot \EE_k\left[\|\bar{A}_k - A_k\|^2 | E_k\right]\cdot\PP_k[E_k] + \EE_k\left[(\|\bar{A}_k^{-1}\| + \|A_k^{-1}\|)^2 | E_k^c\right]\cdot\PP_k[E_k^c] \nonumber{} \\
&\leq \frac{9}{4q_{\min}^4}\cdot \EE_k\left[\|\bar{A}_k - A_k\|^2\right] + \frac{4}{q_{\min}^2}\cdot\PP_k[E_k^c] \nonumber{} \\
&\leq \frac{9}{4q_{\min}^4}\cdot\EE_k\left[\|\bar{A}_k - A_k\|_F^2\right] + \frac{72\rho_k^j}{q_{\min}^4} \label{eq.inverse_diff_variance} \\
&= \frac{9}{2q_{\min}^4}\cdot\EE_k\left[\|\bar{J}_k - \nabla c(X_k)^T\|_F^2\right] + \frac{72\rho_k^j}{q_{\min}^4} \nonumber{} \\
&\leq \frac{9\rho_k^j}{2q_{\min}^4} + \frac{72\rho_k^j}{q_{\min}^4} = \frac{153\rho_k^j}{2q_{\min}^4}, \nonumber{}
\end{align}
where the second inequality follows Lemma~\ref{lem.singular_values}, the third inequality is using \eqref{eq.probability_Ekc}, the second equality comes from the structure of $\bar{A}_k$ and $A_k$, and the last inequality is from Assumption~\ref{ass.estimate_asymptotic}. Now from the Cauchy-Schwarz inequality, \eqref{eq.basic_condition} and \eqref{eq.inverse_diff_variance}, we can bound the second term on the right hand side of \eqref{eq.sol_variance} as
\bequation\label{eq.variance_ub_2}
\EE_k\left[\left\|(\bar{A}_k^{-1} - A_k^{-1})\bbmatrix \nabla f(X_k) \\ c(X_k) \ebmatrix\right\|^2\right] \leq \EE_k\left[\left\|\bar{A}_k^{-1} - A_k^{-1}\right\|^2\right]\cdot \left\|\bbmatrix \nabla f(X_k) \\ c(X_k) \ebmatrix\right\|^2 \leq \frac{153\rho_k^j}{2q_{\min}^4}\cdot (\kappa_{\nabla f} + \kappa_c)^2.
\eequation
Lastly, to bound the last term on the right hand side of \eqref{eq.sol_variance}, it follows the Cauchy-Schwarz inequality, \eqref{eq.basic_condition}, Lemma~\ref{lem.singular_values}, Assumption~\ref{ass.estimate_asymptotic} and \eqref{eq.inverse_diff} that
\begin{align}
\EE_k&\left[2\bbmatrix \nabla f(X_k) \\ c(X_k) \ebmatrix^T(\bar{A}_k^{-1} - A_k^{-1})\bar{A}_k^{-1}\bbmatrix \bar{G}_k - \nabla f(X_k) \\ \bar{C}_k - c(X_k) \ebmatrix\right] \nonumber{} \\
&\quad \leq 2\left\|\bbmatrix \nabla f(X_k) \\ c(X_k) \ebmatrix\right\| \cdot \EE_k\left[\left\|\bar{A}_k^{-1} - A_k^{-1}\right\|\cdot\|\bar{A}_k^{-1}\|\right]\cdot\EE_k\left[\left\|\bbmatrix \bar{G}_k - \nabla f(X_k) \\ \bar{C}_k - c(X_k) \ebmatrix\right\|\right] \nonumber{} \\
&\quad \leq 2\left(\kappa_{\nabla f} + \kappa_c\right) \cdot \frac{1}{q_{\min}}\cdot \EE_k\left[\left\|\bar{A}_k^{-1} - A_k^{-1}\right\|\right]\cdot \sqrt{\EE_k\left[\left\|\bbmatrix \bar{G}_k - \nabla f(X_k) \\ \bar{C}_k - c(X_k) \ebmatrix\right\|^2\right]} \label{eq.variance_ub_3} \\
&\quad \leq \frac{2\left(\kappa_{\nabla f} + \kappa_c\right)}{q_{\min}}\cdot\left(\frac{3\sqrt{\rho_k^j}}{\sqrt{2}q_{\min}^2} + \frac{36\rho_k^j}{q_{\min}^3}\right)\cdot \sqrt{\rho_k^g + \rho_k^c} \nonumber{} \\
&\quad \leq \frac{6\left(\kappa_{\nabla f} + \kappa_c\right)}{q_{\min}^4}\cdot\left(\frac{q_{\min}}{\sqrt{2}} + 12\sqrt{\rhomax}\right)\cdot \sqrt{\rho_k^j(\rho_k^g + \rho_k^c)}, \nonumber{}
\end{align}
where the last inequality follows Assumption~\ref{ass.estimate_asymptotic} that $\rho_k^j \leq \sqrt{\rhomax} \cdot \sqrt{\rho_k^j}$ for all $k\in\NN{}$. Combining \eqref{eq.variance_ub_1}, \eqref{eq.variance_ub_2} and \eqref{eq.variance_ub_3}, and choosing sufficiently large values for constants $\omega_2$ and $\omega_3$,~\ie{},
\begin{align*}
\omega_2 := \frac{153}{2q_{\min}^4}\cdot (\kappa_{\nabla f} + \kappa_c)^2,
\quad
\omega_3 := \frac{6\left(\kappa_{\nabla f} + \kappa_c\right)}{q_{\min}^4}\cdot\left(\frac{q_{\min}}{\sqrt{2}} + 12\sqrt{\rhomax}\right),
\end{align*}
we conclude the first part of the statement.

Now we are ready to prove the second part of the statement. For any iteration $k\in\NN{}$, using the triangle inequality, the fact of $\EE_k[\bar{G}_k] = \nabla f(X_k)$ (see Assumption~\ref{ass.estimate_asymptotic}) and the Cauchy-Schwarz inequality, we have
\bequationn
\baligned
\left|\EE_k\left[\bar{G}_k^T\bar{D}_k - \nabla f(X_k)^TD_k\right]\right| &\leq \left|\EE_k\left[(\bar{G}_k - \nabla f(X_k))^T(\bar{D}_k - D_k)\right]\right| + \left|\EE_k\left[(\bar{G}_k - \nabla f(X_k))^TD_k\right]\right| \\
&\quad\quad + \left|\EE_k\left[\nabla f(X_k)^T(\bar{D}_k - D_k)\right]\right| \\
&\leq \frac{1}{2}\EE_k\left[\|\bar{G}_k - \nabla f(X_k)\|^2\right] + \frac{1}{2}\EE_k\left[\|\bar{D}_k - D_k\|^2\right] + \|\nabla f(X_k)\|\cdot \left\|\EE_k\left[\bar{D}_k - D_k\right]\right\| \\
&\leq \frac{\rho_k^g}{2} + \frac{\rho_k^g + \rho_k^c}{2q_{\min}^2} + \frac{\omega_2}{2}\cdot\rho_k^j + \frac{\omega_3}{2}\cdot\sqrt{\rho_k^j(\rho_k^g+\rho_k^c)} + \kappa_{\nabla f}\omega_1\cdot\sqrt{\rho_k^j} \\
&\leq \frac{\rho_k^g}{2} + \frac{\rho_k^g + \rho_k^c}{2q_{\min}^2} + \omega_4\cdot \sqrt{\rho_k^j} + \frac{\omega_3}{2}\cdot\sqrt{\rho_k^j(\rho_k^g+\rho_k^c)},
\ealigned
\eequationn
where the second last inequality comes from \eqref{eq.basic_condition}, Assumption~\ref{ass.estimate_asymptotic}, Lemma~\ref{lem.solution_bias}, and the first part of the statement in this lemma, and the last inequality follows the fact of $\rho_k^j \leq \sqrt{\rhomax} \cdot \sqrt{\rho_k^j}$ (see Assumption~\ref{ass.estimate_asymptotic}) and choosing a suitable value for $\omega_4$, e.g., $\omega_ 4 = \frac{\sqrt{\rhomax}}{2}\omega_2 + \kappa_{\nabla f}\omega_1$. Therefore, the proof is completed.
\eproof

\section{Proof of Lemma~\ref{lem.prob-tau-good-k}}
\label{sec:proof-lem-prob-tau-good-k}

The proof of Lemma~\ref{lem.prob-tau-good-k} relies on the following lemmas. We first
show that the function \(T_d\) is Lipschitz continuous, which helps to bound the difference between \(T_d(Z_k + \bar{\Delta}_k, H_k)\) and \(T_d(Z_k, H_k)\) using \(\norm{\bar{\Delta}_k}\).

\blemma\label{lem.Td-lip-cont}
Given any constant $\Lambda > 0$, \(f(x, Y) := x^{\top} Y^{-1} x\) is Lipschitz continuous on $\mathcal{S}_{\Lambda} := \{(x,Y)\in\RR^{n}\times \mathbb{S}^n: \|x\| \leq \Lambda, Y^TY \succeq \frac{1}{\Lambda}I \}$.
\elemma

\begin{proof}[Proof of Lemma~\ref{lem.Td-lip-cont}]
To conclude the statement, it is enough to show that there exists a constant \(L > 0\) such that for any $\{(x_1,Y_1),(x_2,Y_2)\}\subset \mathcal{S}_{\Lambda}$, it always holds that
\begin{align*}
\abs{f(x_1, Y_1) - f(x_2, Y_2)} \leq L \left( \norm{x_1 - x_2} + \norm{Y_1 - Y_2} \right).
\end{align*}
We first fix \(Y\) terms and compare the function values at different $x$ values. For any $\{(x_1,Y),(x_2,Y)\}\subset\mathcal{S}_{\Lambda}$, by the Cauchy-Schwarz inequality and the definition of $\mathcal{S}_{\Lambda}$, we have
\begin{align*}
\abs{f(x_1, Y) - f(x_2, Y)}
& = \abs{ x_1^{\top} Y^{-1} x_1 - x^{\top}_2 Y^{-1} x_2 }
= \abs{ (x_1 + x_2)^{\top} Y^{-1} (x_1 - x_2)} \\
& \leq \norm{x_1 + x_2} \cdot \norm{Y^{-1}} \cdot \norm{x_1 - x_2} \leq 2\Lambda^{\frac{3}{2}}\cdot\|x_1-x_2\| \leq L_x \norm{x_1 - x_2}
\end{align*}
for some large enough constant \(L_x \geq 2\Lambda^{\frac{3}{2}}\).
Next, we bound the partial derivative of \(f(x, Y)\) with respect to \(Y\) for any fixed \(x\). Specifically, by~\cite[Equation (61)]{petersen2008matrix}, the Cauchy-Schwarz inequality, and a large enough constant \(L_Y \geq \Lambda^3\), we have
\begin{align*}
\norm{\partial_Y f(x, Y)} = \norm{- Y^{-T} x x^{\top} Y^{-T}}
\leq \norm{x x^{\top}} \norm{Y^{-T}}^2 \leq \|x\|^2\|Y^{-1}\|^2 \leq \Lambda^3 \leq L_Y.
\end{align*}
Therefore, by the triangle inequality, it holds that
\begin{align*}
\abs{f(x_1, Y_1) - f(x_2, Y_2)} \leq \abs{f(x_1, Y_1) - f(x_1, Y_2)} + \abs{f(x_1, Y_2) - f(x_2, Y_2)} \leq L_Y\|Y_1 - Y_2\| + L_x\|x_1 - x_2\|,
\end{align*}
and we conclude the proof by taking $L = \max\{L_x,L_Y\} > 0$.
\end{proof}

We then establish a useful probability inequality in Lemma~\ref{lem:prob-lb}, building on the fundamental Markov-type probabilistic inequality provided in Lemma~\ref{lem:chebyshev-at-zero}.
\blemma\label{lem:chebyshev-at-zero}
Suppose that \(X\in\RR\) is a random variable with \(\Expect \left[ X^2 \right] < +\infty\) and \(\Expect \left[ X \right] > 0\). Then,
\begin{align}
\Prob \left[ X > 0 \right] \geq \frac{
\left(\Expect \left[ X \right]\right)^2
}{
\left(\Expect \left[ X \right]\right)^2 + \Var \left[ X \right]
}, \label{eq:prob-lb-0}
\end{align}
where $\Var \left[ X \right]$ represents for the variance of $X$.
\elemma
\bproof
To begin with, notice that \(\Expect \left[ X^2 \right] \geq \left(\Expect \left[X\right]\right)^2 > 0\). Then, by the Cauchy-Schwarz inequality and properties of integration, one has
\begin{align*}
0 < \Expect \left[ X \right]
=
\Expect \left[ X \cdot \ones\left( X > 0 \right) \right] +
\Expect \left[ X \cdot \ones\left( X \leq 0 \right) \right]
\leq
\Expect \left[ X \cdot \ones\left( X > 0 \right) \right]
\leq
\sqrt{\Expect \left[ X^2 \right]\cdot \Prob\left[ X > 0 \right]},
\end{align*}
from which \(\displaystyle \Prob [X > 0] \geq \frac{\left( \Expect \left[ X \right] \right)^{2}}{\Expect \left[ X^2 \right]}\), and~\eqref{eq:prob-lb-0} follows from \(\Var \left[ X \right] = \Expect \left[ X^2 \right] - \left( \Expect \left[ X \right] \right)^2\).
\eproof

\blemma\label{lem:prob-lb}
Suppose that $\Delta$ is a random vector satisfying
\begin{align*}
\Expect \left[ \Delta \right] = 0,
\quad
0 \in \supp{\Delta},
\ \ \text{and} \ \
\Expect \left[ \norm{\Delta}^2 \right] < +\infty.
\end{align*}
Then, for any $\{\epsilon,M\}\subset\RR_{>0}$, it holds for any $r \in (0, \frac{\epsilon}{2M}]$ that
\begin{align}
\Prob \left[ \epsilon - M \norm{\Delta} > 0 \right]
& \geq
\Prob \left[(\epsilon - M \norm{\Delta}) \cdot \ones \left( \norm{\Delta} < r \right) > 0 \right] \nonumber{} \\
& \geq
\frac{
\Expect \left[ \left(\epsilon - M \norm{\Delta}  \right) \cdot \ones \left( \norm{\Delta} < r \right)  \right]
}{
\epsilon - M \Expect \left[ \norm{\Delta} \cdot \ones \left( \norm{\Delta} < r \right) \right] + \left( M^2 / \epsilon\right) \Expect \left[ \norm{\Delta}^2 \right]
}.
\label{eq:prob-lb}
\end{align}
\elemma
\bproof[Proof of Lemma~\ref{lem:prob-lb}]
Let \(X := (\epsilon - M \norm{\Delta}) \cdot \ones \left( \norm{\Delta} < r \right)\). By the choice of \(r\) and \(0 \in \supp{\Delta}\), we have
\begin{align*}
&0 < \epsilon/2 \cdot \Prob \left[ \norm{\Delta} < r \right]
\leq \Expect \left[ (\epsilon - M \norm{\Delta}) \cdot \ones \left( \norm{\Delta} < r \right)\right], \\
\text{and} \quad &\Expect \left[ (\epsilon - M \norm{\Delta})^2 \cdot \ones \left( \norm{\Delta} < r \right) \right]
\leq \epsilon^2 < +\infty.
\end{align*}
Consequently, applying Lemma~\ref{lem:chebyshev-at-zero} to \(X\) yields
\begin{align}
\Prob \left[\epsilon - M \norm{\Delta} > 0\right]
\geq
\Prob \left[ \left(\epsilon - M \norm{\Delta}  \right) \cdot \ones \left( \norm{\Delta} < r \right) > 0 \right]
= \Prob \left[ X > 0 \right]
\geq \frac{\left( \Expect \left[ X \right] \right)^2}{\left( \Expect \left[ X \right] \right)^2 + \Var \left[ X \right]}, \label{eq:prob-lb-raw}
\end{align}
where the first inequality is because \(\norm{\Delta} < r \in (0,  \frac{\epsilon}{2M}]\) implies \(\epsilon - M\norm{\Delta} > \epsilon/2 > 0\). Next, we will derive a lower bound for~\eqref{eq:prob-lb-raw} by constructing an upper bound for \(\Var \left[ X \right]\). Since \(X = (\epsilon - M \norm{\Delta}) \cdot \ones{} \left( \norm{\Delta} < r \right) \), one can expand \(\Var{} \left[ X \right]\) with \(X_1 := \epsilon \cdot \ones{} \left( \norm{\Delta} < r \right)\) and \(X_2 := M \norm{\Delta} \cdot \ones \left( \norm{\Delta} < r \right)\) as
\begin{align*}
\Var{} \left[ X \right]
= \Var{} \left[ X_1 - X_2 \right]
= \Var{} \left[ X_1 \right] + \Var{} \left[ X_2 \right]
- 2 \Cov{} \left[ X_1, X_2 \right],
\end{align*}
where $\Cov{} \left[ X_1, X_2 \right]$ represents the covariance between random variables $X_1$ and $X_2$. Moreover, when $\|\Delta\| < r$ and $r\in (0,  \frac{\epsilon}{2M}]$, we have \(\norm{\Delta} < r \leq \epsilon / (2M) < \epsilon / M\), which further implies \(M \norm{\Delta}^2 / \epsilon < \norm{\Delta}\) and
\begin{align}
\Var{} \left[ X_2 \right]
& = M^2 \left( \Expect \left[ \norm{\Delta}^2 \cdot \ones \left( \norm{\Delta} < r \right) \right]
- \left( \Expect \left[ \norm{\Delta} \cdot \ones \left( \norm{\Delta} < r \right) \right] \right)^2 \right) \nonumber{} \\
& \leq M^2 \left( \Expect \left[ \norm{\Delta}^2 \cdot \ones \left( \norm{\Delta} < r \right) \right]
-  \Expect \left[ \frac{M}{\epsilon} \norm{\Delta}^2 \cdot \ones \left( \norm{\Delta} < r \right) \right]
\cdot \Expect \left[ \norm{\Delta} \cdot \ones \left( \norm{\Delta} < r \right) \right]
 \right) \nonumber{} \\
& = M^2 \Expect \left[ \norm{\Delta}^2 \cdot \ones \left( \norm{\Delta} < r \right) \right]
\left( 1 - \frac{M}{\epsilon} \Expect \left[ \norm{\Delta} \cdot \ones \left( \norm{\Delta} < r \right) \right] \right) \nonumber{} \\
& = \frac{M^2}{\epsilon} \Expect \left[ \norm{\Delta}^2 \cdot \ones \left( \norm{\Delta} < r \right) \right]
\left(
  \epsilon \cdot \Prob \left[ \norm{\Delta} \geq r \right]
+ \Expect \left[ (\epsilon - M \norm{\Delta}) \cdot \ones \left( \norm{\Delta} < r \right) \right]
\right) \nonumber{} \\
& = M^2 \Expect \left[ \norm{\Delta}^2 \cdot \ones \left( \norm{\Delta} < r \right) \right]
\Prob \left[ \norm{\Delta} \geq r \right]
+ \frac{M^2}{\epsilon} \Expect \left[ \norm{\Delta}^2 \cdot \ones \left( \norm{\Delta} < r \right) \right]
\cdot \Expect \left[ X \right]. \label{eq:prob-lb-var-x2}
\end{align}
Expanding \(\Var{} \left[ X_1 \right]\) and \(\Cov \left[ X_1, X_2 \right]\) leads to
\begin{align}
\Var{} \left[X_1 \right]
& = \Expect \left[ \epsilon^2 \cdot \ones \left( \norm{\Delta} < r \right) \right]
- \epsilon^2 \Prob \left[ \norm{\Delta} < r \right]^2
= \epsilon^2 \Prob \left[ \norm{\Delta} < r \right] \Prob \left[ \norm{\Delta} \geq r \right] \label{eq:prob-lb-var-x1} \\
\text{and} \quad\Cov{} \left[ X_1, X_2 \right]
& = \Expect \left[ X_1 \cdot X_2 \right] - \Expect \left[ X_1 \right] \Expect \left[ X_2 \right] \nonumber{} \\
& = \epsilon M \Expect \left[ \norm{\Delta} \cdot \ones \left( \norm{\Delta} < r \right) \right]
- \epsilon \cdot \Prob \left[ \norm{\Delta} < r \right]
\cdot M \Expect \left[ \norm{\Delta} \cdot \ones \left( \norm{\Delta} < r \right) \right] \nonumber{} \\
& = \epsilon M \Prob \left[ \norm{\Delta} \geq r \right] \cdot \Expect \left[ \norm{\Delta} \cdot \ones \left( \norm{\Delta} < r \right) \right]. \label{eq:prob-lb-cov}
\end{align}
It follows by
combining~\eqref{eq:prob-lb-var-x2},~\eqref{eq:prob-lb-var-x1}, and~\eqref{eq:prob-lb-cov} that
\begin{align}
\Var{} \left[ X \right]
& = \Var{} \left[ X_1 \right] + \Var{} \left[ X_2 \right] - 2\Cov \left[ X_1, X_2 \right]
\nonumber{} \\
& \leq \Prob \left[ \norm{\Delta} \geq r \right] \cdot \left(
  \epsilon^2 \cdot \Prob \left[ \norm{\Delta} < r \right]
- 2 \epsilon M \Expect \left[ \norm{\Delta} \cdot \ones \left( \norm{\Delta} < r \right) \right]
+ M^2 \Expect \left[ \norm{\Delta}^2 \cdot \ones \left( \norm{\Delta} < r \right) \right]
 \right) \nonumber{} \\
& \qquad + \frac{M^2}{\epsilon} \Expect \left[ \norm{\Delta}^2 \cdot \ones \left( \norm{\Delta} < r \right) \right]
\cdot \Expect \left[ X \right] \nonumber{} \\
& = \Prob \left[ \norm{\Delta} \geq r \right] \cdot \Expect \left[ (\epsilon - M \norm{\Delta})^2 \cdot \ones \left( \norm{\Delta} < r \right) \right] +
\frac{M^2}{\epsilon} \Expect \left[ \norm{\Delta}^2 \cdot \ones \left( \norm{\Delta} < r \right) \right]
\cdot \Expect \left[ X \right] \nonumber{} \\
& \leq \Prob \left[ \norm{\Delta} \geq r \right] \cdot \epsilon \cdot \Expect \left[ X \right]
+ \frac{M^2}{\epsilon} \Expect \left[ \norm{\Delta}^2 \cdot \ones \left( \norm{\Delta} < r \right) \right] \cdot \Expect \left[ X \right], \label{eq:prob-lb-var-x}
\end{align}
where the last inequality is because \((\epsilon - M \norm{\Delta})^2 \leq \epsilon(\epsilon - M\norm{\Delta})\) when \(\norm{\Delta} < r \in (0,  \frac{\epsilon}{2M}]\). Then, the desired result follows, because
\begin{align*}
&\Prob \left[ \left(\epsilon - M \norm{\Delta}  \right) \cdot \ones \left( \norm{\Delta} < r \right) > 0 \right] = \Prob\left[X > 0\right]
\geq \frac{\left(\Expect \left[ X \right]\right)^2}{\left(\Expect \left[ X \right]\right)^2 + \Var \left[ X \right] } = \frac{\Expect \left[X\right]}{\Expect \left[X\right] + \frac{\Var \left[ X \right]}{\Expect \left[X\right]}} \\
\geq \ & \frac{
\Expect \left[ \left(\epsilon - M \norm{\Delta}  \right) \cdot \ones \left( \norm{\Delta} < r \right)  \right]
}{
\Expect \left[ \left(\epsilon - M \norm{\Delta}  \right) \cdot \ones \left( \norm{\Delta} < r \right)  \right] + \left( \Prob \left[ \norm{\Delta} \geq r \right] \cdot \epsilon + \left( M^2 / \epsilon \right) \Expect \left[ \norm{\Delta}^2 \cdot \ones \left( \norm{\Delta} < r \right) \right] \right)
} \\
\geq \ & \frac{
\Expect \left[ \left(\epsilon - M \norm{\Delta}  \right) \cdot \ones \left( \norm{\Delta} < r \right)  \right]
}{
\epsilon - M \Expect \left[ \norm{\Delta} \cdot \ones \left( \norm{\Delta} < r \right) \right] + \left( M^2 / \epsilon\right) \Expect \left[ \norm{\Delta}^2 \right].
}
\end{align*}
where the second inequality is by \(\Expect \left[ X \right] > 0\) and~\eqref{eq:prob-lb-var-x}.
\eproof

Now we are ready to present the proof of Lemma~\ref{lem.prob-tau-good-k}.

\bproof[Proof of Lemma~\ref{lem.prob-tau-good-k}]
When \(\cT_k^{\trial} \geq \tauzero\),~\eqref{eq:prob-tau-good-main-k} holds trivially, as \(\{\bar{\cT}_k\}\) is monotonically decreasing (see Lemma~\ref{lem.merit_parameter}). When \(\norm{c(X_k)}_1 = 0\), we have \(\cT_k^{\trial} = +\infty \geq \tauzero\), reducing to the previous case. Therefore, it is sufficient to focus on the case where \(\cT^{\trial}_k < \tauzero\) and \(\norm{c(X_k)}_1 > 0\). To begin with, one has
\begin{align}
\Prob_k \left[ \bar{\cT}_k \leq \cT_k^{\trial} \right]
& \geq \Prob_k \left[
    (1 - \epsilon_{\tau}) \bar{\cT}^{\trial}_k < \cT_k^{\trial} \right] \nonumber{} \\
& \geq \Prob_k \left[
    (1 - \epsilon_{\tau}) \bar{\cT}^{\trial}_k < \cT_k^{\trial},
     T_n(\bar{Z}_k) > 0, T_d(\bar{Z}_k,H_k) > 0  \right] \nonumber{} \\
& = \Prob_k \left[
    (1 - \epsilon_{\tau}) \bar{\cT}^{\trial}_k < \cT_k^{\trial}, T_n(\bar{Z}_k)> 0 \right] \label{eq:key_eq_0}
\end{align}
where the first inequality is by Lemma~\ref{lem.merit_parameter}, and the last equality is because \((1 - \epsilon_{\tau}) \bar{\cT}_k^{\trial} < \cT_k^{\trial} < \tauzero\) implies that \(\bar{\cT}_k^{\trial} < +\infty\) and $\cT_k^{\trial} < +\infty$, in which case \(T_d(\bar{Z}_k,H_k) > 0\) and $T_d(Z_k,H_k) > 0$. Consequently, by the definitions of \(\cT_k^{\trial}\) and \(\bar{\cT}_k^{\trial}\), it follows that
\begin{align}
\Prob_k \left[ (1-\epsilon_{\tau})\bar\cT_k^{\trial} < \cT_k^{\trial}, T_n(\bar{Z}_k) > 0 \right]
= \Prob_k \left[\frac{T_d(\bar{Z}_k, H_k)}{T_n(\bar{Z}_k)}
> (1-\epsilon_{\tau}) \cdot \frac{T_d(Z_k,H_k)}{T_n(Z_k)}, T_n(\bar{Z}_k) > 0 \right].
\label{eq:prob-tau-good-step-1}
\end{align}
Moreover, due to \inrevise{Assumption~\ref{ass.prob}, Lemma~\ref{lem.singular_values} and Lemma~\ref{lem.Td-lip-cont}}, $T_n(\cdot)$ and $T_d(\cdot, H_k)$ are Lipschitz continuous for all $k\in\NN{}$. Consequently, there exists a sufficiently large constant \(\nctwo\label{cnst:apx-lip-const} > 0\) that only depends on \(\mathcal{X}\) and \(q_{\text{min}}\) such that for all $k\in\NN$ and $Z\in\RR^{n+m+mn}$, one has
\begin{align*}
\abs{T_d(Z,H_k) - T_d(Z_k, H_k)} \leq \rctwo{cnst:apx-lip-const} \norm{Z - Z_k} \quad
\text{ and }\quad \abs{T_n(Z) - T_n(Z_k)} \leq \rctwo{cnst:apx-lip-const} \norm{Z - Z_k}.
\end{align*}
Then, when \(\cT^{\trial}_k < \tauzero\), it holds that
\begin{align}
\eqref{eq:prob-tau-good-step-1}
= ~ & ~ \Prob_k \left[ T_d(\bar{Z}_k, H_k) T_n(Z_k)
> (1-\epsilon_{\tau}) \cdot T_d(Z_k, H_k) T_n(\bar{Z}_k), T_n(\bar{Z}_k) > 0\right] \nonumber{} \\
\geq ~ & ~ \Prob_k \left[ T_n(Z_k) \left( T_d(Z_k, H_k)
- \rctwo{cnst:apx-lip-const} \norm{\bar{Z}_k - Z_k}  \right)
> (1-\epsilon_{\tau}) \cdot T_d(Z_k, H_k) \left( T_n(Z_k) + \rctwo{cnst:apx-lip-const} \norm{\bar{Z}_k - Z_k} \right),
 T_n(\bar{Z}_k) > 0 \right] \nonumber{} \\
\geq ~ & ~
\Prob_k \left[ \epsilon_{\tau} T_n(Z_k) T_d(Z_k, H_k)
- \rctwo{cnst:apx-lip-const} (T_n(Z_k) + (1 - \epsilon_{\tau}) \cdot T_d(Z_k, H_k))\norm{\bar{Z}_k - Z_k} > 0, T_n(Z_k) > \rctwo{cnst:apx-lip-const} \norm{Z_k - \bar{Z}_k}  \right] \nonumber{} \\
= ~ & ~
\Prob_k \left[ \epsilon_{\tau} T_n(Z_k) - \rctwo{cnst:apx-lip-const} \left(\frac{T_n(Z_k)}{T_d(Z_k, H_k)} + (1-\epsilon_{\tau}) \right)\norm{\bar{Z}_k - Z_k} > 0, \norm{Z_k - \bar{Z}_k} < \frac{T_n(Z_k)}{\rctwo{cnst:apx-lip-const}}
 \right] \nonumber{} \\
\geq ~ & ~
\Prob_k \left[ \epsilon_{\tau} \norm{c(X_k)}_1
- \rctwo{cnst:apx-lip-const}
\left( \frac{\tauzero}{(1-\sigma)} + (1-\epsilon_{\tau}) \right)\norm{\bar{\Delta}_k} > 0,
\norm{\bar{\Delta}_k} < \frac{\epsilon_{\tau} \norm{c(X_k)}_1}{\rctwo{cnst:apx-lip-const}}
 \right],
\label{eq:key_eq_1}
\end{align}
where the first inequality follows from \(T_n(Z_k) = \norm{c(X_k)}_1 > 0\) and \(T_d(Z_k, H_k) > 0\), and the last inequality is because \(\cT^{\trial}_k = (1 - \sigma) T_n(Z_k) / T_d(Z_k, H_k) < \tauzero\). Therefore, by letting \(\epsilon_k := \epsilon_{\tau} \norm{c(X_k)}_1\) and choosing \(\rc{cnst:lem.prob-tau-good-k} := \rctwo{cnst:apx-lip-const} \max\left\{1, \frac{\tauzero}{1 - \sigma} + ( 1 - \epsilon_{\tau} )\right\}\) , one obtains
\begin{equation}\label{eq:key_eq_2}
\begin{aligned}
\eqref{eq:key_eq_1}
& \geq \Prob_k \left[ \epsilon_k - \rc{cnst:lem.prob-tau-good-k} \norm{\bar{\Delta}_k} > 0, \norm{\bar{\Delta}_k} < \frac{\epsilon_k}{2 \rc{cnst:lem.prob-tau-good-k}} \right]
= \Prob_k \left[ \left(\epsilon_k - \rc{cnst:lem.prob-tau-good-k} \norm{\bar{\Delta}_k}  \right)
\cdot \ones \left( \norm{\bar{\Delta}_k} < \frac{\epsilon_k}{2 \rc{cnst:lem.prob-tau-good-k}}  \right) > 0, \right] \\
& \geq
\frac{
\Expect_k \left[ \left(\epsilon_k - \rc{cnst:lem.prob-tau-good-k} \norm{\bar{\Delta}_k}  \right) \cdot \ones \left( \norm{\bar{\Delta}_k} < \frac{\epsilon_k}{2 \rc{cnst:lem.prob-tau-good-k}} \right)  \right]
}{
\epsilon_k - \rc{cnst:lem.prob-tau-good-k} \Expect_k \left[ \norm{\bar{\Delta}_k} \cdot \ones \left( \norm{\bar{\Delta}_k} < \frac{\epsilon_k}{2 \rc{cnst:lem.prob-tau-good-k}} \right) \right] + \left( \rc{cnst:lem.prob-tau-good-k}^2 / \epsilon_k \right) \Expect_k \left[ \norm{\bar{\Delta}_k}^2 \right]
},
\end{aligned}
\end{equation}
where the last inequality applies Lemma~\ref{lem:prob-lb} to the filtered probability space associated to \inrevise{\(\mathcal{F}'_k\)}. Combining \eqref{eq:key_eq_0}--\eqref{eq:key_eq_2}, we complete the proof.
\eproof

\end{document}